\documentclass{uthesis10} 

\usepackage{amsmath, amsthm, amssymb}
 \usepackage[pdftex]{graphicx}
\usepackage{makeidx}
\usepackage[pdftex,bookmarks=true]{hyperref}
\makeindex

\usepackage{amssymb}
\usepackage{verbatim}

\usepackage{amscd}
      \usepackage[nohug,heads=littlevee]{diagrams}
      \diagramstyle[labelstyle=\scriptstyle]
\usepackage{DCpic}

\newcommand{\catfont}[1]{\mathcal{#1}}
\newcommand{\slot}{\underline{\hspace{.4cm}}}
\newcommand{\filtfont}[1]{\catfont{#1}}
\newcommand{\extDim}{\text{dim} \hspace{.1cm}}
\newcommand{\isomorphic}{\simeq}
\newcommand{\intHom}{\text{\underline{Hom}}}
\newcommand{\dual}[1]{#1^{\vee}}
\newcommand{\isoarrow}{\stackrel{\isomorphic}{\longrightarrow}}
\newcommand{\gen}[1]{\langle #1 \rangle}
\newcommand{\res}{ | }
\newcommand{\vlongrightarrow}{\overrightarrow{\qquad}}
\newcommand{\vvlongrightarrow}{\overrightarrow{\qquad\qquad}}
\newcommand{\vvvlongrightarrow}{\overrightarrow{\qquad\qquad\qquad}}
\newcommand{\myand}{\hspace{.05cm} \wedge \hspace{.05cm}}
\newcommand{\myor}{\hspace{.05cm} \vee \hspace{.05cm}}
\newcommand{\mynot}{\neg}
\newcommand{\myimplies}{\implies}
\newcommand{\myiff}{\Longleftrightarrow}

\newcommand{\myeq}{\doteq}

\newcommand{\nin}{\notin}
\newcommand{\nothing}{\emptyset}

\newcommand{\multsymbol}{*}
\newcommand{\invlim}{\varprojlim}
\newcommand{\dirlim}{\varinjlim}
\newcommand{\ideal}{\lhd}
\newcommand{\languageSpace}{\hspace{1cm}}
\newcommand{\abs}[1]{|#1|}

\newcommand{\eqspace}{\hspace{.5cm}}

\newcommand*{\uprod}{\mathchoice{\underset{\scriptscriptstyle\filtfont{U}}{\textstyle\prod}\displaystyle}
{\prod_{\scriptscriptstyle\filtfont{U}} \hspace{-.05cm}\textstyle}
{\scriptscriptstyle\prod_{\scriptscriptstyle\filtfont{U}}
\hspace{-.05cm}\scriptstyle}
{\prod_{\filtfont{U}}\scriptscriptstyle}}

\newcommand*{\resprod}{\mathchoice{\underset{\scriptscriptstyle R}{\textstyle\prod}\displaystyle}
{\prod_{\scriptscriptstyle R} \hspace{-.05cm} \textstyle}
{\scriptscriptstyle\prod_{\scriptscriptstyle R} \hspace{-.05cm}
\scriptstyle} {\prod_{R}\scriptscriptstyle}}

\newcommand*{\hprod}{\mathchoice{\underset{\scriptscriptstyle H}{\textstyle\prod}\displaystyle}
{\prod_{\scriptscriptstyle H} \hspace{-.05cm} \textstyle}
{\scriptscriptstyle\prod_{\scriptscriptstyle H} \hspace{-.05cm}
\scriptstyle} {\prod_{H}\scriptscriptstyle}}

\newcommand*{\hnprod}[1]{\mathchoice{\underset{\scriptscriptstyle H \scriptscriptstyle \leq \scriptscriptstyle #1}{\textstyle\prod}\displaystyle}
{\prod_{\scriptscriptstyle H \scriptscriptstyle \leq
\scriptscriptstyle #1} \hspace{-.05cm} \textstyle}
{\scriptscriptstyle\prod_{\scriptscriptstyle H \scriptscriptstyle
\leq \scriptscriptstyle #1} \hspace{-.05cm} \scriptstyle}
{\prod_{H \leq #1}\scriptscriptstyle}}

\newarrow{Iso} <--->

\theoremstyle{plain}
\newtheorem{thm}{Theorem}[section]
\newtheorem{cor}[thm]{Corollary}
\newtheorem{lem}[thm]{Lemma}
\newtheorem{prop}[thm]{Proposition}
\theoremstyle{definition}
\newtheorem{defn}{Definition}[section]
\theoremstyle{remark}

\numberwithin{equation}{section}

\begin{document}          

\title{Ultraproducts of Tannakian Categories and Generic Representation Theory of Unipotent Algebraic Groups}

\author{Michael Crumley}


\CopyRightPage{yes}  

\MyDocument{Dissertation}  

\MyDegree{Doctor of Philosophy Degree in Mathematics}

\ConferDate{August}{2010}

\MyAdvisor{Dr.~Paul R.~Hewitt}


  \SecondCommitteeMember{Dr.~Charles J.~Odenthal}
   \ThirdCommitteeMember{Dr.~Martin R.~Pettet}
  \FourthCommitteeMember{Dr.~Gerard Thompson}
   \FifthCommitteeMember{Dr.~Steve Smith}

\GradDean{Dr.~Patricia R.~Komuniecki}{Dean}

\maketitle 


\begin{abstractpage}

Let $G$ be an affine group scheme defined over a field $k$, and
denote by $\text{Rep}_k G$ the category of finite dimensional
representations of $G$ over $k$.  The principle of tannakian
duality states that any neutral tannakian category is tensorially
equivalent to $\text{Rep}_k G$ for some affine group scheme $G$
and field $k$, and conversely.

Originally motivated by an attempt to find a first-order
explanation for generic cohomology of algebraic groups, we study
neutral tannakian categories as abstract first-order structures
and, in particular, ultraproducts of them. One of the main
theorems of this dissertation is that certain naturally definable
subcategories of these ultraproducts are themselves neutral
tannakian categories, hence tensorially equivalent to
$\text{Comod}_A$ for some Hopf algebra $A$ over a field $k$. We
are able to give a fairly tidy description of the representing
Hopf algebras of these categories, and explicitly compute them in
several examples.  The work done in this vein constitutes roughly
half of this dissertation.

The second half is much less abstract in nature, as we turn our
attention to working out the representation theories of certain
unipotent algebraic groups, namely the additive group $G_a$ and
the Heisenberg group $H_1$. The results we obtain for these groups
in characteristic zero are not at all new or surprising, but in
positive characteristic they perhaps are.  In both cases we obtain
that, for a given dimension $n$, if $p$ is large enough with
respect to $n$, all $n$-dimensional modules for these groups in
characteristic $p$ are given by commuting products of
representations, with the constituent factors resembling
representations of the same group in characteristic zero.  This
has led us to define the `height-restricted ultraproduct' of the
categories $\text{Rep}_{k_i} G$ for a sequence of fields $k_i$ of
increasing positive characteristic, and the above result can be
summarized by saying that these height-restricted ultraproducts
are tensorially equivalent to $\text{Rep}_{k} G^n$, where $G^n$
denotes a direct product of copies of $G$ and $k$ is a certain
field of characteristic zero. We later use these results to
extrapolate some generic cohomology results for these particular
unipotent groups.

\end{abstractpage}



\begin{dedication}
\vspace{2cm}
\begin{center} {\LARGE \textit{To Sarah} }
\end{center}
\end{dedication}


\begin{acknowledgments}

It is not often in one's life that the opportunity arises to truly
thank the people that have made a real impact upon his existence.
I would therefore like to take this opportunity to express my
sincerest thanks

\begin{description}
\item[]to my advisor Dr.~Paul Hewitt, without whose guidance this
dissertation could not have possibly come into being.  He has at
all times been utterly unselfish with his expertise, insight, and
optimism.  He challenged me when I was wrong, and encouraged me
when I was right.  A more willing and able advisor one could not
ask for.

\item[]to Dr.'s Charles Odenthal, Martin Pettet, Gerard Thompson
and Steve Smith, for their careful reading of this manuscript and
their many thoughtful critiques and suggestions;

\item[]to my father Michael, for his encouragement and support,
even when he had no idea what I was doing with my weekends;

\item[]to my mother Valerie, who instilled in me my earliest love
of learning, and fostered what she called my `insatiable
curiosity';

\item[]to my brother Josh, for his encouragement, and for being
the first of our family to discover the beauty of mathematics;

\item[]to my aunt Ellie, uncle Chad, Nanna, and cousins Andrew,
Jessie, and Lexi, for making me a part of their family;

\item[]to my guitar teacher Kevin Smith, who was the first to show
me that a life of scholarship was a life well spent;

\item[]to John, Jay, Geoff and Andrew, for their encouragement,
support, and friendship;

\item[]to Petko, my roommate and mathematical brother;

\item[]to all of my aunts, uncles, cousins, and dearly departed
grandma;

\item[]to all of the excellent teachers and mentors I've had
throughout the years;

\item[]to my cat Mam and my fish, for reminding me that happiness
has shockingly little to do with money;

\item[]to British Petroleum and the owners of Deepwater Horizon,
for proving to us that intelligence is a thing utterly divorced
from wisdom;

\item[]to major league umpire Jim Joyce, for aspiring to do his
job the best he could even when it was the least popular thing in
the world to do at that moment;

\item[]to the Detroit Tigers, because they are due for a series
win very soon;

\item[]to the makers of Guinness draft beer and Jim Beam bourbon
whiskey, for delaying the completion of this dissertation by at
least six months;

\item[]to Ludwig van Beethoven and Johnny Cash, and they know why;

\item[]and to Sarah, for her support, encouragement,
understanding, faith, strength, patience, and love.

\end{description}
\end{acknowledgments}

\tableofcontents 


\newtheorem{theorem}{Theorem}[section]
\newtheorem{lemma}[theorem]{Lemma}
\newtheorem{proposition}[theorem]{Proposition}
\newtheorem{example}[theorem]{Example}
\newtheorem*{maintheorem}{Main Theorem}
\newtheorem{definition}{Definition}

\CaptionFormat{hang}  

\StartDocumentText 




\chapter{Introduction}

\newcommand{\extdim}{\extDim}

Consider the following two theorems:

\begin{thm} (see Corollary 3.4 of \cite{FP}) Let $G$ be a simple,
simply connected algebraic group defined and split over
$\mathbb{F}_p$, and $\lambda$ a dominant weight.  If $p$ is
sufficiently large with respect to $G$, $\lambda$ and $n$, then
the dimension of $H^n(G(\mathbb{F}_p),S(\lambda))$ is independent
of $p$.
\end{thm}

\begin{thm}
Let $\phi$ be a first-order statement in the language of fields
such that $\phi$ is true for every characteristic zero field. Then
$\phi$ is true for all fields of sufficiently large positive
characteristic.
\end{thm}

The first is a classic generic cohomology theorem; if you can
assume such and such a thing to be large (in this case,
characteristic), cohomology stabilizes.  The second is a textbook
exercise in model theory, an easy consequence of the compactness
theorem for first-order logic.  The analogy between these two
statements has been the broad motivation for the following: is
there a first-order explanation for the phenomenon of generic
cohomology?

Our investigations into this question have, as fate would have it,
led us far astray from our original objective.  The majority of
this dissertation is devoted to the study of \textit{neutral
tannakian categories} as abstract first-order structures (roughly
speaking, the categories which can in some sense be thought of as
$\text{Rep}_k G$ for some affine group scheme $G$ and field $k$),
and in particular, ultraproducts of them. To this end we identify
certain subcategories of these ultraproducts which themselves are
neutral tannakian categories, hence tensorially equivalent to
$\text{Comod}_{A}$ for some Hopf algebra $A$.  We are able to
provide a general formula for $A$, and explicitly compute it in
several examples.

For the remainder we turn away from ultraproducts, and instead to
the study of the concrete representation theories of certain
unipotent algebraic groups, namely the additive group $G_a$ and
the Heisenberg group $H_1$.  For both groups we obtain a certain
`generic representation theory' result: that while the
characteristic $p >0$ and characteristic zero theories of both
can, by and large, be expected to bear little resemblance to one
another, if instead one is content to keep positive characteristic
large with respect to dimension, there is in fact a very strong
correspondence between the two.  These results are later codified
by considering the `height-restricted ultraproduct' of these
groups for increasing characteristic, and from them we are able to
generate some modest, `height-restricted' generic cohomology
results for these groups.

\section{Preliminaries}

For an algebraic group $G$ defined over $\mathbb{Z}$ and a field
$k$, $\text{Rep}_k G$ is the category of finite dimensional
representations of $G$ over $k$.  This category is tensorially
equivalent to $\text{Comod}_{A \otimes k}$, where $A$ is the
representing Hopf algebra of $G$ over $\mathbb{Z}$, and we
generally prefer to think of it as the latter.  If $k_i$ is a
collection of fields indexed by $I$ and $\filtfont{U}$ a
non-principal ultrafilter over $I$, we consider the ultraproduct
of the categories $\text{Comod}_{A_i}$ with respect to
$\filtfont{U}$, with $A_i = A \otimes k_i$, which we denote as
$\uprod \text{Comod}_{A_i}$.

The language over which these categories are realized as
first-order structures, which we call the `language of abelian
tensor categories' (section \ref{chapterTheLanguageOf}), includes
symbols denoting an element being an object or morphism,
composition of morphisms, addition of morphisms, morphisms
pointing from one object to another, and notably, a symbol for the
tensor product (of objects and morphisms).  It also includes
symbols denoting certain natural transformations on the category,
necessary to describe certain regularity properties of the tensor
product, e.g.~being naturally associative and commutative.  The
primary reason we have chosen these symbols is

\begin{thm}(see chapter \ref{chapterFirstOrderDefinabilityOf})
In the language of abelian tensor categories, the statement ``is a
tannakian category'' is a first-order sentence.
\end{thm}

Chapter \ref{chapterTannakianDuality} is devoted to giving an
explicit definition of a tannakian category.  Suffice it to say
for the moment, it is an abelian category $\catfont{C}$, endowed
with a bifunctor $\otimes:\catfont{C} \times \catfont{C}
\rightarrow \catfont{C}$, which satisfies a plethora of regularity
conditions, e.g. being naturally associative and possessing
internal Homs.  We say that a tannakian category $\catfont{C}$ is
\textit{neutral} (see definition
\ref{defnneutraltannakiancategory}) if it comes equipped with a
\textit{fibre functor}, i.e.~an exact, faithful, $k$-linear tensor
preserving functor \index{$\omega$} $\omega$ from $\catfont{C}$ to
$\text{Vec}_k$ (the category of finite dimensional vector spaces
over $k$, where $k$ is the field
$\text{End}_\catfont{C}(\underline{1})$, and $\underline{1}$
denotes the identity object of $\catfont{C}$). The motivation for
the definition of a neutral tannakian category is the following
theorem.

\begin{thm}(see theorem 2.11 of \cite{deligne})
Let $\catfont{C}$ be a neutral tannakian category over the field
$k$ with fibre functor $\omega$. Then

\begin{enumerate}
\item{The functor $\text{Aut}^\otimes(\omega)$ on $k$-algebras is
representable by an affine group scheme $G$} \item{$\omega$
defines an equivalence of tensor categories between $\catfont{C}$
and $\text{Rep}_k G$}
\end{enumerate}
\end{thm}

The moral: a neutral tannakian category is (tensorially equivalent
to) the category of finite dimensional representations of an
affine group scheme over a field, and vice versa.  (Section
\ref{recoveringanalgebraicgroupsection} is devoted to describing
how one goes about, in principle, recovering the representing Hopf
algebra of a neutral tannakian category.)

\section{The Restricted Ultraproduct of Neutral Tannakian
Categories}

As is argued in section \ref{FODofExt1Section}, the basic concepts
of cohomology of modules (at least in the case of $\text{Ext}^1$)
are quite naturally expressible in the language of abelian tensor
categories.  Therefore, to study cohomology over a particular
group $G$ and field $k$, a reasonable object of study is the
category $\text{Rep}_{k}G$ as a first-order structure in this
language.  But further, we are interested in studying
\textit{generic} cohomology; that is, for a fixed group $G$ and
sequence of fields $k_i$, we would like to know if a particular
cohomological computation eventually stabilizes for large enough
$i$.  We are then drawn to the study of not the single category
$\text{Rep}_k G$ for fixed $k$, but rather the infinite sequence
of the categories $\text{Rep}_{k_i} G$.  And as ultraproducts of
relational structures, by design, tend to preserve only those
first-order properties which are true `almost all of the time', it
is for this reason that we have chosen to study ultraproducts of
categories of the form $\text{Rep}_{k_i} G$, which we denote by
$\uprod \text{Rep}_{k_i} G$.

While being a tannakian category is a first-order concept, the
property of being endowed with a fibre functor, so far as we can
tell, is not. If $\catfont{C}_i$ is a sequence of tannakian
categories neutralized by the fibre functors $\omega_i$, the
natural attempt to endow $\uprod \catfont{C}_i$ with a fibre
functor would go as follows.
 Define a functor $\omega$ on $\uprod \catfont{C}_i$ which takes an
object $[X_i] \in \uprod \catfont{C}_i$ to $\uprod \omega_i(X_i)$
(ultraproduct of vector spaces; see section
\ref{ultraproductsofvectorspaces}), and similarly for a morphism
$[\phi_i]$ (ultraproduct of linear maps; see section
\ref{lineartransformationsandmatricessubsection}).  But this will
not do; $\omega([X_i])$ will in general be infinite dimensional
(proposition \ref{prop1}), specifically disallowed by the
definition of a fibre functor. Further, for any collection of
vector spaces $V_i$ and $W_i$ over the fields $k_i$, we have a
natural injective map $\uprod V_i \otimes \uprod W_i \rightarrow
\uprod V_i \otimes W_i$ (section \ref{tensorproductssubsection}).
But unless at least one of the collections is boundedly finite
dimensional, this will not be an isomorphism; thus $\omega$ will
not be tensor preserving in general. We therefore make the
following compromise:

\begin{defn}
The \textbf{restricted ultraproduct} of the $\catfont{C}_i$,
denoted $\resprod \catfont{C}_i$, is the full subcategory of
$\uprod \catfont{C}_i$ consisting of those objects $[X_i]$ such
that the dimension of $\omega_i(X_i)$ is bounded.
\end{defn}

Then we indeed have

\begin{thm}(see theorems \ref{resprodistannakianthm} and
\ref{resprodisneutralthm}) $\resprod \catfont{C}_i$ is a tannakian
category, neutralized over the field $k=\uprod k_i$ by the functor
$\omega$ described above.
\end{thm}

Thus, $\resprod \catfont{C}_i$ is tensorially equivalent to
$\text{Comod}_{A_\infty}$ for some Hopf algebra $A_\infty$ over
the field $k = \uprod k_i$.  The question then: what is
$A_\infty$?

The obvious first guess, that it is the ultraproduct of the Hopf
algebras $A_i$ representing each of the $\catfont{C}_i$, is not
correct; problem being, this is not a Hopf algebra at all.  We
start by defining a map $\Delta$ on $\uprod A_i$ by the formula
$\uprod A_i \stackrel{[\Delta_i]}{\vlongrightarrow} \uprod A_i
\otimes A_i$ (the ultraproduct of the maps $\Delta_i$).  But
again, unless the $A_i$ are boundedly finite dimensional, we
cannot expect this $\Delta$ to point to $\uprod A_i \otimes \uprod
A_i \subset \uprod A_i \otimes A_i$ in general. So we make another
compromise:

\begin{defn} (see section
\ref{sectionTheRestrictedUltraproductOfHopfAlgebras}) The
\textbf{restricted ultraproduct} of the Hopf algebras $A_i$,
denoted $A_R$, is the collection of all $[a_i] \in \uprod A_i$
such that the rank of $a_i$ is bounded.
\end{defn}

Defining exactly what ``rank'' means here takes some doing, so we
defer it; suffice it to say, $A_R$ can indeed be given the
structure of a coalgebra, under the definition of $\Delta$ given
above. We are able to prove
\ref{chapterTheRepresentingHopfAlgebraOf}

\begin{thm}
\label{introThm1}The representing Hopf algebra of the restricted
ultraproduct of the categories $\text{Comod}_{A_i}$ is isomorphic
to the restricted ultraproduct of the Hopf algebras $A_i$.
\end{thm}

In section \ref{sectionExamples} we explicitly work out $A_\infty$
for a few examples.  If $G$ is a finite group defined over
$\mathbb{Z}$ with representing Hopf algebra $A$, and if $k_i$ is
any collection of fields, then the Hopf algebras $A \otimes k_i$
are constantly finite dimensional, whence the full ultraproduct
$\uprod A \otimes k_i$ \textit{is} in fact a Hopf algebra.  In
this case $A_\infty$ can be identified with $A \otimes \uprod
k_i$, whence $\resprod \text{Rep}_{k_i} G \isomorphic
\text{Rep}_{\uprod k_i} G$.

For non-finite groups, the situation becomes considerably more
delicate. As an example, consider the multiplicative group $G=G_m$
(subsection \ref{subsectionTheGroupGm}) and let $k_i$ be any
collection of fields. For a fixed ultrafilter $\filtfont{U}$, let
$\uprod \mathbb{Z}$ denote the ultrapower of the integers. Then we
can identify $A_\infty$ as the $k = \uprod k_i$-span of the formal
symbols $x^{[z_i]}$, $[z_i] \in \uprod \mathbb{Z}$, with $\Delta$
and $\text{mult}$ defined by
\begin{gather*}
 A_\infty = \text{span}_k(x^{[z_i]}: [z_i] \in \uprod
 \mathbb{Z})\\
 \Delta: x^{[z_i]} \mapsto x^{[z_i]} \otimes x^{[z_i]} \\
 \text{mult}: x^{[z_i]} \otimes x^{[w_i]} \mapsto x^{[z_i + w_i]}
\end{gather*}

We also note here that chapter \ref{adualitytheoremchapter}
contains an interesting theorem about finite dimensional
subcoalgebras of Hopf algebras which was necessary to prove
theorem \ref{introThm1}, but is certainly of interest in its own
right, and requires no understanding of ultraproducts.

\section{From Ultraproducts to Generic Cohomology}

The reason we chose to study these categories in the first place
is because cohomology of modules (at least in the $\text{Ext}^1$
case) is a naturally expressible concept in the language of
abelian tensor categories. That is (see section
\ref{FODofExt1Section})

\begin{prop}
For fixed $n$, the statement $\phi(M,N) \stackrel{\text{def}}{=} $
``$\text{Ext}^1(M,N)$ has dimension $n$'' is a first-order formula
in the language of abelian tensor categories.
\end{prop}

Here we have adopted the view that $\text{Ext}^1(M,N)$, relative
to a given abelian category, consists of equivalence classes of
module extensions of $M$ by $N$, as opposed to the more standard
definition via injective or projective resolutions; necessary,
since the category $\text{Rep}_k G$ will in general not have
enough injective or projective objects (due to it consisting of
only \emph{finite} dimensional representations of $G$ over $k$).
Suppose then that $M$ and $N$ are modules for $G$ over
$\mathbb{Z}$, and that $k_i$ is a collection of fields. We wish to
discover whether the quantity
\[ \extdim \text{Ext}^1_{G(k_i)}(M,N) \]
stabilizes for large $i$.  We have a criterion for this to be
true.

\begin{thm}
Let $M_i$ and $N_i$ denote the images inside $\text{Rep}_{k_i}(G)$
of the modules $M$ and $N$, and let $[M_i],[N_i]$ denote the
images of the tuples $(M_i),(N_i)$ inside the category $\resprod
\text{Rep}_{k_i}(G)$. Then if the computation $\extdim
\text{Ext}^1([M_i],[N_i])$ is both finite and the same inside the
category $\resprod \text{Rep}_{k_i}(G)$ for \emph{every} choice of
non-principal ultrafilter, the computation $\extdim
\text{Ext}^1_{G(k_i)}(M_i,N_i)$ is the same for all but finitely
many $i$.
\end{thm}

\begin{proof}
The key fact (and the reason we restrict to $\text{Ext}^1$ in the
first place) is that computing $\text{Ext}^1$ in the restricted
ultracategory $\resprod \catfont{C}_i$ is the same as doing so in
the full ultracategory $\uprod \catfont{C}_i$, since the extension
module of any $1$-fold extension of $[M_i]$ by $[N_i]$ has bounded
dimension $\text{dim}(M_i) + \text{dim}(N_i)$. First-order
statements that are true in $\uprod \text{Rep}_{k_i}(G)$ for
\textit{every} choice of non-principal ultrafilter correspond to
statements that are true for all but finitely many of the
categories $\uprod \text{Rep}_{k_i}(G)$, namely the statement
``$\extdim \text{Ext}^1_{G(k_i)}(M_i,N_i) = n$''.
\end{proof}

Our attempts to extend these results to the case of
$\text{Ext}^n$, $n>1$, have so far met with resistance; for more
on this see section \ref{TheDifficultyWithSection}.

\section{Generic Representation Theory of Unipotent Groups}

Beginning in chapter \ref{themethodchapter} we take a break from
working with ultraproducts, and instead focus on the concrete
representation theory of two unipotent algebraic groups, both in
zero and positive characteristic. Starting with the additive group
\index{$G_a$} $G_a$ (chapter \ref{chapterTheAdditiveGroup}) we
prove

\begin{thm}
\label{GacharzeropanaloguetheoremIntro} (see theorem
\ref{Gacharzeropanaloguetheorem})
\begin{enumerate}
\item{Let $k$ have characteristic zero.  Then every
$n$-dimensional representation of $G_a$ over $k$ is given by an $n
\times n$ nilpotent matrix $N$ over $k$ according to the formula
\[ e^{xN} \]
} \item{Let $k$ have positive characteristic $p$.  Then if $p >>
n$, every $n$-dimensional representation of $G_a$ over $k$ is
given by a finite ordered sequence $N_i$ of $n \times n$ commuting
nilpotent matrices over $k$ according to the formula
\[ e^{xN_0} e^{x^p N_1} \ldots e^{x^{p^m} N_m} \]
}
\end{enumerate}
\end{thm}
This is our first indication of a connection between the
characteristic zero theory of a unipotent group and its positive
characteristic theory for $p$ large with respect to dimension.  In
chapter \ref{TheHeisenbergGroup} we obtain an identical result for
the Heisenberg group \index{$H_1$} $H_1$:

\begin{thm}
(see theorems \ref{H1charZeroBakerThm} and \ref{H1charpBakerThm})
\begin{enumerate}
\item{Let $k$ have characteristic zero.  Then every
$n$-dimensional representation of $H_1$ is given by a triple
$X,Y,Z$ of $n \times n$ nilpotent matrices over $k$ satisfying
$Z=[X,Y]$ and $[X,Z] = [Y,Z] = 0$, according to the formula
\[ e^{xX + yY + (z-xy/2)Z} \]
} \item{Let $k$ have positive characteristic $p$.  Then if $p >>
n$, every $n$-dimensional representation of $H_1$ over $k$ is
given by a sequence $X_0,Y_0,Z_0,X_1,Y_1,Z_1\ldots,X_m, Y_m,Z_m$
of $n \times n$ nilpotent matrices over $k$ satisfying $Z_i =
[X_i,Y_i]$, $[X_i,Z_i] = [Y_i,Z_i] = 0$, and whenever $i \neq j$,
$X_i,Y_i,Z_i$ all commute with $X_j,Y_j,Z_j$, according to the
formula
\[ e^{xX_0+yY_0 + (z-xy/2)Z_0} e^{x^p X_1 + y^p Y_1 + (z^p-x^p
y^p/2)Z_1} \ldots e^{x^{p^m} X_m + y^{p^m} Y_m + (z^{p^m} -
x^{p^m} y^{p^m}/2)Z_m}
\] }
\end{enumerate}
\end{thm}

We see then that for $p$ sufficiently larger than dimension,
characteristic $p$ representations for these unipotent groups are
simply commuting products of representations, each of which `look
like' a characteristic zero representation, with each factor
accounting for one of its `Frobenius layers'.  It is this
phenomenon to which the phrase `generic representation theory' in
the title refers.

Note that this is only a theorem about $p >> \text{dimension}$;
for any positive characteristic field $k$, once dimension becomes
too large in the category $\text{Rep}_k H_1$, the analogy
completely breaks down, and representations of $H_1$ over $k$ can
be expected to bear no resemblance to representations in
characteristic zero.

\section{The Height-Restricted Ultraproduct}

Let $G$ be either of the two above discussed unipotent groups, $k$
a field of characteristic $p>0$, $V$ a representation of $G$ over
$k$. Suppose that $V$ is of the form described in part 2.~of the
two preceding theorems.  Then we define the \index{height}
\textbf{height} of $V$ to be $m+1$, that is, the number of
Frobenius layers in the representation. For instance, in the case
of $G_a$, height is simply the largest $m$ such that $x^{p^{m-1}}$
occurs in the matrix formula of the representation.

Now let $k_i$ be a sequence of fields of strictly increasing
positive characteristic, and let $\catfont{C}_i = \text{Rep}_{k_i}
G$.  For an object $[X_i]$ of the restricted ultraproduct
$\resprod \catfont{C}_i$ (i.e.~where the $X_i$ have bounded
dimension), by the above two theorems, for large enough $i$, $X_i$
will be of the aforementioned form, so that for all but finitely
many $i$, the height of $X_i$ is well-defined.  We therefore
define the \textbf{height} of $[X_i]$ as the essential supremum of
$\{ \text{height}(X_i):i \in I\}$. Note that the height of a given
object $[X_i]$ of $\resprod \catfont{C}_i$ might well be infinite.
In case it is not, we define

\begin{defn}(see definition \ref{defnHightResUprod})
The \textbf{height-restricted ultraproduct} of the categories
$\catfont{C}_i = \text{Rep}_{k_i} G$, for $k_i$ of increasing
positive characteristic, is the full subcategory of the restricted
ultraproduct $\resprod \catfont{C}_i$ consisting of those objects
$[X_i]$ of finite height.  We denote this category by $\hprod
\catfont{C}_i$.  For $n \in \mathbb{N}$, we denote by $\hnprod{n}
\catfont{C}_i$ the full subcategory of $\hprod \catfont{C}_i$
consisting of those objects of height no greater than $n$.
\end{defn}

We have seen already that, for $p$ sufficiently large with respect
to dimension, representations of $G$ in characteristic $p$
resemble representations of $G^n$ in characteristic zero.  We
shall also see later that this resemblance is functorial, in the
sense that the analogy carries over to morphisms between the
representations, and to various other constructions, e.g.~direct
sums and tensor products. The most compact way to express this is

\begin{thm}(see theorem \ref{heightResMainTheorem})
If $k_i$ is a sequence of fields of strictly increasing positive
characteristic, the category $\hprod \text{Rep}_{k_i} G$ is
tensorially equivalent to $\text{Rep}_k G^\infty$, where
$G^\infty$ denotes a countable direct power of $G$, and $k$ is the
ultraproduct of the fields $k_i$. Similarly, for any $n \in
\mathbb{N}$, $\hnprod{n} \text{Rep}_{k_i} G$ is tensorially
equivalent to $\text{Rep}_k G^n$.
\end{thm}

Note in particular that the group $G^n$ obtained is independent of
the choice of non-principal ultrafilter, and while the field $k$
does vary, it will in all cases have characteristic zero.

\section{Height-Restricted Generic Cohomology}

\label{GenericCohomologyforUnipotentAlgebraicGroupsinLargeCharacteristicIntroSection}

This last result will allow us to derive some generic cohomology
theorems for the unipotent algebraic groups discussed above, at
least for the case of $\text{Ext}^1$.  Rather than state the
theorem precisely here it will be much more illuminating to
illustrate with an example, which is worked out in more detail in
section \ref{GenCohomExampleSection}.

Let $G = G_a$ and let $k$ have characteristic $p >0$.  Then direct
computation shows (using theorem \ref{Gacharptheorem}) that a
basis for $\text{Ext}^1_{G_a(k)}(k,k)$ is given by the sequence of
linearly independent extensions
\[ \xi_m: 0 \rightarrow k \rightarrow \left(%
\begin{array}{cc}
  1 & x^{p^m} \\
   & 1 \\
\end{array}%
\right) \rightarrow k \rightarrow 0 \] for $m=0,1, \ldots$.  This
is obviously infinite dimensional, so we ask the more interesting
question: what happens when we restrict the height of the
extension module?  Specifically, let
$\text{Ext}^{1,h}_{G_a(k)}(k,k)$ denote the space of equivalence
classes of extensions whose extension module has height no greater
than $h$. Then of course $\{\xi_m:m=0,1, \ldots, h-1\}$ forms a
basis for it, and its dimension is $h$.

Now assume $k$ has characteristic zero and compute
$\text{Ext}^1_{G^h(k)}(k,k)$.  Then similarly (using theorems
\ref{Gacharzerothm} and \ref{directproducttheorem}) we have the
basis
\[ \xi_m: 0 \rightarrow k \rightarrow \left(%
\begin{array}{cc}
  1 & x_m \\
   & 1 \\
\end{array}
\right)\rightarrow k \rightarrow 0 \] for $m=0,1, \ldots, h-1$,
where $x_m$ denotes the $m^{\text{th}}$ free variable of $G^h$.
Thus we see that, for $k_i$ of sufficiently large positive
characteristic
\[ \extDim \text{Ext}^{1,h}_{G_a(k_i)}(k_i,k_i) =
\extDim \text{Ext}^1_{G_a^h(\uprod k_i)}(\uprod k_i,\uprod k_i) \]

This example is misleading in that the above equality holds for
all primes $p$ (due to the small dimension of the extension
modules of the extensions); in general we merely claim that the
above holds for sufficiently large characteristic. We shall prove
in section \ref{GenericCohomologyforExt1} that this is quite a
general phenomenon.

\section{Notational Conventions}

Throughout this dissertation our convention for expressing
composition of maps is as follows.  If $\phi:X \rightarrow Y$ and
$\psi:Y \rightarrow Z$ are maps, then the composition  $X
\stackrel{\phi}{\longrightarrow} Y
\stackrel{\psi}{\longrightarrow} Z$ shall be expressed as
\[ \phi \circ \psi \] However, if $x$ is an element of $X$, then the element $z \in Z$
gotten by evaluating $\phi \circ \psi$ at $x$ shall be expressed
as
\[\psi(\phi(x))\]
In other words, when we are expressing a composition with no
inputs we shall write functions on the right, and when we are
evaluating a composition at an input we shall write functions on
the left. The reader need only remember that whenever he sees the
symbol $\circ$ as in $\phi \circ \psi$, we are writing functions
on the right, and when he instead sees the parenthetical notation
$\psi(\phi(x))$ we are writing functions on the left.  In
particular, we shall never write $(x)(\phi \circ \psi)$, $(\phi
\circ \psi)(x)$, $((x)\phi)\psi$, $\phi(\psi(x))$, or $\psi \circ
\phi$.

A matrix, if we wish to emphasize what its entries are, is
generally written in the notation $(a_{ij})$, suppressing mention
of its dimensions. If $\phi:U \rightarrow V$ is a linear map
between finite dimensional vector spaces, and if $(k_{ij})$ is the
matrix representing $\phi$ in certain bases, then it is understood
to be doing so be passing column vectors to the right of
$(k_{ij})$.  In particular, if $\phi:U \rightarrow V$, $\psi:V
\rightarrow W$ are linear maps, and if $(k_{ij})$ represents
$\phi$ and $(l_{ij})$ represents $\psi$, then the matrix product
$(l_{ij})(k_{ij})$ represents the map $\phi \circ \psi$.

\index{$\uprod M_i$} If $M_i$ is a collection of relational
structures in a common signature indexed by the set $I$, and if
$\filtfont{U}$ is a non-principal \index{ultrafilter} ultrafilter
over $I$, we denote by $\uprod M_i$ the \index{ultraproduct}
ultraproduct of these structures with respect to $\filtfont{U}$.
The reader may consult the appendix for a review of ultrafilters
and ultraproducts in general.  In several instances in this
dissertation we shall be considering certain substructures of
ultraproducts, e.g.~the restricted ultraproduct of the neutral
tannakian categories $\catfont{C}_i$, which in this case we denote
by $\resprod \catfont{C}_i$.  Note that, in this case and in
several others, to avoid using a double subscript, we have dropped
reference to the particular non-principal ultrafilter being
applied; as it will always be assumed to be fixed but arbitrary,
no confusion should result.

The reader is encouraged to consult the index for a more complete
list of commonly used symbols.

\chapter{Algebraic Groups, Hopf Algebras, Modules, Comodules and Cohomology}

Here we review the basic facts concerning the duality between
algebraic groups and their Hopf algebras, modules for algebraic
groups vs.~comodules for Hopf algebras, and the definitions
concerning cohomology of modules and comodules.  We shall also
define the equivalent categories $\text{Rep}_k G$ and
$\text{Comod}_A$, where $A$ is the representing Hopf algebra of
$G$, and the important constructions within them.  We shall mostly
be content to recording definitions and important theorems, only
rarely supplying proofs. The reader may consult \cite{jantzen},
\cite{hopfalgebras}, and \cite{waterhouse} for a more thorough and
excellent account of what follows.  We particularly recommend the
first several chapters of \cite{waterhouse} for those less
accustomed to the `functorial' view of algebraic groups we shall
be adopting. \cite{humphreys} and \cite{borel} are also excellent
references, but with much less of an emphasis on this functorial
view.

\section{Algebraic Groups, Coalgebras and Hopf Algebras}

In this dissertation, ``algebraic group'' over a ring $k$ shall
always mean a particular kind of affine group scheme. It is at
this level of generality in which we will operate throughout.  If
$k$ is any commutative ring with identity (and all rings will
assumed to be so) a \textbf{$k$-algebra} shall always mean a
commutative $k$-algebra with identity.

\begin{defn}(see section $1.2$ of \cite{waterhouse})
\index{affine group scheme} An \textbf{affine group scheme} over a
commutative ring $k$ with identity is a representable covariant
functor from the category of all $k$-algebras to the category of
groups. We say it is an \index{algebraic group} \textbf{algebraic
group} if the representing object of the functor is finitely
generated as a $k$-algebra, and a \textbf{finite group} if it is
finitely generated as a $k$-module.
\end{defn}

\begin{defn}(see section $1.1$ of \cite{hopfalgebras})
\index{coalgebra} \label{coalgebradefn} Let $k$ be a ring, $C$ a
$k$-module.  $C$ is called a $k$-\textbf{coalgebra} if it comes
equipped with $k$-linear maps \index{$\Delta$} $\Delta:C
\rightarrow C \otimes C$ \index{co-multiplication}
(co-multiplication) and \index{$\varepsilon$} $\varepsilon:C
\rightarrow k$ \index{co-unit} (co-unit) making the following two
diagrams commute:
\begin{equation}
\label{hopf1}
\begin{diagram}
C & \rTo^\Delta & C \otimes C \\
\dTo^\Delta & & \dTo_{1 \otimes \Delta} \\
C \otimes C & \rTo_{\Delta \otimes 1} & C \otimes C \otimes C
\end{diagram}
\end{equation}
\begin{equation}
\label{hopf2}
\begin{diagram}
& & C & & \\
& \ldTo^{\isomorphic} & & \rdTo^\isomorphic & \\
k \otimes C & & \dTo^\Delta & & C \otimes k \\
& \luTo_{\varepsilon \otimes 1} & & \ruTo_{1 \otimes \varepsilon}
\\
& & C \otimes C & & \\
\end{diagram}
\end{equation}
A $k$-\textbf{bialgebra} is a $k$-module $C$ which is
simultaneously a $k$-algebra and a $k$-coalgebra in such a way
that $\Delta$ and $\varepsilon$ are algebra maps.  A bialgebra $C$
is called a \index{Hopf algebra} \textbf{Hopf algebra} if it comes
equipped with a $k$-algebra map \index{$S$} $S:C \rightarrow C$
(co-inverse or \index{antipode} antipode) making the following
commute:
\begin{equation}
\label{hopf3}
\begin{diagram}
A & \rTo^\Delta & A \otimes A & &\\
  & & & \rdTo^{S \otimes 1} & \\
\dTo^{\varepsilon} & & & & A \otimes A \\
 & & & \ldTo_{\text{mult}} & \\
 k & \rTo & A & & \\
\end{diagram}
\end{equation}
A \index{coalgebra!morphism of} \textbf{morphism} between the
$k$-coalgebras $(C,\Delta)$ and $(C^\prime, \Delta^\prime)$ is a
$k$-linear map $\phi:C \rightarrow C^\prime$ such that the
following diagram commutes:
\begin{diagram}
C & \rTo^\phi & C^\prime \\
\dTo^\Delta & & \dTo_{\Delta^\prime} \\
C \otimes C & \rTo_{\phi \otimes \phi} & C^\prime \otimes C^\prime
\\
\end{diagram}

\end{defn}

\begin{thm}(see section $1.4$ of \cite{waterhouse})
The representing object of an affine group scheme over $k$ is a
Hopf algebra over $k$.  Conversely, any Hopf algebra over $k$
defines an affine group scheme over $k$.
\end{thm}

Let $G(\slot) = \text{Hom}_k(A,\slot)$ be an affine group scheme
over $k$ represented by the Hopf algebra $A$. As the names
suggest, the co-multiplication, co-unit, and co-inverse maps
attached to a Hopf algebra encode the group multiplication,
identity, and inversion, respectively.  If $R$ is a $k$-algebra
then an element of $G(R)$ is by definition a $k$-homomorphism
$\phi:A \rightarrow R$.  Then the map $\Delta$ tells us how to
multiply elements of $G(R)$; given $\phi,\psi:A \rightarrow R$
their product, call it $\phi * \psi$, is defined to be the unique
map making the following diagram commute:
\begin{diagram}
A & \rTo^\Delta & A \otimes A \\
\dTo^{\phi * \psi} & & \dTo_{\phi \otimes \psi} \\
R & \lTo_{\text{mult}} & R \otimes R \\
\end{diagram}
Similarly the inverse of the element $\phi$, call it
$\text{inv}(\phi)$, is the unique map making
\begin{diagram}
A & \rTo^S & A \\
\dTo^{\text{inv}(\phi)} & \ldTo_\phi & \\
R & & \\
\end{diagram}
commute.  Finally, the identity element $e$ of $G(R)$ is defined
by the commutativity of the diagram.
\begin{diagram}
A & \rTo^{\varepsilon} & k \\
\dTo^{e} & \ldTo & \\
R & & \\
\end{diagram}

With this in mind, diagrams \ref{hopf1}, \ref{hopf2} and
\ref{hopf3} are really not mysterious at all. \ref{hopf1} merely
encodes the fact that the multiplication defined by $\Delta$ is
associative, \ref{hopf2} corresponds to the statement that $A
\stackrel{\varepsilon}{\longrightarrow} k \rightarrow R$ is always
the identity element of $G(R)$, and one can probably guess what
fact about groups \ref{hopf3} represents.

When we refer to a Hopf algebra we shall often write it as
$(A,\Delta,\varepsilon)$, emphasizing the fact that these are
often the only pieces of information we require for a given
purpose.  Besides, just as inverses can be discovered by looking
at the multiplication table of a group, so also is the map $S$
completely determined by the map $\Delta$ (and the same can be
said for the map $\varepsilon$).

Essential to the study of representable functors in any category
is the so-called Yoneda lemma, which tells us that natural
transformations from representable functors to other functors are
quite easy to describe.

\begin{lem}
\index{Yoneda lemma} \label{YonedaLemma} (Yoneda lemma; see
section $1.3$ of \cite{waterhouse}) Let $\catfont{C}$ be any
locally small category (so that Hom-sets are actually sets),
$G,H:\catfont{C} \rightarrow \text{Sets}$ any two (covariant) set
valued functors, and suppose that $G$ is representable by the
object $A \in \catfont{C}$ (so that $G(\slot) =
\text{Hom}_{\catfont{C}}(A, \slot)$).  Let $\Phi$ be a natural
transformation from the functor $G$ to $H$.  Then for any object
$X$ of $\catfont{C}$ and element $\phi$ of
$\text{Hom}_{\catfont{C}}(A, X)$, $\Phi_X(\phi) = (H
\phi)(\Phi_A(1_A))$.
\end{lem}

\begin{proof}

Given $\phi:A \rightarrow X$, consider the commutative diagram
\begin{diagram}
G(A) & \rTo^{\Phi_A} & H(A) \\
\dTo^{G\phi} & & \dTo_{H \phi} \\
G(X) & \rTo_{\Phi_X} & H(X) \\
\end{diagram}
Here $G \phi$ refers to the map that sends $\psi:A \rightarrow A$
to $\psi \circ \phi$, and $H \phi$ refers to whatever map $H$
sends $\phi$ to.  Start with $1_A \in G(A)$ in the upper left
corner, chase it around both paths to $H(X)$, and you get the
equation claimed.
\end{proof}

As the names suggest, there is a duality between $k$-algebras and
$k$-coalgebras.

\begin{defn}(see section $1.3$ of \cite{hopfalgebras})
Let $k$ be a field, $(C,\Delta,\varepsilon)$ a $k$-coalgebra.  The
\textbf{dual algebra} to $C$ is the set $C^*$ of linear
functionals on $C$ endowed with the following multiplication: if
$\phi,\psi \in C^*$, then $\text{mult}(\phi \otimes \psi)$ is the
$k$-linear map given by the composition
\[C \stackrel{\Delta}{\longrightarrow} C \otimes C \stackrel{\phi \otimes \psi}{\vlongrightarrow} k
\otimes k \isomorphic k \]
\end{defn}

If $A$ is an infinite dimensional algebra, it is generally not
possible to introduce a coalgebra structure on the entire dual
space of $A$.  We can however introduce one on a certain subspace.
\begin{defn}(see section $1.5$ of \cite{waterhouse})
\index{coalgebra!finite dual} Let $k$ be a field, $A$ a
$k$-algebra. The \textbf{finite dual coalgebra} of $A$, denoted
\index{$A^\circ$} $A^\circ$, is the subspace of $A^*$ consisting
of those linear functionals on $A$ which kill an ideal of $A$ of
finite codimension.  For $\alpha \in A^\circ$, we define
$\Delta(\alpha)$ as follows: let $\alpha$ act on $A \otimes A$ by
the composition
\[ A \otimes A \stackrel{\text{mult}}{\vlongrightarrow} A
\stackrel{\alpha}{\longrightarrow} k \] and then pass this
composition to the isomorphism $(A \otimes A)^\circ \isomorphic
A^\circ \otimes A^\circ$.  We define $\varepsilon:A^\circ
\rightarrow k$ by $\varepsilon(\alpha) = \alpha(1)$.
\end{defn}

If $A$ is finite dimensional, then $A^\circ$ is all of $A^*$. In
this case there is a natural isomorphism $A \isomorphic A^{\circ
*}$ of algebras, and likewise a natural isomorphism $C \isomorphic
C^{* \circ}$ of coalgebras.

This duality is functorial.  Given an algebra map $\phi:A
\rightarrow B$ we get a coalgebra map \index{$\phi^\circ$}
$\phi^\circ:B^\circ \rightarrow A^\circ$ defined by, for $\beta
\in B^\circ$, $\phi^\circ(\beta)$ is the composition
\[ A \stackrel{\phi}{\longrightarrow} B
\stackrel{\beta}{\longrightarrow} k \] In a similar fashion we get
algebra maps from coalgebra maps.

\section{$G$-modules and $A$-comodules}

\begin{defn}(see section $3.1$ of \cite{waterhouse})
\index{representation} Let $G$ be an affine group scheme over the
ring $k$, and $V$ a $k$-module. A \textbf{linear representation}
of $G$ on $V$ is a natural transformation from the functor
$G(\slot)$ to the functor $GL_V(\slot)$, where $\text{GL}_V(R)
\stackrel{def}{=} \text{Aut}_R(V \otimes R)$.  We say also then
that $V$ is a $G$-\textbf{module}.
\end{defn}

The concept of a linear representation, put this way, is a bit
intimidating.  However, it is a consequence of (the proof of) the
Yoneda lemma that linear representations correspond to very
concrete things, called comodules.

\begin{defn}(see definition $2.1.3$ of \cite{hopfalgebras})
 Let $(C, \Delta, \varepsilon)$ be a coalgebra
over the ring $k$, $V$ a $k$-module.  $V$ is called a (right)
\index{comodule} $C$-\textbf{comodule} if it comes equipped with a
$k$-linear map $\rho:V \rightarrow V \otimes C$ such that the
following two diagrams commute:
\begin{equation}
\label{comod1}
\begin{diagram}
V & \rTo^\rho & V \otimes C \\
\dTo^\rho & & \dTo_{1 \otimes \Delta} \\
V \otimes C & \rTo_{\rho \otimes 1} & V \otimes C \otimes C \\
\end{diagram}
\end{equation}
\begin{equation}
\label{comod2}
\begin{diagram}
V & \rTo^\rho & V \otimes C \\
& \rdTo_{\isomorphic} & \dTo_{1 \otimes \varepsilon} \\
& & V \otimes k \\
\end{diagram}
\end{equation}
\end{defn}

\begin{thm}(see section $3.2$ of \cite{waterhouse})
\index{comodule!in relation to representations}
\label{modulecomodulecorrespondencethm} If $G$ is an affine group
scheme represented by the Hopf algebra $A$, then linear
representations of $G$ on $V$ correspond to $A$-comodule
structures on $V$.
\end{thm}
It is for this reason that, in this dissertation, we shall quite
often confuse the notions of $G$-module and $A$-comodule, and
shall sometimes speak glibly of $A$-modules, representations for
$A$, comodules for $G$, etc.

Here is the correspondence.  Given an $A$-comodule $(V,\rho)$, we
get a representation $\Phi$ of $G$ on $V$ as follows: if $g \in
G(R)$, then $g$ is by definition a $k$-homomorphism from $A$ to
$R$. Define $g$ to act on $V \otimes R$ via the composition
\begin{equation}
\label{comodtorep}
 V \stackrel{\rho}{\vlongrightarrow} V \otimes A \stackrel{1
\otimes g}{\vlongrightarrow} V \otimes R
\end{equation}
and then extend to $V \otimes R$ by $R$-linearity.  Conversely,
let $\Phi:G \rightarrow \text{GL}_V$ be a representation.  Then we
may ask how $\text{id}_A \in G(A)$ acts on $V$, that is, what is
the map
\[ \Phi_A(\text{id}_A):V \otimes A \rightarrow V \otimes A \]
As we demand this map to be $A$-linear it is necessarily
determined by its restriction to $V \isomorphic V \otimes 1$, and
it is this map, call it $\rho$, which gives $V$ the structure of
an $A$-comodule. A Yoneda lemma type argument guarantees that, for
any $g \in G(R)$, $\Phi_R(g)$ is given by equation
\ref{comodtorep}.

Let us be more explicit.  Suppose $k$ is a field and $(V,\rho)$ a
finite dimensional $A$-comodule.  Fix a basis $e_1, \ldots, e_n$
of $V$ and write
\[ \rho: e_j \mapsto \sum_{i=1}^n e_i \otimes a_{ij} \]
Then the matrix $(a_{ij})$ is the `formula' for the representation
of $G$ on $V$.  That is, for $g \in G(R)$, $g$ acts on $V \otimes
R$ via the matrix $(g(a_{ij}))$ in the given basis, and then
extending by $R$-linearity.

Comparing equations \ref{hopf1} and \ref{hopf2} with equations
\ref{comod1} and \ref{comod2}, we see that $(A,\Delta)$ itself
qualifies as a (usually infinite dimensional) $A$-comodule, and we
call it the \index{representation!regular} \textbf{regular
representation}. Among other reasons, this is an important
representation because

\begin{thm}(see section $3.5$ of \cite{waterhouse})
\label{regreplemma} \label{regreptheorem}
 If $(V,\rho)$ is an $n$-dimensional $A$-comodule, then $V$ is embeddable in the
$n$-fold direct sum of the regular representation.
\end{thm}

Over fields, we have the following elementary yet eminently useful
results.

\begin{thm}\index{coalgebra!fundamental theorem of}(Fundamental theorem of coalgebras;
see theorem $1.4.7$ of \cite{hopfalgebras})
\label{fundamentaltheoremofcoalgebras} Let $k$ be a field, $C$ a
$k$-coalgebra. Then $C$ is the directed union of its finite
dimensional subcoalgebras.
\end{thm}

\begin{thm}\index{comodule!fundamental theorem of} (Fundamental theorem of comodules;
see theorem $2.1.7$ of \cite{hopfalgebras})
\label{fundamentaltheoremofcomodules} Let $k$ be a field, $C$ a
$k$-coalgebra, $V$ a $C$-comodule.  Then $V$ is the directed union
of its finite dimensional subcomodules.
\end{thm}

\section{The Categories $\text{Rep}_k G$ and $\text{Comod}_A$}

Let $k$ be a field, $G$ an affine group scheme over $k$, $A$ its
representing Hopf algebra.

\begin{defn}
\index{$\text{Comod}_A$} $\text{Comod}_A$ is the category whose
objects are finite dimensional $A$-comodules, and whose
\index{comodule!morphism of} morphisms between $(V,\rho)$ and
$(W,\mu)$ are those $k$-linear maps $\phi:V \rightarrow W$ making
the following commute:
\begin{diagram}
V & \rTo^\phi & W \\
\dTo^\rho & & \dTo_\mu \\
V \otimes A & \rTo_{\phi \otimes 1} & W \otimes A \\
\end{diagram}

\end{defn}

\begin{defn}
\index{$\text{Rep}_k G$} $\text{Rep}_k G$ is the category whose
objects are finite dimensional vector spaces $V$ with a prescribed
$G$-module structure $\Phi:G \rightarrow \text{GL}_V$, and whose
\index{representation!morphism of} morphisms between
$(V,\Phi),(W,\Psi)$ are those linear maps $\phi:V \rightarrow W$
such that, for every $k$-algebra $R$ and element $g \in G(R)$, the
following commutes:
\begin{diagram}
V \otimes R & \rTo^{\phi \otimes 1} & W \otimes R \\
\dTo^{\Phi_R(g)} & & \dTo_{\Psi_R(g)} \\
V \otimes R & \rTo_{\phi \otimes 1} & W \otimes R \\
\end{diagram}
\end{defn}

Note well that whenever we write $\text{Comod}_A$ or $\text{Rep}_k
G$ it consists of only \textit{finite dimensional}
modules/comodules, in contrast to the notation of some other
authors.

The correspondence between $G$-modules and $A$-comodules given in
theorem \ref{modulecomodulecorrespondencethm} is actually a
functorial one (and it would be of no use otherwise).

\begin{thm}
The categories $\text{Rep}_k G$ and $\text{Comod}_A$ are
equivalent as $k$-linear abelian tensor categories.
\end{thm}

This theorem is merely the assertion that a linear map $\phi:V
\rightarrow W$ is a morphism for $V$ and $W$ as $G$-modules if and
only if it is for $V$ and $W$ as $A$-comodules, and that the
tensor structure on $\text{Comod}_A$ (defined below) induces the
usual tensor structure we expect to see in $\text{Rep}_k G$.

\begin{defn}
\label{comoduleConstructionsDefn} Let $A$ be a Hopf algebra,
$(V,\rho)$, $(W,\mu)$ finite dimensional $A$-comodules.

\begin{enumerate}
\index{comodule!direct sum of} \item{Their \textbf{direct sum} is
the $A$-comodule with underlying vector space $V \oplus W$ and
comodule map given by the composition
\[ V \oplus W \stackrel{\rho \oplus \mu}{\vlongrightarrow} (V
\otimes A) \oplus (W \otimes A) \isomorphic (V \oplus W) \otimes A
\]
} \item{Their \index{comodule!tensor product of} \textbf{tensor
product} is the $A$-comodule with underlying vector space $V
\otimes W$ and comodule map given by the composition
\[ V \otimes W \stackrel{\rho \otimes \mu}{\vlongrightarrow} (V
\otimes A) \otimes (W \otimes A) \stackrel{1 \otimes \text{Twist}
\otimes 1}{\vvlongrightarrow} V \otimes W \otimes A \otimes A
\stackrel{1 \otimes 1 \otimes \text{mult}}{\vvlongrightarrow} V
\otimes W \otimes A \] }
 \item{The \textbf{tensor product} of two
morphisms $\phi:V \rightarrow X$, $\psi:W \rightarrow Y$ is the
usual tensor product of linear maps $\phi \otimes \psi:V \otimes W
\rightarrow X \otimes Y$} \item{The \textbf{trivial}
$A$-\textbf{comodule}, or \textbf{trivial representation}, is the
$A$-comodule having underlying vector space $k$ and comodule map
$\rho:k \rightarrow k \otimes A$ given by $\rho: 1 \mapsto 1
\otimes 1$}
 \item{The \index{comodule!dual of} \textbf{dual} of
$(V,\rho)$ is the $A$-comodule with underlying vector space $V^*$
and comodule map $\mu:V^* \rightarrow V^* \otimes A$ defined as
follows.  Let $\rho^*:V^* \rightarrow \text{Hom}_k(V,A)$ be the
map that sends the functional $\phi:V \rightarrow k$ to the
composition $V \stackrel{\rho}{\longrightarrow} V \otimes A
\stackrel{\phi \otimes 1}{\vlongrightarrow} k \otimes A
\isomorphic A$.  Then $\mu$ is the composition
\[ V^* \stackrel{\rho^*}{\longrightarrow} \text{Hom}_k(V,A)
\isomorphic V^* \otimes A \stackrel{1 \otimes S}{\vlongrightarrow}
V^* \otimes A \] where $S$ denotes the antipode of $A$} \item{The
\index{comodule!internal Hom of} \textbf{internal Hom} of $V$ and
$W$, denoted \index{$\intHom(V,W)$} $\text{\underline{Hom}}(V,W)$,
is the tensor product of the dual of $V$ with $W$.}

\end{enumerate}
\end{defn}

An alternative, basis dependent definition of the dual $(V^*,\mu)$
of the comodule $(V,\rho)$ is as follows.  Pick a basis $\{e_1,
\ldots, e_n\}$ of $V$, and let $\{\alpha_1,\ldots, \alpha_n\}$ be
the dual basis of $V^*$.  Write $\rho:e_j \mapsto \sum_i e_i
\otimes a_{ij}$, and set $\bar{\rho}:\alpha_j \mapsto  \sum_i
\alpha_i \otimes a_{ji}$ (note the transpose being applied). Then
$\mu$ is the composition
\[ V^* \stackrel{\bar{\rho}}{\longrightarrow} V^* \otimes A
\stackrel{1 \otimes S}{\vlongrightarrow} V^* \otimes A \] In other
words, if $(a_{ij})$ is the matrix formula for a representation on
$V$ in a given basis, then the dual representation on the dual
space $V^*$, in the dual basis, is the inverse of the transpose of
$(a_{ij})$.

\section{Cohomology of Comodules}

\label{CohomologyOfComodulesSection}

The relevant definitions for cohomology in $\text{Comod}_A$ make
sense in any $k$-linear abelian category, so we state them at this
level of generality.  The reader may consult \cite{weibel} or
\cite{benson} for a more thorough introduction to these matters.

\begin{defn}
Let $\catfont{C}$ be a $k$-linear abelian category, $M,N$ objects
of $\catfont{C}$.  An $n$-fold \index{extension}\textbf{extension}
of $M$ by $N$ is an exact sequence

\[ \xi: 0 \rightarrow N \rightarrow X_{n-1} \rightarrow \ldots \rightarrow
X_0 \rightarrow M \rightarrow 0 \]
Two extensions are
\index{extension!equivalence of} \textbf{equivalent} if there
exist morphisms $\phi_i:X_i \rightarrow Y_i$ such that
\begin{diagram}
0 & \rTo & N & \rTo & X_{n-1} & \rTo & \ldots & \rTo &  X_0 & \rTo
& M & \rTo & 0
\\
& & \dEq & & \dTo^{\phi_{n-1}} & & & &  \dTo^{\phi_0} & & \dEq \\
0 & \rTo & N & \rTo & Y_{n-1} & \rTo & \ldots & \rTo & Y_0 & \rTo
& M & \rTo & 0
\\
\end{diagram}
commutes.  The set of equivalence classes (with respect to the
equivalence relation generated by the relation `being equivalent')
of $n$-fold extensions of $M$ by $N$ is denoted
\index{$\text{Ext}^n(M,N)$} $\text{Ext}^n(M,N)$.
\end{defn}

In case $n > 1$, the term `equivalent' is abusive; it is a not
necessarily symmetric or transitive relation.  Nonetheless, this
relation generates a unique equivalence relation, and it is this
relation with respect to which $\text{Ext}^n(M,N)$ is defined.

On the other hand, this relation \emph{is} a bona fide equivalence
relation on $\text{Ext}^1(M,N)$.  This is because, as can be
shown, the map $\phi_0$ given by the definition of equivalence
must necessarily be an isomorphism.

Let
\[ \xi: 0 \rightarrow N \rightarrow  X_{n-1} \rightarrow \ldots \rightarrow
X_0 \rightarrow M \rightarrow 0 \]
\[ \chi: 0 \rightarrow N  \rightarrow Y_{n-1} \rightarrow \ldots \rightarrow
Y_0 \rightarrow M \rightarrow 0 \] be two $n$-fold extensions of
$M$ by $N$ with $n >1$.  The \index{extension!Baer sum of}
\textbf{Baer sum} of $\xi$ and $\chi$, denoted $\xi \oplus \chi$,
is the extension gotten as follows. Let $\Gamma$ and $\Omega$ be
the pullback/pushout respectively of the following two diagrams
\[
\begin{diagram}
\Omega & \rTo & Y_0 \\
\dTo & & \dTo \\
X_0 & \rTo & M \\
\end{diagram}
\hspace{1.5cm}
\begin{diagram}
\Gamma & \lTo & Y_{n-1} \\
\uTo & & \uTo \\
X_{n-1} & \lTo & N \\
\end{diagram}
\]
Then $\xi \oplus \chi$ is the extension
\[ 0 \rightarrow N \rightarrow \Gamma \rightarrow X_{n-2} \oplus
Y_{n-2}  \rightarrow \ldots \rightarrow X_1 \oplus Y_1 \rightarrow
\Omega \rightarrow M \rightarrow 0 \]

For $n=1$, the Baer sum is defined slightly differently.  Let
\[ \xi:0 \rightarrow N \stackrel{\phi_1}{\longrightarrow} X_1 \stackrel{\psi_1}{\longrightarrow} M \rightarrow 0
\]
\[\chi:0 \rightarrow N \stackrel{\phi_2}{\longrightarrow} X_2 \stackrel{\psi_2}{\longrightarrow} M \rightarrow 0
\]
be two $1$-fold extensions of $M$ by $N$.  Let $X$ be the pullback
of $X_1$ and $X_2$ under $M$, and $\phi^\prime,\bar{\phi}$ the
unique maps making
\[
\begin{diagram}
N & &  & & \\
& \rdTo(4,2)^{\phi_1} \rdTo(2,4)_{-\phi_2}  \rdDashto^{\phi^\prime} &   & & \\
 & & X & \rTo_{\pi_1} & X_1 \\
& & \dTo_{\pi_2} & & \dTo_{\psi_1} \\
& & X_2 & \rTo_{\psi_2} & M \\
\end{diagram}
\hspace{1.5cm}
\begin{diagram}
N & &  & & \\
& \rdTo(4,2)^{\phi_1} \rdTo(2,4)_{0}  \rdDashto^{\bar{\phi}} &   & & \\
 & & X & \rTo_{\pi_1} & X_1 \\
& & \dTo_{\pi_2} & & \dTo_{\psi_1} \\
& & X_2 & \rTo_{\psi_2} & M \\
\end{diagram}
\]
commute.  Let $X \stackrel{\tau}{\longrightarrow} Y$ be the
cokernel of $\phi^\prime$, and $\psi$ the unique map making
\begin{diagram}
N & \rTo^{\phi^\prime} & X & \rTo^{\pi_1 \psi_1} & M \\
& & \dTo^\tau & \ruDashto_\psi \\
& & Y \\
\end{diagram}
commute.  Set $\phi = \bar{\phi} \tau$.  Then the Baer sum $\xi
\oplus \chi$ is the extension
\[ 0 \rightarrow N \stackrel{\phi}{\longrightarrow} Y
\stackrel{\psi}{\longrightarrow} M \rightarrow 0 \]

Let $a \neq 0$ be a scalar, $\xi$ an $n$-fold extension as above,
and let $M \stackrel{\phi}{\longrightarrow} X_{n-1}$ be the first
map in the extension. Then we define the \index{extension!scalar
multiplication of} \textbf{scalar multiplication} of $a$ on $\xi$
to be the extension
\[ a\xi: 0 \rightarrow M \stackrel{\frac{1}{a} \phi}{\longrightarrow} X_{n-1} \rightarrow \ldots \rightarrow
X_0 \rightarrow N \rightarrow 0 \] with all of the other maps and
objects staying the same.  For $a = 0$, we define $0 \xi$ to be
the trivial extension (defined below).

\begin{thm}
For any two objects $M$ and $N$ of a $k$-linear abelian category,
Baer sum and scalar multiplication respect equivalence classes,
and $\text{Ext}^n(M,N)$ is a vector space under those operations.
\end{thm}

The additive identity of $\text{Ext}^n(M,N)$ is called the
\textbf{trivial} or \textbf{split} extension.  In the case of
$\text{Ext}^1$, it can be identified as the equivalence class of
the extension $0 \rightarrow N \rightarrow N \oplus M \rightarrow
M \rightarrow 0$.

\chapter{Tannakian Duality}

\label{chapterTannakianDuality}

Here we present the basic facts concerning the correspondence
between categories of finite dimensional representations of affine
group schemes over a field, and so-called neutral tannakian
categories.  We have no intention of giving a full account,
especially concerning proofs; we shall mostly content ourselves
with giving the definition of a neutral tannakian category, and,
following the proof given in \cite{deligne}, describing a method
for recovering the representing Hopf algebra of such a category.
The reader may consult \cite{saavedra} or \cite{breen}, but the
development given here follows almost exclusively that of
\cite{deligne}.  The reader may also consult \cite{freyd} for an
excellent introduction to the theory of abelian categories, and
\cite{maclane} for a good account of tensor categories in the
abstract (there referred to as \emph{monoidal} categories, but
without many of the assumptions we shall be placing on them).

The theory of tannakian categories, while having broader
implications than what we will be discussing is, as far as we are
concerned, a successful attempt to answer three natural questions
about the category $\text{Rep}_k G$ (equivalently,
$\text{Comod}_A$, where $A$ is the representing Hopf algebra of
$G$).  Firstly, to what extent does purely categorical information
about the category $\text{Rep}_k G$ determine the group $G$?  The
answer is, completely, if one allows for one piece of external
information (called a \emph{fibre functor}). Secondly, can we
recover in some constructive way the Hopf algebra $A$ from
$\text{Comod}_A$? The answer again is yes, and it is to this that
we will devoting most of our time. Finally, is there a set of
axioms one can write down which are equivalent to a category being
$\text{Rep}_k G$ for some $k$ and $G$?  The answer again is yes,
and these axioms serve as the definition for a neutral tannakian
category.

\section{A Motivating Example}

To motivate the definitions given in the next section it will help
to keep in mind the simplest and yet most important example, the
category \index{$\text{Vec}_k$} $\text{Vec}_k$ consisting of all
finite dimensional vector spaces over a field $k$, with morphisms
being $k$-linear maps between the vector spaces (or, if you like,
$\text{Rep}_k G_0$, where $G_0$ is the trivial group represented
by the Hopf algebra $k$).  Firstly, $\text{Vec}_k$ is a $k$-linear
\index{category!abelian} abelian category. This means, among other
things, that the Hom-sets themselves have a $k$-linear structure
on them, and composition of morphisms is bilinear with respect to
this structure. It also means that $\text{Vec}_k$ satisfies some
nice regularity conditions, and that certain desirable
constructions are always possible within it: finite biproducts
always exist (in the form of the usual direct sum of vector
spaces), kernels and cokernels always exist, and all monomorphisms
and epimorphisms are normal (every injective map is the kernel of
its cokernel, and every surjective map is the cokernel of its
kernel).

As it happens, this is not quite enough to recover fully the fact
that $\text{Vec}_k$ is indeed $\text{Vec}_k$. Enter the
\index{category!tensor} tensor product. To every pair of vector
spaces $V$ and $W$ we assign an object called $V \otimes W$, and
to every pair of morphisms $V \stackrel{\phi}{\rightarrow} X$ and
$W \stackrel{\psi}{\rightarrow} Y$ we assign a morphism, denoted
$V \otimes W \stackrel{\phi \otimes \psi}{\longrightarrow} X
\otimes Y$.  We also have that, for every composable pair
$\phi,\psi$ and composable pair $a,b$, $(\phi \otimes a) \circ
(\psi \otimes b) = (\phi \circ \psi) \otimes (a \circ b)$. This
amounts to the assertion that $\otimes$ is a bifunctor on
$\text{Vec}_k$. This bifunctor $\otimes$ is a
\index{category!abelian tensor} bilinear functor in the sense
that, for any $c \in k$, $(c \phi + \psi) \otimes \eta$ = $c (\phi
\otimes \eta) + (\psi \otimes \eta)$, and similarly for the other
slot.

We know also that $\otimes$ is a commutative operation.  That is,
for every pair of vector spaces $A$ and $B$ there is a natural
isomorphism \index{$\text{comm}$} $A \otimes B
\stackrel{\text{comm}_{A,B}}{\isomorphic} B \otimes A$ (namely the
map $a \otimes b \mapsto b \otimes a$). Naturality here means that
for every pair of maps $A \stackrel{\phi}{\rightarrow} X$ and $B
\stackrel{\psi}{\rightarrow} Y$ the following commutes:
\begin{diagram}
A \otimes B & \rTo^{\text{comm}_{A,B}} & B \otimes A \\
\dTo^{\phi \otimes \psi} & & \dTo_{\psi \otimes \phi} \\
X \otimes Y & \rTo_{\text{comm}_{X,Y}} & Y \otimes X \\
\end{diagram}
Similarly, $\otimes$ is naturally associative, given by the
natural isomorphism \index{$\text{assoc}$} $(A \otimes B) \otimes
C \stackrel{\text{assoc}_{A,B,C}}{\isomorphic} A \otimes (B
\otimes C)$ (namely $(a \otimes b) \otimes c \mapsto a \otimes (b
\otimes c)$).

$\text{Vec}_k$ also has an \index{identity object} identity object
for $\otimes$, namely the vector space $k$.  This means that to
every vector space $V$, there is a natural isomorphism
\index{$\text{unit}$} $V \stackrel{\text{unit}_V}{\isomorphic} k
\otimes V$ (namely $v \mapsto 1 \otimes v$), satisfying diagrams
analogous to the above. In the context of abstract tannakian
categories this identity object is denoted as
\index{$\underline{1}$} $\underline{1}$.

We also mention that the isomorphisms $\text{comm}$ and
$\text{assoc}$ satisfy some coherence conditions with one another.
These are expressed by the so-called pentagon and hexagon axioms,
to be discussed in the next section.

Recall the universal bilinear mapping property of the tensor
product.  For every $V$ and $W$ there is a bilinear map $\otimes:
V \times W \rightarrow V \otimes W$ with the following property:
to every bilinear map $V \times W \stackrel{\phi}{\rightarrow} Z$
there is a unique linear map $V \otimes W
\stackrel{\psi}{\rightarrow} Z$ such that the following diagram
commutes:
\begin{diagram}
V \otimes W & \rTo^\psi & Z \\
\uTo^\otimes & \ruTo_\phi & \\
V \times W & & \\
\end{diagram}
What this gives us is a natural isomorphism between linear maps on
$V \otimes W$ and bilinear maps on $V \times W$.  But a bilinear
map $V \times W \rightarrow Z$ is just another name for a linear
map $V \rightarrow \text{Hom}(W,Z)$.  Thus we have an isomorphism
\[ \text{Hom}(V \otimes W,Z) \isomorphic \text{Hom}(V,
\text{Hom}(W,Z)) \] Another way of stating this is that
$\text{Hom}(\slot \otimes W,Z)$ is a representable functor, and
its representing object is $\text{Hom}(W,Z)$.  $\text{Vec}_k$
enjoys the property that, for any objects $V$ and $W$,
$\text{Hom}(V,W)$ is also an object of $\text{Vec}_k$.  For a
category in which this is not exactly the case, we have a
different name, when it exists, for the representing object of
$\text{Hom}(\slot \otimes W,Z)$: we call it \index{internal Hom}
\emph{internal Hom}, and denote it by \index{$\intHom(V,W)$}
$\intHom(W,Z)$. The above discussion can thus be summed up as
saying that, in $\text{Vec}_k$, internal Homs always exist.

Of special interest then is, for any vector space $V$, the object
$\intHom(V,\underline{1})$, which in $\text{Vec}_k$ can be
identified as the space of linear functionals on $V$.  We denote
this object by \index{internal Hom!dual} $\dual{V}$.  It is well
known that any vector space $V$ is naturally isomorphic to
$V^{\vee \vee}$ via the map $v \mapsto \text{ev}_v$, where
$\text{ev}_v$ is the map that evaluates any functional $V
\rightarrow k$ at the element $v$.  We say then that all objects
of $\text{Vec}_k$ are \index{reflexive object} \emph{reflexive} in
the sense that $v \mapsto \text{ev}_v$ is always an isomorphism.

For any vector spaces $X_1,X_2,Y_1,Y_2$, there is an obvious
isomorphism $\intHom(X_1,Y_1) \otimes \intHom(X_2,Y_2) \isomorphic
\intHom(X_1 \otimes X_2,Y_1 \otimes Y_2)$, namely the map that
sends the element $\phi \otimes \psi$ to the map of the same name.
This isomorphism can be thought of as expressing the fact that the
tensor product, acting on $\intHom(X,Y)$ as an object, is
compatible with its action on it as a Hom-set.

In $\text{Vec}_k$, $\text{End}(\underline{1}) = \text{End}(k)$ can
be identified with the field $k$ itself, if we take addition to be
addition of maps and multiplication to be composition of
morphisms.  Thus we can say that $\text{End}(\underline{1})$ is a
field. With this in mind, and given everything else we've done, we
can actually define the $k$-linear structure on $\text{Vec}_k$
without assuming it. If $c:k \rightarrow k$ is the linear map
given by multiplication by the constant $c \in k$ and $\phi:V
\rightarrow W$ any linear map, we can define $c \phi$ as the
composition
\[ V \stackrel{\text{unit}_V}{\longrightarrow} k \otimes V
\stackrel{c \otimes \phi}{\longrightarrow} k \otimes W
\stackrel{\text{unit}_W^{-1}}{\longrightarrow} W \]

Finally, there is a so-called \emph{fibre functor} on
$\text{Vec}_k$, i.e.~an exact, k-linear tensor preserving functor
from $\text{Vec}_k$ to $\text{Vec}_{\text{End}(\underline{1})}$.
Here, we can simply take this functor to be the identity
$\text{Vec}_k \rightarrow \text{Vec}_k$.

The preceding discussion amounts to the assertion that
$\text{Vec}_k$ is a neutral tannakian category, which we formally
define now.

\section{Definition of a Neutral Tannakian Category}
\label{defntannakian}
\begin{defn}
\label{defnabeliancategory} An \index{category!abelian}
\textbf{abelian category} is a category $\catfont{C}$ with the
following properties:

\begin{enumerate}
\item{For all objects $A, B \in \catfont{C}$, $\text{Hom}(A,B)$ is
endowed with the structure of an abelian group, and composition of
morphisms is bilinear with respect to this structure} \item{Every
pair of objects in $\catfont{C}$ has a biproduct} \item{Every
morphism has a kernel and a cokernel} \item{Every monomorphism is
the kernel of some morphism, and every epimorphism is the cokernel
of some morphism}

\end{enumerate}

\end{defn}

\begin{defn}
\label{defntensorcategory}
 Let \index{category!tensor} $\catfont{C}$ be a category
endowed with a bifunctor $\otimes:\catfont{C} \times \catfont{C}
\rightarrow \catfont{C}$, and denote by, for objects $A,B$ and
morphisms $\phi,\psi$, $A \otimes B
\stackrel{\text{def}}{=}\otimes(A,B)$ and $\phi \otimes \psi
\stackrel{\text{def}}{=}\otimes(\phi,\psi)$.  Then $\otimes$ is
called a \textbf{tensor product}, and $\catfont{C}$ is called a
\textbf{tensor category}, if the following hold:

\begin{enumerate}

\item{There is a functorial isomorphism \index{$\text{assoc}$}
$\text{assoc}_{X,Y,Z}:X \otimes (Y \otimes Z) \isomorphic (X
\otimes Y) \otimes Z$} \item{There is a functorial isomorphism
\index{$\text{comm}$} $\text{comm}_{X,Y}: X \otimes Y \isomorphic
Y \otimes X$} \item{There is an \index{identity object}
\textbf{identity object}, denoted \index{$\underline{1}$}
$\underline{1}$, and a functorial isomorphism
\index{$\text{unit}$} $\text{unit}_V:V \isomorphic \underline{1}
\otimes V$ inducing an equivalence of categories $\catfont{C}
\rightarrow \catfont{C}$} \item{(pentagon axiom) For all objects
$X,Y,Z$ and $T$, the following commutes:
\begin{diagram}
& & (X \otimes Y) \otimes (Z \otimes T) & & \\
& \ruTo^{\text{assoc}} & & \rdTo^{\text{assoc}} & \\
X \otimes (Y \otimes (Z \otimes T)) & & & & ((X \otimes Y) \otimes
Z) \otimes T \\
\dTo^{1 \otimes \text{assoc}} & & & & \dTo_{\text{assoc} \otimes
1} \\
X \otimes((Y\otimes Z) \otimes T) &  & \rTo_{\text{assoc}} & & (X \otimes (Y \otimes Z)) \otimes T \\
\end{diagram}
(the obvious sub-scripts on $\text{assoc}$ have been omitted) }
\item{(hexagon axiom) For all objects $X,Y$ and $Z$, the following
commutes:
\begin{diagram}
 & & (X \otimes Y) \otimes Z & & \\
 & \ruTo^{\text{assoc}} & & \rdTo^{\text{comm}} & \\
 X \otimes (Y \otimes Z) & & & & Z \otimes (X \otimes Y) \\
 \dTo^{1 \otimes \text{comm}} & & & & \dTo_{\text{assoc}} \\
 X \otimes (Z \otimes Y) & & & & (Z \otimes X) \otimes Y \\
 & \rdTo_{\text{assoc}} & & \ruTo_{\text{comm} \otimes 1} \\
 & & (X \otimes Z) \otimes Y & & \\
\end{diagram}
} \item{For all objects $X$ and $Y$, the following commute:
\[
\begin{diagram}
X \otimes Y & \rTo^{\text{unit}} & \underline{1} \otimes (X
\otimes Y) \\
& \rdTo_{\text{unit} \otimes 1} & \dTo_{\text{assoc}} \\
& & (\underline{1} \otimes X) \otimes Y \\
\end{diagram}
\hspace{2.5cm}
\begin{diagram}
X \otimes Y & \rTo^{\text{unit} \otimes 1} & (\underline{1}
\otimes X) \otimes Y
\\
\dTo^{1 \otimes \text{unit}} & & \dTo_{\text{comm} \otimes 1} \\
X \otimes (\underline{1} \otimes Y) & \rTo_{\text{assoc}}& (X
\otimes
\underline{1}) \otimes Y \\
\end{diagram}
\]
}

\end{enumerate}

\end{defn}

For reasons of convenience our definition of a tensor category is
a slight deviation from that given on page 105 of \cite{deligne}.
There conditions 3.~and 6.~above are replaced with the seemingly
weaker requirement that there exist an identity object $U$ and
isomorphism $u:U \rightarrow U \otimes U$ such that $X \mapsto U
\otimes X$ is an equivalence of categories.  However, proposition
1.3 on that same page makes it clear that Deligne's definition
implies ours, so we have not changed anything.

What we call a tensor category others might call a \emph{monoidal}
category, and our demand that it be, e.g., naturally commutative
is quite often not assumed by other authors.  Saavedra in
\cite{saavedra} would in fact call this an ACU tensor category,
indicating that is associative, commutative, and unital.  We shall
have no occasion to consider any tensor categories but this kind,
so we call them simply tensor categories.

The significance of the pentagon and hexagon axioms is that,
loosely speaking, they introduce enough constraints to ensure that
any diagram that should commute, does commute. The reader should
see \cite{maclane} and \cite{deligne} for more on this.

We also note that the identity object $\underline{1}$ and the
isomorphism $\text{unit}$ are not demanded to be unique.  However,
any two such are isomorphic up to a unique isomorphism commuting
with the unit isomorphisms, so it is unique for all intents and
purposes; see proposition 1.3 of \cite{deligne}.

We define an \index{category!abelian tensor} \textbf{abelian
tensor category} to be an abelian category equipped with a tensor
product in such a way that $\otimes$ is a bi-additive functor,
i.e.~$(\phi + \psi) \otimes a = \phi \otimes a + \psi \otimes a$,
and similarly for the other slot.

Our goal now is to define what is called a \emph{rigid abelian
tensor category}.  This is defined to be an abelian tensor
category in which all internal Homs exist, every object is
reflexive, and for all objects $X_1, X_2, Y_1, Y_2$, a certain
natural map
\[ \intHom(X_1,Y_1) \otimes \intHom(X_2,Y_2) \rightarrow
\intHom(X_1 \otimes X_2, Y_1 \otimes Y_2) \]
 is always an
isomorphism.  We define now what these things mean in purely
categorical terms.

Let $X$ and $Y$ be objects of the tensor category $\catfont{C}$
and consider the contravariant functor $\text{Hom}(\slot \otimes
X,Y)$ from $\catfont{C}$ to the category of sets. It sends any
object $T$ to $\text{Hom}(T \otimes X,Y)$ and any morphism $T
\stackrel{\phi}{\rightarrow} W$ to the morphism $\text{Hom}(W
\otimes X,Y) \stackrel{\hat{\phi}}{\rightarrow} \text{Hom}(T
\otimes X,Y)$ defined by, for $W \otimes X
\stackrel{\psi}{\rightarrow} Y$, the image of $\psi$ under
$\hat{\phi}$ is the composition
\[ T \otimes X \stackrel{\phi \otimes 1}{\longrightarrow} W
\otimes X \stackrel{\psi}{\longrightarrow} Y \]
 Suppose that, for fixed $X$ and $Y$, this functor is representable, and call the
representing object \index{$\intHom(V,W)$} $\intHom(X,Y)$.
$\intHom(X,Y)$ is by definition an \index{internal Hom}
\label{internalHomdefnPage} \textbf{internal Hom} object for $X$
and $Y$ (so called because, e.g.~in $\text{Vec}_k$, internal Hom
is just Hom). Then we have a natural isomorphism $\Omega$ going
from the functor $\text{Hom}(\slot,\intHom(X,Y))$ to the functor
$\text{Hom}(\slot \otimes X,Y)$.  If we plug in the object
$\intHom(X,Y)$ to each slot and apply $\Omega_{\intHom(X,Y)}$ to
the element $\text{id} \in \text{Hom}(\intHom(X,Y),\intHom(X,Y))$,
we get a map $\intHom(X,Y) \otimes X \rightarrow Y$, which by
definition we call $\text{ev}_{X,Y}$ (so called because, in
$\text{Vec}_k$, $\text{ev}_{X,Y}$ is the evaluation map $\phi
\otimes x \mapsto \phi(x)$).

The significance of the map $\text{ev}_{X,Y}$ is that, analyzing
the situation in light of the Yoneda lemma (page
\pageref{YonedaLemma}), we find that for any morphism $\phi$ in
$\text{Hom}(T,\intHom(X,Y))$, applying the isomorphism $\Omega_T$
yields
\[ \Omega_T(\phi) = (\phi \otimes 1) \circ \text{ev}_{X,Y} \]
which is to say, the following is always commutative:
\begin{diagram}
\label{rigidity1}
T \otimes X & & \\
\dTo^{\phi \otimes 1} & \rdTo^{\Omega_T(\phi)} & \\
\intHom(X,Y) \otimes X & \rTo_{\text{ev}_{X,Y}} & Y \\
\end{diagram}

We suppose now that for any objects $X$ and $Y$, $\intHom(X,Y)$
exists. We define the \index{internal Hom!dual} \textbf{dual} of
$X$, denoted $\dual{X}$, to be $\intHom(X,\underline{1})$, where
$\underline{1}$ is the identity object for our tensor category,
and we simply write $\text{ev}_X$ for
$\text{ev}_{X,\underline{1}}$, which is a map $\dual{X} \otimes X
\rightarrow \underline{1}$.  In $\text{Vec}_k$, $\text{ev}_X$ is
the familiar map $\phi \otimes x \mapsto \phi(x)$.

We now proceed to define a map $\iota_X:X \rightarrow
X^{\vee\vee}$, which in $\text{Vec}_k$ will correspond to the
usual evaluation map $x \mapsto (\phi \mapsto \phi(x))$. We have,
for any $X$, an isomorphism $\Omega_X$ between
\[ \text{Hom}(X,X^{\vee \vee})
\stackrel{\Omega_X}{\longrightarrow} \text{Hom}(X \otimes
X^{\vee},\underline{1}) \] Define $\iota_X:X \rightarrow X^{\vee
\vee}$ to be the map on the left hand side of the above
isomorphism which corresponds to the composition $X \otimes
\dual{X} \stackrel{\text{comm}}{\longrightarrow} \dual{X} \otimes
X \stackrel{\text{ev}_X}{\longrightarrow} \underline{1}$ on the
right hand side.  In other words, $\iota_X$ is the unique map
making the following commute:
\begin{equation}
\label{rigidity2}
\begin{diagram}
X \otimes \dual{X}  & & & & \\
 & \rdTo^{\text{comm}} & & & \\
 \dTo^{\iota_X \otimes 1} & & \dual{X} \otimes X & & \\
 & & & \rdTo^{\text{ev}_X} & \\
 X^{\vee \vee} \otimes X^{\vee} &  & \rTo_{\text{ev}_{X^\vee}} & &
 \underline{1} \\
 \end{diagram}
\end{equation}
We define an object $X$ of $\catfont{C}$ to be \index{reflexive
object} \label{reflexiveObjectdefnPage} \textbf{reflexive} if the
map $\iota_X$ just constructed is an isomorphism.

Consider the following composition:
\begin{gather}
\label{CommAssocCompEquation} (\intHom(X_1, Y_1) \otimes
\intHom(X_2,Y_2)) \otimes (X_1 \otimes
X_2) \isoarrow \\
(\intHom(X_1,Y_1) \otimes X_1) \otimes (\intHom(X_2,Y_2) \otimes
X_2)) \stackrel{\text{ev} \otimes \text{ev}}{\longrightarrow} Y_1
\otimes Y_2 \nonumber
\end{gather}
 where the first isomorphism is the obvious one built by
application of the comm and assoc isomorphisms.  Call the above
composition $\Psi$.  Then there is a unique map, call it $\Phi$,
making the following commute:
\begin{equation}
\begin{diagram}
\label{rigidity3}
 (\intHom(X_1,Y_1) \otimes \intHom(X_2,Y_2))
\otimes (X_1 \otimes
X_2) & & \\
\dTo^{\Phi \otimes 1} & \rdTo^\Psi & \\
\intHom(X_1 \otimes X_2,Y_1 \otimes Y_2) \otimes (X_1 \otimes X_2)
& \rTo_{\text{ev}} & Y_1 \otimes Y_2 \\
\end{diagram}
\end{equation}

\begin{defn}
\label{defnrigidity}
 An abelian tensor category is called
\index{category!rigid abelian tensor}\textbf{rigid} if
$\intHom(X,Y)$ exists for all $X$ and $Y$, all objects are
reflexive, and for any quadruple $X_1, X_2, Y_1, Y_2$, the map
$\Phi$ just constructed is an isomorphism.
\end{defn}

Next consider the Hom-set $\text{End}(\underline{1})$.  As
composition is demanded to be bilinear with respect to the
additive structure on Hom-sets, $\text{End}(\underline{1})$ is a
ring with unity under addition and composition.  We say then

\begin{defn}
A rigid abelian tensor category is called a
\index{category!tannakian} \textbf{tannakian category} if the ring
$\text{End}(\underline{1})$ is a field.
\end{defn}

Assume now that $\catfont{C}$ is tannakian and let $k$ be the
field $\text{End}(\underline{1})$. Then $\catfont{C}$ has a
$k$-linear structure forced upon it as follows.  If $c$ is an
element of $k$ (that is, a morphism $\underline{1}
\stackrel{c}{\rightarrow} \underline{1}$) and $\phi$ is a morphism
from $V$ to $W$, then we define the scalar multiplication $c \phi$
to be the composition
\[ V \stackrel{\text{unit}_V}{\longrightarrow} \underline{1}
\otimes V \stackrel{c \otimes \phi}{\longrightarrow} \underline{1}
\otimes W \stackrel{\text{unit}_W^{-1}}{\longrightarrow} W \] As
$\otimes$ acts bilinearly on morphisms, we have $c(\phi+\psi) =
c\phi + c\psi$ for all $c \in \text{End}(\underline{1})$,
$\phi,\psi \in \text{Hom}(V,W)$.  Thus, $\text{Hom}(V,W)$ is a
vector space over $k$.

If $\catfont{C}$ is any category, then being tannakian is not
quite enough for us to conclude that it is the category of
representations of some affine group scheme over the field
$\text{End}(\underline{1})$. We need the additional fact that
objects of $\catfont{C}$ can, in some sense, be `thought of' as
concrete finite dimensional vector spaces, and morphisms as
concrete linear maps between them.  This is the role fulfilled by
a \emph{fibre functor}, a certain kind of functor from
$\catfont{C}$ to $\text{Vec}_k$, where $k$ is the field
$\text{End}(\underline{1})$.

We need to define first what is meant by a \emph{tensor functor}
$F:\catfont{C} \rightarrow \catfont{D}$, where $\catfont{C}$ and
$\catfont{D}$ are tensor categories.  We denote with the same
symbol $\otimes$ the tensor product in both categories.  We denote
by $\text{assoc}$ the requisite associativity isomorphism in
$\catfont{C}$, and by $\text{assoc}^\prime$ that in $\catfont{D}$;
similarly for the natural isomorphisms comm, unit, and the
identity object $\underline{1}$.

\begin{defn}
\label{defntensorfunctor} \index{tensor functor}
 Let $\catfont{C}$ and $\catfont{D}$ be
tensor categories, and $F:\catfont{C} \rightarrow \catfont{D}$ a
functor. $F$ is a \textbf{tensor functor} if there is a functorial
isomorphism $c_{X,Y}:F(X) \otimes F(Y) \isoarrow F(X \otimes Y)$
satisfying

\begin{enumerate}
\item{For all objects $X,Y,Z$ of $\catfont{C}$, the following
commutes:
\begin{diagram}
 & & F(X) \otimes F(Y \otimes Z) & & \\
 & \ruTo^{1 \otimes c} & & \rdTo^{c} & \\
F(X) \otimes (F(Y) \otimes F(Z)) & & & & F(X \otimes (Y \otimes Z)) \\
 \dTo^{\text{assoc}^\prime} & & & & \dTo_{F(\text{assoc})} \\
 (F(X) \otimes F(Y)) \otimes F(Z) & & & & F((X \otimes Y) \otimes Z) \\
 & \rdTo_{c \otimes 1} & & \ruTo_{c} \\
 & & F(X \otimes Y) \otimes F(Z) & & \\
\end{diagram}
} \item{For all objects $X$ and $Y$ of $\catfont{C}$, the
following commutes:
\begin{diagram}
F(X) \otimes F(Y) & \rTo^c & F(X \otimes Y) \\
\dTo^{\text{comm}^\prime} & & \dTo_{F(\text{comm})} \\
F(Y) \otimes F(X) & \rTo_c & F(Y \otimes X) \\
\end{diagram}
} \item{If $(\underline{1},\text{unit})$ is an identity object of
$\catfont{C}$, then there is an identity object
$(\underline{1}^\prime,\text{unit}^\prime)$ of $\catfont{D}$ such
that $F(\underline{1}) = \underline{1}^\prime$, and for every
object $X$ of $\catfont{C}$, $\text{unit}^\prime_{F(x)} =
F(\text{unit}_X) \circ c_{X,\underline{1}}$.}
\end{enumerate}
\end{defn}

Again, condition 3.~appears stronger than condition (c) on page
114 of \cite{deligne}, but they are actually equivalent, again by
proposition 1.3 of \cite{deligne}.

If $\catfont{C}$ is any tannakian category over the field $k =
\text{End}(\underline{1})$ then we define a \index{fibre functor}
\textbf{fibre functor} on $\catfont{C}$, usually denoted
\index{$\omega$} $\omega$, to be any exact, faithful, $k$-linear
tensor functor from $\catfont{C}$ to $\text{Vec}_k$.  We can
finally define

\begin{defn}
\label{defnneutraltannakiancategory} A \index{category!neutral
tannakian} \textbf{neutral tannakian category} is a tannakian
category equipped with a fibre functor.
\end{defn}

\section{Recovering an Algebraic Group from a Neutral Tannakian
Category}

\label{recoveringanalgebraicgroupsection}

Let $G$ be an affine group scheme over the field $k$ and let
$\omega:\text{Rep}_k G \rightarrow \text{Vec}_k$ be the forgetful
functor, i.e.~the functor which sends every representation of $G$
to its underlying $k$-vector space, and every map to itself.  If
$R$ is a $k$-algebra, we define $\text{Aut}^\otimes(\omega)(R)$ to
be the collection of tensor preserving automorphisms of the
functor $\omega^R: \text{Rep}_k G \rightarrow \text{Mod}_R$.  Here
$\omega^R$ is the functor which sends any object $X$ of
$\text{Rep}_k G$ to the $R$-module $X \otimes R$ (strictly
speaking we should write $\omega(X) \otimes R$, but we
deliberately confuse $X$ with its underlying vector space
$\omega(X)$ to keep the notation simple), and sends any morphism
$X \stackrel{\phi}{\rightarrow} Y$ to the $R$-linear map $X
\otimes R \stackrel{\phi \otimes 1}{\longrightarrow} Y \otimes R$.
To be more explicit

\begin{defn}
An element of $\text{Aut}^\otimes(\omega)(R)$ is a family
$(\lambda_X:X \in \text{Rep}_k G)$, where each $\lambda_X$ is an
$R$-linear automorphism of $X \otimes R$, subject to
\begin{enumerate}
\item{$\lambda_{\underline{1}}$ is the identity map on $R
\isomorphic k \otimes R$} \item{$\lambda_{X \otimes Y} = \lambda_X
\otimes \lambda_Y$ for all $X, Y \in \text{Rep}_k G$} \item{For
all morphisms $X \stackrel{\phi}{\rightarrow} Y$ in $\text{Rep}_k
G$, the following commutes:
\begin{diagram}
X \otimes R & \rTo^{\lambda_X} & X \otimes R \\
\dTo^{\phi \otimes 1} & & \dTo_{\phi \otimes 1} \\
Y \otimes R & \rTo_{\lambda_Y} & Y \otimes R \\
\end{diagram}
}

\end{enumerate}

\end{defn}

If $g$ is an element of the group $G(R)$ then it is trivial to
verify that $g$ defines an element of
$\text{Aut}^\otimes(\omega)(R)$, just by working through the
definitions. If, for $X \in \text{Rep}_k G$, we write $g_X$ for
the automorphism $X \otimes R \rightarrow X \otimes R$ defined by
the representation $G \rightarrow \text{GL}(X)$, then the above
three requirements are all satisfied. $\underline{1}$ is of course
the trivial representation in $\text{Rep}_k G$, so $g$ acts
identically on $\underline{1}$ by definition. The action of $g_{X
\otimes Y}$ on $X \otimes Y$ is defined by the equation $g_{X
\otimes Y} = g_X \otimes g_Y$, giving us 2., and 3.~is true since
morphisms in $\text{Rep}_k G$ must by definition commute with the
action of $g$.

Thus, we have a natural map from the functor $G$ to the functor $
\text{Aut}^\otimes(\omega)$.  We can now state one half of the
principle of tannakian duality:

\begin{thm}(proposition $2.8$ of \cite{deligne})
If $G$ is an affine group scheme over the field $k$, then the
natural map of functors $G \rightarrow \text{Aut}^\otimes(\omega)$
is an isomorphism.
\end{thm}

Stated more plainly: the only tensor preserving automorphisms of
the functor $\omega^R$ are ones that are given by elements of
$G(R)$. We see then that the category $\text{Rep}_k G$, along with
the forgetful functor $\omega$, completely determines the group
$G$: it can be recovered as the affine group scheme
$\text{Aut}^\otimes(\omega)$.

This first half of our main theorem points the way to the second
half.  Starting with an abstract neutral tannakian category
$\catfont{C}$ with fibre functor $\omega$, the functor $G =
\text{Aut}^\otimes(\omega)$ is itself an affine group scheme such
that $\catfont{C}$ is tensorially equivalent to $\text{Rep}_k G$.
That is:

\begin{thm}(theorem $2.11$ of \cite{deligne})
\index{tannakian duality} \label{tannakiandualitythm} Let
$\catfont{C}$ be a \index{category!neutral tannakian} neutral
tannakian category over the field $k$ with fibre functor $\omega$.
Then

\begin{enumerate}
\item{The functor $\text{Aut}^\otimes(\omega)$ on $k$-algebras is
representable by an \index{affine group scheme} affine group
scheme $G$} \item{$\omega$ defines an equivalence of tensor
categories between $\catfont{C}$ and \index{$\text{Rep}_k
G$}$\text{Rep}_k G$}
\end{enumerate}
\end{thm}

This is the principle of tannakian duality.

The remainder of this section is devoted to, following the proof
of the above theorem given in \cite{deligne}, giving an
`algorithm' of sorts for recovering the representing Hopf algebra
$A$ from an abstract neutral tannakian category.  We shall not
justify most of the steps taken; the interested reader should see
the actual proof for this.  For the remainder, $\catfont{C}$ will
denote a fixed neutral tannakian category over the field $k$ with
fibre functor $\omega$. We denote the image of the object $X$
under the functor $\omega$ as $\omega(X)$, and the image of a
morphism $\phi$ under $\omega$ simply as $\phi$.  The usual names
for $\underline{1}$, comm, assoc, $\otimes$, etc.~also hold here.

\begin{defn}(see page $133$ of \cite{deligne})
\index{principal subcategory}  For an object $X$ of $\catfont{C}$,
\index{$\gen{X}$} $\gen{X}$, the \textbf{principal subcategory}
generated by $X$, is the full subcategory of $\catfont{C}$
consisting of those objects which are isomorphic to a subobject of
a quotient of $X^n = X \oplus \ldots \oplus X$ for some $n$.
\end{defn}

Note firstly that $\gen{X}$ is not itself a tannakian category, in
general not being closed under the tensor product; it is however a
$k$-linear abelian category.  Note that $Y \in \gen{X}$ if and
only if $\gen{Y} \subset \gen{X}$.  We can say then that
$\catfont{C}$ is the direct limit of its principal subcategories,
with the direct system being the inclusions $\gen{Y} \subset
\gen{X}$ when applicable.

\begin{defn}(see lemma $2.13$ of \cite{deligne})
\index{$\text{End}(\omega \res \gen{X})$} \label{endomegadefn} For
$X$ an object of $\catfont{C}$, we define $\text{End}(\omega \res
\gen{X})$, the collection of all \index{endomorphisms of a fibre
functor} \textbf{endomorphisms of the fibre functor} $\omega$
restricted to $\gen{X}$, to consist of families $\lambda =
(\lambda_Y: Y \in \gen{X})$ such that $\lambda_Y:\omega(Y)
\rightarrow \omega(Y)$ is a $k$-linear map, and for every
$\catfont{C}$-morphism $Y \stackrel{\phi}{\rightarrow} Z$, the
following commutes:
\begin{diagram}
\omega(Y) & \rTo^{\lambda_Y} & \omega(Y) \\
\dTo^{\phi} & & \dTo_\phi \\
\omega(Z) & \rTo_{\lambda_Z} & \omega(Z) \\
\end{diagram}

\end{defn}

An important point, used often in this dissertation, is the fact
that every $\lambda \in \text{End}(\omega \res \gen{X})$ is
determined by $\lambda_X$. If $\iota_i:X \rightarrow X^n$ is the
$i^\text{th}$ inclusion map then
\begin{diagram}
\omega(X) & \rTo^{\iota_i} & \omega(X^n) \\
\dTo^{\lambda_X} & & \dTo_{\lambda_{X^n}} \\
\omega(X) & \rTo_{\iota_i} & \omega(X^n) \\
\end{diagram}
must commute for every $i$, which clearly forces $\lambda_{X^n} =
\lambda_X^n$.  If $Y$ is a subobject of some $X^n$ with $Y
\stackrel{\iota}{\longrightarrow} X^n$ injective, then $\lambda_Y$
must commute with
\begin{diagram}
\omega(Y) & \rInto^\iota & \omega(X^n) \\
\dTo^{\lambda_Y} & & \dTo_{\lambda_X^n} \\
\omega(Y) & \rInto^\iota & \omega(X^n) \\
\end{diagram}
and since $\iota$ is injective, $\lambda_Y$ is unique in this
respect.  Finally, if $Z$ is a quotient of some $Y \in \gen{X}$
with $Y$ a subobject of $X^n$, then we have a surjective map $Y
\stackrel{\pi}{\longrightarrow} Z$ and the commutative diagram
\begin{diagram}
\omega(\text{ker}(\pi)) & \rInto^\iota & \omega(Y) & \rOnto^\pi & \omega(Z) \\
\dTo^{\lambda_{\text{ker}(\pi)}} & & \dTo^{\lambda_Y} & & \dTo_{\lambda_Z} \\
\omega(\text{ker}(\pi)) & \rInto^\iota & \omega(Y) & \rOnto^\pi & \omega(Z) \\
\end{diagram}
with $\iota$ the inclusion of the kernel of $\pi$ into $Y$.  By
commutativity of the left square $\lambda_Y$ must stabilize
$\omega(\text{ker}(\pi))$.  This shows that there is at most one
$\lambda_Z$ making this diagram commute, hence $\lambda_Z$ is
determined by $\lambda_X$ as well.

Therefore it does no harm to confuse $\text{End}(\omega \res
\gen{X})$ with $\{\lambda_X : \lambda \in \text{End}(\omega \res
\gen{X})\}$, its image in $\text{End}(\omega(X))$; we refer to
this subalgebra of $\text{End}(\omega(X))$ as \index{$L_X$}
\label{LXdefn} $L_X$.

\label{endomegadiscussion}

Now suppose that $X \in \gen{Y}$, which is the same as saying
$\gen{X} \subset \gen{Y}$.  If $\lambda \in \text{End}(\omega \res
\gen{Y})$ it is straightforward to check that $\lambda_X \in L_X$.
This gives, for every such $X$ and $Y$, a canonical map from $L_Y$
to $L_X$, denoted \index{$T_{X,Y}$} $T_{X,Y}$; we call this the
\index{transition mapping} \textbf{transition mapping} from $L_Y$
to $L_X$. It is clear from the definition that, for $X \in
\gen{Y}$ and $Y \in \gen{Z}$, then also $X \in \gen{Z}$, and the
diagram
\begin{diagram}
L_Z & \rTo^{T_{Y,Z}} & L_Y \\
& \rdTo_{T_{X,Z}} & \dTo_{T_{X,Y}} \\
& & L_X \\
\end{diagram}
commutes, which give the $L_X, X \in \catfont{C}$ the structure of
an inverse system.

For each $X \in \catfont{C}$ let $B_X$ be the dual coalgebra to
$L_X$. Then from the $k$-algebra maps $L_Y
\stackrel{T_{X,Y}}{\longrightarrow} L_X$ we get $k$-coalgebra maps
$B_X \stackrel{T_{X,Y}^\circ}{\longrightarrow} B_Y$.  Thus, for
objects $X,Y$ and $Z$ of $\catfont{C}$ with $X \in \gen{Y}$ and $Y
\in \gen{Z}$, the diagram
\begin{diagram}
B_Z & \lTo^{T_{Y,Z}^\circ} & B_Y \\
& \luTo_{T_{X,Z}^\circ} & \uTo_{T_{X,Y}^\circ} \\
& & B_X \\
\end{diagram}
commutes, giving the $B_X, X \in \catfont{C}$ the structure of a
direct system.  Then let
\[ B = \dirlim_{X \in \catfont{C}} B_X \]
\label{pagerefForRecoveringA_TannakianDuality} be its direct
limit;  this $B$ is the underlying coalgebra of our eventual Hopf
algebra.

We now define an equivalence of categories $F:\catfont{C}
\rightarrow \text{Comod}_B$ which carries the fibre functor
$\omega$ into the forgetful functor $\text{Comod}_B \rightarrow
\text{Vec}_k$, that is, such that the diagram
\begin{diagram}
\catfont{C} & \rTo^F & \text{Comod}_B \\
 & \rdTo_{\omega} & \dTo_{\text{forget}} \\
 & & \text{Vec}_k \\
\end{diagram}
commutes.  Let $A$ be a finite dimensional $k$-algebra, $A^\circ$
its dual coalgebra, and $V$ a finite dimensional $k$-vector space.
We define a map $\text{Hom}_k(A \otimes V,V)
\stackrel{\Phi}{\longrightarrow} \text{Hom}_k(V,V\otimes A^\circ)$
 as follows.  For $\rho \in
\text{Hom}(A \otimes V,V)$, $\Phi(\rho)$ is the composition
\begin{gather*}
 V \stackrel{\isomorphic}{\longrightarrow} k \otimes V
\stackrel{\text{diag} \otimes \text{Id}}{\vvlongrightarrow}
\text{End}_k(A^\circ) \otimes V \stackrel{\isomorphic \otimes
\text{Id}}{\vvlongrightarrow} A^{\circ *} \otimes A^\circ \otimes
V \\
 \stackrel{\text{Id} \otimes \text{Twist}}{\vvlongrightarrow}
A^{\circ *} \otimes V \otimes A^\circ \stackrel{\isomorphic
\otimes \text{Id} \otimes \text{Id}}{\vvlongrightarrow} A \otimes
V \otimes A^\circ \stackrel{\rho \otimes
\text{Id}}{\vvlongrightarrow} V \otimes A^\circ
\end{gather*}
Reading from left to right, the various maps occurring in this
composition are defined as: $\isomorphic$ is the canonical
isomorphism $V \isomorphic k \otimes V$, $\text{diag}$ is the map
that sends $1 \in k$ to $\text{Id} \in \text{End}_k(A^\circ)$,
$\isomorphic$ is the canonical isomorphism $\text{End}_k(A^\circ)
\isomorphic A^{\circ *} \otimes A^\circ$, $\text{Twist}$ is the
obvious commutativity isomorphism, and $\isomorphic$ is the
canonical isomorphism $A^{\circ *} \isomorphic A$.

\begin{lem}
Let $A$ be a finite dimensional $k$-algebra, $A^\circ$ its dual
coalgebra, and $V$ a finite dimensional $k$-vector space.  Then
the map $\Phi:\text{Hom}_k(A \otimes V,V) \longrightarrow
\text{Hom}_k(V,V\otimes A^\circ)$ just defined is a bijection.
Further, $\rho \in \text{Hom}_k(A \otimes V,V)$ defines a valid
$A$-module structure on $V$ if and only if $\Phi(\rho) \in
\text{Hom}_k(V,V \otimes A^\circ)$ defines a valid
$A^\circ$-comodule structure on $V$.

\end{lem}

The reader should see proposition 2.2.1 of \cite{hopfalgebras} for
a proof of this fact.  However, be aware that our map $\Phi$ is
actually the inverse of the map they consider, and we have
replaced the coalgebra $C$ and dual algebra $C^*$ with the
coalgebra $A^\circ$ and algebra $A^{\circ *} \isomorphic A$.

For any $X \in \catfont{C}$ the vector space $\omega(X)$ is in the
obvious way a module for the $k$-algebra $L_X$. Then according to
the previous lemma $\omega(X)$ carries with it also the structure
of a comodule over $B_X$, call it $\rho_X$.  Then if $B_X
\stackrel{\phi_X}{\longrightarrow} B$ is the canonical map given
by the definition of $B$ as a direct limit, we get a $B$-comodule
structure on $\omega(X)$, call it $\rho$, via the composition
\[ \rho: \omega(X) \stackrel{\rho_X}{\longrightarrow} \omega(X) \otimes B_X
\stackrel{1 \otimes \phi_X}{\vlongrightarrow} \omega(X) \otimes B
\]
If $X \in \gen{Y}$ for some $Y$, then similarly $\omega(X)$ is a
module over $L_Y$ (via the transition mapping $L_Y
\stackrel{T_{X,Y}}{\longrightarrow} L_X$), hence a comodule over
$B_Y$, and yet again over $B$.  The various commutativities of the
relevant diagrams ensure that we will get the same $B$-comodule
structure on $\omega(X)$ no matter which principal subcategory we
consider it to be an object of.  We therefore define the image of
the object $X$ under the functor $F$ to be
\[ F(X) = (\omega(X),\rho) \]
That is, the $B$-comodule with underlying vector space $\omega(X)$
and comodule map $\rho:\omega(X) \rightarrow \omega(X) \otimes B$
just defined.

It is tedious but straightforward to argue that, if $X
\stackrel{\phi}{\longrightarrow} Y$ is a morphism in the category
$\catfont{C}$, then working through the definitions of $L_X$,
$B_X$ and $B$, (the image under the fibre functor of) $\phi$ is
actually a map of $B$-comodules.  We therefore define the image of
a morphism $\phi$ under $F$ to be the same map between the vector
spaces $\omega(X)$ and $\omega(Y)$.

\begin{thm}
\label{thegeneralequivalencethm} The functor $F:\catfont{C}
\rightarrow \text{Comod}_B$ just defined is an equivalence of
categories.
\end{thm}

That $F$ is a faithful functor is clear from the fact that
$\omega$ is as well.  The claim that $F$ is both full and
essentially surjective is however by no means obvious; see
proposition 2.14 of \cite{deligne} for a proof of this.

We have thus far recovered a $k$-coalgebra $B$ and an equivalence
between our abstract neutral tannakian category $\catfont{C}$ and
$\text{Comod}_B$.  What remains is to recover the multiplication
on $B$.  As the usual tensor product on comodules over a Hopf
algebra is defined in terms of its multiplication, it is not
surprising that we should turn the process around to recover the
multiplication from the tensor product.

Let $B$ be a $k$-coalgebra and $u: B \otimes_k B \rightarrow B$ be
any $k$-homomorphism. Then we can define a bifunctor $\phi^u:
\text{Comod}_B \times \text{Comod}_B \rightarrow \text{Comod}_B$
as follows: it sends the pair of comodules $(X,\rho),(Y,\mu)$ to
the comodule $\phi^u(X,Y)$ having underlying vector space $X
\otimes_k Y$ and comodule map given by the composition
\[ X \otimes Y \stackrel{\rho \otimes \mu}{\vlongrightarrow} X
\otimes B \otimes Y \otimes B \stackrel{1 \otimes \text{Twist}
\otimes 1}{\vvlongrightarrow} X \otimes Y \otimes B \otimes B
\stackrel{1 \otimes 1 \otimes u}{\vlongrightarrow} X \otimes Y
\otimes B \] (In case $B$ is a Hopf algebra and $u$ is mult, then
this is by definition the tensor product on $\text{Comod}_B$.)
What is not quite obvious is that in fact all such bifunctors
arise in this fashion.

\begin{prop}(see proposition 2.16 of \cite{deligne})
\label{tensortomultprop} For any $k$-coalgebra $B$, the map $u
\mapsto \phi^u$ defines a bijective correspondence between the set
of all $k$-homomorphisms $u:B \otimes B \rightarrow B$, and the
set of all bifunctors $F:\text{Comod}_B \times \text{Comod}_B
\rightarrow \text{Comod}_B$ having the property that the
underlying vector space of $F(X,Y)$ is the tensor product of the
underlying vector spaces of $X$ and $Y$.
\end{prop}

So then, let us define a bifunctor on $\text{Comod}_B$ which for
the moment we call $\square$. $F$ is an equivalence, so it has an
`inverse' functor, call it $F^{-1}$.  Then for two $B$-comodules
$S$ and $T$, we set
\[ S \hspace{.05cm} \square \hspace{.05cm} T \stackrel{\text{def}}{=} F(F^{-1}(S) \otimes
F^{-1}(T)) \]
where $\otimes$ on the right refers to the given
tensor structure on $\catfont{C}$.  We define $\square$ to act on
morphisms in an analogous fashion.  As the diagram
\begin{diagram}
\catfont{C} & \rTo^F & \text{Comod}_B \\
 & \rdTo_{\omega} & \dTo_{\text{forget}} \\
 & & \text{Vec}_k \\
\end{diagram}
commutes, it is easy to see that $\square$ as a bifunctor
satisfies the hypothesis of the previous proposition; hence,
$\square$ is uniquely of the form $\phi^u$ for some $u:B \otimes B
\rightarrow B$.  This $u$, call it now mult, is the recovered
multiplication on $B$, finally giving $B$ the structure of a Hopf
algebra.

We close by mentioning that the necessary conditions needed for
$B$ to be a commutative Hopf algebra follow from certain
properties assumed about $\otimes$ on $\catfont{C}$. For instance,
it is the existence of the natural isomorphisms comm and assoc
which guarantees that mult should be a commutative and associative
operation, and the existence of an identity element for mult
follows from the existence of the identity object $\underline{1}$
for $\otimes$.  The interested reader should see pg.~137 of
\cite{deligne} for more on this.

\section{Recovering a Hopf Algebra in Practice}

Here we record some results which will later be useful in
computing the representing Hopf algebra for a given neutral
tannakian category, according to the method outlined in the
previous section.

\subsection{A Categorical Lemma}

Much of the work done in this dissertation entails the computing
of direct/inverse limits over very large and unwieldy collections
of objects. The following lemma will allow us at times to
drastically simplify our computations.

If $(I,\leq)$ is a directed set, we say that $(J,\leq)$ is a
\textbf{sub-directed set} if $J$ is a subset of $I$, $J$ is
directed, and whenever $j \leq k$ in $J$, $j \leq k$ in $I$.  Note
that we do not demand the converse to hold; in case it does, we
call $J$ \textbf{full}. We say the sub-directed set $J$ is
\textbf{essential} in $I$ if, for every $i \in I$, there is a $j
\in J$ such that $i \leq j$ in $I$.

\begin{lem}
\label{directlimitlemma}
 Let $\catfont{C}$ be any category, $I$ a
directed set, $\{X_i\}$ a collection of objects indexed over $I$,
and $\{X_i \stackrel{\phi_{ij}}{\longrightarrow} X_j\}$ a direct
system for the $X_i$ over $I$, and let
\begin{diagram}
& & X_I & & \\
 & \ruTo^{\phi_i} & & \luTo^{\phi_j} \\
 X_i & & \rTo_{\phi_{ij}} & & X_j \\
\end{diagram}
be the direct limit of this system. Let $J \subset I$ be a (not
necessarily full) sub-directed set, and let
\begin{diagram}
& & X_J & & \\
& \ruTo^{\psi_i} & & \luTo^{\psi_j} \\
X_i & & \rTo_{\phi_{ij}} & & X_j \\
\end{diagram}
be the direct limit over $J$.  Then if $J$ is essential in $I$,
these two direct limits are isomorphic, via a unique isomorphism
commuting with the canonical injections.

\end{lem}

\begin{proof}
It is well known that any two direct limits for the same system
are isomorphic in the above mentioned way.  Thus, we will prove
the theorem by showing that the $X_J$, the direct limit over $J$,
can also be made into a direct limit object for the $X_i$ over all
of $I$. For any $i \in I$, define a map $X_i
\stackrel{\rho_i}{\rightarrow} X_J$ as $\psi_i$ if $i \in J$, and
in case $i \notin J$, as the composition
\[ X_i \stackrel{\phi_{ij}}{\longrightarrow} X_j
\stackrel{\psi_j}{\longrightarrow} X_J \] where $j$ is any member
of $J$ such that $i \leq j$.  This is well-defined: if $k \in J$
is any other such that $i \leq k$, let $l$ be an upper bound for
$j$ and $k$ in $J$.  Then every sub-triangle of the diagram
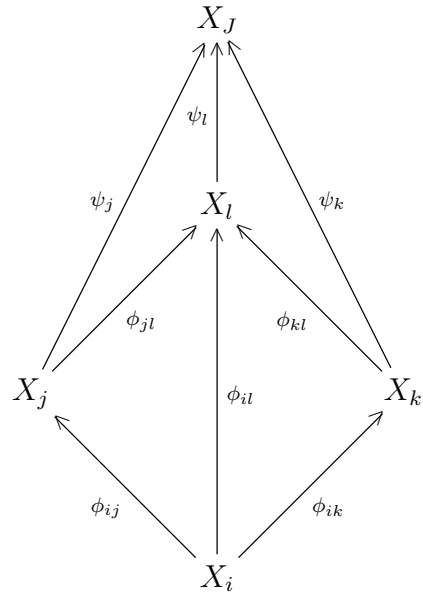
\begin{diagram}
 & & X_J & & \\
 & \ruTo(2,4)^{\psi_j} & \uTo^{\psi_l} & \luTo(2,4)^{\psi_k} & \\
 &  & X_l & & \\
 & \ruTo_{\phi_{jl}} & & \luTo_{\phi_{kl}} & \\
 X_j & & \uTo_{\phi_{il}} & & X_k \\
 & \luTo_{\phi_{ij}} & & \ruTo_{\phi_{ik}} & \\
 & & X_i & & \\
\end{diagram}
commutes, and thus so does the outermost diamond.

We claim that with these $X_i \stackrel{\rho_i}{\rightarrow} X_J$,
$X_J$ is a direct limit for the $X_i$ over all of $I$.  Let $Y$ be
any object with morphisms $X_i \stackrel{t_i}{\longrightarrow} Y$
such that, for every $i \leq j \in I$, the following commutes:
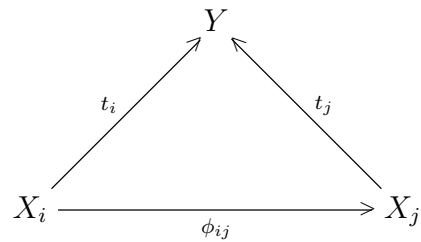
\begin{diagram}
 & & Y & & \\
 & \ruTo^{t_i} & & \luTo^{t_j} & \\
 X_i & & \rTo_{\phi_{ij}} & & X_j \\
\end{diagram}
Then this diagram obviously commutes for every $i \leq j \in J$,
and the universal property of $X_J$ guarantees a unique map $X_J
\stackrel{t}{\longrightarrow} Y$ making
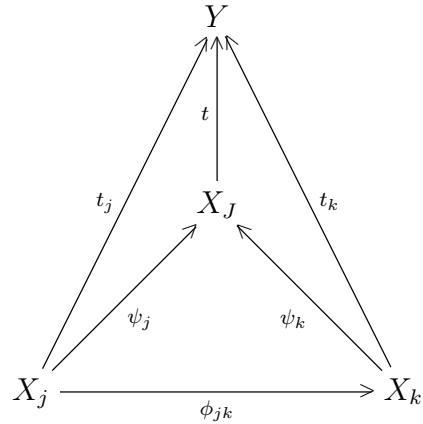
\begin{diagram}
 & & Y& & \\
 & \ruTo(2,4)^{t_j} & \uTo^{t} & \luTo(2,4)^{t_k} & \\
 &  & X_J & & \\
 & \ruTo_{\psi_j} & & \luTo_{\psi_k} & \\
 X_j & & \rTo_{\phi_{jk}} & & X_k \\
\end{diagram}
commute for every $j,k \in J$.  But this map $t$ also satisfies
the universal property required for $X_J$ to be a direct limit
over all of $I$.  For if $i,l \in I$ with $i \leq l$, then let
$j,k \in J$ with $i \leq j$, $l \leq k$, and $j \leq k$, and the
following also commutes:
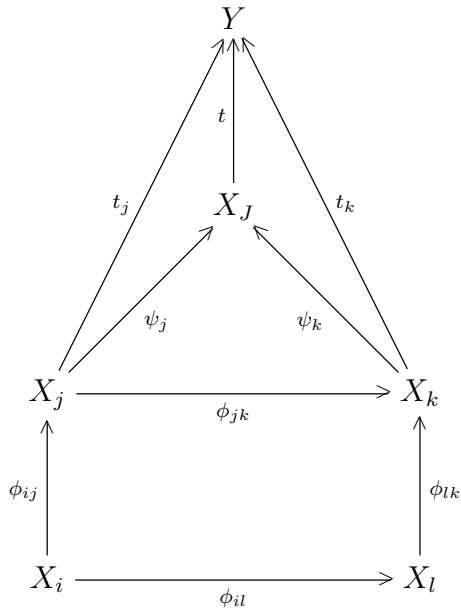
\begin{diagram}
 & & Y& & \\
 & \ruTo(2,4)^{t_j} & \uTo^{t} & \luTo(2,4)^{t_k} & \\
 &  & X_J & & \\
 & \ruTo_{\psi_j} & & \luTo_{\psi_k} & \\
 X_j & & \rTo_{\phi_{jk}} & & X_k \\
 \uTo^{\phi_{ij}} & & & & \uTo_{\phi_{lk}} \\
 X_i & & \rTo_{\phi_{il}} & & X_l \\
\end{diagram}
and hence so does
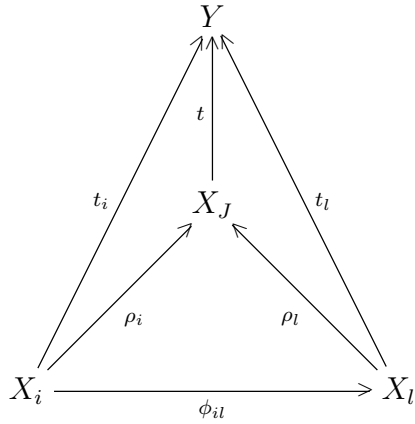
\begin{diagram}
 & & Y & & \\
 & \ruTo(2,4)^{t_i} & \uTo^{t} & \luTo(2,4)^{t_l} & \\
 &  & X_J & & \\
 & \ruTo_{\rho_i} & & \luTo_{\rho_l} & \\
 X_i & & \rTo_{\phi_{il}} & & X_l \\
\end{diagram}
This map $t$ is still unique, since satisfying universality over
all of $I$ is clearly a more stringent requirement than doing so
over all of $J$.

\end{proof}

There is an obvious analogue to this lemma as concerns inverse
limits which we state but do not prove.

\begin{lem}
\label{inverselimitlemma} \label{directlimitlemma}
 Let $\catfont{C}$ be any category, $I$ a
directed set, $\{X_i\}$ a collection of objects indexed over $I$,
and $\{X_i \stackrel{\tau_{ij}}{\longleftarrow} X_j\}$ an inverse
system for the $X_i$ over $I$, and let
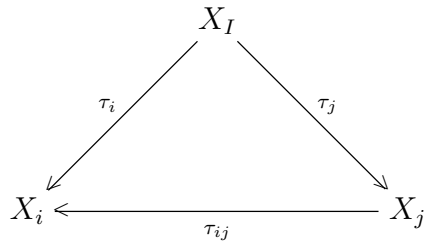
\begin{diagram}
& & X_I & & \\
 & \ldTo^{\tau_i} & & \rdTo^{\tau_j} \\
 X_i & & \lTo_{\tau_{ij}} & & X_j \\
\end{diagram}
be the inverse limit of this system. Let $J \subset I$ be a (not
necessarily full) sub-directed set, and let
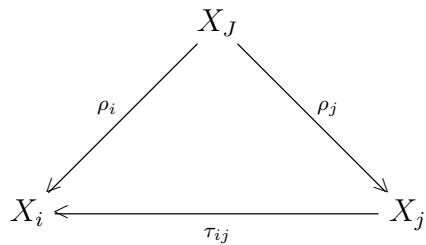
\begin{diagram}
& & X_J & & \\
& \ldTo^{\rho_i} & & \rdTo^{\rho_j} \\
X_i & & \lTo_{\tau_{ij}} & & X_j \\
\end{diagram}
be the inverse limit over $J$.  Then if $J$ is essential in $I$,
these two inverse limits are isomorphic, via a unique isomorphism
commuting with the canonical projections.

\end{lem}

\subsection{Computing $\text{End}(\omega \res \gen{X})$}

Recall from page \pageref{endomegadiscussion} that if $X$ is an
object of $\catfont{C}$, we define $L_X$ to be the subalgebra of
$\text{End}_k(\omega(X))$ consisting of those linear maps which
are `starting points' for a full endomorphism of the fibre functor
restricted to $\gen{X}$.  Here we describe a practical method for
computing $L_X$, which is gleaned from pages $132$, $133$ of
\cite{deligne}.  The definition of $L_X$ given, \textit{a priori},
seems to require that we look at arbitrarily large powers of $X$
to discover if a given transformation of $\omega(X)$ is or is not
in $L_X$, but the method described here shows that we need only
look inside a fixed power of $X$ ($X^{\text{dim}(\omega(X))}$ in
fact).

Let $n = \text{dim}(\omega(X))$ and write $X^n = X_1 \oplus X_2
\oplus \ldots \oplus X_n$, where each $X_i$ is simply a labelled
copy of $X$.  If $Y \stackrel{\psi}{\longrightarrow} X^n$ is any
embedding then we can write $\psi = \psi_1 \oplus \ldots \oplus
\psi_n$, where $Y \stackrel{\psi_i}{\vlongrightarrow} X_i$ is the
$i^{\text{th}}$ component of $\psi$.

As $\omega(X)$ is $n$-dimensional, so is $\omega(X)^*$, so fix an
isomorphism $\alpha:k^n \rightarrow \omega(X)^*$, and let $e_1 =
(1,0,\ldots,0), \ldots, e_n = (0, \ldots, 0,1)$ be the standard
basis of $k^n$.  From this $\psi$ and $\alpha$ we define a linear
map $\psi_\alpha:\omega(Y) \rightarrow \omega(X)^* \otimes
\omega(X)$ as follows: for a vector $y \in \omega(Y)$,
\[ \psi_\alpha(y) = \sum_{i=1}^n \alpha(e_i) \otimes \psi_i(y) \]
If we identify $\omega(X)^* \otimes \omega(X)$ with
$\text{End}_k(\omega(X))$ in the usual fashion, we may speak of
whether the image of $\psi_\alpha$ does or does not contain the
element $\text{id}:\omega(X) \rightarrow \omega(X)$.  Further,
exactness and faithfulness of the functor $\omega$ imply that,
just as in $\text{Vec}_k$, the concept of a ``smallest'' object
having a given property make sense.  So we define

\begin{defn}
For an object $X$ and fixed isomorphism $\alpha:k^n \rightarrow
\omega(X)^*$, $P_X^\alpha$ is the smallest subobject of $X^n$
having the property that the image of $\omega(P_X^\alpha)$ under
$\psi_\alpha$ contains $\text{id}:\omega(X) \rightarrow
\omega(X)$, where $\psi$ is the embedding $P_X^\alpha \rightarrow
X^n$.
\end{defn}

It is a completely non-obvious fact that

\begin{thm}(see lemmas 2.12 and 2.13 of \cite{deligne})
\label{L_Xtheorem}
 For any $\alpha$, the image of
$\omega(P_X^\alpha)$ under $\psi_\alpha$ is the algebra $L_X$.
\end{thm}

We have no intention of justifying this, although we do mention
the reason that $\alpha$ can be chosen arbitrarily.  Consider the
subobject $X^n$ itself of $X^n$, with the embedding being $\psi =
\text{id}$. Certainly if one chooses a different $\beta$ then the
subobjects $P_X^\alpha$ and $P_X^\beta$ will be different, but
their images under $\psi_\alpha$ and $\psi_\beta$ respectively
will not change.  Obviously $\psi_\alpha$ and $\psi_\beta$ are
isomorphisms of vector spaces, and thus we have a commutative
diagram

\begin{diagram}
\omega(X^n) & \rTo^{\phi_{\beta,\alpha}} & \omega(X^n) \\
\dTo^{\psi_\alpha} & \ldTo_{\psi_\beta} & \\
\dual{\omega(X)} \otimes \omega(X) & & \\
\end{diagram}
where $\phi_{\beta,\alpha}$ is a linear isomorphism.  But it can
in fact be shown $\phi_{\beta,\alpha}$ must in fact be (the image
of) an actual isomorphism between the object $X^n$ and itself in
the original category.  Such a $\phi_{\beta,\alpha}$ must then
preserve the notion `smallest subobject', and so the computation
will always yield the same subspace $L_X$ of
$\text{End}(\omega(X))$.

Example: consider the following module $X$ for the additive group
$G_a$ over a field $k$, with matrix formula
\[
\left(
\begin{array}{cc}
  1 & x \\
  0 & 1 \\
\end{array}
\right)
\]
in the basis $f_1,f_2$ for $\omega(X)$. We will compute
$\text{End}(\omega \res \gen{X})$ using the method outlined above.
This is a $2$-dimensional module, so we consider the module $X^2$,
with matrix formula
\[
\left(
\begin{array}{cccc}
  1 & x &  &  \\
   & 1 &  &  \\
   &  & 1 & x \\
   &  &  & 1 \\
\end{array}
\right)
\]
in the basis $f_{1,1},f_{1,2},f_{2,1},f_{2,2}$ for $\omega(X^2)$.
Next we consider an arbitrary subobject of $X^2$. As any subobject
factors through the identity mapping $X^2 \rightarrow X^2$, it
does no harm to choose $\psi = 1$.  Then the coordinate maps
$\psi_1,\psi_2:\omega(X^2) \rightarrow \omega(X)$, using the usual
canonical injections $X \rightarrow X^2$, are
\begin{equation*}
\begin{split}
& \psi_1:f_{11} \mapsto f_{1}, \quad f_{12} \mapsto f_{2}, \quad
f_{21} \mapsto 0, \quad f_{22} \mapsto 0 \\
& \psi_2:f_{11} \mapsto 0, \quad f_{12} \mapsto 0, \quad f_{21}
\mapsto f_{1}, \quad f_{22} \mapsto f_{2}
\end{split}
\end{equation*}
For the isomorphism $\alpha:k^2 \rightarrow \omega(X)^*$, let's
keep life simple and choose $\alpha:(1,0) \mapsto f_1^*,(0,1)
\mapsto f_2^*$, where
 $f_1^*,f_2^*$ is the dual basis for $\omega(X)^*$.
From this $\psi$ and $\alpha$ we compute
 $\psi_\alpha:\omega(X^2) \rightarrow \omega(X)^* \otimes
 \omega(X)$, which has formula
\[ \psi_\alpha(x) = \alpha((1,0)) \otimes \psi_1(x) +
 \alpha((0,1)) \otimes \psi_2(x) \]
and thus
\begin{equation*}
\begin{split}
\psi_\alpha:& f_{11} \mapsto f_1^* \otimes f_1 \\
 &f_{12} \mapsto f_1^* \otimes f_2 \\
 &f_{21} \mapsto f_2^* \otimes f_1 \\
 &f_{22} \mapsto f_2^* \otimes f_2
\end{split}
\end{equation*}

Now, to compute $P_X^\alpha$, we ask: what is the smallest
subobject of $X^2$ such that, under the map $\psi_\alpha$,
contains the identity map $X \rightarrow X$?  We identify the
identity map of course as the element $f_1^* \otimes f_1 + f_2^*
\otimes f_2 \in \omega(X)^* \otimes \omega(X)$.  This projects
back to, under $\psi_\alpha$, the element $f_{11} + f_{22}$ of
$\omega(X^2)$, which, in the given bases, corresponds to the
vector $(1,0,0,1)$.  A quick computation shows that the smallest
subspace of $\omega(X)^2$ stable under $X$ and containing $f_{11}
+ f_{22}$ is
\[ \text{span}(f_{11} + f_{22}, f_{21}) \]
which, under $\psi_\alpha$ and then the isomorphism $\omega(X)^*
\otimes \omega(X) \isomorphic \text{End}_k(\omega(X))$, maps to
the span of the transformations
\[
\left(
\begin{array}{cc}
  1 & 0 \\
  0 & 1 \\
\end{array}
\right), \left(
\begin{array}{cc}
  0 & 1 \\
  0 & 0 \\
\end{array}
\right)
\]
where we have written these transformations as matrices in the
bases $f_1, f_2$ for $\omega(X)$.  Thus, $\text{End}(\omega \res
\gen{X})$ can be identified with the algebra of all $2 \times 2$
matrices of the form
\[
\left(
\begin{array}{cc}
  a & b \\
  0 & a \\
\end{array}
\right)
\]
for arbitrary $a$ and $b$.

Another Example: Consider the module
\[
\left(
\begin{array}{ccc}
  x &  &  \\
   & x^2 &  \\
   &  & x^3 \\
\end{array}
\right)
\]
for the multiplicative group $G_m$.  Skipping all the mumbo-jumbo
with $\psi$ and $\alpha$, all we really have to do is find the
invariant subspace of
\[
\left(
\begin{array}{ccccccccc}
  x &  &  &  &  &  &  &  &  \\
   & x^2 &  &  &  &  &  &  &  \\
   &  & x^3 &  &  &  &  &  &  \\
   & &  & x &  &  &  &  &  \\
   &  &  &  & x^2 &  &  &  &  \\
   &  &  &  &  & x^3 &  &  &  \\
   &  &  &  &  &  & x &  & \\
   &  &  &  &  &  &  & x^2 &  \\
   &  &  &  &  &  &  &  & x^3 \\
\end{array}
\right)
\]
generated by the vector $(1,0,0,0,1,0,0,0,1)$.  This we compute to
be the span of the vectors
$(1,0,0,0,0,0,0,0,0),(0,0,0,0,1,0,0,0,0)$, and
$(0,0,0,0,0,0,0,0,1)$.  These in turn project to the span of the
matrices
\[
\left(
\begin{array}{ccc}
  1 & 0 & 0 \\
  0 & 0 & 0 \\
  0 & 0 & 0 \\
\end{array}%
\right), \left(
\begin{array}{ccc}
  0 & 0 & 0 \\
  0 & 1 & 0 \\
  0 & 0 & 0 \\
\end{array}%
\right), \left(
\begin{array}{ccc}
  0 & 0 & 0 \\
  0 & 0 & 0 \\
  0 & 0 & 1 \\
\end{array}
\right)
\]
That is, we can identify $\text{End}(\omega \res \gen{X})$ with
the collection of all diagonal transformations on $\omega(X)$.

\chapter{First-Order Definability of Tannakian Categories}

\label{chapterFirstOrderDefinabilityOf}
\label{axiomsfortannakiancategories}

The goal of this chapter is to prove that, in a certain
appropriately chosen language, the sentence
\index{category!tannakian}``is a tannakian category'' is
first-order. Note that we do not claim that the property of `being
neutral' is necessarily first-order.

\section{The Language of Abelian Tensor Categories}

\label{chapterTheLanguageOf}

The title of this section is a misnomer for two reasons.  First,
the article ``the'' implies there is only one such language, and
this is certainly not the case.  It is however, as far as the
author can tell, the most natural and minimal choice for our
purposes. Secondly, not all structures in the `language of abelian
tensor categories' are abelian tensor categories, and we may as
well have called it the `language of tannakian categories'.  But
the name seems natural enough.

Our language is purely relational; it has no function or constant
symbols.  The primitive symbols are \index{$\text{assoc}$}
\index{$\text{comm}$} \index{$\text{unit}$}
\begin{gather*}
\index{first-order language!of abelian tensor categories} \slot
\in \text{Mor} \languageSpace \slot \in \text{Ob} \languageSpace
\slot:\slot \rightarrow \slot \languageSpace \slot \circ \slot
\myeq \slot \languageSpace \slot + \slot \myeq \slot \\
 \slot \otimes \slot \myeq \slot \languageSpace
\text{assoc}_{\slot, \slot, \slot} \myeq \slot \languageSpace
\text{comm}_{\slot,\slot} \myeq \slot \languageSpace
\text{unit}_{\slot} \myeq \slot \languageSpace
\end{gather*}
The intended interpretation of the symbols are as follows.  $x \in
\text{Mor}$ expresses that $x$ is a morphism in the category, $x
\in \text{Ob}$ that $x$ is an object.  $x:y \rightarrow z$
expresses that the morphism $x$ points from the object $y$ to $z$,
$x \circ y \myeq z$ expresses that the morphism $x$ composed with
$y$ is equal to $z$, and $x+y \myeq z$ expresses that the morphism
$x$ added to $y$ is equal to $z$.  $x \otimes y \myeq z$ means
that the tensor product of $x$ and $y$ is equal to $z$, and this
could either mean tensor product of objects or tensor product of
morphisms.

The symbols $\text{assoc},\text{comm}$, and $\text{unit}$ stand
for the requisite natural isomorphisms present in an abelian
tensor category.  For instance, $\text{assoc}_{x,y,z} \myeq t$
expresses that the associativity isomorphism $(x \otimes y)
\otimes z \isomorphic x \otimes (y \otimes z)$ attached to the
objects $x,y$ and $z$ is equal to $t$, and similarly for
$\text{comm}_{x,y} \myeq t$. $\text{unit}_{x} \myeq y$ expresses
that $y$ is the natural isomorphism between the object $x$ and $x
\otimes \underline{1}$, where $\underline{1}$ is an identity
object for the tensor category.

The reader should take care not to automatically identify the
symbol $\myeq$ occurring as a sub-symbol of the above symbols with
actual equality of elements; $\myeq$ is purely formal in this
context. For a random structure in the above signature it is
entirely possible to have four elements $a,b,c,d$ such that $a
\circ b \myeq c$, $a \circ b \myeq d$, but not $c = d$, where this
last equation is actual equality of elements.  Of course, we chose
the symbol $\myeq$ because, modulo the theory we are going to
write down, $\myeq$ does in fact behave like equality.

To make anything we are about to do manageable, we must find a way
to treat certain of the relational symbols in our language as if
they were functional right from the start, and hence to treat
expressions such as $x \circ y$ as if they were terms. For
instance, we would like to be able to write the sentence
\[ (\forall x,y)(x \circ y = y \circ x) \]
and attach the intended meaning to it.  But as it stands, this is
not a sentence in our language.  Here is how this can be remedied.
In the case of $\circ$ we can treat the symbol $x \circ y$ as a
term as follows.  If $\Phi(z)$ is any formula, $x$ and $y$
variables, then we define $\Phi(x \circ y)$ to be the formula
\[ (\forall t)(x \circ y \myeq t \myimplies \Phi(t)) \]
where $t$ is some variable not occurring in $\Phi$ and not equal
to $x$ or $y$.  By iteration of this process we can in fact treat
any `meaningful composition' of variables as a term.  By
meaningful composition we mean: any variable $x$ is a meaningful
composition, and if $\Psi,\Sigma$ are meaningful compositions,
then so is $(\Psi) \circ (\Sigma)$ (for instance, $((x \circ y)
\circ z) \circ (s \circ x)$ is a meaningful composition).  Then
for a meaningful composition $(\Psi) \circ (\Sigma)$ and formula
$\Phi(x)$, we define $\Phi((\Psi) \circ (\Sigma))$ by induction on
the length of the composition to be
\[ (\forall s,t,r)((s = \Psi \myand t = \Sigma \myand s \circ t \myeq r) \myimplies \Phi(r)) \]
where $s,t$ and $r$ are some not already being used variables.
This formula is well-defined, since the formulas $s = \Psi$,
$t=\Sigma$, and $\Phi(r)$ are by induction.  For example then, the
formula $x \circ y = y \circ x$ literally translates to
\[(\forall s)(y \circ x \myeq s \myimplies (\forall r)(x \circ y
\myeq r \myimplies r = s)) \] The same trick can obviously be
applied to the symbols involving
$+,\otimes,\text{assoc},\text{comm}$, and $\text{unit}$.  Thus we
can be confident in the meaning of something like
\[ (\forall x,y,z)( (x + y) \circ z = (x \circ z) + (y \circ z))
\]

Let us agree on some abbreviations.  All capital English letter
variables ($A,B,X,Y$, etc.) are understood to range over objects,
lower case Greek letters ($\phi,\psi,\alpha,\beta$, etc.) over
morphisms, and if we wish to be nonspecific we will use lower case
English letters ($a,b,x,y$, etc.).  So if $\Phi(x)$ is a formula,
we define $(\forall X) \Phi(X)$ to mean $(\forall x)(x \in
\text{Ob} \myimplies \Phi(x))$, and $(\forall \psi)\Phi(\psi)$
means $(\forall x)(x \in \text{Mor} \myimplies \Phi(x))$.  The
formula $(\exists X)\Phi(X)$ stands for $(\exists x)(x \in
\text{Ob} \myand \Phi(x))$, and similarly for $(\exists
\psi)\Phi(\psi)$. $(\forall x) \Phi(x)$ and $(\exists x) \Phi(x)$
mean exactly what they say.

If $a_1, \ldots, a_n,x,y$ are variables then $a_1, \ldots, a_n :x
\rightarrow y$ is shorthand for $a_1:x \rightarrow y \myand \ldots
\myand a_n:x \rightarrow y$.  $(\forall a_1, \ldots, a_n:x
\rightarrow y)\Phi(a_1, \ldots, a_n)$ is shorthand for $(\forall
a_1, \ldots, a_n)(a_1, \ldots, a_n:x \rightarrow y \myimplies
\Phi(a_1, \ldots, a_n))$, and $(\exists a_1, \ldots, a_n:x
\rightarrow y) \Phi(a_1, \ldots, a_n)$ is shorthand for $(\exists
a_1, \ldots, a_n)(a_1, \ldots, a_n:x \rightarrow y \myand
\Phi(a_1, \ldots, a_n))$. We make identical definitions for the
expressions $x_1, \ldots, x_n \in \text{Ob}$ and $x_1, \ldots, x_n
\in \text{Mor}$.

If $x$ and $y$ are variables, we define the formula $\text{Dom}(x)
\myeq y$ to mean $(\exists z)(x:y \rightarrow z)$, and we make an
analogous definition for $\text{Codom}(x) \myeq y$.  We can treat
$\text{Dom}$ and $\text{Codom}$ as if they were functions by
declaring: if $\Phi(x)$ is a formula, we define
$\Phi(\text{Dom}(x))$ to mean $(\forall y)(\text{Dom}(x) \myeq y
\myimplies \Phi(y))$, and similarly for $\text{Codom}$.

The remainder of this chapter is devoted to proving, piecemeal,
that the statement ``is a tannakian category'' is expressible by a
first-order sentence in the language of abelian tensor categories.

\label{firstorderdefinabilitychapter}
\section{Axioms for a Category}
\index{category}
\begin{enumerate}

\item{Every element of $\catfont{C}$ is either an object or a
morphism, but not both:
\[ (\forall x)((x \in \text{Ob} \myor x \in \text{Mor}) \myand
\mynot(x \in \text{Ob} \myand x \in \text{Mor})) \] } \item{All
arrows are morphisms, and all vertices are objects:
\[ (\forall x, y, z)(x:y \rightarrow z \myimplies (x \in \text{Mor}
\myand y \in \text{Ob} \myand z \in \text{Ob})) \] }
 \item{Every
morphism points to and from exactly one object:
\[ (\forall \phi)(\exists!X,Y)(\phi:X \rightarrow Y) \]
}
 \item{Composition only makes sense on morphisms:
\[ (\forall x,y,z)(x \circ y \myeq z \implies x,y,z \in \text{Mor}) \]
} \item{Composition only makes sense between composable morphisms,
and the composition points where it should:
\begin{gather*}
(\forall \phi,\psi, \eta)( \phi \circ \psi \myeq \eta \myimplies
(\text{Codom}(\phi) = \text{Dom}(\psi) \\
 \myand \eta:\text{Dom}(\phi) \rightarrow \text{Codom}(\psi)))
\end{gather*} }
\item{Composition is a function on composable arrows:
\[ (\forall \phi,\psi)(\text{Codom}(\phi) = \text{Dom}(\psi) \myimplies (\exists!\eta)(\phi \circ \psi \myeq \eta)) \]}
\item{Composition is associative:
\[ (\forall x,y,z)((\exists t)((x \circ y) \circ z = t) \myimplies (x \circ (y \circ z) = (x \circ y) \circ z))  \] }
\end{enumerate}

We define the formula $x \myeq 1_y$ to mean that $x$ is a
two-sided identity morphism for $y$, i.e.~as the formula $(x:y
\rightarrow y) \myand (\forall z)((x \circ z =t \myor z \circ x =
t) \myimplies z = t)$. We write $x \myeq 1$ to mean that $x$ is an
identity morphism for some object, i.e. $(\exists X)(x \myeq
1_X)$. If $\Phi(x)$ is a formula, we define $\Phi(1_x)$ to mean
$(\forall y)(y \myeq 1_x \myimplies \Phi(y))$, and similarly for
$1$.
\begin{itemize}
\item[8.]{Every object has an identity morphism:
\[ (\forall X)(\exists \phi)(\phi \myeq 1_X) \]
}
\end{itemize}

\section{Axioms for an Abelian Category}
\label{axiomsforanabeliancategory}

In this section we build axioms amounting to the statement that a
given category is abelian, using definition
\ref{defnabeliancategory} as our guide.

\begin{enumerate}

\item{Addition is only defined on addable morphisms, and their sum
points where it should: \begin{gather*}
 (\forall x,y,z)(x+y \myeq
z \myimplies (x,y,z \in \text{Mor} \myand \text{Dom}(x) =
\text{Dom}(y) \myand \text{Dom}(y) \\
= \text{Dom}(z) \myand \text{Codom}(x) = \text{Codom}(y) \myand
\text{Codom}(y) = \text{Codom}(z))) \end{gather*}}
 \item{Addition is a function on addable
morphisms:
\[ (\forall x,y)((\text{Dom}(x) = \text{Dom}(y) \myand
\text{Codom}(x) = \text{Codom}(y)) \myimplies (\exists!z)(x + y =
z)) \] }
 \item{Addition is associative and commutative:
\[ (\forall x,y,z)((\exists t)((x+y)+z = t) \myimplies (x+y = y+x \myand (x+y)+z = x+(y+z))) \]
}
\end{enumerate}

Define the formula $x \myeq 0_{y,z}$ to mean $x$ is an additive
identity for $\text{Hom}(x,y)$.  That is, $(x:y \rightarrow z)
\myand (\forall t:y \rightarrow z)(x+t = t)$. Define $x \myeq 0$
to be $(\exists X,Y)(x=0_{X,Y})$.  If $\Phi(x)$ is a formula,
$\Phi(0_{x,y})$ means $(\forall z)(z = 0_{x,y} \myimplies
\Phi(z))$, and similarly for $0$.
\begin{itemize}
\item[4.]{Existence of zero morphisms for addition:
\[ (\forall A,B)(\exists \phi)(x \myeq 0_{A,B})\] }
\end{itemize}

Define the formula $x\myeq-y$ to mean $x$ is an additive inverse
for $y$.  That is, $x+y = 0$.  For a formula $\Phi(x)$, $\Phi(-y)$
is shorthand for the formula $(\forall x)(x \myeq -y \myimplies
\Phi(x))$.
\begin{itemize}
 \item[5.]{Existence of additive inverses:
\[(\forall \phi)(\exists x)(x = -\phi) \] }
\item[6.]{Bilinearity of composition over addition:
\begin{gather*} (\forall A,B,C,D)(\forall
\eta,\phi,\psi,\nu)((\eta:A \rightarrow B \myand
 \phi,\psi:B \rightarrow C  \myand \nu:C \rightarrow D)
\myimplies \\ ( \eta \circ (\phi + \psi) = \eta \circ \phi + \eta
\circ \psi \myand (\phi + \psi) \circ \eta = \phi \circ \eta +
\psi \circ \eta)) \end{gather*} }
\end{itemize}

The definition of an abelian category calls for the existence of
pair-wise biproducts, kernels and cokernels, and normality of
monomorphisms and epimorphisms.  Here we give first-order
definitions of these concepts.

Let $A$ and $B$ be objects.  Then a biproduct, which we denote as
$A \oplus B$, is a diagram
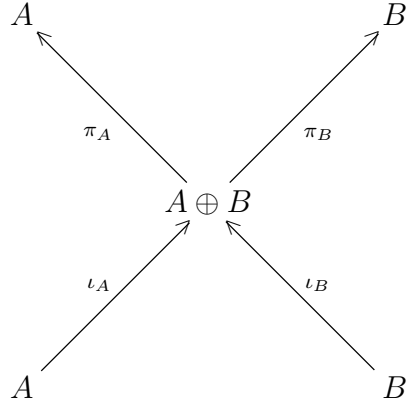
\begin{diagram}
A & & & & B \\
 & \luTo_{\pi_A} & & \ruTo_{\pi_B} & \\
 & & A \oplus B & & \\
 & \ruTo^{\iota_A} & & \luTo^{\iota_B} & \\
 A & & & & B \\
\end{diagram}
with the following properties: $\pi_A \circ \iota_A + \pi_B \circ
\iota_B = 1_{A \oplus B}$, $\iota_A \circ \pi_A = 1_A$, $\iota_B
\circ \pi_B = 1_B$, $\iota_A \circ \pi_B = 0$, and $\iota_B \circ
\pi_A = 0$.  One could clearly write down these conditions as a
first-order formula.  Thus, for objects $A$ and $B$ we define the
formula $\text{Sum}(Z;A,B)$ to mean that there exists maps
$\iota_A, \iota_B, \pi_A, \pi_B$ satisfying all the above
criteria.

Let $\psi:A \rightarrow B$ be a morphism. A kernel for $\psi$ is
by definition a map $k$ pointing from some object $K$ to $A$ such
that $k \circ \psi = 0_{K,B}$, and for any object $C$ and map
$\rho:C \rightarrow A$ with $\rho \circ \psi = 0_{C,B}$ there is a
unique map $\hat{\rho}:C \rightarrow K$ such that $\hat{\rho}
\circ k = \rho$.  Again, this is clearly first-order.  Thus, we
define: $\text{ker}(k;\psi)$ means that the morphism $k$ is a
kernel for $\psi$. The same obviously holds for the dual concept
of cokernel, so we define $\text{coker}(c;\psi)$ in like fashion.

In the language of categories saying that a morphism is an
epimorphism is to say that it is right-cancellative.  That is,
$\psi:A \rightarrow B$ is an epimorphism if for any maps
$\eta,\nu:B \rightarrow C$, $\psi \circ \eta = \psi \circ \nu$
implies that $\eta=\nu$. Again, this is clearly a first-order
concept, so we define the formula $\text{epic}(\phi)$ to mean
$\phi$ is an epimorphism, and likewise $\text{monic}(\phi)$ that
$\phi$ is a monomorphism.
\begin{itemize}
\item[7.]{Every pair of objects has a biproduct:
\[ (\forall A,B)(\exists Z)(\text{Sum}(Z;A,B)) \]
}
 \item[8.]{Every morphism has a kernel and a cokernel:
\[ (\forall \phi)(\exists k,c)(\text{ker}(k;\phi) \myand \text{coker}(c;\phi)) \] }
\item[9.]{Every monomorphism is normal:
\[ (\forall \phi)(\text{monic}(\phi) \myimplies (\exists \psi)(\text{ker}(\phi;\psi)) \] }
\item[10.]{Every epimorphism is normal:
\[ (\forall \phi)(\text{epic}(\phi) \myimplies (\exists
\psi)(\text{coker}(\phi;\psi))\] }
\end{itemize}

\section{Axioms for an Abelian Tensor Category}

Here we build axioms asserting that given abelian tensor category
is an abelian tensor category, per definition
\ref{defntensorcategory}. Our first task is to assert that
$\otimes$ is a bi-additive functor.

\begin{enumerate}

\item{Every pair of morphisms and objects has a unique tensor
product: \[ (\forall X,Y)(\exists!Z)(X \otimes Y \myeq Z) \myand
(\forall \phi,\psi)(\exists!\eta)(\phi \otimes \psi \myeq \eta)
\]
} \item{The tensor product of objects is an object, that of
morphisms is a morphism, and there's no such thing as a tensor
product of an object and a morphism:
\[ (\forall X,Y)(X \otimes Y \in \text{Ob}) \myand (\forall
\phi,\psi)(\phi \otimes \psi \in \text{Mor}) \myand (\forall
X,\psi)(\nexists x)(X \oplus \psi \myeq x \myor \psi \oplus X
\myeq x)
\]
} \item{The tensor product of morphisms points where it should:
\[ (\forall \phi,\psi)(\forall A,B,X,Y)((\phi:A \rightarrow B
\myand \psi:X \rightarrow Y) \myimplies (\phi \otimes \psi:A
\otimes X \rightarrow B \otimes Y)) \] }
 \item{The tensor product
preserves composition:
\begin{gather*}
 (\forall \phi,\psi,\eta,\nu)(\forall
A,B,X,Y,S,T)((\phi:A \rightarrow X \myand \psi:B \rightarrow Y
\myand \eta:X \rightarrow S \myand \nu:Y \rightarrow T) \\
\myimplies ((\phi \circ \eta) \otimes (\psi \circ \nu)) = (\phi
\otimes \psi) \circ (\eta \otimes \nu))) \end{gather*} }
 \item{The tensor product preserves
identity:
\[ (\forall A,B)((1_A \otimes 1_B) = 1_{A \otimes B}) \]
} \item{The tensor product is a bi-additive functor:
\begin{gather*}
 (\forall A,B,X,Y)(\forall \phi,\psi:A \rightarrow B)(\forall
\eta:X \rightarrow Y) \\ ( ((\phi + \psi) \otimes \eta = \phi
\otimes \eta + \psi \otimes \eta) \myand (\eta \otimes (\phi +
\psi) = \eta \otimes \phi + \eta \otimes \psi)) \end{gather*} }
\end{enumerate}

Next we assert that the natural isomorphisms assoc, comm, and unit
are doing the job we need them to.  We start with assoc.
\begin{itemize}
 \item[7.]{assoc accepts
objects and returns morphisms:
\[ (\forall x,y,z,t)(\text{assoc}_{x,y,z} \myeq t \myimplies (x,y,z \in
\text{Ob} \myand t \in \text{Mor})) \] }
 \item[8.]{assoc is a function
on triples of objects:
\[ (\forall X,Y,Z)(\exists!\phi)(\text{assoc}_{X,Y,Z} \myeq \phi) \]
}
\end{itemize}

If $\Phi(x)$ is a formula and $a,b,c$ variables, by
$\Phi(\text{assoc}_{a,b,c})$ we mean the formula $(\forall
t)(\text{assoc}_{a,b,c} \myeq t \myimplies \Phi(t))$.
\begin{itemize}
\item[9.]{assoc points where it should:
\[ (\forall X,Y,Z)(\text{assoc}_{X,Y,Z}:X \otimes (Y \otimes Z)
\rightarrow (X \otimes Y) \otimes Z) \] }
\end{itemize}

We define a formula $\text{iso}(\phi)$ to mean that the morphism
$\phi$ is an isomorphism: $(\forall A,B)(\phi:A \rightarrow B
\myimplies (\exists \psi:B \rightarrow A)(\phi \circ \psi = 1_A
\myand \psi \circ \phi = 1_B))$.
\begin{itemize}
 \item[10.]{assoc is always an
isomorphism:
\[ (\forall X,Y,Z)(\text{iso}(\text{assoc}_{X,Y,Z})) \]
} \item[11.]{assoc is a natural transformation:
\[ (\forall X,Y,Z,R,S,T)(\forall \phi:X \rightarrow R,\psi:Y
\rightarrow S,\eta:Z \rightarrow T)(\text{the following commutes:}
\]
\begin{diagram}
X \otimes (Y \otimes Z) & \rTo^{\text{assoc}_{X,Y,Z}} & (X \otimes
Y) \otimes Z \\
\dTo^{\phi \otimes (\psi \otimes \eta)} & & \dTo_{(\phi \otimes
\psi) \otimes \eta} \\
R \otimes (S \otimes T) & \rTo_{\text{assoc}_{R,S,T}} & (R \otimes
S) \otimes T \\
\end{diagram}
}
\end{itemize}

We make the necessary assertions and definitions for comm and unit
in like fashion.

\begin{itemize}
 \item[12.]{comm accepts
objects and returns morphisms:
\[ (\forall x,y,z)(\text{comm}_{x,y} \myeq z \myimplies (x,y \in
\text{Ob} \myand z \in \text{Mor})) \] }

 \item[13.]{comm is a function
on pairs of objects:
\[ (\forall X,Y)(\exists!\phi)(\text{comm}_{X,Y} \myeq \phi) \]
}
\end{itemize}

If $\Phi(x)$ is a formula and $a,b$ variables, by
$\Phi(\text{comm}_{a,b})$ we mean the formula $(\forall
t)(\text{assoc}_{a,b} \myeq t \myimplies \Phi(t))$.

\begin{itemize}
\item[14.]{comm points where it should:
\[ (\forall X,Y)(\text{comm}_{X,Y}:X \otimes Y
\rightarrow Y \otimes X) \] }
\end{itemize}

\begin{itemize}

 \item[15.]{comm is always an
isomorphism:
\[ (\forall X,Y)(\text{iso}(\text{comm}_{X,Y})) \]
}

\item[16.]{comm is a natural transformation:
\[ (\forall X,Y,R,S)(\forall \phi:X \rightarrow R,\psi:Y
\rightarrow S)(\text{the following commutes:}
\]
\begin{diagram}
X \otimes Y & \rTo^{\text{comm}_{X,Y}} & Y \otimes X \\
\dTo^{\phi \otimes \psi} & & \dTo_{\psi \otimes \phi} \\
R \otimes S & \rTo_{\text{comm}_{R,S}} & S \otimes R \\
\end{diagram}
}
\end{itemize}

\begin{itemize}
 \item[17.]{unit accepts
objects and returns morphisms:
\[ (\forall x,y)(\text{unit}_{x} \myeq y \myimplies (x \in
\text{Ob} \myand y \in \text{Mor})) \] }

 \item[18.]{unit is a function
on objects:
\[ (\forall X)(\exists!\phi)(\text{unit}_{X} \myeq \phi) \]
}
\end{itemize}

If $\Phi(x)$ is a formula and $a$ a variable, by
$\Phi(\text{unit}_{a})$ we mean the formula $(\forall
t)(\text{unit}_{a} \myeq t \myimplies \Phi(t))$.
\begin{itemize}
\item[19.]{unit is always an isomorphism:
\[ (\forall X)(\text{iso}(\text{unit}_X)) \]}
\end{itemize}

For unit, we must make the additional assertion that there exists
an identity object for $\otimes$.
\begin{itemize}
 \item[20.]{unit has an identity
object associated to it:
\[ (\exists U)(\forall X)(\text{unit}_X:X \rightarrow U \otimes
X) \]}
\end{itemize}

We define the formula $\text{id}^\otimes(u)$ to mean that $u$ is
an \index{identity object} identity object associated to unit.
That is, $(\forall X)(\text{unit}_X: X \rightarrow u \otimes X)$.
\begin{itemize}
\item[21.]{unit is a natural transformation:
\[ (\forall X,Y)(\forall \phi:X \rightarrow Y)(\forall U)(\text{id}^\otimes(U) \myimplies \text{the following commutes:}
\]
\begin{diagram}
X & \rTo^{\text{unit}_{X}} & U \otimes X \\
\dTo^{\phi} & & \dTo_{1_U \otimes \phi} \\
Y & \rTo_{\text{unit}_{Y}} & U \otimes Y \\
\end{diagram}
}
\end{itemize}

We must now assert that the functor $X \mapsto
\text{id}^\otimes(u) \otimes X$ is an equivalence, which is to say
that it is full, faithful, and essentially surjective.  Essential
surjectivity is already asserted by previous axioms: for every
$X$, $X$ is isomorphic to $\text{id}^\otimes(u) \otimes X$.  Thus
we must assert that it is full and faithful.
\begin{itemize}
 \item[22.]{The functor $X \mapsto
\text{id}^\otimes(u) \otimes X$ is full and faithful:
\[ (\forall \phi)(\forall X,Y,U)((\text{id}^\otimes(U) \myand \phi:U \otimes X \rightarrow U
\otimes Y) \myimplies (\exists ! \psi)(\psi:X \rightarrow X \myand
1_X \otimes \psi = \phi)) \] }
\end{itemize}

All that is left then is to assert the various coherence
conditions among assoc, comm, and unit, conditions 4, 5, and 6 in
definition \ref{defntensorcategory}.  All of these statements are
of the form, for some fixed $n$, ``for all objects $X_1, \ldots,
X_n$, the following diagram commutes.''  These are plainly
first-order, so we do not repeat them.

\section{Axioms for a Tannakian Category}

 Here we must assert the \index{category!rigid abelian tensor} rigidity of the
abelian tensor category $\catfont{C}$ (definition
\ref{defnrigidity}) and that $\text{End}(\underline{1})$ is a
field, modulo all of our previous axioms. The first condition for
rigidity is the existence of an \index{internal Hom} internal Hom
object for every pair of objects (see page
\pageref{internalHomdefnPage}). This is by definition an object
\index{$\intHom(V,W)$} $\intHom(X,Y)$ such that the functors
$\text{Hom}(\slot,X \otimes Y)$ and
$\text{Hom}(\slot,\intHom(X,Y))$ are naturally isomorphic. Here is
how we can define this in a first-order fashion.

 Suppose that $Z$ is an internal Hom object
for $X$ and $Y$. Then we have a natural isomorphism of functors
$\text{Hom}(\slot,Z) \stackrel{\Phi}{\longrightarrow}
\text{Hom}(\slot \otimes X, Y)$. The Yoneda lemma (page
\pageref{YonedaLemma}) guarantees that this map $\Phi$ must take
the following form: for an object $T$ and map $T
\stackrel{\phi}{\longrightarrow} Z$, $\Phi_T(\phi)$ is the unique
map making the diagram
\begin{diagram}
T \otimes X & & \\
\dTo^{\phi \otimes 1} & \rdTo^{\Phi_T(\phi)} & \\
Z \otimes X & \rTo_{\text{ev}} & Y \\
\end{diagram}
commute, where we have given the name $\text{ev}$ to the element
$\Phi_Z(1_Z) \in \text{Hom}(Z \otimes X,Y)$.  And in fact, any
such map $Z \otimes X \rightarrow Y$ gives you a natural
transformation between the functors $\text{Hom}(\slot,Z)$ and
$\text{Hom}(\slot \otimes X, Y)$.  Thus, to assert the existence
of a natural isomorphism, we need only assert the existence of a
map $\text{ev}:Z \otimes X \rightarrow Y$ such that the natural
transformation $\Phi$ it defines gives a bijection
$\text{Hom}(T,Z) \stackrel{\Phi_T}{\vlongrightarrow} \text{Hom}(T
\otimes X,Y)$ for every $T$. We therefore define the formula
$\intHom(Z,\text{ev};X,Y)$ to mean that the object $Z$ and
morphism $\text{ev}$ form an internal Hom pair for $X$ and $Y$:
\begin{gather*} (\text{ev}:Z \otimes X \rightarrow Y) \myand (\forall
T)(\forall \psi:T \otimes X \rightarrow Y)(\exists! \phi:T
\rightarrow Z)(\text{the following commutes: } \\
\begin{diagram}
T \otimes X & & \\
\dTo^{\phi \otimes 1} & \rdTo^{\psi} & \\
Z \otimes X & \rTo_{\text{ev}} & Y \\
\end{diagram}
\end{gather*}
\begin{itemize}
\item[1.]{Every pair of objects has an internal Hom: \[ (\forall
X,Y)(\exists Z)(\exists \phi)(\intHom(Z,\phi;X,Y)) \] }
\end{itemize}

\label{intHomisFODpage} Recall now the definition of
\index{reflexive object} reflexivity of the object $X$ (see page
\pageref{reflexiveObjectdefnPage}); it is the assertion that a
certain map $\iota_X:X \rightarrow X^{\vee \vee}$ is an
isomorphism.  This map is defined by the property that it uniquely
makes
\begin{diagram}
X \otimes \dual{X}  & & & & \\
 & \rdTo^{\text{comm}} & & & \\
 \dTo^{\iota_X \otimes 1} & & \dual{X} \otimes X & & \\
 & & & \rdTo^{\text{ev}_X} & \\
 X^{\vee \vee} \otimes X^{\vee} &  & \rTo_{\text{ev}_{X^\vee}} & &
 \underline{1} \\
\end{diagram}
commute.  But of course $\iota_X$ is not really unique, since
neither is e.g.~$\dual{X}$, since there are in general many
(mutually isomorphic) choices for internal Hom. But for a fixed
choice of the various internal Hom objects referenced in this
diagram, it is unique.  We therefore define $\text{incl}(\iota;X)$
to mean that $\iota$ qualifies as one of these maps:
\begin{gather*}
(\exists T,R,U)(\exists
\phi,\psi)(\text{id}^\otimes(U) \myand \intHom(T,\phi;X,U) \myand
\intHom(R,\psi;T,U) \\
\myand \iota:X \rightarrow R \myand \text{the following commutes:
} \\
\begin{diagram}
X \otimes T  & & & & \\
 & \rdTo^{\text{comm}_{X,T}} & & & \\
 \dTo^{\iota \otimes 1_T} & & T \otimes X & & \\
 & & & \rdTo^{\phi} & \\
 R \otimes T &  & \rTo_{\psi} & & U \\
\end{diagram}
\end{gather*}
\begin{itemize}
\item[2.]{All objects are reflexive:
\[ (\forall X)(\forall \iota)(\text{incl}(\iota;X) \myimplies
\text{iso}(\iota)) \] }
\end{itemize}

Our last task in defining rigidity is to assert that the map
$\Phi$ referenced in diagram \ref{rigidity3} is an isomorphism.
First, given objects $R,S,T$ and $U$, let us define the
isomorphism $(R \otimes S) \otimes (T \otimes U) \isomorphic (R
\otimes T) \otimes (S \otimes U)$ referenced in diagram
\ref{CommAssocCompEquation}.  This is gotten by composition of the
following sequence of commutativity and associativity isomorphisms
(the subscripts of which we suppress): \begin{gather*}  (R \otimes
S) \otimes (T \otimes U) \stackrel{\text{assoc}}{\vlongrightarrow}
R \otimes (S \otimes (T \otimes U)) \stackrel{1 \otimes
\text{assoc}}{\vvlongrightarrow} R \otimes ((S \otimes T) \otimes
U) \\
 \stackrel{1 \otimes ( \text{comm} \otimes 1)}{\vvlongrightarrow}
R \otimes (( T \otimes S) \otimes U)
\stackrel{\text{assoc}}{\vlongrightarrow} (R \otimes ( T \otimes
S)) \otimes U \stackrel{\text{assoc} \otimes 1}{\vvlongrightarrow}
(( R \otimes T) \otimes S) \otimes U)
\\
 \stackrel{\text{assoc}}{\vlongrightarrow} (R \otimes T) \otimes
(S \otimes U)
\end{gather*}
 Define the formula $\text{ISO}(R,S,T,U,\Psi)$ to
mean that $\Psi$ is the above composition with respect to the
objects $R,S,T$ and $U$.  Next we must define the map $\Phi$
defined by the commutativity of diagram \ref{rigidity3}.  This is
done using a similar strategy to that used to define reflexivity.
Define the formula $\text{QUAD}(X_1,X_2,Y_1,Y_2,\Phi)$ to mean
that $\Phi$ qualifies as one of the maps referenced in diagram
\ref{rigidity3} with respect to the objects $X_1,X_2,Y_1,Y_2$:
\begin{gather*}(\exists Z_1,Z_2,Z)(\exists
\text{ev}_1,\text{ev}_2,\text{ev},\Psi)(\intHom(Z_1,\text{ev}_1;X_1,Y_1)
\myand \intHom(Z_2,\text{ev}_2;X_2,Y_2) \\  \myand
\intHom(Z,\text{ev};X_1 \otimes X_2,Y_1 \otimes Y_2) \myand
\text{ISO}(Z_1,Z_2,X_1,X_2,\Psi) \myand \text{the following
commutes:} \\
\begin{diagram}
(Z_1 \otimes Z_2) \otimes (X_1 \otimes X_2) & \rTo^\Psi & (Z_1
\otimes X_1) \otimes (Z_2 \otimes X_2) \\
\dTo^{\Phi \otimes 1} & & \dTo_{\text{ev}_1 \otimes \text{ev}_2}
\\
Z \otimes (X_1 \otimes X_2) & \rTo_{\text{ev}} & Y_1 \otimes Y_2
\\
\end{diagram}
\end{gather*}
\begin{itemize}
\item[3.]{For all objects $X_1,X_2,Y_1,Y_2$, the map $\Phi$ in
diagram \ref{rigidity3} is an isomorphism:
\[(\forall X_1,X_2,Y_1,Y_2)(\forall
\Phi)(\text{QUAD}(X_1,X_2,Y_1,Y_2,\Phi) \myimplies \text{$\Phi$ is
an isomorphism}) \] }
\end{itemize}

 Finally, we have to assert that the ring
$\text{End}(\underline{1})$ is a field.  Given any object $X$, our
previous axioms already assert that $\text{End}(X)$ is a ring with
unity, so let $\text{field}(X)$ be the assertion that
$\text{End}(X)$ is a commutative ring with inverses:
\[ (\forall \phi,\psi:X \rightarrow X)( \phi \circ \psi = \psi
\circ \phi) \myand (\forall \phi:X \rightarrow X)(\mynot(\phi = 0)
\myimplies (\exists \psi)(\phi \circ \psi = 1_X)) \]
\begin{itemize}
\item[4.]{$\text{End}(\underline{1})$ is a field:
\[ (\forall U)(\text{id}^\otimes(U) \implies \text{field}(U)) \]
}
\end{itemize}

We have proved that the statement ``is a tannakian category'' is
expressible by a first-order sentence in the language of abelian
tensor categories.

\chapter{Subcategories of Tannakian Categories}

In this chapter we record some results on (abelian, tannakian)
subcategories of (abelian, tannakian) categories which will be
needed later. For instance we will prove useful criteria which
allow us to conclude that a given subcategory of a neutral
tannakian category is also neutral tannakian.

\begin{prop}
\label{tannakiansubcategorylemma}
 Let $\catfont{C}$ be a full
subcategory of the tannakian category $\catfont{D}$.  Then
$\catfont{C}$ is tannakian if it is closed under the taking of
biproducts, subobjects, quotients, tensor products, duals, and
contains an identity object.
\end{prop}

\begin{proof}
 We first show that $\catfont{C}$ is abelian.
$\catfont{C}$ is a full subcategory of $\catfont{D}$, and so
obviously Hom-sets still have the structure of an abelian group,
and composition is still bilinear. Consequently, being a zero or
identity morphism in $\catfont{C}$ is coincident with being one in
$\catfont{D}$. If $A \oplus B$ is the $\catfont{D}$-biproduct of
the $\catfont{C}$-objects $A$ and $B$, then we have a
$\catfont{D}$-diagram
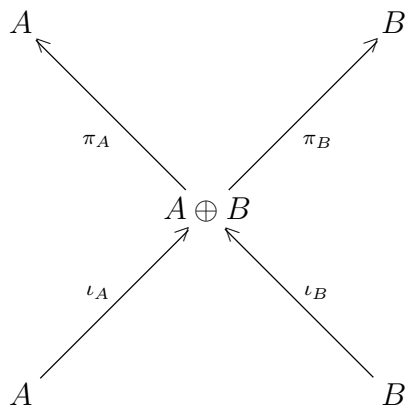
\begin{diagram}
A & & & & B \\
 & \luTo_{\pi_A} & & \ruTo_{\pi_B} & \\
 & & A \oplus B & & \\
 & \ruTo^{\iota_A} & & \luTo^{\iota_B} & \\
 A & & & & B \\
\end{diagram}
such that $\pi_A \circ \iota_A + \pi_B \circ \iota_B = 1_{A \oplus
B}$, $\iota_A \circ \pi_A = 1_A$, $\iota_B \circ \pi_B = 1_B$,
$\iota_A \circ \pi_B = 0$, and $\iota_B \circ \pi_A = 0$.  But all
these maps exist in $\catfont{C}$ as well, along with the given
relations, so this diagram constitutes a $\catfont{C}$-biproduct
for $A$ and $B$.

Let $A \stackrel{\phi}{\longrightarrow} B$ be a
$\catfont{C}$-morphism, and $K \stackrel{k}{\longrightarrow} A$
its $\catfont{D}$-kernel; this map exists in $\catfont{C}$ as
well, since $K$ is a subobject of $A$. If $L
\stackrel{\psi}{\longrightarrow} A$ is any $\catfont{C}$-morphism
with $\psi \circ \phi = 0$ in $\catfont{C}$, then this composition
is zero in $\catfont{D}$ as well; consequently, there is a unique
morphism $L \stackrel{\bar{\psi}}{\vlongrightarrow} K$ such that
$\bar{\psi} \circ k = \psi$.  As $\catfont{C}$ is full, this map
$\bar{\psi}$ exists also in $\catfont{C}$, and is clearly still
unique.  Thus $k$ is a $\catfont{C}$-kernel for $\phi$ as well,
which shows that all kernels exist in $\catfont{C}$.  An analogous
proof holds for the existence of cokernels, using the fact that
$\catfont{C}$ is closed under quotients.

Let $A \stackrel{\phi}{\longrightarrow} B$ be a
$\catfont{C}$-monomorphism; we claim that it is also a
monomorphism in $\catfont{D}$.  Let $X
\stackrel{\psi}{\longrightarrow} A$ be any $\catfont{D}$-morphism
such that $\psi \circ \phi = 0$; we wish to show that $\psi = 0$.
As every morphism in an abelian category factors through an
epimorphism and a monomorphism (page 199 of \cite{maclane}), we
have a commutative diagram
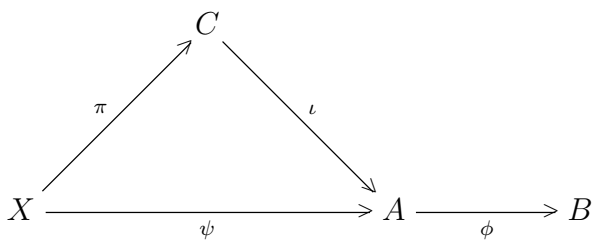
\begin{diagram}
 & & C & & & & \\
 & \ruTo^\pi & & \rdTo^\iota & & & \\
 X & & \rTo_\psi & & A & \rTo_\phi & B \\
\end{diagram}
where $\pi$ and $\iota$ are a
$\catfont{D}$-epimorphism/monomorphism respectively, giving $\pi
\circ \iota \circ \phi = 0$.  As $\pi$ is epic, we have $\iota
\circ \phi = 0$.  But $C$ is a subobject of $A$, hence a member of
$\catfont{C}$, and so by the $\catfont{C}$-monomorphic property of
$\phi$, $\iota = 0$.  Thus $\psi = \pi \circ \iota$ equals $0$ as
well, and we have shown that $\phi$ is a
$\catfont{D}$-monomorphism.

So if $A \stackrel{\phi}{\longrightarrow} B$ is any
$\catfont{C}$-monomorphism, it is also a
$\catfont{D}$-monomorphism, thus normal in $\catfont{D}$.  Then
let $B \stackrel{\psi}{\longrightarrow} C$ be the
$\catfont{D}$-morphism for which $\phi$ is the
$\catfont{D}$-kernel.  Again $\psi$ factors through an epimorphism
and a monomorphism, and we have a commutative diagram
\begin{diagram}
& & & & X& & \\
& & & \ruTo^\pi & & \rdTo^\iota & \\
A & \rTo_\phi & B & & \rTo_\psi & & C \\
\end{diagram}
The map $B \stackrel{\pi}{\longrightarrow} X$ exists in
$\catfont{C}$, $X$ being a quotient of $B$.  We claim that $\phi$
is a $\catfont{C}$-kernel for $\pi$.  If $L
\stackrel{\eta}{\longrightarrow} B$ is any $\catfont{C}$-map such
that $\eta \circ \pi = 0$, then also $\eta \circ \psi = \eta \circ
\pi \circ \iota = 0$; as $\phi$ is a kernel for $\psi$, there is a
unique map $L \stackrel{\bar{\eta}}{\vlongrightarrow} A$ such that
$\bar{\eta} \circ \phi = \eta$, which satisfies the universal
property of $\phi$ being a $\catfont{C}$-kernel for $\pi$.
Therefore all monomorphisms are normal in $\catfont{C}$. An
analogous argument shows that all epimorphisms in $\catfont{C}$
are normal.  Therefore $\catfont{C}$ is an abelian category.

Since the tensor product of two objects in $\catfont{C}$ is also
in $\catfont{C}$, so also is the tensor product of two morphisms,
since $\catfont{C}$ is full.  For objects $A,B$ and $C$ of
$\catfont{C}$, the associativity map $(A \otimes B) \otimes C
\stackrel{\text{assoc}_{A,B,C}}{\vvlongrightarrow} A \otimes (B
\otimes C)$ exists in $\catfont{C}$, and is clearly still natural.
Just as `monomorphic' and `epimorphic' are identical concepts in
$\catfont{C}$ and $\catfont{D}$, so is `isomorphic', and thus
assoc is a natural isomorphism in $\catfont{C}$.  Analogous
statements hold for the requisite isomorphisms comm and unit, the
latter existing in $\catfont{C}$ since the identity element of
$\catfont{D}$ is stipulated to exist in $\catfont{C}$.  The
coherence conditions 4., 5.~and 6.~of definition
\ref{defntensorcategory} clearly also still hold, as well as the
bilinearity of $\otimes$.

In any tensor category, one can in fact identify the object
$\intHom(A,B)$ with $\dual{A} \otimes B$; as duals are assumed
exist in $\catfont{C}$, so also do all internal Homs, as well as
the requisite `ev' maps since $\catfont{C}$ is full. The remaining
conditions of definition \ref{defnrigidity} merely stipulate that
certain maps must be isomorphisms; as $\catfont{C}$ is full, these
maps also exist in $\catfont{C}$, and are isomorphisms since they
are in $\catfont{D}$.  And of course, $\text{End}(\underline{1})$
is still a field. This completes the proof.

\end{proof}

\begin{lem}
\label{pullbackpushoutlemma} Let $\catfont{C}$ be a full abelian
subcategory of the abelian category $\catfont{D}$ which is closed
under the taking of biproducts, subobjects, and quotients.

\begin{enumerate}
\item{The $\catfont{C}$-diagram
\begin{diagram}
X & \rTo^{\pi_1} & X_1 \\
\dTo^{\pi_2} & & \dTo_{\phi_1} \\
X_2 & \rTo_{\phi_2} & Z \\
\end{diagram}
is a $\catfont{C}$-pullback for $X_2
\stackrel{\phi_2}{\longrightarrow} Z
\stackrel{\phi_1}{\longleftarrow} X_1$ if and only if it is also a
$\catfont{D}$-pullback for $X_2 \stackrel{\phi_2}{\longrightarrow}
Z \stackrel{\phi_1}{\longleftarrow} X_1$}
 \item{The
$\catfont{C}$-diagram
\begin{diagram}
X & \lTo^{\iota_1} & X_1 \\
\uTo^{\iota_2} & & \uTo_{\psi_1} \\
X_2 & \lTo_{\psi_2} & Z \\
\end{diagram}
is a $\catfont{C}$-pushout for $X_2
\stackrel{\psi_2}{\longleftarrow} Z
\stackrel{\psi_1}{\longrightarrow} X_1$ if and only if it is also
a $\catfont{D}$-pushout for $X_2 \stackrel{\psi_2}{\longleftarrow}
Z \stackrel{\psi_1}{\longrightarrow} X_1$}
\end{enumerate}
\end{lem}

\begin{proof}
We will prove this for pullbacks, leaving the pushout case to the
reader.  Suppose first that
\begin{diagram}
X & \rTo^{\pi_1} & X_1 \\
\dTo^{\pi_2} & & \dTo_{\phi_1} \\
X_2 & \rTo_{\phi_2} & Z \\
\end{diagram}
is a $\catfont{D}$-pullback diagram.  Let $T$ be an object of
$\catfont{C}$, and suppose we have a commutative diagram
\begin{diagram}
T & \rTo^{\rho_1} & X_1 \\
\dTo^{\rho_2} & & \dTo_{\phi_1} \\
X_2 & \rTo_{\phi_2} & Z \\
\end{diagram}
Then by the universal property of being a $\catfont{D}$-pullback,
there is a unique $\catfont{D}$-map $\rho:T \rightarrow X$ such
that $\rho \circ \pi_1 = \rho_1$ and $\rho \circ \pi_2 = \rho_2$.
As $\catfont{C}$ is full, this map $\rho$ exists in $\catfont{C}$
as well, satisfies these relations, and is clearly still unique.
Thus this diagram constitutes a $\catfont{C}$-pullback as well.

Conversely, suppose the above is a $\catfont{C}$-pullback diagram.
As $\catfont{D}$ is abelian, we know  that $X_1
\stackrel{\phi_1}{\longrightarrow} Z
\stackrel{\phi_2}{\longleftarrow} X_2$ \emph{has} a
$\catfont{D}$-pullback, say
\begin{diagram}
U & \rTo^{\mu_1} & X_1 \\
\dTo^{\mu_2} & & \dTo_{\phi_1} \\
X_2 & \rTo_{\phi_2} & Z \\
\end{diagram}
The proof of theorem 2.15 of \cite{freyd} shows that $U$ can
always be taken to be, up to isomorphism, a certain subobject of
$X_1 \oplus X_2$.  As $\catfont{C}$ is closed under the taking of
biproducts and subobjects, $U$ is an object of $\catfont{C}$, and
as $\catfont{C}$ is full, $\mu_1$ and $\mu_2$ are morphisms in
$\catfont{C}$; thus, this diagram belongs to $\catfont{C}$.  Then
by the above, as this diagram is a $\catfont{D}$-pullback, so also
is it a $\catfont{C}$-pullback. As any two pullbacks in an abelian
category are isomorphic up to a unique isomorphism, we must have
that
\begin{diagram}
T & \rTo^{\rho_1} & X_1 \\
\dTo^{\rho_2} & & \dTo_{\phi_1} \\
X_2 & \rTo_{\phi_2} & Z \\
\end{diagram}
is a $\catfont{D}$-pullback as well.
\end{proof}

\begin{lem}
\label{freydlemma}
 Let $\catfont{D}$ be an abelian category,
$\catfont{C}$ a non-empty full abelian subcategory of
$\catfont{D}$.  Then exact sequences in $\catfont{C}$ are also in
$\catfont{D}$ if and only if for every morphism $A
\stackrel{\phi}{\longrightarrow} B$ in $\catfont{C}$, there is a
$\catfont{D}$-kernel of $\phi$, $\catfont{D}$-cokernel of $\phi$,
and $\catfont{D}$-direct sum which all lie in $\catfont{C}$.
\end{lem}

\begin{proof}
See theorem 3.41 of \cite{freyd}.
\end{proof}

In the case of $1$-fold extensions in a $k$-linear abelian
category, we shall need something slightly stronger.

\begin{prop}
\label{linIndExtensionsInSubcategoriesProp}
 Let $\catfont{D}$ be a $k$-linear abelian category,
$\catfont{C}$ a non-empty full $k$-linear abelian subcategory of
$\catfont{D}$.  Let $M$ and $N$ be objects of $\catfont{C}$ and
let $\xi_1, \ldots, \xi_m$ be a sequence of $1$-fold extensions of
$N$ by $M$ in $\catfont{C}$.  Then the $\xi_j$ are linearly
independent in $\catfont{C}$ if and only if they are linearly
independent in $\catfont{D}$.
\end{prop}

\begin{proof}
By `linearly independent', we mean with respect to the $k$-vector
space structure defined by the Baer sum on $\text{Ext}^1(M,N)$
(see section \ref{CohomologyOfComodulesSection}).

Denote by $\text{EXT}^1_{\catfont{C}}(M,N)$ the collection of all
$1$-fold extensions of $M$ by $N$ in $\catfont{C}$, and define
similarly $\text{EXT}^1_{\catfont{D}}(M,N)$ (this is different
from $\text{Ext}^1_{\catfont{C}}(M,N)$, which is the collection of
all \emph{equivalence classes} of $1$-fold extensions). By the
previous lemma, for every $\xi \in
\text{EXT}^1_{\catfont{C}}(M,N)$, $\xi$ is also a member of
$\text{EXT}^1_{\catfont{D}}(M,N)$, whence we have a map
$\text{EXT}^1_{\catfont{C}}(M,N) \rightarrow
\text{EXT}^1_{\catfont{D}}(M,N)$.

We claim firstly that this map respects equivalence of extensions.
If $\xi:0 \rightarrow N \rightarrow X \rightarrow M \rightarrow
0$, $\chi: 0 \rightarrow N \rightarrow Y \rightarrow M \rightarrow
0$ are two $\catfont{C}$-equivalent extensions, then we have a
$\catfont{C}$-isomorphism $\phi:X \rightarrow Y$ making
\begin{diagram}
\xi:0 & \rTo & N & \rTo & X & \rTo & M & \rTo & 0 \\
 & & \dEq & & \dTo^\phi & & \dEq \\
\chi:0 & \rTo & N & \rTo & X & \rTo & M & \rTo & 0 \\
\end{diagram}
commute.  But as $\phi$ is a $\catfont{D}$-isomorphism as well,
$\xi$ and $\chi$ are also $\catfont{D}$-equivalent.  Thus, our map
$\text{EXT}^1_{\catfont{C}}(M,N) \rightarrow
\text{EXT}^1_{\catfont{D}}(M,N)$ is actually a map
$\text{Ext}^1_{\catfont{C}}(M,N) \rightarrow
\text{Ext}^1_{\catfont{D}}(M,N)$. This map is injective, for if
$\xi$ and $\chi$ are $\catfont{D}$-equivalent according to the
above diagram, then they are also $\catfont{C}$-equivalent, since
the map $\phi$ exists in $\catfont{C}$.

What is left then is to verify that this map is linear.  Let $k$
be a scalar and $\xi:0 \rightarrow N
\stackrel{\phi}{\longrightarrow} X
\stackrel{\psi}{\longrightarrow} M \rightarrow 0$ a
$\catfont{C}$-extension of $M$ by $N$.  Then the scalar
multiplication of $k$, when $k \neq 0$, is defined to be
\[ k\xi:0 \rightarrow N \stackrel{k^{-1} \phi}{\longrightarrow}  X
\stackrel{\psi}{\longrightarrow} M \rightarrow 0 \]
and in case
$k=0$, as the trivial extension.  As scalar multiplication of
morphisms and trivial extensions are defined the same way in
$\catfont{C}$ as in $\catfont{D}$, so also is scalar
multiplication of extensions.

Let $\xi,\chi$ be two extensions.  To compute the Baer sum $\xi
\oplus \chi$, we are asked to compute a certain pullback, to
compute a pair of unique maps pushing through a pullback, to
compute a certain cokernel of one of these maps, to compute a
unique map pushing through a cokernel, and finally to compute a
certain composition.  Lemma \ref{pullbackpushoutlemma}, as well as
the proof of lemma \ref{tannakiansubcategorylemma}, show that all
of these constructions lead to the same answer whether done in
$\catfont{C}$ and $\catfont{D}$. We conclude that if $\xi \oplus
\chi = \eta$ in $\catfont{C}$, so also does this equation hold in
$\catfont{D}$. This completes the proof.
\end{proof}

\begin{lem}
A full tannakian subcategory of a neutral tannakian category is
also neutral (over the same field and via the restriction of the
same fibre functor).
\end{lem}

\begin{proof}
We are given a neutral tannakian category $\catfont{D}$ with fibre
functor $\omega:\catfont{D} \rightarrow \text{Vec}_k$, where $k$
is the field $\text{End}(\underline{1})$.  We want to show that
this functor restricted to $\catfont{C}$, which we still call
$\omega$, qualifies as a fibre functor on $\catfont{C}$. Looking
at the conditions of definition \ref{defntensorfunctor} it is easy
to verify that $\omega$ restricted to $\catfont{C}$ is still a
tensor functor. The requisite isomorphism $c_{X,Y} : \omega(A
\otimes B) \isomorphic \omega(A) \otimes \omega(B)$ still exists,
is natural, and still satisfies the relevant diagrams, since e.g.
assoc in $\catfont{C}$ is the same as assoc in $\catfont{D}$.
Faithfulness and $k$-linearity are also clearly still satisfied;
all that remains to check is exactness. By proposition
\ref{tannakiansubcategorylemma}, $\catfont{C}$ and $\catfont{D}$
satisfy the hypothesis of lemma \ref{freydlemma}, and thus exact
sequences in $\catfont{C}$ are also in $\catfont{D}$.  As $\omega$
preserves exact sequences in $\catfont{D}$, so must it also when
restricted to $\catfont{C}$.
\end{proof}

\chapter{Some Ultraproduct Constructions}

In the next chapter we shall be studying ultraproducts of
tannakian categories.  In this chapter we define several
ultraproduct constructions and record several results on them that
will soon be necessary; the learned reader may wish to treat this
chapter merely as a reference. The reader may also consult the
appendix for a review of ultrafilters and ultraproducts in
general.

In this dissertation, if $M_i$ is a collection of relational
structures in a common first-order signature, indexed by $I$, and
if \index{$\filtfont{U}$} $\filtfont{U}$ is a non-principal
ultrafilter on $I$, we denote by \index{$\uprod M_i$} $\uprod M_i$
the \index{ultraproduct} ultraproduct of those structures with
respect to $\filtfont{U}$.  For a tuple of elements $(x_i)$ from
the $M_i$, we denote by \index{$[x_i]$} $[x_i]$ its equivalence
class, that is, its image as an element of $\uprod M_i$. When we
make statements like ``the ultraproduct of vector spaces is a
vector space over the ultraproduct of the fields'', it will always
be the case that these ultraproducts, both for fields and vector
spaces, are being taken with respect to the same fixed ultrafilter
(as indeed it makes no sense to assume otherwise).

\section{Fields}

\label{ultraproductsoffields}

\index{ultraproduct!of fields}

Let $k_i$ be a sequence of fields indexed by $I$.  We treat the
$k_i$ as structures in the language \index{first-order language!of
fields} $+, \multsymbol, -, 0,1$ with the obvious interpretation.
For brevity we write the term $x \multsymbol y$ as the
juxtaposition $xy$.  As always we fix a non-principal ultrafilter
$\filtfont{U}$ on $I$ throughout.

\begin{prop}
\label{ultraprodoffieldsisafieldprop} $\uprod k_i$ is a field.
\end{prop}

\begin{proof}
Simply realize that the axioms for a field are first-order
sentences in this language:

\begin{enumerate}
\item{$(\forall x,y,z)((x+y)+z = x+(y+z) \myand x(yz)=(xy)z)$}
\item{$(\forall x,y)(x+y=y+x \myand xy=yx)$} \item{$(\forall x)(1x
= x \myand 0+x=x)$} \item{$(\forall x)(x+(-x) = 0)$}
\item{$(\forall x,y,z)(x(y+z) = xy+xz)$} \item{$(\forall
x)(\mynot(x=0) \myimplies (\exists y)(xy=1))$}
\item{$\mynot(1=0)$}
\end{enumerate}

Now apply corollary \ref{loscor}.

\end{proof}

\begin{prop}
Suppose that $k_i$ is a sequence of fields of strictly increasing
positive characteristic.  Then $\uprod k_i$ has characteristic
zero.
\end{prop}

\begin{proof}
For a fixed prime $p$, let $\text{char}_p$ be the statement $1+1+
\ldots + 1 = 0 $ ($p$-occurrences of $1$).  As the $k_i$ have
strictly increasing characteristic, for fixed $p$, $\text{char}_p$
is false in all but finitely many of them.  Thus $\mynot
\text{char}_p$ holds on a cofinite set, which is always large, and
so $\mynot \text{char}_p$ holds in the ultraproduct.  This goes
for every $p$, which is equivalent to $\uprod k_i$ having
characteristic zero.
\end{proof}

\section{Vector Spaces}

\label{ultraproductsofvectorspaces}

\index{ultraproduct!of vector spaces}

Let $V_i$ be an indexed collection of vector spaces over the
fields $k_i$.  We treat the $V_i$ simply as structures in the
signature \index{first-order language!of vector spaces} $+,0$,
i.e.~as abelian groups, forgetting for the moment the scalar
multiplication.  Then $\uprod V_i$ is also an abelian group in
this signature, and the addition is of course given by $[v_i] +
[w_i] \stackrel{\text{def}}{=} [v_i + w_i]$.

Let $k= \uprod k_i$ be the ultraproduct of the fields $k_i$ as in
the previous section.  We assume as always that both of these
ultraproducts are being taken with respect to the same fixed
non-principal ultrafilter.

\begin{thm}
$\uprod V_i$ is a vector space over $\uprod k_i$, under the scalar
multiplication
\[ [a_i][v_i] \stackrel{\text{\emph{def}}}{=} [a_i v_i] \]
\end{thm}

\begin{proof}
The given multiplication is well-defined: If $(a_i),(b_i)$ are
equal on the large set $J$, and if $(v_i),(w_i)$ are equal on the
large set $K$, then $(a_i v_i)$ and $(b_i w_i)$ are equal on at
least the large set $J \cap K$.  It is routine to verify that this
definition satisfies the axioms of a vector space.
\end{proof}

\begin{prop}
The finite collection of linear equations of the form
\[ [a_i]_1[v_i]_1 + \ldots + [a_i]_n [v_i]_n = [0] \]
is true in $\uprod V_i$ if and only if the corresponding
collection of linear equations
\[ a_{i,1}v_{i,1}+ \ldots + a_{i,n} v_{i,n} = 0 \]
is true for almost every $i$.
\end{prop}

\begin{proof} For the forward implication, the claim is obvious in the case of a single
equation, since the first equation is equivalent to
$[a_{i,1}v_{i,1} + \ldots + a_{i,n} v_{i,n}] = [0]$.  For a finite
set of equations, each individual equation holds on a large set,
and taking the finite intersection of these large sets, we see
that there is a large set on which all the equations hold.  The
reverse implication is obvious.
\end{proof}

\begin{prop}
\index{ultraproduct!of bases} \label{basesarepreservedprop} A
finite set of vectors $[e_i]_1, \ldots, [e_i]_n$ is a basis for
$\uprod V_i$ over $\uprod k_i$ if and only if, for almost every
$i$, the set of vectors $e_{i,1}, \ldots, e_{i,n}$ is a basis for
$V_i$ over $k_i$.
\end{prop}

\begin{proof}
Given a linear dependence $[a_i]_1[e_i]_1 + \ldots + [a_i]_n
[e_i]_n = [0]$ in $\uprod V_i$, we get a linear dependence for
almost every $i$, by the previous proposition.  If $[a_i]_j \neq
[0]$, then $a_{i,j} \neq 0$ for almost every $i$, and by taking
another intersection we get a non-trivial dependence in almost
every $i$.  Conversely, if we have a non-trivial dependence
$a_{i,1} e_{i,1} + \ldots + a_{i,n} e_{i,n} = 0$ in almost every
$i$, the equation $[a_i]_1 [e_i]_1 + \ldots + [a_i]_n [e_i]_n =
[0]$ holds in $\uprod V_i$.  By lemma \ref{coveringlemma}, at
least one of the $[a_i]_j$ must be non-zero.

If $e_{i,j}$ span $V_i$ for almost every $i$, then for every
$[v_i] \in \uprod V_i$, we have an almost everywhere valid
equation $a_{i,1}e_{i,1} + \ldots + a_{i,n}e_{i,n} = v_i$ in
$V_i$, which projects to an equation $[a_i]_1 [e_i]_1 + \ldots +
[a_i]_n [e_i]_n = [v_i]$, showing that the $[e_i]_j$ span $\uprod
V_i$.  Conversely, if the $e_{i,j}$, almost everywhere, do not
span $V_i$, choose $v_i$ for each of those slots which are not in
the span of the $e_{i,j}$.  Then neither can $[v_i]$ be in the
span of the $[e_i]_j$, lest we project back to an almost
everywhere linear combination for $v_i$ in terms of the $e_{i,j}$.
\end{proof}

\begin{prop}
\label{prop1} For a fixed non-negative integer $n$, $\uprod V_i$
has dimension $n$ over $\uprod k_i$ if and only if almost every
$V_i$ has dimension $n$ over $k_i$.  $\uprod V_i$ is infinite
dimensional over $\uprod k_i$ if and only if, for every $n$,
almost every $V_i$ does not have dimension $n$ over $k_i$.
\end{prop}

\begin{proof}
Apply the previous proposition.
\end{proof}

We see then that if $V_i$ is almost everywhere of dimension $n <
\infty$, it does no harm to assume that it has dimension $n$
everywhere.  We call such collections \index{constantly finite
dimensional} \textbf{constantly finite dimensional} or sometimes
\index{boundedly finite dimensional} \textbf{boundedly finite
dimensional}. With few exceptions it is these types of collections
of vector spaces we will be concerning ourselves with.

\subsection{Linear Transformations and Matrices}

\label{lineartransformationsandmatricessubsection}

For a collection of vector spaces $V_i$, $W_i$ and linear maps
$\phi_i:V_i \rightarrow W_i$, denote by $[\phi_i]$ the linear map
$\uprod V_i \rightarrow \uprod W_i$ defined by
\index{ultraproduct!of linear transformations} \index{$[\phi_i]$}
\begin{equation}
\label{ultraproductsoflineartransformations}
 [\phi_i]([v_i]) = [\phi_i(v_i)]
\end{equation}
Denote by $\uprod \text{Hom}_{k_i}(V_i,W_i)$ the collection of all
such transformations of the form $[\phi_i$].  So long as the $V_i$
are constantly finite dimensional, we are justified in using this
notation because

\begin{prop}
\label{homprop} If the $V_i$ are of constant finite dimension,
$[\phi_i] = [\psi_i]$ as linear transformations if and only if,
for almost every $i$, $\phi_i = \psi_i$ as linear transformations.
Further, $[\phi_i] \circ [\psi_i] = [\phi_i \circ \psi_i]$,
$[\phi_i] + [\psi_i] = [\phi_i + \psi_i]$, and for an element
$[a_i]$ of $\uprod k_i$, $[a_i][\phi_i] = [a_i \phi_i]$.
\end{prop}

\begin{proof}
The `if' direction is obvious.  For the converse, let $[e_i]_1,
\ldots, [e_i]_n$ be a basis for $\uprod V_i$, whence, for almost
every $i$, $e_{i,1}, \ldots, e_{i,n}$ is a basis for $V_i$.
$[\phi_i] = [\psi_i]$ if and only if they agree on this basis, so
let $J_m, m = 1 \ldots n$ be the large set on which $\phi_i(e_i^m)
= \psi_i(e_i^m)$.  Then the finite intersection of these $J_m$, on
which $\phi_i$ and $\psi_i$ agree on every basis element, thus on
which $\phi_i = \psi_i$, is large.  The last three claims of the
proposition are now obvious.
\end{proof}

If the $V_i$ are of unbounded dimensionality, the theorem does not
hold; the proof falls apart when we try to take the intersection
of the $J_m$, which in this case may well be an infinite
intersection, and not guaranteed to be large.

\begin{prop}
\label{homlemma} If $V_i$, $W_i$ are constantly finite dimensional
collections of vector spaces, then
\[ \text{Hom}_{\uprod k_i}(\uprod V_i, \uprod W_i) \isomorphic \uprod
\text{Hom}_{k_i}(V_i,W_i)\]
\end{prop}

\begin{proof}
It is easy to verify that the right hand side of the above claimed
isomorphism is always included in the left hand side (even if the
$V_i$ or $W_i$ are not boundedly finite dimensional).  For the
other inclusion, pick bases $[e_i^1],\ldots, [e_i^n]$ and
$[f_i^1],\ldots,[f_i^m]$ of $\uprod V_i$ and $\uprod W_i$
respectively.  For a linear transformation $\phi$ in the left hand
side of the claimed isomorphism, write it as an $n \times m$
matrix
\[
\left(%
\begin{array}{ccc}
  [a_i^{1,1}] & \ldots & [a_i^{1,n}] \\
  \vdots &  & \vdots \\
  \left[a_i^{m,1}\right] & \ldots & \left[a_i^{m,n}\right] \\
\end{array}%
\right)
\]
in the given bases.  Then one can verify by hand that $\phi$ is of
the form $[\phi_i]$, where each $\phi_i$ is the transformation
$V_i \rightarrow W_i$ given by the matrix
\[
\left(%
\begin{array}{ccc}
  a_i^{1,1} & \ldots & a_i^{1,n} \\
  \vdots &  & \vdots \\
  a_i^{m,1} & \ldots & a_i^{n,m} \\
\end{array}%
\right)
\]
in the bases $e_i^1,\ldots, e_i^n$, $f_i^1, \ldots, f_i^n$.
\end{proof}

We can therefore always assume that a linear transformation
$\uprod V_i \rightarrow \uprod W_i$ is uniquely of the form
$[\phi_i]$, so long as the $V_i$ and $W_i$ are of constant finite
dimension. This theorem is not true if the $V_i$ and $W_i$ are not
both boundedly finite dimensional; the forward inclusion fails.

\begin{defn}

\index{ultraproduct!of matrices}

Let $M_i$ be a sequence of $n \times m$ matrices over the fields
$k_i$, given by
\[
\left(%
\begin{array}{ccc}
  a_i^{1,1} & \ldots & a_i^{1,n} \\
  \vdots &  & \vdots \\
  a_i^{m,1} & \ldots & a_i^{m,n} \\
\end{array}%
\right)
\]
Then we define the ultraproduct of these matrices, denoted
$[M_i]$, to be the $n \times m$ matrix over the field $\uprod k_i$
given by
\[
\left(%
\begin{array}{ccc}
  [a_i^{1,1}] & \ldots & [a_i^{1,n}] \\
  \vdots &  & \vdots \\
  \left[a_i^{m,1}\right] & \ldots & \left[a_i^{m,n}\right] \\
\end{array}%
\right)
\]
\end{defn}

\begin{prop}
\label{matrixlemma} Let $V_i$, $W_i$ have constant dimension $n$
and $m$, with bases $e_i^1, \ldots, e_i^n$, $f_i^1, \ldots, f_i^m$
respectively. If $\phi_i:V_i \rightarrow W_i$ is represented by
the $n \times m$ matrix $M_i$ in the given bases, then $[\phi_i]$
is represented by the matrix $[M_i]$ in the bases $[e_i^1],
\ldots, [e_i^n]$,$[f_i^1],\ldots,[f_i^m]$.
\end{prop}

\begin{proof}
Obvious.
\end{proof}

Applying propositions \ref{homlemma} and \ref{matrixlemma}
together yield

\begin{cor}
For integers $m$ and $n$ and fields $k_i$,
\[ \text{Mat}_{n,m}(\uprod k_i) \isomorphic \uprod
\text{Mat}_{n,m}(k_i) \] where the latter stands for all matrices
of the form $[M_i]$, $M_i \in \text{Mat}_{n,m}(k_i)$.
\end{cor}

As we have seen, an ultraproduct of linear transformations
preserves composition, addition, and scalar multiplication.  By
induction on complexity, it thus also preserves any equation
involving a finite combination of these three operations, and by
considering a finite intersection of large sets, the same is true
for any finite collection of such equations.  We state this as a
theorem.

\begin{thm}
\label{transformationsarenicethm} Let $[\phi_i]_1, \ldots,
[\phi_i]_n$ be a finite collection of linear transformations, all
between ultraproducts of constantly finite dimensional vector
spaces. Then a finite collection of equations among the
$[\phi_i]_j$ involving addition of maps, composition, and scalar
multiplication is valid if and only if the corresponding
collection of equations among the $\phi_{i,1}, \ldots, \phi_{i,n}$
is valid almost everywhere.
\end{thm}

By proposition \ref{matrixlemma}, the same is true for matrices:

\begin{cor}
\label{matricesarenicethm} The same is true for a finite
collection $[M_i]_1, \ldots, [M_i]_k$ of matrices over $\uprod
k_i$, if we replace `addition of maps', `composition', and `scalar
multiplication of maps' with `addition of matrices',
`multiplication of matrices', and `scalar multiplication of
matrices'.
\end{cor}

\begin{prop}
\label{ultrainjectivityprop} Over collections of constantly finite
dimensional vector spaces, ultraproducts preserve injectivity,
surjectivity, kernels and cokernels.
\end{prop}

\begin{proof}
Let $V_i \stackrel{\phi_i}{\longrightarrow} W_i$ be a collection
of linear maps.  Suppose first that almost every $\phi_i$ is
injective.  Then if $[v_i] \neq [0]$, $v_i \neq 0$ for almost
every $i$, and taking the intersection of these two large sets,
$\phi_i(v_i) \neq 0$ for almost every $i$, which is the same as
saying $[\phi_i]([v_i]) \neq [0]$.  Thus $[\phi_i]$ is injective.

For the converse, we must use the constant finite dimensionality
of the $V_i$.  Let $[e_i]_1, \ldots, [e_i]_n$ be a basis for
$\uprod V_i$, so that $e_{i,1},\ldots, e_{i,n}$ is a basis for
almost every $V_i$, say on the large set $J$. Suppose  that almost
every $\phi_i$ is not injective, say on the large set $K$, and for
each $i \in K$ let $v_i = a_{i,1} e_{i,1} + \ldots + a_{i,n}
e_{i,n}$ be a non-zero vector such that $\phi_i(v_i) = 0$.  Then
at least one of the $a_{i,j}$ is non-zero for each $i$.  By lemma
\ref{coveringlemma}, at least one of $[a_i]_j$ is non-zero in
$\uprod k_i$.  Then we see that $[\phi_i]([a_i]_1 [e_i]_1 + \ldots
[a_i]_n [e_i]_n) = [0]$, but $[a_i]_m \neq 0$; hence $[\phi_i]$ is
not injective.

The proof of surjectivity is similarly proved, using instead the
constant dimensionality of the $W_i$.

Suppose that, for almost every $i$, $\phi_i$ is a kernel map for
$\psi_i$.  This is the assertion that $\phi_i$ is injective, that
everything in the image of $\phi_i$ is killed by $\psi_i$, and
that nothing outside the image of $\phi_i$ is killed by $\psi_i$.
That $[\phi_i]$ is also injective has already been proved.  To say
that $[v_i]$ is in the image of $[\phi_i]$ is equivalent to saying
that $v_i$ is in the image of $\phi_i$ for almost every $i$, and
to say that $[\psi_i]$ kills $[w_i]$ is equivalent to saying that
$\psi_i$ kills $w_i$ for almost every $i$; the same goes for their
negations.

The case of cokernels is proved similarly, using instead the fact
that surjectivity is preserved.
\end{proof}

\begin{prop}
\label{ultraexactsequencethm} \label{ultraexactsequenceprop} Over
collections of constantly finite dimensional vector spaces, the
collection of diagrams
\[ 0 \rightarrow X_i^1 \stackrel{\phi_i^1}{\longrightarrow} \ldots
\stackrel{\phi_i^n}{\longrightarrow} X_i^{n+1} \rightarrow 0 \] is
almost everywhere exact if and only if the corresponding sequence
\[ [0] \rightarrow [X_i]_1 \stackrel{[\phi_i]_1}{\vlongrightarrow}
\ldots \stackrel{[\phi_i]_n}{\vlongrightarrow} [X_i]_{n+1}
\rightarrow [0] \] is exact.
\end{prop}

\begin{proof}
The assertion that the sequence
\[ X \stackrel{\phi}{\longrightarrow} Y
\stackrel{\psi}{\longrightarrow} Z \] is exact amounts to the
assertion that $\phi$ is a kernel for $\psi$, minus the
requirement that $\phi$ be injective, which (the proof of)
proposition \ref{ultrainjectivityprop} shows to be preserved by
ultraproducts.  Checking that the above sequences are exact
amounts to checking finitely many sub-sequences of this form,
which is preserved by ultraproducts.
\end{proof}

\subsection{Tensor Products}

\index{ultraproduct!of tensor products}

\label{tensorproductssubsection}

\begin{prop}
\label{ultratensorproductprop} Let $V_i,W_i$ be (not necessarily
boundedly finite dimensional) collections of vector spaces.  Then
there is a natural injective map
\[ \uprod V_i \otimes \uprod W_i \stackrel{\Phi}{\vlongrightarrow}
\uprod V_i \otimes W_i \] given by $[v_i] \otimes [w_i] \mapsto
[v_i \otimes w_i]$.  The image of \index{$\Phi$} $\Phi$ consists
exactly of those elements having bounded tensor length. The map
$\Phi$ is an isomorphism if and only if at least one of the
collections $V_i$ or $W_i$ are of \index{boundedly finite
dimensional} bounded finite dimension.
\end{prop}

\begin{proof}
That the map $\Phi$ is well-defined is easy to verify, remembering
of course that the same ultrafilter applies to all ultraproducts
under consideration. Injectivity is likewise easy to verify.  Any
element on the left hand side, by the very definition of tensor
product, has bounded tensor length, and hence so must its image on
the right hand side. Conversely, if $ [\sum_{j=1}^n v_{ij} \otimes
w_{ij} ]$ is an element of bounded tensor length on the right hand
side, then $\sum_{j=1}^n [v_i]_j \otimes [w_i]_j$ is a pre-image
for it on the left.

To prove the isomorphism claim: if $V$ and $W$ are vector spaces
of finite dimension $n$ and $m$ respectively, then the maximum
tensor length of any element of $V \otimes W$ is $\text{min}(n,m)$
(see lemma \ref{minTensorLengthLemma}).  Then if say $V_i$ is of
bounded finite dimension $n$, any $[x_i] \in \uprod V_i \otimes
W_i$ is almost everywhere a sum of no more than $n$ simple
tensors, and is in the image of $\Phi$.

Conversely, suppose neither of $V_i$ or $W_i$ are of bounded
dimension.  For each $i$, choose $x_i \in V_i \otimes W_i$ such
that $x_i$ is of maximum possible tensor length; we claim that
$[x_i] \in \uprod V_i \otimes W_i$ is of unbounded tensor length,
hence not in the image of $\Phi$.  If not, say the tensor length
of $x_i$ is almost everywhere bounded by $n$.  This gives a large
set on which the statement ``at least one of $V_i$ or $W_i$ has
dimension $n$'' is true. This large set is covered by the union of
$\{i \in I: \text{$V_i$ has dimension $n$}\}$ and $\{i \in I:
\text{$W_i$ has dimension $n$}\}$, so by lemma
\ref{coveringlemma}, at least one of them must be large. This
gives a large set on which at least one of $V_i$ or $W_i$ is of
bounded dimension, a contradiction.
\end{proof}

This next proposition justifies our calling the map $\Phi$
`natural'.

\begin{prop}
\label{ultratensormapsprop}
 If $V_i$, $W_i$, $X_i$, $Y_i$ are
collection of vector spaces, and $\phi_i:V_i \rightarrow X_i$,
$\psi_i:W_i \rightarrow Y_i$ linear maps, then the following
commutes:
\begin{equation}
\begin{diagram}
\label{naturalityofPhidiagram} \uprod V_i
\otimes \uprod W_i& \rTo^{\Phi} & \uprod V_i \otimes W_i \\
\dTo^{[\phi_i] \otimes
[\psi_i]}  & &  \dTo_{[\phi_i \otimes \psi_i]}\\
\uprod X_i
\otimes \uprod Y_i  & \rTo_{\Phi} & \uprod X_i \otimes Y_i  \\
\end{diagram}
\end{equation}

\end{prop}

\begin{proof}
Let $[v_i] \otimes [w_i]$ be a simple tensor in $\uprod V_i
\otimes \uprod W_i$.  Chasing it both ways gives the same result:
\begin{diagram}
[v_i] \otimes [w_i] & \rTo^\Phi & [v_i \otimes w_i] \\
\dTo^{[\phi_i] \otimes [\psi_i]} & & \dTo_{[\phi_i \otimes w_i]} \\
[\phi_i(v_i)] \otimes [\psi_i(w_i)] &  \rTo_\Phi  & [\phi(v_i)
\otimes \psi_i(w_i)] \\
\end{diagram}

\end{proof}

\section{Algebras and Coalgebras}

\index{ultraproduct!of algebras} If $(L_i,\text{mult}_i)$ is a
collection of algebras over the fields $k_i$ then it is easy to
verify that $\uprod L_i$ is an algebra over the field $\uprod
k_i$, under the obvious definitions of addition, multiplication,
and scalar multiplication.  The multiplication on $\uprod L_i$ is
in particular defined as the composition
\[ \text{mult}:\uprod L_i \otimes \uprod L_i
\stackrel{\Phi}{\longrightarrow} \uprod L_i \otimes L_i
\stackrel{[\text{mult}_i]}{\vlongrightarrow} \uprod L_i \]

\index{ultraproduct!of coalgebras} Alas, for coalgebras, things
are not so easy. Here is what can go wrong.  Suppose
$(C_i,\Delta_i,\varepsilon_i)$ is a collection of coalgebras over
the fields $k_i$.  Then $\uprod C_i$ is at least a vector space
over $\uprod k_i$.  Now let us try to define a co-multiplication
 map $\Delta$ on $\uprod C_i$.  We start by writing

\[ \Delta: \uprod C_i \stackrel{[\Delta_i]}{\vlongrightarrow} \uprod C_i
\otimes C_i \] But as it stands, this won't suffice; we need
$\Delta$ to point to $\uprod C_i \otimes \uprod C_i$.  As is shown
in proposition \ref{ultratensorproductprop}, unless the $C_i$ are
of boundedly finite dimension, we only have an inclusion $\uprod
C_i \otimes \uprod C_i \stackrel{\Phi}{\vlongrightarrow} \uprod
C_i \otimes C_i$, whose image consists of those elements of
bounded tensor length, and for a typical collection $C_i$ of
coalgebras it is usually a simple matter to come up with an
element $[c_i] \in \uprod C_i$ such that $[\Delta_i(c_i)]$ has
unbounded tensor length.  Thus, the $\Delta$ constructed above
cannot be expected to point to $\uprod C_i \otimes \uprod C_i$ in
general (this problem is dealt with at length in section
\ref{sectionTheRestrictedUltraproductOfHopfAlgebras}).

Nonetheless, if the $C_i$ \emph{are} of boundedly finite
dimension, then the map $\Phi$ is an isomorphism, whence we can
define
\[ \Delta:\uprod C_i \stackrel{[\Delta_i]}{\vlongrightarrow} \uprod C_i \otimes C_i
\stackrel{\Phi^{-1}}{\vlongrightarrow} \uprod C_i \otimes \uprod
C_i \] Likewise, we define a co-unit map by
\[ \varepsilon: \uprod C_i
\stackrel{[\varepsilon_i]}{\vlongrightarrow} \uprod k_i \] and we
have

\begin{prop}
If $(C_i,\Delta_i,\varepsilon_i)$ is a collection of boundedly
finite dimensional coalgebras over the fields $k_i$, then $\uprod
C_i$ is a coalgebra over the field $\uprod k_i$, under the
definitions of $\Delta$ and $\varepsilon$ given above.
\end{prop}

\begin{proof}
We must verify diagrams \ref{hopf1} and \ref{hopf2} of definition
\ref{coalgebradefn}.  Consider
\begin{diagram}
\text{ } & \lTo & \uprod C_i \ & \rdTo(4,2) \\
 & & \dTo^{[\Delta_i]}  &\rdTo(4,2)^{\isomorphic} \rdTo_{[\isomorphic_i]} &  & \\
 \dTo^\Delta & & \uprod C_i \otimes C_i & \rTo^{[1 \otimes
 \varepsilon_i]} & \uprod C_i \otimes k_i & \rTo^{\Phi^{-1}} &
 \uprod C_i \otimes k_i \\
  & & \dTo^{\Phi^{-1}} & & & \ruTo(4,2)_{1 \otimes \varepsilon} & \\
 \text{ } & \rTo & \uprod C_i \otimes \uprod C_i \\
\end{diagram}
Commutativity of the top middle triangle follows from the
everywhere commutativity of it, which is diagram \ref{hopf2}
applied to each $C_i$.  The rest of the subpolygons are easy to
verify, whence we have commutativity of the outermost, which is
diagram \ref{hopf2}.  Diagram \ref{hopf1} can be proved in a
similar fashion.
\end{proof}

\begin{prop}
\index{$A^\circ$} \label{uprodoverdualsprop} Let $L_i$ be a
collection of boundedly finite dimensional algebras over the
fields $k_i$. Then there is a natural isomorphism of coalgebras
\[ \uprod L_i^\circ  \isomorphic \left(\uprod L_i\right)^{\circ}  \]
which sends the tuple of functionals $[\phi_i:L_i \rightarrow
k_i]$ on the left to the functional $[\phi_i]:\uprod L_i
\rightarrow \uprod k_i$ on the right.
\end{prop}

\begin{proof}
Call the claimed isomorphism $\Psi$. That it is an isomorphism of
vector spaces is clear from proposition \ref{homlemma}.  To see
that it is a map of coalgebras we must verify commutativity of
\begin{diagram}
\uprod L_i^\circ & \rTo^\Psi & (\uprod L_i)^\circ \\
\dTo^{[\Delta_i]} & & \\
\uprod L_i^\circ \otimes L_i^\circ & & \dTo_{\Delta} \\
\dTo^{\Phi^{-1}} & & \\
\uprod L_i^\circ \otimes \uprod L_i^\circ & \rTo_{\Psi \otimes
\Psi} & (\uprod L_i)^\circ \otimes (\uprod L_i)^\circ \\
\end{diagram}
where $\Delta_i$ denotes the coalgebra structure on $L_i^\circ$
and $\Delta$ that on $(\uprod L_i)^\circ$.  Let $\text{mult}_i$ be
the multiplication on the algebra $L_i$ and $\text{mult}$ that on
$\uprod L_i$, so by definition $\text{mult} = \Phi \circ
[\text{mult}_i]$.

Let $[\alpha_i:L_i \rightarrow k_i]$ be an arbitrary element of
$\uprod L_i^\circ$ and let us chase it both ways.  Working
downward first, we ask how $\Phi^{-1}([\Delta_i]([\alpha_i]))$
acts on $\uprod L_i \otimes \uprod L_i$; it does so by the
composition
\[ \uprod L_i \otimes \uprod L_i
\stackrel{\Phi}{\longrightarrow} \uprod L_i \otimes L_i
\stackrel{[\text{mult}_i]}{\vlongrightarrow} \uprod L_i
\stackrel{[\alpha_i]}{\longrightarrow} \uprod k_i \] Next we ask
how $\Delta(\Psi([\alpha_i]))$ acts on $\uprod L_i \otimes \uprod
L_i$.  It does so by the composition
\[ \uprod L_i \otimes \uprod L_i
\stackrel{\text{mult}}{\vlongrightarrow} \uprod L_i
\stackrel{[\alpha_i]}{\longrightarrow} \uprod k_i \] But
$\text{mult}$ is \emph{defined} to be $\Phi \circ
[\text{mult}_i]$, and so these actions are equal.  This completes
the proof.
\end{proof}

\chapter{The Restricted Ultraproduct of Neutral Tannakian Categories}

\label{TheMainTheorem} \label{themaintheoremchapter}

In this chapter we prove one of the main theorems of this
dissertation, namely that a certain natural subcategory of an
ultraproduct of neutral tannakian categories is also neutral
tannakian.

\section{Smallness of the Category $\text{Rep}_k G$}

Before beginning in earnest, we pause in this section to address a
subtle but important point.  If one wishes to consider an
ultraproduct of a collection of `things', those things must be
sets; in particular, they must be \index{relational structure}
relational structures. Thus, if one wishes to consider the
ultraproduct of a collection of categories, then those categories
must be small categories, and the necessary (abelian, tensor,
etc.) structure on the \index{category} categories must be
realized as actual relations and functions on that set. This
forces the question: for an affine group scheme $G$ and field $k$,
can $\text{Rep}_k G$ be taken to be small, up to tensorial
equivalence?

This question is not fully addressed in this dissertation, but we
shall at least give here some arguments that lead us to believe
that this is a fair assumption.  In particular, we shall argue why
we believe that the category $\text{Vec}_k$ can be taken to be
small, up to tensorial equivalence.  Similar arguments we believe
should apply to the category $\text{Rep}_k G$ for arbitrary $k$
and $G$.

For the remainder of this section, we shall use the term
\emph{small} and \emph{tensorially small} in an abusive sense; the
category $\catfont{C}$ shall be said to be \emph{small} if it is
equivalent to a small category (even though she itself may not
be), and the tensor category $\catfont{C}$ shall be said to be
\emph{tensorially small} if there is a tensor preserving
equivalence between $\catfont{C}$ and a small tensor category (see
definitions \ref{defntensorcategory} and \ref{defntensorfunctor}).

Consider the category $\text{Vec}_k$ of finite dimensional vector
spaces over a field $k$, which can be identified as $\text{Rep}_k
G_0$, the category of finite dimensional representations of the
trivial group $G_0$. Denote further by $\text{VEC}_k$ the category
of \emph{all} vector spaces over $k$ (finite dimensional or not).
We observe first that the category $\text{VEC}_k$ should by no
means be assumed to be small.  Even her skeleton would consist of
objects of every possible dimension over $k$, and hence of sets of
every possible cardinality. If this skeleton were indeed a set, we
could take the union of all objects contained in that set, and
therefore arrive at a set of cardinality greater than that of any
other set. This is anathema according to the basic tenets of set
theory.

But the category $\text{Vec}_k$ is indeed small. To see this, we
shall follow page 93 of \cite{maclane} in observing that any
category is equivalent (though not necessarily \emph{tensorially}
equivalent) to its skeleton.  To realize the skeleton of
$\text{Vec}_k$ as a small category, for each $n$ we take the set
$V_n = k^n$, i.e.~the collection of all formal linear combinations
of $k$ over the set $\{1,2, \ldots, n\}$ with the obvious
$k$-vector space structure. Likewise define $\text{Hom}(V_n,V_m)$
to be the set of all functions from $V_n$ to $V_m$ which qualify
as $k$-linear maps under the given vector space structures.  Then
the $V_n$ and $\text{Hom}(V_n,V_m)$, themselves being a collection
of sets indexed by the sets $\mathbb{N}$ and $\mathbb{N}^2$, can
indeed be collected into a single set.  Thus $\text{Vec}_k$ is a
small category.

What is far less obvious is that $\text{Vec}_k$ is
\emph{tensorially} small. To illustrate the problem, we again
direct the reader to page 164 of \cite{maclane}, where the author
shows that the skeleton of the category of sets cannot be given
the structure of a tensor category (in this case defined as the
usual cartesian product of sets).  This is the reason, after all,
that we bother with the assoc, comm, and unit isomorphisms in a
tannakian category; demanding, for example, that $(X \otimes Y)
\otimes Z = X \otimes (Y \otimes Z)$ (strict equality) is simply
too stringent. For similar reasons we do not believe that it
suffices to endow the skeleton of $\text{Vec}_k$ with the
structure of a tensor category in the usual sense.

Here is what we believe is a possible approach to remedying this.
Denote by $\catfont{C}_0$ the skeleton of $\text{Vec}_k$ as
defined above, and for each $n \in \mathbb{N}$, denote by
$\catfont{C}_n$ the following category. The objects of
$\catfont{C}_n$ are the objects of $\catfont{C}_{n-1}$, along with
all pairwise tensor products of objects in $\catfont{C}_{n-1}$
(via whatever standard construction one likes, e.g.~as a certain
quotient of the free vector space on $V \times W$; see section 1.7
of \cite{greub}). For objects $V,W \in \catfont{C}_n$, we let
$\text{Hom}_{\catfont{C}_n}(V,W)$ be
$\text{Hom}_{\catfont{C}_{n-1}}(V,W)$ if $V$ and $W$ are in
$\catfont{C}_{n-1}$, and if not, as the collection of all
functions from $V$ to $W$ which qualify as $k$-linear maps.
Finally, we define $\catfont{C}$ to be the union of the categories
$\catfont{C}_n$ for $n=0,1,2, \ldots$.

Note that, since $\catfont{C}_0$ already contains an isomorphic
copy of every finite dimensional vector space over $k$,
$\catfont{C}$ contains no new objects up to isomorphism.  The
whole point of bothering with these new objects is so as not to
encounter any paradoxes similar to that described on page 164 of
\cite{maclane}.  Proving rigorously that this category satasfies
the axioms of a tannakian category would no doubt require
significant effort, but we believe that it could be done.

We would define the primitive relations of the language of abelian
tensor categories on this structure in the obvious manner.  For
instance, the relation $\phi \circ \psi \myeq \eta$ would hold
precisely when $\eta$ is the composition of $\phi$ and $\psi$ in
the usual sense, and similarly for $\phi + \psi \myeq \eta$.
Importantly, we would define the relation $X \otimes Y \myeq Z$ to
hold when $Z$ is the unique object of $\catfont{C}_{n+1}$ such
that $X \otimes Y = Z$, where $n$ is the least integer such that
$X$ and $Y$ are both objects of $\catfont{C}_n$.  The relation
$\text{assoc}_{X,Y,Z} \myeq \phi$ would hold when $\phi$ is the
unique map $(X \otimes Y) \otimes Z \rightarrow X \otimes (Y
\otimes Z)$ such that $\phi$ sends $(x \otimes y) \otimes z$ to $x
\otimes (y \otimes z)$, and similarly for $\text{comm}_{X,Y} \myeq
\phi$.  As for the unit relation, denote by $\underline{1}$ the
unique $1$-dimensional vector space in the skeleton
$\catfont{C}_0$, and define $\text{unit}_X \myeq \phi$ to hold
when $\phi$ is the unique map $X \rightarrow \underline{1} \otimes
X$ such that $\phi:x \mapsto 1 \otimes x$.

Given that one could verify that this category $\catfont{C}$
satisfies the axioms of a tannakian category, showing it to be
tensorially equivalent to the usual $\text{Vec}_k$, and hence
small, should be straightforward; simply define $F:\catfont{C}
\rightarrow \text{Vec}_k$ to be the inclusion functor.  This
functor is clearly full, faithful, and essentially surjective,
hence an equivalence.  Showing $F$ to be tensor preserving (see
definition \ref{defntensorfunctor}) would likewise be
straightforward.  Finally, apply definition 1.10 and proposition
1.11 of \cite{deligne} to see that if $F$ is an equivalence, and
if it is a tensor functor, then it is also a \emph{tensor
equivalence}, in the sense that its inverse can also be taken to
be tensor preserving.  Thus $F$ is an equivalence of abelian
tensor categories, and $\text{Vec}_k$ is a tensorially small
category.

\subsection{A Quotient Category Approach}

Here we mention briefly a possible alternative to the ultraproduct
approach taken in this dissertation, one which replaces the
ultraproduct with a certain \index{quotient category} quotient
category.

Let $G_i$ be a collection of affine group schemes over the fields
$k_i$, and let $\catfont{C}_i$  be the category $\text{Rep}_{k_i}
G_i$.  Denote by $\prod_{i \in I} \catfont{C}_i$ the
\textit{product} of the categories $\catfont{C}_i$; that is, the
category whose objects are all possible tuples of objects $(X_i:i
\in I)$, and whose morphisms are all possible tuples of morphisms
$(\phi_i:i \in I)$, with the obvious definitions of morphism
composition, addition of morphisms, tensor product of objects,
etc.  Fix a non-principal ultrafilter $\filtfont{U}$ on $I$, and
for objects $(X_i),(Y_i) \in \prod_{i \in I} \catfont{C}_i$,
define the following congruence relation $\sim$ on
$\text{Hom}((X_i),(Y_i))$: $(\phi_i) \sim (\psi_i)$ if and only if
the subset of $I$ on which $\phi_i = \psi_i$ is large.  Note that,
if $(\phi_i),(\psi_i):(X_i) \rightarrow (Y_i)$ with $(\phi_i) \sim
(\psi_i)$, and if $(\rho_i),(\mu_i):(Y_i) \rightarrow (Z_i)$ with
$(\rho_i) \sim (\mu_i)$, then $(\phi_i \circ \rho_i) \sim (\psi_i
\circ \mu_i)$, as the intersection of two large sets is also large
(see definition \ref{defnFilter}). Thus $\sim$ is indeed a
congruence relation.

Let $\catfont{C}^\prime$ denote the quotient category of $\prod_{i
\in I} \catfont{C}_i$ with respect to $\sim$.  Then under the
assumption of the previous section, namely that any
$\text{Rep}_{k} G$ is tensorially equivalent to a small category,
we have

\begin{thm}
The quotient category \index{quotient category}
$\catfont{C}^\prime$ and the ultraproduct category $\uprod
\catfont{C}_i$ are equivalent as abelian tensor categories.
\end{thm}

\begin{proof}
For each $i$, denote by $\overline{\catfont{C}}_i$ the posited
small category, and by $F_i:\overline{\catfont{C}}_i \rightarrow
\catfont{C}_i$ the posited tensor equivalence (i.e., the inclusion
functor), of the previous section.  Then by proposition 1.11 of
\cite{deligne}, let $G_i$ be the tensor preserving `inverse' of
$F_i$.  Define a functor $G: \catfont{C}^\prime \rightarrow \uprod
\overline{\catfont{C}}_i$ as follows: for objects, $G((X_i)) =
[G_i(X_i)]$, and for morphisms, $G([\phi_i]) = [G_i(\phi_i)]$.
That $G$ is essentially surjective is clear from the essential
surjectivity of each $G_i$, and similarly the fullness and
faithfulness of $G$ follows.  Thus $G$ is an equivalence.

By hypothesis each $G_i$ is a tensor equivalence, and so comes
equipped with a functorial isomorphism $c^i_{X_i,Y_i}:G_i(X_i)
\otimes G(Y_i) \isoarrow G_i(X_i \otimes Y_i)$ for all $X_i,Y_i
\in \catfont{C}_i$.  Then define a functorial isomorphism $c$ as
follows.  For each pair of objects $(X_i),(Y_i)$ in the quotient
category $\catfont{C}^\prime$, define $c_{(X_i),(Y_i)}:[G_i(X_i)]
\otimes [G_i(Y_i)] \rightarrow [G_i(X_i \otimes Y_i)]$ as the
composition
\[ c_{(X_i),(Y_i)}: [G_i(X_i)]
\otimes [G_i(Y_i)] \stackrel{\Phi}{\longrightarrow} [G_i(X_i)
\otimes G_i(Y_i)] \stackrel{[c^i_{X_i,Y_i}]}{\vlongrightarrow}
[G_i(X_i \otimes Y_i)]
\] where $\Phi$ is the natural injection defined in proposition \ref{ultratensorproductprop}
, and $[c^i_{X_i,Y_i}]$ is the ultraproduct of the linear maps
$c^i_{X_i,Y_i}$ (see equation
\ref{ultraproductsoflineartransformations}). So equipped with $c$,
we believe $G$ can now be shown to be tensor preserving according
to definition \ref{defntensorfunctor}. Apply again proposition
1.11 of \cite{deligne} to see that $G$ is indeed a tensor
equivalence.
\end{proof}

The chief disadvantage of ultraproducts, as highlighted, is that
the constituent categories of the ultraproduct must be shown to be
small relational structures; the quotient category approach does
away with this requirement.  On the other hand, the chief
advantage of the ultraproduct approach is that one immediately has
$\L$os' theorem, which allows us to pass immediately from
first-order statements on the factors to first-order statements in
the ultraproduct.  In particular, we may conclude immediately that
$\uprod \catfont{C}_i$ is a tannakian category, simply by virtue
of the fact that `being tannakian' is a first-order sentence in
the language of abelian tensor categories. Of course, given the
tensorial smallness of $\text{Rep}_k G$, since the ultraproduct
and quotient categories are tensorially equivalent, we conclude
that a $\L$os' theorem-type result must indeed hold for the
quotient category as well; but this can no longer be assumed, and
must be proven, and is the chief disadvantage of the quotient
category approach.

\section{The Restricted Ultraproduct}

For the remainder of this dissertation, by a (abelian, tensor,
etc.) category, we shall always mean a small category realized as
a structure in the language of abelian tensor categories, and by a
tannakian category, we shall mean a structure satisfying the
axioms given in chapter \ref{axiomsfortannakiancategories}.

\begin{thm}
\index{ultraproduct!of tannakian categories}
\label{resprodisatannakiancategorythm} Let $\catfont{C}_i$ be a
sequence of tannakian categories indexed by $I$, $\filtfont{U}$ an
ultrafilter on $I$. Then \index{$\uprod \catfont{C}_i$} $\uprod
\catfont{C}_i$ is a tannakian category.
\end{thm}

\begin{proof}
The property of being tannakian, by the work done in chapter
\ref{axiomsfortannakiancategories}, is expressible by a
first-order sentence in the language of these structures.  By
\index{Los' theorem} $\L$os's theorem (corollary \ref{loscor}),
the same sentence is true in the ultraproduct.
\end{proof}

A word or two about what $\uprod \catfont{C}_i$ actually looks
like. If \index{$[x_i]$} $[x_i]$ is an element of $\uprod
\catfont{C}_i$, then $[x_i]$ is an object or a morphism of $\uprod
\catfont{C}_i$ according to whether the set on which $x_i \in
\text{Ob}$ or $x_i \in \text{Mor}$ is large.  The axioms of a
category state that exactly one of these statements hold in every
slot, and in an ultrafilter, exactly one of a subset of $I$ or its
complement is large.  It thus does no harm to think of every
element $[x_i]$ of $\uprod \catfont{C}_i$ as being represented by
a tuple $(x_i)$ consisting either entirely of objects or entirely
of morphisms, since it is necessarily equivalent to a tuple (many
in fact) of one of these forms.

If \index{$[\phi_i]$} $[\phi_i],[\psi_i],[\eta_i]$ are morphisms
of $\uprod \catfont{C}_i$, then the relation $[\phi_i] \circ
[\psi_i] \myeq [\eta_i]$ holds if and only if the relation $\phi_i
\circ \psi_i \myeq \eta_i$ hold for almost every $i$.  Similarly
the tensor product of the objects \index{$[X_i]$} $[X_i]$ and
$[Y_i]$ is $[Z_i]$, where $Z_i$ denotes the unique object such
that $X_i \otimes Y_i \myeq Z_i$.  In short, $[\phi_i] \circ
[\psi_i] = [\phi_i \circ \psi_i]$, and $[X_i] \otimes [Y_i] = [X_i
\otimes Y_i]$.

By $\L$os's theorem, the same is true for anything that can be
expressed as a first-order concept in our language. For example,
since `being internal Hom' is first-order (see page
\pageref{intHomisFODpage}) we conclude immediately that an
internal Hom object for $[X_i]$ and $[Y_i]$ is necessarily an
object $[Z_i]$, where almost every $Z_i$ is an internal Hom object
for $X_i$ and $Y_i$. Importantly, an identity object for $\uprod
\catfont{C}_i$ is a tuple $[U_i]$ such that $U_i$ is an identity
object for almost every $i$, and an endomorphism of $[U_i]$ is an
element $[\phi_i]$ consisting of morphisms which point from $U_i$
to itself almost everywhere.

One should take care however not to be hasty in concluding that a
given categorical concept is inherited by $\uprod \catfont{C}_i$
from the $\catfont{C}_i$, if you do not know beforehand that the
concept is first-order. The following example illustrates this.

Take $\catfont{C}_i = \text{Vec}_k$ for a fixed field $k$, indexed
by $I = \mathbb{N}$.  For objects $A$ and $B$ of $\text{Vec}_k$,
consider the (non first-order) categorical statement ``$A$ is
isomorphic to an $n$-fold direct sum of $B$ for some $n$''.  Now
the following are both first-order: ``$X$ is isomorphic to $Y$''
and for \emph{fixed} $n$, ``$A$ is an $n$-fold direct sum of
$B$''. This means we can identify $[X_i]^n$ with $[X_i^n]$, and
for fixed $n$, objects $[Y_i]$ that are isomorphic to $[X_i]^n$
with tuples of objects $(Y_i)$ which are almost everywhere
isomorphic to $(X_i^n)$. So let $V_i \in \catfont{C}_i$ be an
$i$-dimensional vector space, and $W_i \in \catfont{C}_i$ a
$1$-dimensional vector space. Then the statement ``$V$ is
isomorphic to an $n$-fold direct sum of $W$ for some $n$'' is true
of every $V_i$ and $W_i$; but the statement is clearly \emph{not}
true of the elements $[V_i]$ and $[W_i]$ inside the category
$\uprod \catfont{C}_i$.  This observation in fact \emph{proves}
that the statement ``$A$ is isomorphic to an $n$-fold direct sum
of $B$ for some $n$'' is not first-order.

In what follows we fix, for each $i$, an identity object for
$\catfont{C}_i$, and denote it by $\underline{1}_i$. We denote
simply by $\underline{1}$ the object $[\underline{1}_i]$ of
$\uprod \catfont{C}_i$.

\begin{prop}
Let $\catfont{C}_i$ be a sequence of tannakian categories, and
denote by $k_i$ the field $\text{End}(\underline{1}_i)$.  Then
$\text{End}(\underline{1})$ can be identified with $k = \uprod
k_i$, the ultraproduct of the fields $k_i$.
\end{prop}

\begin{proof}
As mentioned, $\text{End}(\underline{1})$ consists exactly of
those elements $[\phi_i]$ such that $\phi_i$ is almost everywhere
an endomorphism of $\underline{1}_i$.  But this is exactly the
description of $\uprod k_i$ (see section
\ref{ultraproductsoffields}), if we identify $k_i =
\text{End}(\underline{1}_i)$. The multiplication and addition in
$\text{End}(\underline{1})$ and $\uprod k_i$ are clearly
compatible with this identification, since multiplication is
composition of maps, and addition is addition of morphisms.
\end{proof}

Now assume that $\catfont{C}_i$ is a sequence of \emph{neutral}
tannakian categories, and denote by $\omega_i$ the fibre functor
on each $\catfont{C}_i$.  As mentioned in the introduction, we see
no way to endow $\uprod \catfont{C}_i$ with a fibre functor, at
least not one that is compatible with the each of the $\omega_i$
($\uprod \catfont{C}_i$ might thus be an interesting example of a
non-neutral tannakian category, but that is not investigated in
this dissertation). Instead we look to a certain subcategory of
$\uprod \catfont{C}_i$.

\begin{defn}

For a sequence of neutral tannakian categories $\catfont{C}_i$,
the \index{restricted ultraproduct!of neutral tannakian
categories} \textbf{restricted ultraproduct} of the
$\catfont{C}_i$, \index{$\resprod \catfont{C}_i$} denoted
$\resprod \catfont{C}_i$, is the full subcategory of $\uprod
\catfont{C}_i$ consisting of those objects \index{$[X_i]$} $[X_i]$
such that $\text{dim}(\omega_i(X_i))$ is almost everywhere
bounded.
\end{defn}

To avoid the use of a double subscript, the notation $\resprod
\catfont{C}_i$ makes no mention of the particular ultrafilter
$\filtfont{U}$ being applied.  As $\filtfont{U}$ is always assumed
to be fixed but arbitrary, no confusion should result.

If $[X_i]$ has almost everywhere bounded dimension, then we may as
well take it to have \emph{everywhere} bounded dimension.  And if
$[X_i]$ is everywhere bounded, $X_i$ takes on only finitely many
values for its dimension; by lemma \ref{coloringlemma}, there is
exactly one dimension $m$ such that the set on which
$\text{dim}(\omega_i(X_i)) = m$ is large. Thus, it does no harm to
think of $\resprod \catfont{C}_i$ as the full subcategory
consisting of those (equivalence classes of) tuples having
\emph{constant} dimension.

\begin{thm}
\label{resprodistannakianthm} $\resprod \catfont{C}_i$ is a
\index{category!tannakian} tannakian subcategory of $\uprod
\catfont{C}_i$.
\end{thm}

\begin{proof}
By lemma \ref{tannakiansubcategorylemma} it is enough to show that
$\resprod \catfont{C}_i$ is closed under the taking of biproducts,
subobjects, quotients, tensor products, duals, and contains the
(an) identity object.

Each $\omega_i$ is $k_i$-linear, hence additive, and theorem 3.11
of \cite{freyd} ensures that $\omega_i$ carries direct sums into
direct sums, and hence biproducts into biproducts.  If $[X_i]$,
$[Y_i]$ have constant dimension, then certainly so do the vector
spaces $\omega_i(X_i) \oplus \omega_i(Y_i) \isomorphic
\omega_i(X_i \oplus Y_i)$.  Thus $\resprod \catfont{C}_i$ is
closed under the taking of biproducts.

As each $\omega_i$ is exact, it certainly preserves injectivity of
maps, i.e.~subobjects.  Then if $[X_i]$ has bounded dimension and
$[Y_i]$ is a subobject of $[X_i]$, likewise $[Y_i]$ must have
bounded dimension, since a vector space has larger dimension than
any of its subobjects.  A similar argument holds for quotients;
thus $\resprod \catfont{C}_i$ is closed under the taking of
quotients and subobjects.

That $\resprod \catfont{C}_i$ is closed under the taking of tensor
products is evident from the definition of a tensor functor; if
$[X_i]$ and $[Y_i]$ have constant dimension $m$ and $n$
respectively, then $\omega_i(X_i \otimes Y_i) \isomorphic
\omega_i(X_i) \otimes \omega_i(Y_i)$ has constant dimension $mn$.

That $\resprod \catfont{C}_i$ has an identity object is similarly
proved; tensor functors by definition carry identity objects to
identity objects, and the only identity objects in $\text{Vec}_k$
are $1$-dimensional vector spaces.

Finally, we must show that the dual of an object $[X_i]$ of
$\resprod \catfont{C}_i$ also has constant dimension.  But this is
evident from proposition 1.9 of \cite{deligne}, which says that
$\omega_i$ carries dual objects to dual objects, and the dual of
any vector space has dimension equal to itself.

\end{proof}

Now define a functor \index{$\omega$} $\omega$ from $\resprod
\catfont{C}_i$ to $\text{Vec}_k$ as follows.  For an object
$[X_i]$ of $\resprod \catfont{C}_i$, we define $\omega([X_i])
\stackrel{\text{def}}{=} \uprod \omega_i(X_i)$ (ultraproduct of
vector spaces; see section \ref{ultraproductsofvectorspaces}), and
for a morphism $[\phi_i]$, we define $\omega([\phi_i])
\stackrel{\text{def}}{=} [\omega_i(\phi_i)]$ (ultraproduct of
linear maps; see page
\pageref{ultraproductsoflineartransformations}).

Since $[X_i] \in \resprod \catfont{C}_i$ is assumed to have
bounded dimension, proposition \ref{prop1} guarantees that
$\omega$ carries $[X_i]$ into a finite dimensional vector space
(hence the reason we restrict to $\resprod \catfont{C}_i$ in the
first place).  As the ultraproduct of maps preserves composition,
and since $1_{\uprod V_i} = [1:V_i \rightarrow V_i]$ (proposition
\ref{homprop}), $\omega$ is evidently a functor.

\begin{thm}
\label{resprodisneutralthm} \label{omegaisafiberfunctorthm}
$\omega$ is a  \index{fibre functor} fibre functor on $\resprod
\catfont{C}_i$.
\end{thm}

\begin{proof}
We prove first that $\omega$ is a \index{tensor functor} tensor
functor. For two objects $[X_i]$, $[Y_i]$ of $\resprod
\catfont{C}_i$, we define the requisite natural isomorphism
$c_{[X_i],[Y_i]}$ of definition \ref{defntensorfunctor} to be the
composition
\[ \uprod \omega_i(X_i) \otimes \uprod \omega_i(Y_i)
\stackrel{\Phi}{\longrightarrow} \uprod \omega_i(X_i) \otimes
\omega_i(Y_i) \stackrel{[c_{X_i,Y_i}]}{\vvlongrightarrow} \uprod
\omega_i(X_i \otimes Y_i)
\]
where $c_{X_i,Y_i}$ denotes the given requisite isomorphism in
each individual category, and $\Phi$ is the natural isomorphism
defined in proposition \ref{ultratensorproductprop}.  We need to
verify that the three conditions of definition
\ref{defntensorfunctor} are satisfied.  Condition 1.~translates to
\begin{diagram}
 & & \omega([X_i]) \otimes \omega([Y_i] \otimes [Z_i]) & & \\
 & \ruTo^{1 \otimes c} & & \rdTo^{c} & \\
\omega([X_i]) \otimes (\omega([Y_i]) \otimes \omega([Z_i])) & & & & \omega([X_i] \otimes ([Y_i] \otimes [Z_i])) \\
 \dTo^{\text{assoc}^\prime} & & & & \dTo_{\omega(\text{assoc})} \\
 (\omega([X_i]) \otimes \omega([Y_i])) \otimes \omega([Z_i]) & & & & \omega(([X_i] \otimes [Y_i]) \otimes [Z_i]) \\
 & \rdTo_{c \otimes 1} & & \ruTo_{c} \\
 & & \omega([X_i] \otimes [Y_i]) \otimes \omega([Z_i]) & & \\
\end{diagram}
where $\text{assoc}^\prime$ denotes the usual associativity
isomorphism in $\text{Vec}_k$, and we have dropped the obvious
subscripts on $c$. The expanded form of this diagram is
\begin{diagram}
\uprod \omega_i(X_i) \otimes \uprod \omega_i(Y_i \otimes Z_i)   & \rTo^\Phi & \uprod \omega_i(X_i) \otimes \omega_i(Y_i \otimes Z_i)\\
\uTo^{[1_i] \otimes [c_{Y_i,Z_i}]}    & & \dTo_{[c_{X_i,Y_i \otimes Z_i}]}       \\
\uprod \omega_i(X_i) \otimes \left( \uprod \omega_i(Y_i) \otimes
\omega_i(Z_i) \right) & & \uprod \omega_i(X_i \otimes (Y_i \otimes Z_i)) \\
\uTo^{[1_i] \otimes \Phi} & & \dTo_{[\omega_i(\text{assoc}_i)]}\\
\uprod \omega_i(X_i) \otimes \left(\uprod \omega_i(Y_i) \otimes
\uprod \omega_i(Z_i) \right) & & \uprod \omega_i((X_i \otimes Y_i)\otimes Z_i) \\
\dTo^{\text{assoc}^\prime} & & \uTo_{[c_{X_i \otimes Y_i,Z_i}]} \\
\left( \uprod \omega_i(X_i) \otimes \uprod \omega_i(Y_i) \right)
\otimes \uprod \omega_i(Z_i) & & \uprod \omega_i(X_i \otimes Y_i) \otimes \omega_i(Z_i)\\
\dTo^{\Phi \otimes [1_i]} & & \uTo_{\Phi}\\
\left( \uprod \omega_i(X_i) \otimes \omega_i(Y_i) \right) \otimes
\uprod \omega_i(Z_i) & \rTo_{[c_{X_i,Y_i}] \otimes [1_i]} & \left(\uprod \omega_i(X_i \otimes Y_i) \right) \otimes \uprod \omega_i(Z_i) \\
\end{diagram}
Now consider the diagram
\begin{diagram}
\text{ }  & \lTo & \uprod \omega_i(X_i) \otimes (\omega_i(Y_i) \otimes \omega_i(Z_i)) & & & & &  \\
 & & \uTo^\Phi & \rdTo(5,1)^{[1_i \otimes c_{Y_i,Z_i}]} & & & & \uprod \omega_i(X_i) \otimes \omega_i(Y_i \otimes Z_i) \\
 & & \uprod \omega_i(X_i) \otimes \left(\uprod \omega_i(Y_i) \otimes \omega_i(Z_i) \right) & & & & \ruTo(5,1)^{([1_i] \otimes [c_{Y_i,Z_i}]) \circ \Phi} & \uTo_{[c_{X_i,Y_i \otimes Z_i}]} \\
 & & \uTo^{[1_i] \otimes \Phi} & & & & & \uprod \omega_i(X_i \otimes (Y_i \otimes Z_i)) \\
\dTo^{[\text{assoc}_i^\prime]} & & \uprod \omega_i(X_i) \otimes \left(\uprod \omega_i(Y_i) \otimes \uprod \omega_i(Z_i) \right) & & & & & \uTo_{[\omega_i(\text{assoc}_i)]} \\
 & & \dTo^{\text{assoc}^\prime \circ (\Phi \otimes [1_i])} & & & & & \uprod \omega_i((X_i \otimes Y_i) \otimes Z_i) \\
 & & \left( \uprod \omega_i(X_i) \otimes \omega_i(Y_i) \right) \otimes \uprod \omega_i(Z_i) & & & & & \uTo_{[c_{X_i \otimes Y_i,Z_i}]} \\
 & & \dTo^\Phi & \rdTo(5,1)^{([c_{X_i,Y_i}] \otimes [1_i]) \circ \Phi} & & & & \uprod \omega_i(X_i \otimes Y_i) \otimes \omega_i(Z_i) \\
\text{ }& \rTo & \uprod (\omega_i(X_i) \otimes \omega_i(Y_i)) \otimes \omega_i(Z_i) & & & & \ruTo(5,1)_{[c_{X_i,Y_i} \otimes 1_i]} & \\
\end{diagram}
where $\text{assoc}_i^\prime$ is the associativity isomorphism in
the category $\text{Vec}_{k_i}$.  This diagram has four simple
subpolygons. The middle polygon consisting of seven vertices is a
contracted version of the previous diagram, and is what we are
trying to prove commutes. Commutativity of the top and bottom
triangles follow directly from the naturality of the isomorphism
$\Phi$ (see diagram \ref{naturalityofPhidiagram} on page
\pageref{naturalityofPhidiagram}), and the left-most polygon can
be verified directly by hand. And since all of the maps are
isomorphisms, some diagram chasing shows that if the outermost six
vertex polygon can be shown to commute, so also does the simple
seven vertex polygon. But the outermost polygon is
\begin{diagram}
 & & \uprod \omega_i(X_i) \otimes \omega_i(Y_i \otimes Z_i) & & \\
 & \ruTo^{[1_i \otimes c_{Y_i,Z_i}]} & & \rdTo^{[c_{X_i,Y_i \otimes Z_i}]} & \\
\uprod \omega_i(X_i) \otimes (\omega_i(Y_i) \otimes \omega_i(Z_i)) & & & & \uprod \omega_i(X_i \otimes (Y_i \otimes Z_i)) \\
 \dTo^{[\text{assoc}_i^\prime]} & & & & \dTo_{[\omega_i(\text{assoc}_i)]} \\
 \uprod (\omega_i(X_i) \otimes \omega_i(Y_i)) \otimes \omega_i(Z_i) & & & & \uprod \omega_i((X_i \otimes Y_i) \otimes Z_i) \\
 & \rdTo_{[c_{X_i,Y_i} \otimes 1_i]} & & \ruTo_{[c_{X_i \otimes Y_i,Z_i}]} \\
 & & \uprod \omega_i(X_i \otimes Y_i) \otimes \omega_i(Z_i) & & \\
\end{diagram}
and by theorem \ref{transformationsarenicethm} commutativity of
this diagram is equivalent to the almost everywhere commutativity
of
\begin{diagram}
 & &  \omega_i(X_i) \otimes \omega_i(Y_i \otimes Z_i) & & \\
 & \ruTo^{1_i \otimes c_{Y_i,Z_i}} & & \rdTo^{c_{X_i,Y_i \otimes Z_i}} & \\
\omega_i(X_i) \otimes (\omega_i(Y_i) \otimes \omega_i(Z_i)) & & & &  \omega_i(X_i \otimes (Y_i \otimes Z_i)) \\
 \dTo^{\text{assoc}_i^\prime} & & & & \dTo_{\omega_i(\text{assoc}_i)} \\
 (\omega_i(X_i) \otimes \omega_i(Y_i)) \otimes \omega_i(Z_i) & & & & \omega_i((X_i \otimes Y_i) \otimes Z_i) \\
 & \rdTo_{c_{X_i,Y_i} \otimes 1_i} & & \ruTo_{c_{X_i \otimes Y_i,Z_i}} \\
 & & \omega_i(X_i \otimes Y_i) \otimes \omega_i(Z_i) & & \\
\end{diagram}
But this is commutative everywhere, as it is merely condition
1.~of definition \ref{defntensorfunctor}, by virtue of each
$\omega_i$ being a tensor functor.  $\omega$ thus satisfies
condition 1.

Condition 2. is proved similarly; consider the diagram
\begin{diagram}
 & & \uprod \omega_i(X_i) \otimes \omega_i(Y_i) & & \\
 & \ruTo^{\Phi} & & \rdTo^{[c_{X_i,Y_i}]} & \\
\uprod \omega_i(X_i) \otimes \uprod \omega_i(Y_i) & & & & \uprod \omega_i(X_i \otimes Y_i) \\
 \dTo^{\text{comm}^\prime} & & \dTo_{[\text{comm}_i^\prime]} & & \dTo_{[\omega_i(\text{comm}_i)]} \\
 \uprod \omega_i(Y_i) \otimes \uprod \omega_i(X_i) & & & & \uprod \omega_i(Y_i \otimes X_i) \\
 & \rdTo_{\Phi} & & \ruTo_{[c_{Y_i,X_i}]} \\
 & & \uprod \omega_i(Y_i) \otimes \omega_i(X_i) & & \\
\end{diagram}
The outermost hexagon is our expanded version of condition 2., and
is what we must prove.  Commutativity of the left trapezoid can be
verified directly by hand. And again by theorem
\ref{transformationsarenicethm}, commutativity of the right
trapezoid is equivalent to the almost everywhere commutativity of
\begin{diagram}
\omega_i(X_i) \otimes \omega_i(Y_i) & \rTo^{c_{X_i,Y_i}} &
\omega_i(X_i \otimes Y_i) \\
\dTo^{\text{comm}_i^\prime} & & \dTo_{\omega_i(\text{comm}_i)} \\
\omega_i(Y_i) \otimes \omega_i(X_i) & \rTo_{c_{Y_i,X_i}} &
\omega_i(Y_i \otimes X_i) \\
\end{diagram}
But this is condition 2.~applied to each individual $\omega_i$,
which commutes by assumption.

For the purposes of this proof we shall replace condition 3.~of
definition \ref{defntensorfunctor} with the seemingly weaker but
equivalent condition given in definition 1.8 of \cite{deligne}:
that whenever $[U_i]$ is an identity object of $\resprod
\catfont{C}_i$ and $[u_i]:[U_i] \rightarrow [U_i] \otimes [U_i]$
an isomorphism, then so is $\omega([U_i])$ and $\omega([u_i])$.
Since any two identity objects of a tensor category are naturally
isomorphic  via a unique isomorphism commuting with the unit maps
(proposition 1.3 of \cite{deligne}), we need only verify this for
a single identity object, namely the pair $[1_i]$ and
$[\text{unit}_{i,\underline{1}_i}]$.  As each $\omega_i$ is a
tensor functor, it sends $1_i$ to an identity object in
$\text{Vec}_{k_i}$, and we know of course that the only identity
objects of $\text{Vec}_{k_i}$ are $1$-dimensional.  Thus
$\omega([1_i]) = \uprod \omega_i(1_i)$ is $1$-dimensional
(proposition \ref{prop1}), thus $\omega([1_i])$ is an identity
object of $\text{Vec}_k$.  And again, as each $\omega_i$ is a
tensor functor, it sends $\text{unit}_{i,\underline{1}_i}$ to an
isomorphism $\omega_i(\underline{1}_i) \rightarrow
\omega_i(\underline{1}_i) \otimes \omega_i(\underline{1}_i)$,
whence $\omega$ sends $[\text{unit}_{i,\underline{1}_i}]$ to an
isomorphism as well.

$\omega$ is $k$-linear by the $k_i$-linearity of each $\omega_i$
and proposition \ref{homprop}:
\begin{equation*}
\begin{split}
 \omega([a_i][\phi_i] + [\psi_i]) &= \omega([a_i \phi_i +
 \psi_i]) \\
&= [\omega_i(a_i \phi_i + \psi_i)] \\
 &= [a_i \omega_i(\phi_i) + \omega_i(\psi_i)] \\
  &= [a_i][\omega_i(\phi_i)] + [\omega_i(\psi_i)] \\
  &= [a_i]\omega([\phi_i]) + \omega([\psi_i])
\end{split}
\end{equation*}

$\omega$ is faithful: if $[\phi_i]$ and $[\psi_i]$ are different
morphisms, then $(\phi_i)$ and $(\psi_i)$ differ on a large set.
By faithfulness of each $\omega_i$, so do $(\omega_i(\phi_i))$ and
$(\omega_i(\psi_i))$, and by proposition \ref{homprop},
$[\omega_i(\phi_i)]$ and $[\omega_i(\psi_i)]$ are different linear
maps.

$\omega$ is exact by the exactness of each $\omega_i$, proposition
\ref{ultraexactsequenceprop}, and the fact that ``is an exact
sequence'' is a first-order concept.  The sequence
\[ [0] \rightarrow [X_i] \stackrel{[\phi_i]}{\vlongrightarrow}
[Y_i] \stackrel{[\psi_i]}{\vlongrightarrow} [Z_i] \rightarrow [0]
\]
in $\resprod \catfont{C}_i$ is exact if and only if the
constituent sequences
\[ 0 \rightarrow X_i \stackrel{\phi_i}{\longrightarrow} Y_i
\stackrel{\psi_i}{\longrightarrow} Z_i \rightarrow 0 \] are almost
everywhere exact, in which case $\omega_i$ of these sequences is
almost everywhere exact, in which case $\omega$ of the first
sequence is exact.  This completes the proof.

\end{proof}

\begin{cor}
If $G_i$ is a sequence of affine group schemes defined over the
fields $k_i$, then \index{restricted ultraproduct!of neutral
tannakian categories} \index{$\resprod \catfont{C}_i$} $\resprod
\text{Rep}_{k_i} G_i$ is (tensorially equivalent to)
$\text{Rep}_{\uprod k_i} G$ for some affine group scheme $G$.
\end{cor}

\begin{proof}
By theorems \ref{resprodistannakianthm} and
\ref{omegaisafiberfunctorthm}, $\resprod \text{Rep}_{k_i} G_i$ is
a \index{category!neutral tannakian} neutral tannakian category
over the field $\uprod k_i$. Apply theorem
\ref{tannakiandualitythm}.
\end{proof}

\chapter{Finite Dimensional Subcoalgebras of Hopf Algebras}

\label{adualitytheoremchapter}

\index{coalgebra!finite dimensional}

In this chapter we take a break entirely from working with
ultraproducts; no understanding of them is required here
whatsoever.  The main theorem of this chapter is perhaps of
interest in its own right, but for our purposes mostly serves as
an invaluable lemma with which to prove the main theorem of the
next chapter.

Here we investigate the special case of when a finite dimensional
comodule $C$ over a Hopf algebra $(A,\Delta,\varepsilon)$ over a
field $k$ is actually a sub-co\emph{algebra} of $A$; our intent is
to show that these satisfy some very nice regularity properties in
terms of how they sit inside the category $\text{Comod}_A$.  To
say that $C \subset A$ is a subcoalgebra is to simply say that the
image of the map $\Delta:A \rightarrow A \otimes A$, when
restricted to $C$, is contained inside $C \otimes C \subset A
\otimes A$, and that we are regarding the map $\Delta$ as the
(left or right, depending) $A$-comodule structure for $C$.
Throughout we will use the same symbols $\Delta$ and $\varepsilon$
for their restrictions to $C$.

In a sense though, the case of a finite dimensional $A$-comodule
being a subcoalgebra of $A$ is really not that special. The
fundamental theorem of coalgebras (theorem
\ref{fundamentaltheoremofcoalgebras}) states that any coalgebra
(and hence Hopf algebra) is a directed union of finite dimensional
subcoalgebras. Further, theorem \ref{regreptheorem} states that
any $A$-comodule can be embedded in some $n$-fold direct sum of
the regular representation. Thus every finite dimensional
$A$-comodule can be embedded in $C^n$, where $C$ is some finite
dimensional subcoalgebra of $A$. We see then that the entire
category $\text{Comod}_A$ can be realized as a direct limit of the
\index{principal subcategory} principal subcategories $\gen{C}$,
where $C$ ranges over all finite dimensional subcoalgebras of $A$.
Anything categorical we can say in general about these
subcoalgebras of $A$ must surely then (and will) be of value.

We would also like to mention that, so far as we can tell, these
results are valid for \emph{any} coalgebra $A$, Hopf algebra or
not.  Nonetheless, as all of our applications of these results
will be toward Hopf algebras, we leave them as stated.

Let $C$ be a subcoalgebra of $(A,\Delta,\varepsilon)$.  Since
$\Delta$ restricts to $C \otimes C$ on $C$, we can think of $C$ as
both a left and a right comodule over $A$. That is
\[ \Delta: C \rightarrow C \otimes C \subset C \otimes A \]
gives a right $A$-comodule structure for $C$, and
\[ \Delta: C \rightarrow C \otimes C \subset A \otimes C \]
gives a left $A$-comodule structure for $C$.  Unless $C$ is
co-commutative we can expect these structures in general to be
quite different.

For the remainder of this chapter denote by $\catfont{C}_R$ the
category of finite dimensional right $A$-comodules, and denote by
$\omega_R$ the fibre (i.e.~forgetful) functor $\catfont{C}_R
\rightarrow \text{Vec}_k$.  Define similarly $\catfont{C}_L$ and
$\omega_L$. For a finite dimensional subcoalgebra $C$ of $A$
denote by $\text{End}_{\catfont{C}_R}(C)$ the algebra of all
endomorphisms on $C$, where we consider $C$ as an object in the
category $\catfont{C}_R$, as defined above; make a similar
definition for $\text{End}_{\catfont{C}_L}(C)$.  Denote as usual
by $\text{End}(\omega_R \res \gen{C})$ the collection of all
natural transformations of the fibre functor $\omega_R$ restricted
to the principal subcategory $\gen{C}$ (see definition
\ref{endomegadefn}), similarly for $\text{End}(\omega_L \res
\gen{C})$

The remainder of this chapter is devoted to proving, piecemeal,
the following:

\begin{thm}
\label{adualitytheorem} Let $C$ be a finite dimensional
subcoalgebra of the Hopf algebra $A$ over the field $k$.  Then

\begin{enumerate}
\item{$\text{End}_{\catfont{C}_R}(C) = \text{End}(\omega_L \res
\gen{C}) = $ \index{$\text{End}(\omega \res \gen{X})$} the
centralizer of $\text{End}_{\catfont{C}_L}(C)$, and
$\text{End}_{\catfont{C}_L}(C) = \text{End}(\omega_R \res \gen{C})
= $ the centralizer of $\text{End}_{\catfont{C}_R}(C)$} \item{All
of the above are canonically isomorphic to the dual algebra of the
coalgebra $C$.}

\end{enumerate}

\end{thm}

It is clear from the remarks on page 135 of \cite{deligne},
combined with lemma 2.13 of the same text, that the author is
quite aware that the algebra $C^*$ is isomorphic to both
$\text{End}(\omega_R \res \gen{C})$ and $\text{End}(\omega_L \res
\gen{C})$.  This is not surprising; we shall argue at the end of
this chapter that this theorem in fact proves that the `algorithm'
given in section \ref{recoveringanalgebraicgroupsection} for
recovering the Hopf algebra $A$ from the category $\text{Comod}_A$
does in fact give the correct answer.  As to the other assertions
of theorem \ref{adualitytheorem}, we are unable to locate any
specific occurrence of them in the literature.

In the statement of the theorem we have deliberately confused (as
we may, by the discussion on page \pageref{endomegadiscussion})
$\text{End}(\omega_R \res \gen{C})$ with its image inside
$\text{End}_{\text{Vec}_k}(\omega_R(C))$.  Note that these are
equalities given in 1.~above, not just isomorphisms.

We will prove first that $C^*$, the dual algebra to the coalgebra
$C$, is isomorphic to $\text{End}_{\catfont{C}_R}(C)$.  We define
maps \begin{gather*} C^* \stackrel{\Omega}{\longrightarrow}
\text{End}_{\catfont{C}_R}(C) \\ \index{$\Gamma$}  C^*
\stackrel{\Gamma}{\longleftarrow} \text{End}_{\catfont{C}_R}(C)
\end{gather*} as follows.  For $\alpha \in C^*$, $\Omega(\alpha)$ is the
composition
\[ C \stackrel{\Delta}{\longrightarrow} C \otimes C
\stackrel{\alpha \otimes 1}{\vlongrightarrow} k \otimes C
\isomorphic C \] and for $\phi \in \text{End}_{\catfont{C}_R}(C)$,
$\Gamma(\phi)$ is the composition
\[ C \stackrel{\phi}{\longrightarrow} C
\stackrel{\varepsilon}{\longrightarrow} k \]

\begin{thm}
\label{OmegaAndGammaTheorem} The maps $\Omega$ and $\Gamma$ are
well-defined algebra maps, and are left and right-sided inverses
for one another, making them both isomorphisms of algebras.
\end{thm}

\begin{proof}
We need to prove first the non-obvious fact that, for any $\alpha
\in C^*$, $\Omega(\alpha)$ is an endomorphism on $C$ as a right
$A$-comodule, that is, that the diagram
\begin{diagram}
C & \rTo^{\Omega(\alpha)} & C \\
\dTo^\Delta & & \dTo_\Delta \\
C \otimes C & \rTo_{\Omega(\alpha) \otimes 1} & C \otimes C \\
\end{diagram}
commutes.  Consider the diagram
\begin{diagram}
C & \rTo^\Delta & C \otimes C & \rTo^{\alpha \otimes 1} & k
\otimes C & \rTo^{\isomorphic} & C \\
\dTo^\Delta & & \dTo_{1 \otimes \Delta} & & \dTo_{1 \otimes
\Delta} & & \dTo_\Delta \\
C \otimes C & \rTo_{\Delta \otimes 1} & C \otimes C \otimes C &
\rTo_{\alpha \otimes 1 \otimes 1} & k \otimes C \otimes C &
\rTo_{\isomorphic} & C \otimes C \\
\end{diagram}
The outermost rectangle is an expanded version of the previous
diagram, and is what we are trying to prove commutes.
Commutativity of the right-most simple rectangle follows directly
from the naturality of $\isomorphic$, commutativity of the middle
rectangle is obvious, and the left-most rectangle is a coalgebra
identity.  Thus the outermost rectangle commutes, and
$\Omega(\alpha)$ is indeed an endomorphism of $C$ as a right
$A$-comodule.

We argue now that $\Omega \circ \Gamma$ and $\Gamma \circ \Omega$
are both the identity.  Let $\alpha \in C^*$, and consider
\begin{diagram}
C & \rTo^\Delta & C \otimes C & \rTo^{\alpha \otimes 1} & k
\otimes C & \rTo^\isomorphic & C & \rTo^\varepsilon & k \\
& \rdTo_{\isomorphic} & \dTo_{1 \otimes \varepsilon} & & \dTo_{1
\otimes \varepsilon} & & & \ruTo(4,2)_{\isomorphic} \\
& & C \otimes k & \rTo_{\alpha \otimes 1} & k \otimes k \\
\end{diagram}
The top line is the map $\Gamma(\Omega(\alpha))$.  We would like
to see that this is equal to $\alpha$, and $\alpha$ is clearly
equal to the bottom three-map composition; thus we seek to prove
commutativity of the outermost polygon.  Commutativity of the
right-most simple polygon follows again from the naturality of
$\isomorphic$, commutativity of the middle square is obvious, and
the left-most triangle is again a coalgebra identity; thus
$\Gamma(\Omega(\alpha)) = \alpha$.

Now let $\phi \in \text{End}_{\catfont{C}_R}(C)$.  Consider
\begin{diagram}
C & \rTo^\Delta & C \otimes C & \rTo^{(\phi \circ \varepsilon)
\otimes 1} & k \otimes C & \rTo^{\isomorphic} & C \\
\dTo^\phi & & \dTo_{\phi \otimes 1} & \ruTo_{\varepsilon \otimes
1} \\
C & \rTo_\Delta & C \otimes C \\
\end{diagram}
The top line is the map $\Omega(\Gamma(\phi))$, which we would
like to see is equal to $\phi$.  Commutativity of the left-most
square is the assertion that $\phi$ is an endomorphism of $C$ as a
right $A$-comodule, and commutativity of the middle triangle is
obvious.  Thus the outermost polygon commutes, giving us
\[ \Omega(\Gamma(\phi)) = \phi \circ (\Delta \circ (\varepsilon
\otimes 1) \circ \isomorphic) \]
But $(\Delta \circ (\varepsilon
\otimes 1) \circ \isomorphic) = 1$ is coalgebra identity, and
hence the right hand side is equal to $\phi$, proving the claim.

We must finally prove that $\Gamma$ is a $k$-algebra map. Recall
the multiplication on $C^*$; it sends the pair of functionals
$\alpha, \beta:C \rightarrow k$ to the functional
\[ C \stackrel{\Delta}{\longrightarrow} C \otimes C
\stackrel{\alpha \otimes \beta}{\vlongrightarrow} k \otimes k
\stackrel{\isomorphic}{\longrightarrow} k \]
 Let $\phi, \psi \in
\text{End}_{R}(C)$, and consider the diagram
\begin{diagram}
C \otimes C & \rTo^{\phi \otimes 1} & C \otimes C &
\rTo^{\varepsilon\otimes 1} & k \otimes C & \rTo^{1 \otimes \psi}
& k \otimes C & \rTo^{1 \otimes \varepsilon} & k \otimes k &
\rTo^{\isomorphic} &  k \\
\uTo^\Delta & & \uTo_\Delta & & \uTo_\isomorphic & &
\uTo_\isomorphic & & & \ruTo(4,2)_\varepsilon \\
C & \rTo_\phi & C & \rEq & C & \rTo_\psi & C \\
\end{diagram}
The composition that starts at the bottom left hand corner, goes
up, and then all the way across, is an expanded version of the map
$\Gamma(\phi) * \Gamma(\psi)$, where $*$ denotes the
multiplication in the algebra $C^*$.  The one that starts at the
bottom left hand corner, goes across, and then diagonally up, is
the map $\Gamma(\phi \circ \psi)$; we want to see of course that
these are equal.  It is enough to show then that all of the simple
polygons commute.  Starting from the left: commutativity of the
first is the assertion that $\phi$ is an endomorphism of $C$ as a
right $A$-comodule, the second is a coalgebra identity, and
commutativity of the third and fourth follow directly from the
naturality of $\isomorphic$.  Therefore $\Gamma$ is a
multiplicative map, and is obviously $k$-linear, since composition
with $\varepsilon$ (or any linear map) is so. Therefore $\Gamma$
is an isomorphism of $k$-algebras.  The same is true of $\Omega$,
since it is the inverse of such a map.
\end{proof}

We claim also that $C^*$ is in much the same way isomorphic to
$\text{End}_{\catfont{C}_L}(C)$, the endomorphism algebra of $C$
as a \emph{left} $A$-comodule. This time we define a map
\[ C^* \stackrel{\Theta}{\longrightarrow}
\text{End}_{\catfont{C}_L}(C) \] as, for $\alpha \in C^*$,
$\Theta(\alpha)$ is the composition
\[ C \stackrel{\Delta}{\longrightarrow} C \otimes C
\stackrel{1 \otimes \alpha}{\vlongrightarrow} C \otimes k
\isomorphic C
\]
(notice the switching of the slots on which $1$ and $\alpha$ act).
We define a map $\Lambda:\text{End}_{\catfont{C}_L}(C) \rightarrow
C^*$ the same way as before: for an endomorphism $\phi$ of $C$ in
the category $\catfont{C}_L$, \label{LambdaDefFDcoalgebras}
$\Lambda(\phi)$ is the composition
\[ C \stackrel{\phi}{\longrightarrow} C
\stackrel{\varepsilon}{\longrightarrow} k \] An proof almost
identical to that of the previous theorem shows again that
$\Theta$ and $\Lambda$ are isomorphisms of $k$-algebras, which we
will not repeat.

\begin{lem}
Let $V$ be a finite dimensional vector space over a field $k$, and
let $r,s \in V \otimes V$ such that $r \neq s$.  Then there exists
a linear functional $\alpha:V \rightarrow k$ such that the
composition
\[ V \otimes V \stackrel{\alpha \otimes 1}{\vlongrightarrow} k
\otimes V \isomorphic V \]
sends $r$ and $s$ to different things.
\end{lem}

\begin{proof}
Let $e_1, \ldots, e_n$ be a basis for $V$.  Then we can write
\[ r = \sum_{i,j} c_{ij} e_i \otimes e_j \]
\[ s = \sum_{i,j} d_{ij} e_i \otimes e_j \]
If $r \neq s$, then $c_{ij} \neq d_{ij}$ for some $i$ and $j$.
Then let $\alpha$ be the functional which sends $e_i$ to $1$ and
all other $e_j$ to $0$.  Then the composition above sends $r$ to
\[ c_{i,1} e_1 + c_{i,2} e_2 + \ldots + c_{i,j} e_j + \ldots +
c_{i,n} e_n \] while it sends $s$ to
\[ d_{i,1} e_1 + d_{i,2} e_2 + \ldots + d_{i,j} e_j + \ldots +
d_{i,n} e_n \] and these are clearly not equal, since $c_{i,j}
\neq d_{i,j}$.
\end{proof}

\begin{thm}
Let $C$ be a finite dimensional subcoalgebra of the Hopf algebra
$A$. Then $\text{End}_{\catfont{C}_L}(C) = $ the centralizer of
$\text{End}_{\catfont{C}_R}(C)$.
\end{thm}

\begin{proof}
Let $\phi \in \text{End}_{\catfont{C}_L}(C)$ and $\psi \in
\text{End}_{\catfont{C}_R}(C)$, and consider the diagram

\begin{diagram}
C & \rTo^\phi & C & \rTo^\psi & C & \rTo^\phi & C \\
\dTo^\Delta & & \dTo^\Delta & & \dTo^\Delta & & \dTo^\Delta \\
C \otimes C & \rTo_{1 \otimes \phi} & C \otimes C & \rTo_{\psi
\otimes 1} & C \otimes C & \rTo_{1 \otimes \phi} & C \otimes C \\
\end{diagram}
Commutativity of all three of the simple squares are merely the
assertions that $\phi$ and $\psi$ are morphisms in the categories
$\catfont{C}_L$ and $\catfont{C}_R$ respectively; thus this entire
diagram commutes.  If we look at the rectangle consisting of the
two left squares we obtain $\phi \circ \psi \circ \Delta = \Delta
\circ (\psi \otimes \phi)$, and looking at the rectangle
consisting of the two right squares we obtain $\psi \circ \phi
\circ \Delta = \Delta \circ (\psi \otimes \phi)$.  Thus we have
\[ \phi \circ \psi \circ \Delta = \psi \circ \phi \circ \Delta \]
But $\Delta$ is injective, and hence $\phi \circ \psi = \psi \circ
\phi$.  This shows that $\text{End}_{\catfont{C}_L}(C)$ is
contained in the centralizer of $\text{End}_{\catfont{C}_R}(C)$.

Now let $\phi$ be any linear map $C \rightarrow C$, and suppose
that it is not a member of $\text{End}_{\catfont{C}_L}(C)$, which
is to say that the diagram
\begin{diagram}
C & \rTo^\phi & C \\
\dTo^\Delta & & \dTo_\Delta \\
C \otimes C & \rTo_{1 \otimes \phi} & C \otimes C \\
\end{diagram}
does \emph{not} commute; we claim there exists a member of
$\text{End}_{\catfont{C}_R}(C)$ with which $\phi$ does not
commute. Recall from theorem \ref{OmegaAndGammaTheorem}  that for
any linear functional $\alpha:C \rightarrow k$, the composition
\[ C \stackrel{\Delta}{\longrightarrow} C \otimes C \stackrel{
\alpha \otimes 1}{\vlongrightarrow} k \otimes C \isomorphic C \]
belongs to $\text{End}_{\catfont{C}_R}(C)$; our job is then to
find an $\alpha$ so that $\phi$ does not commute with this map.
Consider
\begin{diagram}
C & \rTo^\Delta & C \otimes C & \rTo^{\alpha \otimes 1} & k
\otimes C & \rTo^{\isomorphic} & C \\
\dTo^\phi & & \dTo_{1 \otimes \phi} & & \dTo_{1 \otimes \phi} & &
\dTo_\phi \\
C & \rTo_\Delta & C \otimes C & \rTo_{\alpha \otimes 1} & k
\otimes C & \rTo_{\isomorphic} & C \\
\end{diagram}
All of the simple squares of this diagram commute, except for the
left most one, which does not by assumption.  We want to show that
there is an $\alpha$ such that the outermost rectangle does not
commute.  Pick $v \in C$ such that commutativity of the left
square fails, let $r \in C \otimes C$ be its image when chasing it
one way, and $s \in C \otimes C$ its image when chasing it the
other way.  By the previous lemma, pick $\alpha$ such that the
composition
\begin{diagram}
 C \otimes C & \rTo_{\alpha \otimes 1} & k \otimes C &
\rTo_{\isomorphic} & C
\end{diagram}
sends $r$ and $s$ to different things, let's say $m \neq l$.  Then
if we chase $v$ around one path of the outermost rectangle we
arrive at $m$, and the other way, we arrive at $l$; thus the
outermost rectangle does not commute.  This gives a member of
$\text{End}_{\catfont{C}_R}(C)$ with which $\phi$ does not
commute, and the theorem is proved.
\end{proof}

An identical proof shows that $\text{End}_{\catfont{C}_R}(C)$ is
the centralizer of $\text{End}_{\catfont{C}_L}(C)$, which we do
not repeat.

Our last task is to show that $\text{End}_{\catfont{C}_L}(C)$ is
equal to $\text{End}(\omega_R \res \gen{C})$.  Any member $\phi:C
\rightarrow C$ of the latter must at the least make diagrams of
the form
\begin{diagram}
C & \rTo^\psi & C \\
\dTo^\phi & & \dTo_\phi \\
C & \rTo_\psi & C \\
\end{diagram}
commute, where $\psi$ is an arbitrary element of
$\text{End}_{\catfont{C}_R}(C)$. As
$\text{End}_{\catfont{C}_L}(C)$ is equal to the commutator of
$\text{End}_{\catfont{C}_R}(C)$, we already have the forward
inclusion $\text{End}(\omega_R \res \gen{C}) \subset
\text{End}_{\catfont{C}_L}(C)$; it remains to show the reverse.

\begin{lem}
\label{comoduleCoefficientsLemma} Let $(V,\rho),(W,\mu)$ be finite
dimensional comodules over the Hopf algebra $A$. Fix bases $e_1,
\ldots, e_n$, $f_1, \ldots, f_m$ for $V$ and $W$ respectively, and
write
\[ \rho:e_j \mapsto \sum_i e_i \otimes a_{ij} \]
\[ \mu:f_j \mapsto \sum_i f_i \otimes b_{ij} \]
\begin{enumerate}
\item{If $V$ is a subobject of $W$, then each $a_{ij}$ is a linear
combination of the $b_{ij}$.} \item{If $W$ is a quotient object of
$V$, then each $b_{ij}$ is a linear combination of the $a_{ij}$.}
\end{enumerate}
\end{lem}

\begin{proof}
This is a simple fact from linear algebra if we think of
$(a_{ij})$ and $(b_{ij})$ as matrices.  If $\phi:V \rightarrow W$
is a linear map, write it as the matrix $(c_{ij})$ in the relevant
bases. Then $\phi$ being a morphism of $A$-comodules is equivalent
to the matrix equality
\[ (c_{ij})(a_{ij}) = (b_{ij})(c_{ij}) \]
In case $(c_{ij})$ is injective it has a right-sided inverse,
given say by the matrix $(d_{ij})$.  Then $(a_{ij})=
(d_{ij})(b_{ij})(c_{ij})$, and clearly every entry of the right
hand side is a linear combination of the $b_{ij}$; this proves 1.
If $(c_{ij})$ is surjective it has a left-sided inverse, again
call it $(d_{ij})$.  Then we have $(c_{ij})(a_{ij})(d_{ij}) =
(b_{ij})$.  This proves 2.

\end{proof}

\begin{lem}
\label{subalgebrasOfHopfAlgebrasLemma} If $(V,\rho:V \rightarrow V
\otimes A)$ is a right $A$-comodule, and if it belongs to the
principal subcategory $\gen{C}$, then the image of $\rho$ is
contained in $V \otimes C$.
\end{lem}

\begin{proof}
The claim is obvious in case $V$ is direct sum of copies of $C$,
since then its comodule map is the composition
\[ \Delta^n:C^n \stackrel{\Delta^n}{\longrightarrow} (C \otimes C)^n \isomorphic C^n
\otimes C \subset C^n \otimes A \]

Suppose then $(X,\rho)$ is a quotient of some $C^n$.  Then choose
$(a_{ij})$ for $X$ and $(b_{ij})$  for $C^n$ as in the previous
lemma.  Each $b_{ij}$ is in $C$ by assumption, and thus so is each
$a_{ij}$, being a linear combination of the $b_{ij}$.  A similar
argument holds if we consider a subobject of the quotient object
$(X,\rho)$.

\end{proof}

Thus, for any object $(X,\rho) \in \gen{C}$, we can write $\rho:X
\rightarrow X \otimes C$ instead of $\rho:X \rightarrow X \otimes
A$.

Let $\phi:C \rightarrow C$ be any linear map.  Then for $n \in
\mathbb{N}$ we define a map $\phi_{C^n}:C^n \rightarrow C^n$ as
$\phi^n$, that is, the unique linear map commuting with all of the
canonical injections $C \stackrel{\iota_i}{\longrightarrow} C^n$.
Now let $(X,\rho_X)$ be any object of $\gen{C}$. Then there is an
object $(Y,\rho_Y)$ of $\gen{C}$ and a commutative diagram
\begin{diagram}
X & \rInto^\iota & Y & \lOnto^\pi & C^n \\
\dTo^{\rho_X} & & \dTo^{\rho_Y} & & \dTo_{\Delta^n} \\
X \otimes C & \rInto_{\iota \otimes 1} & Y \otimes C & \lOnto_{\pi
\otimes 1}
& C^n \otimes C \\
\end{diagram}
Now, if $\phi$ defines an element of $\text{End}(\omega_R \res
\gen{C})$, there exist unique linear maps $\phi_X$ and $\phi_Y$
making
\begin{diagram}
X & \rInto^\iota & Y & \lOnto^\pi & C^n \\
\dTo^{\phi_X} & & \dTo^{\phi_Y} & & \dTo_{\phi^n} \\
X  & \rInto_{\iota } & Y  & \lOnto_{\pi }
& C^n  \\
\end{diagram}
commute.  But, unless we know \textit{a priori} that $\phi$
defines an element of $\text{End}(\omega_R \res \gen{C})$, all we
can say is that $\phi_Y$ exists, but may not be unique, and that
$\phi_X$ is unique, but may not exist.  Further, if we choose
another such $n$, $(Y,\rho_Y)$, $\pi$ and $\iota$, one cannot
expect to obtain the same linear map $\phi_X$, again unless we
know that $\phi \in \text{End}(\omega_R \res \gen{C})$. For the
moment then, we make the following deliberately ambiguous
definition.

\begin{defn}
\label{phiXdefn} Let $X$ be an object of $\gen{C}$, $\phi:C
\rightarrow C$ any linear map.  Then we define $\phi_X:X
\rightarrow X$ to be any linear map satisfying any one of the
following conditions.

\begin{enumerate}
\item{If $X = C^n$, then $\phi_X = \phi^n$.}
 \item{There exists an
$n$ and a surjective map $C^n \stackrel{\pi}{\longrightarrow} X$
such that $\pi \circ \phi_X = \phi^n \circ \pi$.} \item{There
exists a quotient object $Y$ of $C^n$ such that $\phi_Y$ exists
and satisfies condition 2.~above, and there is an injective map $X
\stackrel{\iota}{\longrightarrow} Y$ such that $\phi_X \circ \iota
= \iota \circ \phi_Y$.}

\end{enumerate}
\end{defn}

So, when we prove theorems about ``the'' map $\phi_X$, it is
understood to apply to any $\phi_X$ satisfying any one of the
above conditions, and under the assumption that it exists in the
first place. It will only later be a consequence of these theorems
that $\phi_X$ is well-defined; that it always exists, and is
always unique.

\begin{lem}
\label{FDsubcoalgebrasphiLemma} Let $(X,\rho_X)$ be an object of
$\gen{C}$, $\phi$ an element of $\text{End}_{\catfont{C}_L}(C)$.
Then the following is always commutative:
\begin{diagram}
X & \rTo^{\phi_X} & X \\
\dTo^{\rho_X} & & \dTo_{\rho_X} \\
X \otimes C & \rTo_{1 \otimes \phi} & X \otimes C \\
\end{diagram}
\end{lem}

\begin{proof}
The claim is obvious if $X = C$, since then the above diagram is
the definition of $\phi$ being an endomorphism of $\catfont{C}$ as
a left $A$-comodule.  If $X = C^n$, consider
\begin{diagram}
C^n & \rTo^{\phi^n} & C^n \\
\dTo^{\Delta^n} & & \dTo_{\Delta^n} \\
(C \otimes C)^n & \rTo^{(1 \otimes \phi)^n} & (C \otimes C)^n \\
\dTo^{\isomorphic} & & \dTo_{\isomorphic} \\
C^n \otimes C & \rTo_{1 \otimes \phi} & C^n \otimes C \\
\end{diagram}
Commutativity of each of the squares is obvious, and commutativity
of the outermost rectangle is what is desired.

Now suppose $C^n \stackrel{\pi}{\longrightarrow} X$ is a quotient
of $C^n$ and that $\phi_X$ satisfies condition 2.~of definition
\ref{phiXdefn}. Consider
\begin{diagram}
\text{ }&&& \rTo^{\pi \otimes 1} &&& \text{ } \\
\uTo &&&&&& \dTo \\
 C^n & \lTo^{\Delta^n} & C^n & \rOnto^\pi & X & \rTo^{\rho_X} & X
\otimes C \\
\dTo^{1 \otimes \phi} & & \dTo^{\phi^n} & & \dTo_{\phi_X} & &
\dTo_{1 \otimes \phi} \\
C^n \otimes C & \lTo_{\Delta^n} & C^n & \rOnto_\pi & X &
\rTo_{\rho_X} & X \otimes C \\
\dTo &&&&&& \uTo \\
\text{ } &&& \rTo_{\pi \otimes 1} &&& \text{ } \\
\end{diagram}
We seek to prove commutativity of the right-most square.
Commutativity of the left most square has been proved, the middle
square commutes by definition, the top and bottom rectangles
commute because $\pi$ is a map of right comodules, and
commutativity of the outermost polygon is obvious.  If one starts
at the second occurrence of $C^n$ at the top and does some diagram
chasing, he eventually obtains  $\pi \circ (\rho_X \circ (1
\otimes \phi)) = \pi \circ (\phi_X \circ \rho_X)$.  But $\pi$ is
surjective, and thus we have $\rho_X \circ (1 \otimes \phi) =
\phi_X \circ \rho_X$, and the claim is proved.

Now suppose $(X,\rho_X)$ is a subobject of the quotient object $Y$
via the map $X \stackrel{\iota}{\longrightarrow} Y$, so that
$\phi_X$ satisfies condition 3.~of definition \ref{phiXdefn}, and
that $\phi_Y$ satisfies condition 2. Consider
\begin{diagram}
\text{ } & & & \lTo_{\iota \otimes 1} & & & \text{ } \\
\dTo &&&&&& \uTo \\
 Y \otimes C & \lTo^{\rho_Y} & Y & \lInto^{\iota} & X &
\rTo^{\rho_X} & X \otimes C \\
\dTo^{1 \otimes \phi} & & \dTo^{\phi_Y} & & \dTo_{\phi_X} & &
\dTo_{1 \otimes \phi} \\
Y \otimes C & \lTo_{\rho_Y} & Y & \lInto_{\iota} & X &
\rTo_{\rho_X} & X \otimes C \\
\uTo &&&&&& \dTo \\
\text{ } &&& \lTo_{\iota \otimes 1} &&& \text{ } \\
\end{diagram}
Commutativity of the right most square is again what we seek to
prove.  We have proved commutativity of the left square, the
middle commutes by definition, the top and bottom rectangles
commute since $\iota$ is a map of comodules, and commutativity of
the outermost rectangle is obvious. Starting at $X$ on the top
line, some diagram chasing shows that $\rho_X \circ (1 \otimes
\phi) \circ (\iota \otimes 1) = \phi_X \circ \rho_X \circ (\iota
\otimes 1)$.  But $\iota \otimes 1$ is injective, whence we have
$\rho_X \circ (1 \otimes \phi) = \phi_X \circ \rho_X$.  The lemma
is proved.
\end{proof}

\begin{prop}
If $(X,\rho_X)$, $(Y,\rho_Y)$ are any objects of $\gen{C}$,
$\psi:X \rightarrow Y$ any morphism in $\catfont{C}_R$, and $\phi$
an element of $\text{End}_{\catfont{C}_L}(C)$, then the following
commutes:
\begin{diagram}
X & \rTo^\psi & Y \\
\dTo^{\phi_X} & & \dTo_{\phi_Y} \\
X & \rTo_\psi & Y \\
\end{diagram}
\end{prop}

\begin{proof}
Consider
\begin{diagram}
X & \rTo^{\phi_X} & X & \rTo^\psi & Y & \rTo^{\phi_Y} & Y \\
\dTo^{\rho_X} & & \dTo^{\rho_X} & & \dTo_{\rho_Y} & &
\dTo_{\rho_Y} \\
X \otimes C & \rTo_{1 \otimes \phi} & X \otimes C & \rTo_{\psi
\otimes 1} & Y \otimes C & \rTo_{1 \otimes \phi} & Y \otimes C \\
\end{diagram}
Commutativity of the right and left squares follow from lemma
\ref{FDsubcoalgebrasphiLemma} and the middle square commutes
because $\psi$ is a morphism in $\catfont{C}_R$.  Thus this entire
diagram commutes. Looking at the left two-square rectangle we have
\[ \phi_X \circ \psi \circ \rho_Y = \rho_X \circ (\psi \otimes
\phi) \] and at the right we have
\[ \psi \circ \phi_Y \circ \rho_Y = \rho_X \circ (\psi \otimes
\phi) \] Thus $(\phi_X \circ \psi) \circ \rho_Y = (\psi \otimes
\phi_Y) \circ \rho_Y$.  But $\rho_Y$ is injective (as all comodule
maps are), hence $\phi_X \circ \psi = \psi \circ \phi_Y$, and the
proposition is proved.

\end{proof}

All that is left to show is that the map $\phi_X$, for $\phi \in
\text{End}_{\catfont{C}_L}(C)$ and $X \in \gen{C}$, actually
exists and is unique.

Uniqueness is immediate: if $\phi_X$ and $\phi_X^\prime$ are any
two maps satisfying definition \ref{phiXdefn}, then they both
satisfy the hypotheses of the previous proposition.  As the
identity map $1:X \rightarrow X$ is a morphism in $\catfont{C}_R$,
the diagram
\begin{diagram}
X & \rTo^{1} & X \\
\dTo^{\phi_X} & & \dTo_{\phi_X^\prime} \\
X & \rTo_{1} & X \\
\end{diagram}
commutes, showing $\phi_X$ and $\phi_X^\prime$ to be equal.

For existence, suppose first that $X$ is a quotient of $C^n$ under
$\pi$. Then $\phi_X$ certainly exists; if $e_i$ is a basis for
$X$, pull each $e_i$ back through $\pi^{-1}$, down through
$\phi^n$, and back through $\pi$.  But we know that $\phi_X$ is
unique, and the only reason that it \emph{would} be unique is
because $\phi^n$ stabilizes the kernel of $\pi$; otherwise there
would be many $\phi_X$ satisfying condition 2.~of definition
\ref{phiXdefn}.  This observation applies to any surjective map on
$C^n$.  As every subobject of $C^n$ is necessarily the kernel of
some surjective map we have proved

\begin{prop}
If $\phi \in \text{End}_{\catfont{C}_L}(C)$, $\phi^n$ stabilizes
all subobjects of $C^n$.
\end{prop}

Now let $Y$ be a subobject of the quotient object $X$, with $X
\stackrel{\pi}{\longrightarrow} X/Y$ the canonical projection.
Then there exists a map $\phi_{X/Y}$ making
\begin{diagram}
X & \rTo^{\pi} & X/Y \\
\dTo^{\phi_X} & & \dTo_{\phi_{X/Y}} \\
X & \rTo^{\pi} & X/Y \\
\end{diagram}
commute.  But this $\phi_{X/Y}$ also makes
\begin{diagram}
C^n & \rTo^{\pi^\prime} & X & \rTo^{\pi} & X/Y \\
\dTo^{\phi^n} & & \dTo^{\phi_X} & & \dTo_{\phi_{X/Y}} \\
C^n & \rTo^{\pi^\prime} & X & \rTo^{\pi} & X/Y \\
\end{diagram}
commute, in particular the outermost rectangle.  Thus $\phi_{X/Y}$
satisfies condition 2.~of definition \ref{phiXdefn}, and is hence
unique; but once again, the only reason this would be true is if
$\phi_X$ stabilized the kernel of $\pi$, namely $Y$.  We have
proved

\begin{prop}
If $X$ is a quotient object of $C^n$ and $\phi_X$ satisfies
condition 2.~of definition \ref{phiXdefn}, then $\phi_X$
stabilizes all subobjects of $X$.
\end{prop}

Finally, if $Y$ is a subobject of the quotient object $X$, then
$\phi_X$ stabilizes $Y$, whence there is a map $\phi_Y$ making
\begin{diagram}
Y & \rInto^\iota & X \\
\dTo^{\phi_Y} & & \dTo_{\phi_X} \\
Y & \rInto^\iota & X \\
\end{diagram}
In all cases then $\phi_X$ exists and is unique.  We have proved

\begin{thm}
For a subcoalgebra $C$ of $A$, $\text{End}_{\catfont{C}_L}(C) =
\text{End}(\omega_R \vert \gen{C})$.
\end{thm}

\section{Corollaries}

Here we record some results based on the above which will be used
later.  We will prove the results for the category
$\catfont{C}_R$; we leave it to the reader to formulate the
obvious analogues for $\catfont{C}_L$.

\begin{prop}
\label{surjectivetransitionmapprop} Let $C$ and $D$ be finite
dimensional subcoalgebras of the Hopf algebra $A$ with $C \subset
D$.  Then the \index{transition mapping} transition mapping
$\text{End}(\omega_R \res \gen{D}) \rightarrow \text{End}(\omega_R
\res \gen{C})$ is dual to the inclusion map $C \rightarrow D$ via
the isomorphism $\text{End}(\omega_R \res \gen{C}) =
\text{End}_{\catfont{C}_L}(C) \isomorphic C^*$.  In particular,
this transition mapping is surjective.
\end{prop}

\begin{proof}
Let $C \stackrel{\iota}{\longrightarrow} D$ be the inclusion
mapping.  Recall from page \pageref{LambdaDefFDcoalgebras} that we
have an isomorphism $\Lambda:\text{End}(\omega_R \res \gen{D})
\longrightarrow D^*$ given by, for $\phi:D \rightarrow D$,
$\Lambda(\phi)$ is the composition
\[ D \stackrel{\phi}{\longrightarrow} D
\stackrel{\varepsilon}{\longrightarrow} k \] and that $\Lambda$
has an inverse $\Theta$ which sends a linear functional $\alpha$
to
\[ D \stackrel{\Delta}{\longrightarrow} D \otimes D \stackrel{1
\otimes \alpha}{\vlongrightarrow} D \otimes k \isomorphic D \] As
$\iota$ is a map of coalgebras $\iota^*$ of algebras, and we have
a commutative diagram
\begin{diagram}
D^* & \rTo^{\iota^*} & C^* \\
\uTo^\Lambda & & \dTo_{\Theta} \\
\text{End}(\omega_R \res \gen{D}) & \rTo_T & \text{End}(\omega_R
\res
\gen{C}) \\
\end{diagram}
where $T$ is some as of yet unidentified algebra map.  Note that
since $\iota$ is injective, $\iota^*$ is surjective, and hence so
is $T$.  We claim that $T$ is the usual transition mapping.

Let $\phi \in \text{End}(\omega_R \res \gen{D})$, and consider
\begin{diagram}
C & \rTo^\Delta & C \otimes C & \rTo^{1 \otimes \iota} & C \otimes
D & \rTo^{1 \otimes \phi} & C \otimes D & \rTo^{1 \otimes
\varepsilon} & C \otimes k & \rTo^{\isomorphic} & C \\
\dTo^\iota & & & \rdTo_{\iota \otimes \iota} & \dTo_{\iota \otimes
1} & & \dTo_{\iota \otimes 1} & & \dTo_{\iota \otimes 1} & &
\dTo_\iota \\
D & & \rTo_\Delta &  & D \otimes D & \rTo_{1 \otimes \phi} & D
\otimes D & \rTo_{1 \otimes \varepsilon} & D \otimes k &
\rTo_{\isomorphic} & D \\
\dTo^\phi & & & & & &  \uTo_\Delta & & & \ruTo(4,2)^1 &  \\
D & & & \rEq & & &  D & &  \\
\end{diagram}
Commutativity of the top left most simple polygon is the assertion
that $\iota$ is a map of coalgebras, and commutativity of all the
other simple polygons in the top row is trivial.  The first
polygon in the bottom row is the assertion that $\phi$ is an
endomorphism on $D$ as a left comodule (which it is, by the
equality $\text{End}(\omega_R \res \gen{D}) =
\text{End}_{\catfont{C}_L}(D)$) and the last polygon is a
coalgebra identity; thus this entire diagram commutes. Now the
composition comprising the entire top line is exactly the map
$T(\phi) \stackrel{\text{def}}{=} \Theta(\iota^*(\Lambda(\phi)))$,
so the outermost polygon of this diagram is
\begin{diagram}
C & \rTo^\iota & D \\
\dTo^{T(\phi)} & & \dTo_\phi \\
C & \rTo_\iota & D \\
\end{diagram}
But, the image of $\phi$ under the transition mapping
$\text{End}(\omega_R \res \gen{D}) \rightarrow \text{End}(\omega_R
\res \gen{C})$ is \emph{by definition} the unique linear map
$\phi_C : C \rightarrow C$ that makes this diagram commute.  Hence
$\phi_C = T(\phi)$, and the proposition is proved.
\end{proof}

\begin{thm}
\label{subcoalgebratheorem} Let $D$ be a finite dimensional
subcoalgebra of $A$, $\text{End}(\omega_R \res \gen{D})
\stackrel{\pi}{\longrightarrow} L$ any surjective mapping of
algebras.  Then there exists a subcoalgebra $C$ of $D$ such that
$L$ is isomorphic to $\text{End}(\omega_R \res \gen{C})$ and
\begin{diagram}
\text{End}(\omega_R \res \gen{D}) & \rTo^\pi & L \\
& \rdTo_{T} & \dTo_{\isomorphic} \\
& & \text{End}(\omega_R \res \gen{C}) \\
\end{diagram}
commutes, where $T$ is the transition mapping.
\end{thm}

\begin{proof}
As $\text{End}(\omega_R \res \gen{D})
\stackrel{\Lambda}{\longrightarrow} D^*$ is an isomorphism of
algebras, any quotient of the former gives rise to a quotient of
the latter via
\begin{diagram}
D^* & \rTo^{\pi^\prime} & C^* \\
\uTo^\Lambda & & \dTo^\isomorphic \\
\text{End}(\omega_R \res \gen{D}) & \rTo_{\pi} & L \\
\end{diagram}
Here we have skipped a step and denoted this algebra by $C^*$,
since, being finite dimensional, it is necessarily the dual
algebra to a unique coalgebra $C$.  Note that $C$ is a
subcoalgebra of $D$ via the map $(\pi^\prime)^*$ and the natural
isomorphism $C^{*\circ} \isomorphic C$.  This means that we can
take $\pi^\prime$ to be dual to the inclusion mapping $C
\stackrel{\iota}{\longrightarrow} D$; let us then replace
$\pi^\prime$ with $\iota^*$.

We need to verify that $L$ can be identified with
$\text{End}(\omega_R \res \gen{C})$.  Consider
\begin{diagram}
D^* & \rTo^{\iota^*} & C^* \\
\uTo^\Lambda & & \dTo^\isomorphic & \rdTo(2,4)^\Theta \\
\text{End}(\omega_R \res \gen{D}) & \rTo_{\pi} & L \\
& \rdTo(4,2)_T & & \rdDotsto^\Sigma \\
&     &                  &  & \text{End}(\omega_R \res \gen{C}) \\
\end{diagram}
Commutativity of the outermost polygon was proved in proposition
\ref{surjectivetransitionmapprop}, and commutativity of the square
is given.  Define the map $\Sigma$ to pass back up through the
isomorphism with $C^*$ and then down through $\Theta$, i.e.
$\Sigma = \isomorphic^{-1} \circ \Theta$. This $\Sigma$ is the
isomorphism we seek; the theorem is proved.
\end{proof}

There is an obvious analogue to this theorem as regards
subalgebras of $\text{End}(\omega_R \res \gen{D})$ as opposed to
quotients, which we will state but not prove.

\begin{thm}
Let $D$ be a finite dimensional subcoalgebra of $A$, and $L
\stackrel{\iota}{\longrightarrow} \text{End}(\omega_R \res
\gen{D}$ any injective mapping of algebras.  Then there exists a
quotient coalgebra $C$ of $D$ such that $L$ is isomorphic to
$\text{End}(\omega_R \res \gen{C})$ and
\begin{diagram}
L & \rTo^\iota & \text{End}(\omega_R \res \gen{D}) \\
& \rdTo_{\isomorphic} & \dTo_{T} \\
& & \text{End}(\omega_R \res \gen{C}) \\
\end{diagram}
commutes, where $T$ is the transition mapping.
\end{thm}

As promised at the beginning of this chapter, these theorems
actually prove that the `algorithm' described in section
\ref{recoveringanalgebraicgroupsection} for recovering the Hopf
algebra $A$ from $\text{Comod}_A$ does in fact give the correct
answer.  (This does \emph{not} prove the general principle of
tannakian duality; here we are assuming from the outset that the
category we are looking at is $\text{Comod}_A$ for some Hopf
algebra $A$).

\begin{thm}
Let $\catfont{C} = \text{Comod}_A$ for some Hopf algebra $A$. Then
$A$ can be recovered as the direct limit of the finite dimensional
coalgebras $\text{End}(\omega \res \gen{X})^\circ$, with the
direct system being the maps $\text{End}(\omega \res
\gen{X})^\circ \stackrel{T_{X,Y}^\circ}{\vlongrightarrow}
\text{End}(\omega \res \gen{Y})^\circ$ whenever $X \in \gen{Y}$,
where $\text{End}(\omega \res \gen{Y})
\stackrel{T_{X,Y}}{\longrightarrow} \text{End}(\omega \res
\gen{X})$ is the transition mapping.
\end{thm}

\begin{proof}
As argued at the beginning of this chapter, the entire category
$\text{Comod}_A$ can be recovered as the direct limit of the
principal subcategories $\gen{C}$, where $C$ ranges over all
finite dimensional subcoalgebras of $A$, with the direct system
being the inclusions $C \subset D$.  Lemma \ref{directlimitlemma}
tells that, for purposes of computing the direct limit, we are
justified in disregarding all objects but these subcoalgebras and
all maps but these inclusions.  But theorem \ref{adualitytheorem}
tells us that $\text{End}(\omega \res \gen{C})^\circ \isomorphic
C$, and proposition \ref{surjectivetransitionmapprop} tells us
that under this isomorphism, the map $\text{End}(\omega \res
\gen{C})^\circ \stackrel{T_{X,Y}^\circ}{\vlongrightarrow}
\text{End}(\omega \res \gen{D})^\circ$ can be identified with the
inclusion map $C \subset D$.  Apply theorem
\ref{fundamentaltheoremofcoalgebras} to see that the direct limit
of these is exactly the Hopf algebra $A$.
\end{proof}

\chapter{The Representing Hopf Algebra of a Restricted Ultraproduct}

\label{chapterTheRepresentingHopfAlgebraOf}

Let $k_i$ be an indexed collection of fields, $(A_i, \Delta_i,
\varepsilon_i)$ a collection of Hopf algebras over those fields,
and $\catfont{C}_i$ the category $\text{Comod}_{A_i}$. The work
done in chapter \ref{TheMainTheorem} tells us that the restricted
ultraproduct $\resprod \catfont{C}_i$ is itself a neutral
tannakian category over the field $k = \uprod k_i$, hence
tensorially equivalent to $\text{Comod}_{A_\infty}$ where
$A_\infty$ is some Hopf algebra over $k$.  The question then: what
is $A_\infty$?

Before starting in earnest, let us examine an obvious first guess:
$A_\infty$ is the ultraproduct of the Hopf algebras $A_i$.  But an
ultraproduct of Hopf algebras is not, in general, a Hopf algebra.
An ultraproduct of algebras over the fields $k_i$ is indeed an
algebra over the field $k_i$, with the obvious definitions of
addition, multiplication, and scalar multiplication.  The problem
comes when we try to give it the structure of a coalgebra,
consistent with the coalgebra structures on each $A_i$.  We start
by writing
\[ \Delta: \uprod A_i \stackrel{[\Delta_i]}{\vlongrightarrow}
\uprod A_i \otimes A_i \] But as it stands, this does not suffice;
we need $\Delta$ to point to $\uprod A_i \otimes \uprod A_i$.
Recall from proposition \ref{ultratensorproductprop} that there is
a natural injective map
\[ \uprod A_i \otimes \uprod A_i \stackrel{\Phi}{\longrightarrow}
\uprod A_i \otimes A_i \] and that, unless the $A_i$ have
boundedly finite dimensionality, it is not surjective.  The image
of $\Phi$ consists exactly of those elements $[v_i]$ which have
bounded tensor length, and for a given collection $A_i$ of
non-boundedly finite dimension, it is a relatively simple matter
to come up with an element $[a_i] \in \uprod A_i$ such that
$[\Delta_i(a_i)]$ has unbounded tensor length. Thus we cannot
expect the image of $\Delta$ constructed above to be contained in
$\uprod A_i \otimes \uprod A_i$ in general.

The next section is devoted to identifying a certain subset of
$\uprod A_i$ which can indeed be given the structure of a
coalgebra, using the definition of $\Delta$ given above.
Thereafter we will show that this coalgebra is in fact a Hopf
algebra, indeed equal to the $A_\infty$ we seek.

\section{The Restricted Ultraproduct of Hopf Algebras}

\label{sectionTheRestrictedUltraproductOfHopfAlgebras}

To allay some of the suspense, we give the following definition,
whose meaning will not be clear until later in this section.

\begin{defn}
\index{restricted ultraproduct!of Hopf algebras} \index{Hopf
algebra!restricted ultraproduct of} The \textbf{restricted
ultraproduct} of the Hopf algebras $A_i$, denoted \index{$A_R$}
$A_R$, is the subset of $\uprod A_i$ consisting of those elements
$[a_i]$ such that $\text{rank}(a_i)$ is bounded.
\end{defn}

Our goal in this section is to define exactly what \index{rank}
``rank'' means in this context, and to show that $A_R$ can be
given the structure of a coalgebra.  Note that the notation $A_R$
makes no mention of the particular ultrafilter $\filtfont{U}$
being applied.  As $\filtfont{U}$ is always understood to be fixed
but arbitrary, no confusion should result.

Let $A$ be a Hopf algebra, $\catfont{C} = \text{Comod}_A$.  For
each $X \in \catfont{C}$ let \index{$L_X$} $L_X$ be the image of
$\text{End}(\omega \res \gen{X})$ inside $\text{End}(\omega(X))$
(see page \pageref{LXdefn}), and \index{$T_{X,Y}$} $L_X
\stackrel{T_{X,Y}}{\longleftarrow} L_Y$ the usual transition
mapping for $X \in \gen{Y}$.  Then we have an inverse system of
algebras, and from it obtain an inverse limit
\begin{diagram}
 & & \invlim_\catfont{C} L_X & & \\
 & \ldTo^{T_Y} & & \rdTo^{T_Z} & \\
 L_Y & & \lTo_{T_{Y,Z}} & & L_Z \\
\end{diagram}

\begin{lem}
\label{T_CissurjectiveLemma_TheRepHopfAlgOf}
 If $C$ is a subcoalgebra of $A$, then $\invlim L_X
\stackrel{T_C}{\longrightarrow} L_C$ is surjective.
\end{lem}

\begin{proof}
By the discussion on page \pageref{adualitytheoremchapter} the
entire category $\catfont{C}$ is generated by the principal
subcategories $\gen{C}$, where $C$ ranges over all finite
dimensional subcoalgebras of $A$, with the direct system being
inclusion mappings $C \subset D$ when applicable. We can then
apply lemma \ref{inverselimitlemma} to see that we might as well
have recovered $\invlim L_X$ with respect to the sub-inverse
system consisting of all subcoalgebras of $A$ under the inclusion
mappings, that is as
\begin{diagram}
 & & \invlim_\catfont{C} L_X & & \\
 & \ldTo^{T_C} & & \rdTo^{T_D} & \\
 L_C & & \lTo_{T_{C,D}} & & L_D \\
\end{diagram}
where $C$ and $D$ range over all subcoalgebras of $A$, and
$T_{C,D}$ defined when $C \subset D$.  Proposition
\ref{surjectivetransitionmapprop} tells us that $T_{C,D}$ is
always surjective, and it is a standard fact about inverse limits
that if this is the case, $T_C$ is always surjective.

\end{proof}

Now apply the finite dual operation to the above inverse system
diagram to obtain a diagram of coalgebras:
\begin{diagram}
 & & (\invlim_\catfont{C} L_X)^\circ & & \\
 & \ruTo^{T_Y^\circ} & & \luTo^{T_Z^\circ} & \\
 L_Y^\circ & & \rTo_{T_{Y,Z}^\circ} & & L_Z^\circ \\
\end{diagram}

The maps $L_Y^\circ \stackrel{T_{Y,Z}^\circ}{\longrightarrow}
L_Z^\circ$ therefore form a direct system; recall from page
\pageref{pagerefForRecoveringA_TannakianDuality} that this is
exactly the direct system from which $A$ can be recovered as its
direct limit. Let us rename $T_{Y,Z}^\circ$ as $\phi_{Y,Z}$, and
so we have the direct limit diagram
\begin{diagram}
 & & \dirlim_\catfont{C} L_X^\circ & & \\
 & \ruTo^{\phi_Y} & & \luTo^{\phi_Z} & \\
 L_Y^\circ & & \rTo_{\phi_{Y,Z}} & & L_Z^\circ \\
\end{diagram}
with $A = \dirlim_\catfont{C} L_X^\circ$.  By the universal
property of direct limits there is a unique map of coalgebras,
call it $\phi:A \rightarrow \left(\invlim L_X\right)^\circ$,
making the following diagram commute for all $Y \in \gen{Z}$:
\begin{equation}
\label{phiDefnDiagram_ResUprodOfHopfAlgebras}
\begin{diagram}
 & & (\invlim_\catfont{C} L_X)^\circ & & \\
 & \ruTo(2,4)^{T_Y^\circ} & \uDashto^{\phi} & \luTo(2,4)^{T_Z^\circ} & \\
 &  & \dirlim_\catfont{C} L_X^\circ & & \\
 & \ruTo_{\phi_Y} & & \luTo_{\phi_Z} & \\
 L_Y^\circ & & \rTo_{\phi_{Y,Z}}  & & L_Z^\circ \\
\end{diagram}
\end{equation}

\begin{prop}
The map $\phi$ is injective.
\end{prop}

\begin{proof}
Recall the concrete definition of a direct limit of algebraic
objects; its underlying set consists of equivalence classes $[a]$
 where $a$ is some element of some $L_Y^\circ$, with $a \in L_Y,$
 $b \in L_Z$ equivalent when there is some $L_T$ with $Y,Z
 \in \gen{T}$ such that $\phi_{Y,T}(a) = \phi_{Z,T}(a)$.  As is
 discussed on page \pageref{adualitytheoremchapter}, for any object $Y$,
 there is a subcoalgebra $C$ of $A$ such that $Y \in \gen{C}$;
 this shows that any element $a$ of any $L_Y^\circ$ is equivalent to
 some element $c$ of $L_C^\circ$ for some subcoalgebra $C$ of $A$.
 Thus, every element of $A= \dirlim L_X^\circ$ can be written as $[c]$, for
 some $c \in L_C^\circ$, $C$ a coalgebra.  Further, given elements $[c]$ and $[d]$ of $\dirlim
 L_X^\circ$, we can clearly choose subcoalgebras $C$ and $D$ of $A$
 so that $c \in L_C^\circ$, $d \in L_D^\circ$, and $C \subset D$ (just enlarge $D$ to be a finite dimensional
 coalgebra containing both $C$ and $D$).

So let $[c], [d] \in \dirlim L_X^\circ$, with $c \in L_C^\circ$,
$d \in L_D^\circ$, and $C \subset D$, and suppose that $\phi$ maps
$[c]$ and $[d]$ to the same element.  Consider the diagram
\begin{diagram}
 & & (\invlim_X L_X)^\circ & & \\
 & \ruTo(2,4)^{T_C^\circ} & \uDashto^{\phi} & \luTo(2,4)^{T_D^\circ} & \\
 &  & \dirlim_X L_X^\circ & & \\
 & \ruTo_{\phi_C} & & \luTo_{\phi_D} & \\
 L_C^\circ & & \rTo_{\phi_{C,D}}  & & L_D^\circ \\
\end{diagram}
To say that $\phi([c]) = \phi([d])$ is the same as saying that
$T_C^\circ(c) = T_D^\circ(d)$.  But $T_D$ is surjective by lemma
\ref{T_CissurjectiveLemma_TheRepHopfAlgOf}, thus $T_D^\circ$ is
injective. By commutativity of
\begin{diagram}
& & (\invlim_X L_X)^\circ & & \\
& \ruTo^{T_C^\circ} & & \luTo^{T_D^\circ} \\
L_C^\circ & & \rTo_{T_{C,D}^\circ} & & L_D^\circ \\
\end{diagram}
we see that we must have $T_{C,D}^\circ(c) = d$; this means that
$[c]=[d]$.
\end{proof}

The map $\phi$ is by no means generally surjective; the author
verified this with a counterexample which he will not burden you
with.

We pause for a moment to see what the map $\phi$ actually looks
like.  A typical element of $A = \dirlim L_X^\circ$ is an
equivalence class $[\alpha:L_Y \rightarrow k]$, where $Y$ is some
object of $\catfont{C}$.  Passing this element through $\phi$ is
the same as pulling it back through $\phi_Y$, and passing back up
through $T_Y^\circ$.  Thus $\phi([\alpha:Y \rightarrow k])$ is the
composition
\[ \invlim_X L_X \stackrel{T_Y}{\longrightarrow} L_Y
\stackrel{\alpha}{\longrightarrow} k \]

Recall the definition of the finite dual $L^\circ$ of the algebra
$L$; it consists of those linear functionals $\alpha:L \rightarrow
k$ which happen to kill an ideal of $L$ having finite codimension.
We therefore define the \index{rank} \textbf{rank} of an element
of $L^\circ$ as the minimum $m$ such that $\alpha$ kills an ideal
of codimension $m$.

\begin{defn}
\index{rank} The \textbf{rank} of an element of $a \in A$  is the
rank of $\phi(a) \in (\invlim L_X)^\circ$.
\end{defn}

\begin{prop}
\label{tensorLengthInTheFiniteDualProp}

Let $(L,\text{mult})$ be any algebra, $(L^\circ, \Delta)$ its
finite dual.  Then if $\alpha \in L^\circ$ has rank $m$,
$\Delta(\alpha)$ can be written as a sum of no more than $m$
simple tensors. Further, we can write
\[ \Delta(\alpha) = \sum_{i=1}^m \beta_i \otimes \gamma_i \]
where each $\beta_i$ and $\gamma_i$ themselves have rank no larger
than $m$.
\end{prop}

\begin{proof}
Let $I \lhd L$ be an ideal of codimension $m$ such that $\alpha(I)
= 0$, and let $e_1, \ldots, e_m \in L$ be such that $e_1 + I,
\ldots, e_m + I$ is a basis for $L/I$, and extend the $e_i$ to a
basis $e_1, \ldots, e_m, f_1, f_2, \ldots $ for all of $L$. Recall
the definition of $\Delta$ in terms of $\text{mult}$; it sends the
functional $\alpha$ to the map
\[ L \otimes L \stackrel{\text{mult}}{\longrightarrow} L
\stackrel{\alpha}{\longrightarrow} k \] and then passes to the
isomorphism $(L \otimes L)^\circ \isomorphic L^\circ \otimes
L^\circ$.  We have a basis for $L \otimes L$, namely those tensors
of the form $e_i \otimes f_j, f_i \otimes e_j, f_i \otimes f_j,
e_i \otimes e_j$.  Since $I$ is an ideal, the only of these basis
elements that might \emph{not} get sent to $I$ under mult are
those of the form $e_i \otimes e_j$, $i,j \leq m$. Then since
$\alpha$ kills $I$, $\Delta(\alpha)$ kills all but the $e_i
\otimes e_j$.  Letting $\gamma_i$ be the functional that sends
$e_i$ to $1$ and everything else to zero, we can write
\[ \Delta(\alpha) = \sum_{i,j} c_{ij} \gamma_i \otimes \gamma_j = \sum_i \gamma_i \otimes \left(\sum_j c_{ij}\gamma_j \right) \]
for some scalars $c_{ij}$, which is a sum of no more than $m$
simple tensors. As each $\gamma_i$ still kills $I$,
$\text{rank}(\gamma_i) \leq m$ for each $i$, and the last claim is
proved as well.
\end{proof}

We shall need the following lemma from linear algebra.

\begin{lem}
\label{minTensorLengthLemma} Let $V,W$ be vector spaces over some
field, and let $\sum_{i=1}^n v_i \otimes w_i \in V \otimes W$.
Then this expression is of minimal tensor length if and only if
the vectors $v_i$ are linearly independent, and the $w_i$ are
linearly independent.
\end{lem}

\begin{proof}
Suppose that one of the collections $v_i$ or $w_i$ are not
linearly independent, let's say the $v_i$.  Then say $v_n$ is in
the span of $v_1, \ldots, v_{n-1}$, and write
\[ v_n = a_1 v_1 + \ldots + a_{n-1} v_{n-1} \]
Then
\begin{equation*}
\begin{split}
\sum_{i=1}^n v_i \otimes w_i & = \sum_{i=1}^{n-1} v_i \otimes w_i
+ v_n \otimes w_n \\ & = \sum_{i=1}^{n-1} v_i \otimes w_i +
(\sum_{i=1}^{n-1} a_i v_i) \otimes w_n \\
& = \sum_{i=1}^{n-1} v_i \otimes (w_i + a_i w_n)
\end{split}
\end{equation*}
which is a sum of less than $n$ simple tensors.  Therefore the
expression is not of minimal tensor length.

Conversely, suppose that $\sum_{i=1}^n v_i \otimes w_i$ can be
reduced in tensor length, and that both the collections $v_i$ and
$w_i$ are linearly independent; we shall force a contradiction.
Suppose we have a tensor length reduction given by the equation
\begin{equation}
\label{tensorLengthLemmaEquation}
 \sum_{i=1}^n v_i \otimes w_i = \sum_{i=1}^n v_i^\prime \otimes
w_i^\prime
\end{equation}
where we let $v_n^\prime = w_n^\prime = 0$.  The $v_i,v_i^\prime$
and $w_i,w_i^\prime$ span finite dimensional subspaces of $V$ and
$W$ respectively; fix bases $\{e_1, \ldots, e_m\}$, $\{f_1,
\ldots, f_l\}$ for these subspaces.  Then write
\[ v_i = \sum_{j=1}^m c_{ij} e_j \hspace{1cm} w_i = \sum_{k=1}^l
d_{ik} f_k \]
\[ v_i^\prime = \sum_{j=1}^m c_{ij}^\prime e_j \hspace{1cm}
w_i^\prime = \sum_{k=1}^l d_{ik}^\prime f_k \] If we plug these
expressions into equation \ref{tensorLengthLemmaEquation} above
and rearrange the summations a bit, we obtain
\[ \sum_{{j=1 \ldots m} \atop {k=1 \ldots l}} \left(\sum_{i=1}^n c_{ij} d_{ik} \right) e_j \otimes f_k
= \sum_{{j=1 \ldots m} \atop {k=1 \ldots l}} \left( \sum_{i=1}^n
c_{ij}^\prime d_{ik}^\prime \right) e_j \otimes f_k \] By matching
coefficients on the linearly independent simple tensors $e_j
\otimes f_k$, we obtain, for every $j$ and $k$, $\sum_{i=1}^n
c_{ij} d_{ik} =  \sum_{i=1}^n c_{ij}^\prime d_{ik}^\prime $.  But
the left hand side is the $(j,k)^{\text{th}}$ entry of the matrix
$(c_{ij})^T(d_{ij})$, and the right hand side the
$(j,k)^{\text{th}}$ entry of $(c_{ij}^\prime)^T(d_{ij}^\prime)$.
Thus, equation \ref{tensorLengthLemmaEquation} is equivalent to
the matrix equation
\[ (c_{ij})^T(d_{ij}) = (c_{ij}^\prime)^T(d_{ij}^\prime) \]

Now since the $v_i$ sit in an $m$-dimensional space and the $w_i$
sit in an $l$-dimensional space, and since we are assuming both
the $v_i$, $w_i$ to be linearly independent, we conclude that $n$
is no bigger than either $m$ or $l$.  Further, the linear
independence of the $v_i$ is equivalent to the linear independence
of the row vectors of the matrix $(c_{ij})$, i.e.~the column
vectors of $(c_{ij})^T$.  Similarly the row vectors of $(d_{ij})$
are linearly independent.  This means that the matrix $(c_{ij})^T$
has fullest possible rank, namely $n$; for the same reason
$(d_{ij})$ has rank $n$.  Viewing the product $(c_{ij})^T(d_{ij})$
as a linear transformation $k^l \rightarrow k^n \rightarrow k^m$,
we see that this product also has rank $n$.

But we claim that $(c_{ij}^\prime)^T(d_{ij}^\prime)$ has rank less
than $n$.  The condition that $v_n^\prime = w_n^\prime = 0$ forces
the matrix $(c_{ij}^\prime)^T$ to have a column of zeroes at the
far right, and $(d_{ij}^\prime)$ to have a row of zeroes at the
bottom.  Then $(c_{ij}^\prime)^T$ has less than $n$ non-zero
column vectors, and so has rank less than $n$; similarly
$(d_{ij})$ has rank less than $n$.  Then clearly also must their
product.

We conclude then that $(c_{ij})^T(d_{ij}) =
(c_{ij}^\prime)^T(d_{ij}^\prime)$ have different rank, a
contradiction.  This completes the proof.
\end{proof}

\begin{cor}
Let $\psi:V \rightarrow W$ be an injective mapping of vector
spaces, and let $v \in V \otimes V$.  Then if $(\psi \otimes
\psi)(v) \in W \otimes W$ can be written as a sum of no more than
$m$ simple tensors, so can $v$.
\end{cor}

\begin{proof}
Write $v = \sum_{i=1}^n v_i \otimes v_i^\prime$ as a minimal sum
of simple tensors, so by lemma \ref{minTensorLengthLemma}the $v_i$
and the $v_i^\prime$ are linearly independent.  As $\psi$ is
injective, the collections $\psi(v_i)$ and $\psi(v_i^\prime)$ are
also linearly independent.  Then the expression $(\psi \otimes
\psi)(v) = \sum_{i=1}^n \psi(v_i) \otimes \psi(v_i^\prime)$, again
by lemma \ref{minTensorLengthLemma}, is of minimal tensor length
in $W \otimes W$.
\end{proof}

We can now prove the key fact which allows us to define a natural
coalgebra structure on $A_R$.

\begin{prop}
\label{RankAndTensorLengthAreNiceProposition} If $a \in A =
\dirlim L_X^\circ$ has rank no greater than $m$, then $\Delta(a)$
can be written as
\[ \Delta(a) = \sum_{i=1}^m b_i \otimes c_i \]
where each $b_i$ and $c_i$ themselves have rank no greater than
$m$.
\end{prop}

\begin{proof}
Let $\Delta$ be the coalgebra structure on $\dirlim L_X^\circ$,
$\Delta^\prime$ that on $(\invlim L_X)^\circ$, and let $\phi(a) =
\alpha$.  As $\alpha$ has rank no greater than $m$, proposition
\ref{tensorLengthInTheFiniteDualProp} tells us that
$\Delta^\prime(\alpha)$ can be written as a sum of no more than
$m$ simple tensors.  As $\phi:\dirlim L_X^\circ \rightarrow
(\invlim L_X)^\circ$ is injective, and since $\phi$ is a map of
coalgebras, we have
\[ (\phi \otimes \phi)(\Delta(a)) = \Delta^\prime(\alpha) \]
which, by the previous corollary, shows that $\Delta(a)$ is a sum
of no more than $m$ simple tensors.  Then write
\[ \Delta(a) = \sum_{i=1}^n b_i \otimes c_i \]
where $n$ is minimal, and in particular so that the $b_i$ and the
$c_i$ are linearly independent, and so that $n \leq m$.  We claim
that all of the $b_i$ and $c_i$ have rank no greater than $m$.

Suppose not, and say $b_1$ has rank greater than $m$.  Let $I
\ideal \invlim L_X$ be an ideal of codimension $m$ killed by
$\alpha$. Then $\phi(b_1)$ cannot kill all of $I$; lets say it
doesn't kill $f \in I$.

Now since the $c_i$ are linearly independent, so are the
$\phi(c_i)$.  Then we can find linearly independent vectors $v_1,
\ldots, v_n \in \dirlim L_X$ such that $\phi(c_i)(v_j) =
\delta_{ij}$ (lemma 1.5.8 of \cite{hopfalgebras}).  Then (under
the isomorphism $(\dirlim L_X)^\circ \otimes (\dirlim L_X)^\circ
\isomorphic (\dirlim L_X \otimes \dirlim L_X)^\circ$), $\sum_i
\phi(b_i) \otimes \phi(c_i)$ does not kill the element $f \otimes
v_1$.

But $\Delta^\prime(\alpha)$ \emph{does} kill $f \otimes v_1$; its
action is given by the composition
\[ \dirlim L_X \otimes \dirlim L_X
\stackrel{\text{mult}}{\vlongrightarrow} \dirlim L_X
\stackrel{\alpha}{\longrightarrow} k \] and as $I$ is an ideal, $f
\otimes v_1$ gets mapped into $I$ under mult, and as $\alpha$
kills $I$, $\Delta^\prime(\alpha)$ kills $f \otimes v_1$.  But
this is absurd, since $\Delta^\prime(\alpha) = \sum_i \phi(a_i)
\otimes \phi(b_i)$.

Thus $\phi(b_1)$ cannot have rank greater than $m$, and the same
argument obviously applies to all of the $b_i$ and $c_i$.  This
completes the proof.
\end{proof}

Now let $(A_i,\Delta_i,\varepsilon_i)$ be an indexed collection of
Hopf algebras over the fields $k_i$, $\filtfont{U}$ an ultrafilter
on $I$, $\uprod A_i$ the ultraproduct (as a vector space) of the
$A_i$, and let $\phi_i:A_i \rightarrow (\invlim L_X)^\circ$ be the
map defined by diagram \ref{phiDefnDiagram_ResUprodOfHopfAlgebras}
for each $A_i$. Define the subset $A_R \subset \uprod A_i$ to
consist exactly of those elements $[a_i] \in \uprod A_i$ such that
$\text{rank}(a_i)$ (defined by each $\phi_i$) is bounded
(equivalently, constant); we call this subset the
\textbf{restricted ultraproduct} of the Hopf algebras $A_i$. We
show now that $A_R$ can be given the structure of a coalgebra.

Consider
\begin{diagram}
A_R & \subset & \uprod A_i & \rTo^{[\Delta_i]} & \uprod A_i
\otimes A_i \\
& \rdDashto(4,4)_\Delta & & & \uTo_{\Phi} \\
& &  & & \uprod A_i \otimes \uprod A_i \\
& & & & \uTo_{\subset} \\
& & & & A_R \otimes A_R
\end{diagram}
We want to show that there is a $\Delta$ making this diagram
commute, which is simply the assertion that the image of
$[\Delta_i]$, when restricted to $A_R$, is contained inside $A_R
\otimes A_R$.  For $[a_i] \in A_R$, say with rank $m$, write
\[ [\Delta_i(a_i)] = [\sum_{j=1}^m b_{ij} \otimes c_{ij} ] \]
where we can take $m$ to be constant over $i$ by proposition
\ref{RankAndTensorLengthAreNiceProposition}; this element is in
the image of $\Phi$.  Pass it down through $\Phi^{-1}$ to
\[ \sum_{j=1}^m [b_{ij}] \otimes [c_{ij}] \]
Again by proposition \ref{RankAndTensorLengthAreNiceProposition}
we can take all of the $b_{ij}$ and $c_{ij}$ to have rank $\leq
m$, showing that this expression is in fact contained in $A_R
\otimes A_R$.  Thus $\Delta$ does indeed exist.

We define a co-unit map $\varepsilon$ from $A_R$ to $\uprod k_i$
in the obvious manner, as the composition
\[ \varepsilon: A_R \subset \uprod A_i
\stackrel{[\varepsilon_i]}{\vlongrightarrow} \uprod k_i \] We
forego the proof that $(A_R,\Delta,\varepsilon)$ satisfy the
relevant diagrams making it a coalgebra.  There are two diagrams
to check, namely diagrams \ref{hopf1} and \ref{hopf2}; but these
follow from the almost everywhere commutativity of these diagrams
with respect to each $(A_i,\Delta_i,\varepsilon_i)$, by
application of the `if' direction of proposition \ref{homprop}
(which holds even in the non-boundedly finite dimensional case),
and the naturality of $\Phi$.

 We have not yet proved that $A_R$ is a Hopf algebra, that it is
closed under multiplication and can be given an antipode map. This
will follow later as we show that $A_R$ is coalgebra isomorphic to
a Hopf algebra, namely $A_\infty$, the representing Hopf algebra
of $\resprod \text{Comod}_{A_i}$.

\section{The Map from $A_\infty$ to $A_R$}

Here we define the map of coalgebras from our representing Hopf
algebra for $\resprod \text{Comod}_{A_i}$, called $A_{\infty}$, to
the coalgebra $A_R$ defined in the previous section, which we will
show in the next section is an isomorphism.

Let us begin by fixing some notation.  $I$ is an indexing set,
$k_i$ is a collection of fields indexed by $I$, $k$ is the
ultraproduct of those fields, $(A_i,\Delta_i,\varepsilon_i)$ is a
collection of Hopf algebras over those fields, $\catfont{C}_i$ is
the category $\text{Comod}_{A_i}$, $\omega_i$ is the fibre
(i.e.~forgetful) functor on each $\catfont{C}_i$, $\catfont{C}$ is
the restricted ultraproduct of the categories $\catfont{C}_i$, and
$\omega$ is the fibre functor (as defined in theorem
\ref{omegaisafiberfunctorthm}) on $\catfont{C}$. For an object $X$
of $\catfont{C}_i$, $L_{X}$ is the image of $\text{End}(\omega_i
\res \gen{X})$ inside $\text{End}(\omega_i(X))$, and similarly for
an object $[X_i]$ of $\catfont{C}$, $L_{[X_i]}$ is the image of
$\text{End}(\omega \res \gen{[X_i]})$ inside
$\text{End}(\omega([X_i]))$.

Each $A_i$ can be recovered as
\[ A_i = \dirlim_{X_i \in \catfont{C}_i} L_{X_i}^\circ \]
and we write the corresponding direct limit diagram as
\begin{equation}
\label{repHopfAlgdiagram1}
\begin{diagram}
& & \dirlim_{\catfont{C}_i} L_{X_i}^\circ & & \\
& \ruTo^{\phi_{Y_i}} & & \luTo^{\phi_{Z_i}} & \\
L_{Y_i}^\circ & & \rTo_{\phi_{Y_i,Z_i}} & & L_{Z_i}^\circ \\
\end{diagram}
\end{equation}
Note that we have used the same symbol $\phi_{\slot}$ for the
several such existing in each category; no confusion should
result.

We also have, in each category, the inverse limit diagram
\begin{diagram}
& & \invlim_{\catfont{C}_i} L_{X_i} & & \\
& \ldTo^{T_{Y_i}} & & \rdTo^{T_{Z_i}} \\
L_{Y_i} & & \lTo_{T_{Y_i,Z_i}} & & L_{Z_i} \\
\end{diagram}
and again we have used $T$ to stand for the transition maps in all
of the categories.  We also have the unique map $\phi_i$ making
\begin{diagram}
 & & (\invlim_{\catfont{C}_i} L_{X_i})^\circ & & \\
 & \ruTo(2,4)^{T_{Y_i}^\circ} & \uDashto^{\phi_i} & \luTo(2,4)^{T_{Z_i}^\circ} & \\
 &  & \dirlim_{\catfont{C}_i} L_{X_i}^\circ & & \\
 & \ruTo_{\phi_{Y_i}} & & \luTo_{\phi_{Z_i}} & \\
 L_{Y_i}^\circ & & \rTo_{\phi_{Y_i,Z_i}}  & & L_{Z_i}^\circ \\
\end{diagram}
commute, as defined by diagram
\ref{phiDefnDiagram_ResUprodOfHopfAlgebras}. As $A_{\infty}$ is by
definition the representing Hopf algebra of $\catfont{C}$, it can
be recovered as a direct limit according to the diagram
\begin{equation}
\label{DirectLimitDiagram_TheMapFromA_Rto}
\begin{diagram}
& & \dirlim_{\catfont{C}} L_{[X_i]}^\circ & & \\
& \ruTo^{\phi_{[Y_i]}} & & \luTo^{\phi_{[Z_i]}} \\
L_{[Y_i]}^\circ & & \rTo_{\phi_{[Y_i],[Z_i]}} & & L_{[Z_i]}^\circ
\\
\end{diagram}
\end{equation}
where the direct system is, as usual, the objects of the category
$\catfont{C}$, with $[X_i] \leq [Y_i]$ meaning that $[X_i] \in
\gen{[Y_i]}$, and the $\phi_{[Y_i],[Z_i]}$ being dual to the
transition mappings $T_{[Y_i],[Z_i]}$.

\begin{prop}
\label{everythingisilluminatedprop}

Let $[X_i],[Y_i]$ be objects of $\catfont{C}$ with $[Y_i] \in
\gen{[X_i]}$.

\begin{enumerate}
\item{ $L_{[X_i]} = \uprod L_{X_i}$ under the isomorphism
$\text{End}_k(\omega([X_i])) \isomorphic \uprod
\text{End}_{k_i}(\omega_i(X_i))$} \item{$L_{[X_i]}^\circ = \uprod
L_{X_i}^\circ$} \item{The \index{transition mapping} transition
mapping $T_{[Y_i],[X_i]}:L_{[X_i]} \rightarrow L_{[Y_i]}$ can be
identified with the ultraproduct of the transition mappings,
$[T_{X_i,Y_i}]:\uprod L_{X_i} \rightarrow \uprod L_{Y_i}$}
\item{The natural $L_{[X_i]}$-module structure on $\omega([X_i])$
can be identified with the ultraproduct of the $L_{X_i}$ module
structures on $\omega_i(X_i)$} \item{The natural
$L_{[X_i]}^\circ$-comodule structure on $\omega([X_i])$ can be
identified with the ultraproduct of the $L_{X_i}^\circ$-comodule
structures on $\omega_i(X_i)$}

\end{enumerate}
\end{prop}

\begin{proof}
Let $[X_i]$ have dimension $n$. To prove the first claim, we work
through the characterization of $L_{[X_i]}$ given by theorem
\ref{L_Xtheorem}. We start by fixing an isomorphism $\alpha:k^n
\rightarrow \omega([X_i])^*$.  As $\omega([X_i])^* = \left(\uprod
\omega_i(X_i)\right)^*$ can be identified with $\uprod
\omega_i(X_i)^*$, $\alpha$ can be uniquely written as
$[\alpha_i:k_i^n \rightarrow \omega_i(X_i)^*]$.  For our $\psi$ we
may as well choose the identity map $[X_i]^n \rightarrow [X_i]^n$,
since any subobject factors through it; then $\psi$ is the
ultraproduct of the identity maps $[\psi^i = 1_i]:[X_i^n]
\rightarrow [X_i^n]$. It is also clear that the map
$\psi_{\alpha}:\omega([X_i]^n) \rightarrow \omega([X_i])^* \otimes
\omega([X_i])$ can be identified with the ultraproduct of the maps
$\psi^i_{\alpha_i}:\omega_i(X_i)^n \rightarrow \omega_i(X_i)^*
\otimes \omega_i(X_i)$, if we allow a factorization through the
isomorphism $\uprod \omega_i(X_i)^* \otimes \omega_i(X_i)
\isomorphic (\uprod \omega_i(X_i))^* \otimes \uprod
\omega_i(X_i)$.

Next we are asked to find $P_{[X_i]}^\alpha$, the smallest
subobject of $[X_i]^n$ such that
$\psi_\alpha(\omega(P_{[X_i]}^\alpha))$ contains
$\text{id}:\omega([X_i]) \rightarrow \omega([X_i])$.
$P_{[X_i]}^\alpha$ is an object of $\resprod \catfont{C}_i$, and
can be written as $P_{[X_i]}^\alpha = [Y_i]$ for some collection
$Y_i$ of objects of $\catfont{C}_i$.  We claim $Y_i =
P_{X_i}^{\alpha_i}$ for almost every $i$. We can identify the
element $\text{id} \in \omega([X_i])^* \otimes \omega([X_i])$ with
the element $[\text{id}_i] \in \uprod \omega_i(X_i)^* \otimes
\omega_i(X_i)$, and the concepts of ``smallest'' and ``subobject
of'' are both first-order.  Thus the following statements are
equivalent:

\begin{enumerate}
\item{$[Y_i]$ is the smallest subobject of $[X_i]^n$ such that
$\psi_\alpha(\omega([Y_i]))$ contains $\text{id}:\omega([X_i])
\rightarrow \omega([X_i])$} \item{For almost every $i$, $Y_i$ is
the smallest subobject of $X_i^n$ such that
$\psi^i_{\alpha_i}(\omega_i(Y_i))$ contains $\text{id}_i :
\omega_i(X_i) \rightarrow \omega_i(X_i)$ }

\end{enumerate}
Then we must have $P_{[X_i]}^\alpha = [P_{X_i}^{\alpha_i}]$,
whence
\[ L_{[X_i]} = \omega(P_{[X_i]}^\alpha) = \uprod
\omega_i(P_{X_i}^{\alpha_i}) = \uprod L_{X_i} \]
and claim 1.~is
proved.

Claim 2.~is immediate, as the taking of duals is known to
distribute over ultraproducts for boundedly finite dimensional
collections of algebras (proposition \ref{uprodoverdualsprop}).

For claim 3.~we note that since $[X_i] \in \gen{[Y_i]}$, $X_i \in
\gen{Y_i}$ for almost every $i$ (lemma \ref{converseUniGenLem}),
and so $T_{X_i,Y_i}$ is defined for almost every $i$.  To prove
the claim we look to the definition of the transition mapping. Let
$[Y_i] \in \gen{[X_i]}$, and suppose for example that $[Y_i]$ is a
subobject of $[X_i]$, under the map $[\iota_i]$. By 1.~above,
every member of $L_{[X_i]}$ is of the form $[\phi_i]$, where
$\phi_i \in L_{X_i}$ for almost every $i$. Then the image of
$[\phi_i]$ under the transition mapping $T_{[Y_i],[X_i]}$ is the
unique map $[\sigma_i]$ that makes
\begin{diagram}
\uprod \omega_i(Y_i) & \rTo^{[\iota_i]} & \uprod \omega_i(X_i)
\\
\dTo^{[\sigma_i]} & & \dTo_{[\phi_i]} \\
\uprod \omega_i(Y_i) & \rTo_{[\iota_i]} & \uprod \omega_i(X_i)
\\
\end{diagram}
commute, which is equivalent to the almost everywhere
commutativity of
\begin{diagram}
\omega_i(Y_i) & \rTo^{\iota_i} & \omega_i(X_i) \\
\dTo^{\sigma_i} & & \dTo_{\phi_i} \\
\omega_i(Y_i) & \rTo_{\iota_i} & \omega_i(X_i) \\
\end{diagram}
which is equivalent to $\sigma_i = T_{Y_i,X_i}(\phi_i)$ for almost
every $i$.  Thus $T_{[Y_i],[X_i]}([\phi_i]) =
[T_{Y_i,X_i}(\phi_i)]$, and claim 3.~is proved.

Claim 4.~is merely the statement that, for $[\phi_i] \in
L_{[X_i]}$ and $[x_i] \in \omega([X_i])$, $[\phi_i]([x_i]) =
[\phi_i(x_i)]$, which is true by definition.  Claim 5.~is
similarly proved.

\end{proof}

Part 2.~of the above proposition tells us that, instead of the
direct limit diagram \ref{DirectLimitDiagram_TheMapFromA_Rto}, we
can write instead
\begin{diagram}
& & \dirlim_{\catfont{C}} \uprod L_{X_i}^\circ & & \\
& \ruTo^{\phi_{[Y_i]}} & & \luTo^{\phi_{[Z_i]}} \\
\uprod L_{Y_i}^\circ && \rTo_{[\phi_{Y_i,Z_i}]} & & \uprod
L_{Z_i}^\circ \\
\end{diagram}
with the understanding that $\phi_{[Z_i]}$ is factoring through
the isomorphism $L_{[Z_i]}^\circ \isomorphic \uprod
L_{Z_i}^\circ$.

\begin{lem}
\label{converseUniGenLem} \label{mylemma1} If $[X_i],[Y_i]$ are
objects of $\catfont{C}$, and if $[X_i] \in \gen{[Y_i]}$, then
$X_i \in \gen{Y_i}$ for almost every $i$.
\end{lem}

\begin{proof}
All of the concepts ``is a subobject of'', ``is a quotient of'',
and (for \emph{fixed} $n$) ``is isomorphic to an $n$-fold direct
sum of'' are first-order statements in the language of abelian
tensor categories. To say that $[X_i] \in \gen{[Y_i]}$ means that,
for some fixed $n$, $[X_i]$ is a subobject of a quotient of
$[Y_i]^n$. Apply theorem \ref{Los'sTheorem} to see that the same
must be true for almost every $i$.
\end{proof}

Now, let us take diagram $\ref{repHopfAlgdiagram1}$ and apply
ultraproducts:
\begin{diagram}
& & \uprod \dirlim_{\catfont{C}_i} L_{X_i}^\circ & & \\
& \ruTo^{[\phi_{Y_i}]} & & \luTo^{[\phi_{Z_i}]} \\
\uprod L_{Y_i}^\circ & & \rTo_{[\phi_{Y_i,Z_i}]} & & \uprod
L_{Z_i}^\circ \\
\end{diagram}
As it stands this diagram is a bit nonsensical: $\uprod
\dirlim_{\catfont{C}_i} L_{X_i}^\circ$ is little more than a set,
being the ultraproduct of a collection of Hopf algebras, and
lacking any kind of coalgebra structure.  We claim however that
each of the maps $[\phi_{Y_i}]$ have their image inside $A_R
\subset \uprod A_i = \uprod \dirlim_{\catfont{C}_i}
L_{X_i}^\circ$, the restricted ultraproduct of the Hopf algebras
$A_i$.

\begin{prop}
The image of each $[\phi_{Y_i}]$, for $[Y_i] \in \catfont{C}$, is
contained inside $A_R$.
\end{prop}

\begin{proof}
Consider, for fixed $i$, the diagram
\begin{diagram}
 & & (\invlim_{\catfont{C}_i} L_{X_i})^\circ & & \\
 & \ruTo(2,4)^{T_{Y_i}^\circ} & \uDashto^{\phi_i} & \luTo(2,4)^{T_{Z_i}^\circ} & \\
 &  & \dirlim_{\catfont{C}_i} L_{X_i}^\circ & & \\
 & \ruTo_{\phi_{Y_i}} & & \luTo_{\phi_{Z_i}} & \\
 L_{Y_i}^\circ & & \rTo_{\phi_{Y_i,Z_i}}  & & L_{Z_i}^\circ \\
\end{diagram}
We claim that if $Y_i$ has dimension $n$, then for any $\alpha \in
L_{Y_i}^\circ$, $\phi_{Y_i}(\alpha)$ has rank no larger than
$n^2$. Commutativity of the above gives
\[ \phi_i(\phi_{Y_i}(\alpha)) = T_{Y_i}^\circ(\alpha) \]
which is equal to the composition
\[ \invlim_{\catfont{C}_i} L_{X_i}
\stackrel{T_{Y_i}}{\vlongrightarrow} L_{Y_i}
\stackrel{\alpha}{\longrightarrow} k \]

Now $L_{Y_i}$, being a subalgebra of
$\text{End}_{\text{Vec}_{k_i}}(\omega_i(Y_i))$, certainly has
dimension no larger than $n^2$; further, the kernel of $T_{Y_i}
\circ \alpha$ contains the kernel of $T_{Y_i}$.  But $T_{Y_i}$ is
an algebra map, and so its kernel \emph{is} an ideal of $\invlim
L_{X_i}$, and has codimension no larger than $n^2$.  Thus
$\phi_i(\phi_{Y_i}(\alpha))$ has rank no larger than $n^2$, hence
by definition $\phi_{Y_i}(\alpha)$ has rank no larger than $n^2$.

Then if $[Y_i] \in \catfont{C}$, say of constant dimension $n$,
and if $[\alpha_i] \in \uprod L_{X_i}^\circ$, then
$[\phi_{Y_i}]([\alpha_i]) \stackrel{\text{def}}{=}
[\phi_{Y_i}(\alpha_i)]$ has bounded rank, each being no larger
than $n^2$; thus it is contained in $A_R$.
\end{proof}

We now have two diagrams
\[
\begin{diagram}
& & \dirlim_{\catfont{C}} \uprod L_{X_i}^\circ & & \\
& \ruTo^{\phi_{[Y_i]}} & & \luTo^{\phi_{[Z_i]}} \\
\uprod L_{Y_i}^\circ && \rTo_{[\phi_{Y_i,Z_i}]} & & \uprod
L_{Z_i}^\circ \\
\end{diagram}
\hspace{2cm}
\begin{diagram}
& & A_R & & \\
& \ruTo^{[\phi_{Y_i}]} & & \luTo^{[\phi_{Z_i}]} \\
\uprod L_{Y_i}^\circ & & \rTo_{[\phi_{Y_i,Z_i}]} & & \uprod
L_{Z_i}^\circ \\
\end{diagram}
\]
where, in the second diagram we have replaced $\uprod
\dirlim_{\catfont{C}_i} L_{X_i}^\circ$ with $A_R$, as we may by
the previous proposition; some routine arguing shows that since
$\phi_{Y_i}: L_{Y_i}^\circ \rightarrow A_i$ is a coalgebra map for
every $i$, then $[\phi_{Y_i}]:\uprod L_{Y_i}^\circ \rightarrow
A_R$ is also a coalgebra map. Some care must be taken here; on the
right, the map $[\phi_{Y_i,Z_i}]$ is defined whenever $Y_i \in
\gen{Z_i}$ for almost every $i$, while on the left, it is only
defined when $[Y_i] \in \gen{[Z_i]}$. Nonetheless, by lemma
\ref{mylemma1}, whenever it is defined on the left, it is defined
on the right. We can now appeal to the universal property of
direct limits to invoke the existence of a unique coalgebra map
$\Omega$ making the following diagram commute:
\begin{diagram}
 & & A_R & & \\
 & \ruTo(2,4)^{[\phi_{Y_i}]} & \uDashto^\Omega & \luTo(2,4)^{[\phi_{Z_i}]} & \\
 &  & \dirlim_{\catfont{C}} \uprod L_{X_i}^\circ & & \\
 & \ruTo_{\phi_{[Y_i]}} & & \luTo_{\phi_{[Z_i]}} & \\
 \uprod L_{Y_i}^\circ & & \rTo_{[\phi_{Y_i,Z_i}]}  & & \uprod L_{Z_i}^\circ \\
\end{diagram}
This $\Omega:A_\infty \rightarrow A_R$ is our claimed isomorphism
of coalgebras, later to be shown, of Hopf algebras.

\section{$\Omega$ is an Isomorphism}

This $\Omega$, while difficult to define, is not that difficult to
describe.  A typical element of $A_\infty = \dirlim \uprod
L_{X_i}^\circ$ looks like
\[ [[\alpha_i:L_{X_i} \rightarrow k_i]_\filtfont{U}]_\catfont{C} \]
That is, it is an equivalence class of equivalence classes of
linear functionals, with $\filtfont{U}$ denoting the equivalence
defined by the ultraproduct with respect to the ultrafilter
$\filtfont{U}$, and $\catfont{C}$ denoting the equivalence defined
by the direct limit over $\catfont{C}$. Each $\alpha_i$ is an
arbitrary linear functional, subject only to the restriction that
the objects $X_i$ have bounded dimension.

 A typical element of $A_R
\subset \uprod \dirlim L_{X_i}^\circ$ on the other hand looks like
\[ [[\alpha_i:L_{X_i} \rightarrow k_i]_{\catfont{C}_i}]_\filtfont{U} \]
where $\catfont{C}_i$ denotes the equivalence defined by the
direct limit over each $\catfont{C}_i$.  The $X_i$ here are not
assumed to have  bounded dimension; only that the functionals
$\alpha_i$ have bounded rank.

The action of $\Omega$ is simple then:
\begin{equation}
\label{actionOfOmegaEquation}
 \Omega: [[\alpha_i:L_{X_i}
\rightarrow k_i]_\filtfont{U}]_{\catfont{C}} \longmapsto
[[\alpha_i:L_{X_i} \rightarrow k_i]_{\catfont{C}_i}]_\filtfont{U}
\end{equation}
 To see it this way,
it is not at all obvious that it is well-defined, or that it is a
map of coalgebras, but we know it is, via the way we constructed
it.

\begin{lem}
\label{FromRankToDimensionLemma} Let $\catfont{C}$ be the category
$\text{Comod}_A$, and let $[\alpha:L_Y \rightarrow k]$ be an
element of $\dirlim_\catfont{C} L_X^\circ$ which has rank no
greater than $m$.  Then $[\alpha:L_Y \rightarrow k]$ can be
written as
\[ [\beta:L_C \rightarrow k] \]
where $C$ is some subcoalgebra of $A$ having dimension no greater
than $m$.
\end{lem}

\begin{proof}
$Y$ is in the principal subcategory generated by some subcoalgebra
$C$ of $A$, and so we have the map $L_Y^\circ
\stackrel{\phi_{Y,C}}{\vlongrightarrow} L_C^\circ$; this shows
that we may as well take $[\alpha:L_Y \rightarrow k]$ to be
$[\gamma:L_C \rightarrow k]$ for some $\gamma \in L_C^\circ$.  To
say that this element has rank no greater than $m$ is to say that
the composition
\[ \invlim_\catfont{C} L_X \stackrel{T_C}{\longrightarrow} L_C
\stackrel{\gamma}{\longrightarrow} k \] kills an ideal $I \ideal
\invlim L_X$ of codimension no greater than $m$.  Let $J$ be the
kernel of $T_C$. We can assume that $I$ contains $J$; if not,
enlarge $I$ to $I+J$, which is still an ideal contained in
$\text{ker}(T_C \circ \gamma)$ having codimension no larger than
$m$.  Then we have a commutative diagram
\begin{diagram}
\invlim_\catfont{C} L_X & \rTo^{T_C} & L_C  \\
& \rdTo_\pi & \dTo_{\pi^\prime}  \\
& & (\invlim_\catfont{C} L_X)/I  \\
\end{diagram}
where $\pi$ is the natural projection, and $\pi^\prime$ is the
unique surjective map gotten because $J \subset I$. As $(\invlim
L_X)/I$ is a quotient algebra of $L_C$, theorem
\ref{subcoalgebratheorem} guarantees that it is isomorphic to
$L_D$ for some subcoalgebra $D$ of $C$, and that under this
identification we can take $\pi^\prime$ to be the transition map
$T_{D,C}$.  Thus we have the commutative diagram
\begin{diagram}
\invlim_\catfont{C} L_X & \rTo^{T_C} & L_C \\
& \rdTo_{T_D} & \dTo_{T_{D,C}}\\
& & L_D \\
\end{diagram}
with $L_D$ of dimension no greater than $m$. And since
$\text{ker}(\gamma) \supset T_C(I) = \text{ker}(T_{D,C})$, there
exists a linear functional $\beta$ making
\begin{diagram}
L_C & \rTo^\gamma & k \\
\dTo^{T_{D,C}} & \ruTo_\beta \\
L_D
\end{diagram}
commute.  By definition then, $\gamma$ and $\beta$ are equal in
the direct limit.  We can thus write
\[ [\gamma:L_C \rightarrow k] = [\beta:L_D: \rightarrow k] \]
Finally, since $L_D$ has dimension no greater than $m$ and $D$ is
a coalgebra, theorem \ref{adualitytheorem} tells us that $D
\isomorphic L_D^\circ$ has dimension no greater than $m$.  This
completes the proof.
\end{proof}

\begin{prop}
\label{OmegaIsSurjectiveProp} The map $\Omega$ is surjective.
\end{prop}

\begin{proof}
Again, a typical element of $A_R$ looks like
\[ [[\alpha_i:L_{X_i} \rightarrow k_i]_{\catfont{C}_i}]_\filtfont{U} \]
with the $\alpha_i$ having constant rank, say $m$.  Then lemma
\ref{FromRankToDimensionLemma} shows we can write this instead as
\[ [[\beta_i:L_{D_i} \rightarrow
k_i]_{\catfont{C}_i}]_\filtfont{U}\] where each $D_i$ is a
subcoalgebra of $A_i$ having dimension no larger than $m$.  Then
the formula given for $\Omega$  at the beginning of this section
(equation \ref{actionOfOmegaEquation}) shows that
\[ [[\beta_i:L_{D_i} \rightarrow
k_i]_\filtfont{U}]_\catfont{C} \] qualifies as a pre-image for our
typical element under $\Omega$.
\end{proof}

\begin{lem}
\label{smallDimensionEmbeddingLemma} Let $G$ be an affine group
scheme represented by the Hopf algebra $A$ over a field $k$.  Let
$(V,\rho)$ be an $n$-dimensional $A$-comodule, fix a basis $e_1,
\ldots, e_n$ for $V$, and let $(a_{ij})$ be the matrix formula of
the representation of $G$ it defines in that basis. Then $C=
\text{span}_k(a_{ij}:1 \leq i,j \leq n)$ is a no more than
$n^2$-dimensional subcoalgebra of $A$. Further, $(V,\rho)$ can be
embedded, as a comodule, into $C^n$.
\end{lem}

\begin{proof}
Apply the comodule identity $\Delta(a_{ij}) = \sum_k a_{ik}
\otimes a_{kj}$ (equation \ref{comoduleequation2}) to see that
$\Delta(C) \subset C \otimes C$, whence $C$ is a subcoalgebra of
$A$.  For the embedding claim, we examine the embedding $V
\rightarrow A^n$ ($n$-fold direct sum of the regular
representation) defined in section 3.5 of \cite{waterhouse}; we
claim that the image of this embedding is in fact contained in
$C^n \subset A^n$.  Let $\Psi:V \otimes A \rightarrow A^n$ denote
the vector space isomorphism $e_i \otimes a \mapsto (0, \ldots,
a,\ldots,0)$ ($a$ in the $i^{\text{th}}$ slot, zeroes elsewhere).
Consider
\begin{diagram}
V & \rTo^\rho & V \otimes A & \rTo^\Psi & A^n \\
 & & & & \dTo_{\Delta^n} \\
 \dTo^\rho & & \dTo^{1 \otimes \Delta} & & (A \otimes A)^n \\
 & & & & \dTo_\isomorphic \\
 V \otimes A & \rTo_{\rho \otimes 1} & V \otimes A \otimes A &
 \rTo_{\Psi \otimes 1} & A^n \otimes A \\
\end{diagram}
Commutativity of the left rectangle is a comodule identity (see
diagram \ref{comod1}), and commutativity of the right rectangle is
obvious, whence this entire diagram commutes.  Note that the
composition that starts at the top right and goes directly down is
by definition the comodule structure on $A^n$ (see definition
\ref{comoduleConstructionsDefn}).  Looking at a condensed version
of the outermost rectangle
\begin{diagram}
V & \rTo^{\rho \circ \Psi}  & A^n \\
 & & \dTo_{\Delta^n} \\
 \dTo^\rho & & (A \otimes A)^n \\
 & & \dTo_\isomorphic \\
 V \otimes A & \rTo_{(\rho \circ \Psi) \otimes 1} & A^n \otimes A
 \\
\end{diagram}
we see that $\rho \circ \Psi$ is an embedding of $V$ into $A^n$.
And if we chase the basis element $e_j$ from $V$ to $A^n$ we
arrive at
\[ e_j \stackrel{\rho}{\longmapsto} \sum_i e_i \otimes a_{ij}
\stackrel{\Psi}{\longmapsto} (a_{1j}, a_{2j}, \ldots, a_{nj}) \]
which is an element of $C^n$.
\end{proof}

\begin{prop}
\label{OmegaIsInjectiveProp} $\Omega$ is injective.
\end{prop}

\begin{proof}

 Let $[[\alpha_i:L_{X_i} \rightarrow k_i]_\filtfont{U}]_\catfont{C}$
and $[[\beta_i:L_{Y_i} \rightarrow k_i]_\filtfont{U}]_\catfont{C}$
be two typical elements of $A_\infty$ such that $\Omega$ maps them
to the same thing.  This means that
\[ [[\alpha_i:L_{X_i} \rightarrow k_i]_{\catfont{C}_i}
]_\filtfont{U} = [[\beta_i:L_{Y_i} \rightarrow
k_i]_{\catfont{C}_i} ]_\filtfont{U} \] which is to say that, for
almost every $i$,
\[ [\alpha_i:L_{X_i} \rightarrow k_i]_{\catfont{C}_i} = [\beta_i:L_{Y_i} \rightarrow
k_i]_{\catfont{C}_i} \] which is to say that, for almost every
$i$, there is a $Z_i$ such that $X_i,Y_i \in \gen{Z_i}$ and
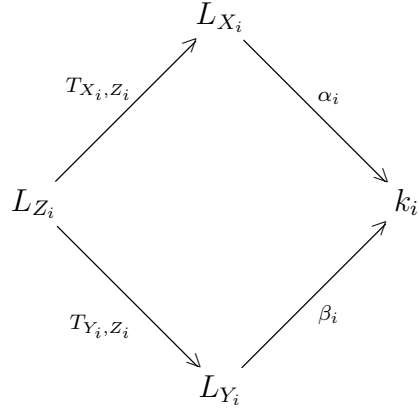
\begin{diagram}
& & L_{X_i} & & \\
& \ruTo^{T_{X_i,Z_i}} & & \rdTo^{\alpha_i} & \\
L_{Z_i} & & & & k_i \\
& \rdTo_{T_{Y_i,Z_i}} & & \ruTo_{\beta_i}& \\
& & L_{Y_i} & & \\
\end{diagram}
commutes. Now the $Z_i$ are, as far as we know, not of bounded
dimension, so we have some work to do.  By lemma
\ref{smallDimensionEmbeddingLemma}, for each $i$ let $C_i$ be a
subcoalgebra of $A_i$ such that $C_i$ has dimension no larger than
$\text{dim}(X_i \oplus Y_i)^2$, and such that $X_i \oplus Y_i$ is
embeddable in $C_i^{\text{dim}(X_i \oplus Y_i)}$ . Note in
particular that this implies that both $[X_i]$ and $[Y_i]$ belong
to the principal subcategory generated by $[C_i]$ (since they are
both subobjects of a subobject of $[C_i^m] = [C_i]^m$ for a fixed
$m$).

For each $i$ let $D_i$ be a subcoalgebra generating all of the
$X_i$, $Y_i$, $Z_i$ and containing $C_i$, which of course we
cannot assume is of bounded dimension. Then for every $i$ we have
a commutative diagram
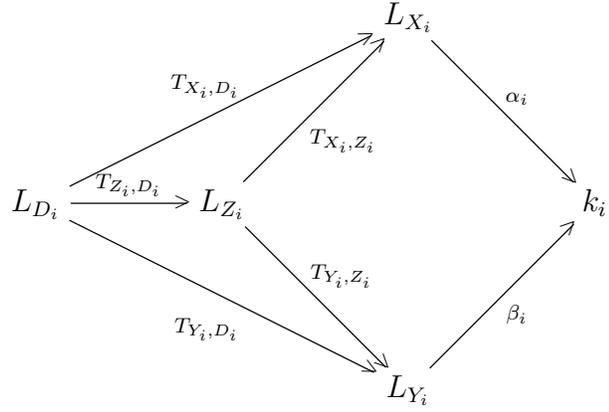
\begin{diagram}
& & & & L_{X_i} & & \\
& &  & \ruTo(4,2)^{T_{X_i,D_i}}  \ruTo_{T_{X_i,Z_i}} & & \rdTo^{\alpha_i} & \\
L_{D_i} & \rTo^{T_{Z_i,D_i}} & L_{Z_i} & & & & k_i \\
& \rdTo(4,2)_{T_{Y_i,D_i}}  & & \rdTo^{T_{Y_i,Z_i}} & & \ruTo_{\beta_i}& \\
& & & & L_{Y_i} & & \\
\end{diagram}
and in particular, the outermost diamond commutes:
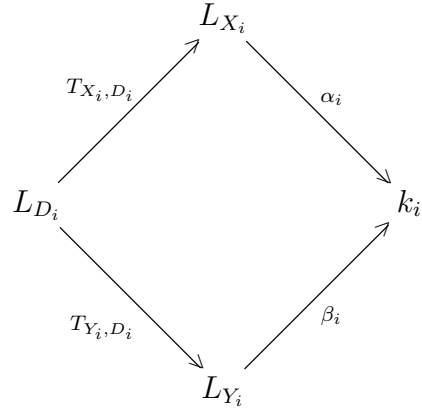
\begin{diagram}
& & L_{X_i} & & \\
& \ruTo^{T_{X_i,D_i}} & & \rdTo^{\alpha_i} & \\
L_{D_i} & & & & k_i \\
& \rdTo_{T_{Y_i,D_i}} & & \ruTo_{\beta_i}& \\
& & L_{Y_i} & & \\
\end{diagram}
and hence so does
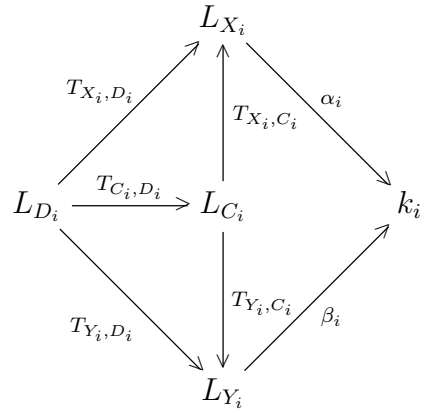
\begin{diagram}
& & L_{X_i} & & \\
& \ruTo^{T_{X_i,D_i}} & \uTo_{T_{X_i,C_i}}& \rdTo^{\alpha_i} & \\
L_{D_i} & \rTo^{T_{C_i,D_i}} & L_{C_i} & & k_i \\
& \rdTo_{T_{Y_i,D_i}} & \dTo_{T_{Y_i,C_i}} & \ruTo_{\beta_i}& \\
& & L_{Y_i} & & \\
\end{diagram}
As $C_i \subset D_i$ are subcoalgebras, proposition
\ref{surjectivetransitionmapprop} tells us that $T_{C_i,D_i}$ is
surjective.  Then commutativity of the above gives $T_{C_i,D_i}
\circ T_{X_i,C_i} \circ \alpha_i = T_{C_i,D_i} \circ T_{Y_i,C_i}
\circ \beta_i$, and since $T_{C_i,D_i}$ is surjective, this gives
us commutativity of
\begin{diagram}
& & L_{X_i} & & \\
& \ruTo^{T_{X_i,C_i}} & & \rdTo^{\alpha_i} & \\
L_{C_i} & & & & k_i \\
& \rdTo_{T_{Y_i,C_i}} & & \ruTo_{\beta_i}& \\
& & L_{Y_i} & & \\
\end{diagram}

Now apply ultraproducts to yield a commutative diagram
\begin{diagram}
& & \uprod L_{X_i} & & \\
& \ruTo^{[T_{X_i,C_i}]} & & \rdTo^{[\alpha_i]} & \\
\uprod L_{C_i} & & & & \uprod k_i \\
& \rdTo_{[T_{Y_i,C_i}]} & & \ruTo_{[\beta_i]}& \\
& & \uprod L_{Y_i} & & \\
\end{diagram}
Note that $[C_i]$, being of bounded dimension, is an object of
$\catfont{C}$.  Then if we identify $\uprod L_{C_i}$ with
$L_{[C_i]}$, $[T_{X_i,C_i}]$ with $T_{[X_i],[C_i]}$, etc.~(as we
may by proposition \ref{everythingisilluminatedprop}),
commutativity of the above implies the equality of $[\alpha_i]$
and $[\beta_i]$ in the direct limit over $\catfont{C}$; that is
\[ [[\alpha_i:L_{X_i} \rightarrow k_i]_\filtfont{U}]_\catfont{C} =
[[\beta_i:L_{Y_i} \rightarrow k_i]_\filtfont{U}]_\catfont{C} \] as
desired.
\end{proof}

\begin{thm}
The representing Hopf algebra of the restricted ultraproduct
$\resprod \text{Comod}_{A_i}$ is coalgebra-isomorphic to the
\index{Hopf algebra!restricted ultraproduct of} \index{restricted
ultraproduct!of Hopf algebras} restricted ultraproduct $A_R$ of
the Hopf algebras $A_i$.
\end{thm}

\begin{proof}
Apply propositions \ref{OmegaIsSurjectiveProp} and
\ref{OmegaIsInjectiveProp}.
\end{proof}

\section{The Equivalence $\prod_{\scriptscriptstyle R} \catfont{C}_i \isomorphic \text{Comod}_{A_R}$}

For a collection of Hopf algebras $A_i$, the previous section
shows that $A_R$ is coalgebra isomorphic to $A_\infty$, the
representing Hopf algebra of $\resprod \text{Comod}_{A_i}$. Then
as $A_\infty$ is a Hopf algebra, so is $A_R$, under whatever
multiplication and antipode map are induced on it by $\Omega$.  We
would like to of course prove that this induced multiplication and
antipode are exactly those inherited by being a subset of $\uprod
A_i$; i.e.~that they are the ultraproduct of the individual
multiplications and antipodes on the $A_i$ restricted to $A_R$.

We will prove this for multiplication; we do not prove it for
antipode, but believe a similar proof to the one we give for
multiplication (using instead the dual construction instead of the
tensor product) could be constructed.

Our first step is to build the equivalence from the category
$\resprod \catfont{C}_i$ to $\text{Comod}_{A_R}$ induced by the
isomorphism $\Omega$; here our work will finally start to pay off,
as this equivalence is quite natural and easy to describe.
Examination of this equivalence will further yield the required
multiplication on $A_R$, as we examine the tensor product on
$\text{Comod}_{A_R}$ induced by this equivalence.

First, following the construction mentioned in theorem
\ref{thegeneralequivalencethm}, we build the equivalence $G:
\resprod \text{Comod}_{A_i} \rightarrow \text{Comod}_{A_\infty}$.
To keep notation simple we use the same symbol $X$ for an object
of $\text{Comod}_A$ and its image under the fibre functor, and
similarly for a morphism.

Let $[X_i,\rho_i:X_i \rightarrow X_i \otimes A_i]$ be an object of
$\resprod \catfont{C}_i$.  The remarks before theorem
\ref{thegeneralequivalencethm} tell us we should define the
$A_\infty$-comodule structure on $\uprod X_i$ to be the
composition
\[ \uprod X_i \stackrel{\rho^\prime}{\longrightarrow} \uprod X_i \otimes
L_{[X_i]}^\circ \stackrel{1 \otimes
\phi_{[X_i]}}{\vvlongrightarrow} (\uprod X_i) \otimes A_\infty \]
where $\rho^\prime$ is the natural $L_{[X_i]}^\circ$-comodule
structure on $\uprod X_i$.  But proposition
\ref{everythingisilluminatedprop} says that we can replace
$L_{[X_i]}^\circ$ with $\uprod L_{X_i}^\circ$, and in so doing can
define $\rho^\prime$ in terms of the individual
$L_{X_i}^\circ$-comodule structures on each $X_i$, whom we call
$\rho_i^\prime$; that is
\begin{diagram}
\uprod X_i & & \rTo^{\rho^\prime} & & \uprod X_i \otimes \uprod
L_{X_i}^\circ & \rTo^{1 \otimes \phi_{[X_i]}} & (\uprod X_i)
\otimes A_\infty \\
& \rdTo_{[\rho_i^\prime]} & & \ruTo_{\Phi^{-1}} \\
 &  & \uprod X_i \otimes L_{X_i}^\circ \\
\end{diagram}
commutes.  The $A_\infty$-comodule structure on $\uprod X_i$, call
it $\rho$, is thus the composition
\[ \uprod X_i \stackrel{[\rho_i^\prime]}{\longrightarrow} \uprod
X_i \otimes L_{X_i}^\circ \stackrel{\Phi^{-1}}{\longrightarrow}
\uprod X_i \otimes \uprod L_{X_i}^\circ \stackrel{1 \otimes
\phi_{[X_i]}}{\vvlongrightarrow} (\uprod X_i) \otimes A_\infty \]
$G([X_i])$ is thus $\uprod X_i$, with the above $A_\infty$
comodule structure.  For a morphism $[\psi_i:X_i \rightarrow Y_i]$
in $\resprod \catfont{C}_i$, we define of course $G([\psi_i])$ to
be $[\psi_i]:\uprod X_i \rightarrow \uprod Y_i$.

The next step is to pass to the isomorphism $\Omega$ to obtain an
equivalence of categories $\resprod \catfont{C}_i \rightarrow
\text{Comod}_{A_R}$.

\begin{thm}
Define a functor $F:\resprod \catfont{C}_i \rightarrow
\text{Comod}_{A_R}$ as follows. $F$ sends the object
$[X_i,\rho_i:X_i \rightarrow X_i \otimes A_i]$ to the vector space
$\uprod X_i$ with the $A_R$-comodule structure
\[ \uprod X_i \stackrel{[\rho_i]}{\vlongrightarrow} \uprod X_i
\otimes A_i \stackrel{\Phi^{-1}}{\vlongrightarrow} \uprod X_i
\otimes \uprod A_i \supset (\uprod X_i) \otimes A_R \]
and sends
the morphism $[\psi:X_i \rightarrow Y_i]$ to $[\psi_i]:\uprod X_i
\rightarrow \uprod Y_i$.  Then $F$ is the equivalence of
categories induced on $G$ by $\Omega$.
\end{thm}

\begin{proof}
Consider the diagram
\begin{diagram}
\uprod X_i & \rTo^{[\rho_i^\prime]} & \uprod X_i \otimes
L_{X_i}^\circ & \rTo^{\Phi^{-1}} & \uprod X_i \otimes \uprod
L_{X_i}^\circ & \rTo^{1 \otimes \phi_{[X_i]}} & (\uprod X_i)
\otimes A_\infty \\
 & \rdTo_{[\rho_i]} & \dTo_{[1 \otimes
\phi_{X_i}]} & & \dTo_{1 \otimes [\phi_{X_i}]} & \rdTo^{1 \otimes
[\phi_{X_i}]} & \dTo_{1 \otimes \Omega} \\
& & \uprod X_i \otimes A_i & \rTo_{\Phi^{-1}} & \uprod X_i \otimes
\uprod A_i & \lInto & (\uprod X_i) \otimes A_R \\
\end{diagram}
The composition that starts at the top left, goes all the away
across, and then down, is the functor gotten from $G$ by $\Omega$;
that is, the top line is the $A_\infty$-comodule structure on
$[X_i]$ under $G$, and then we tack on $1 \otimes \Omega$ to
obtain an $A_R$-comodule structure.  The composition that starts
at the top left, goes diagonally down, and then all the way
across, is the composition referenced in the statement of the
theorem.  We want to see then that this diagram commutes.
Commutativity of the left-most triangle is equivalent to the
almost everywhere commutativity of it, which in turn is simply the
statement that the $A_i$-comodule structure on $X_i$ can be
factored through the $L_{X_i}^\circ$-comodule structure for it,
and through the canonical injection $\phi_{X_i}$.  The next square
follows automatically from the naturality of $\Phi$ (proposition
\ref{ultratensormapsprop}). Commutativity of the next triangle is
obvious, as $[\phi_{X_i}]$ is known to point to $A_R$, and the
last triangle follows from the definition of $\Omega$.  The
theorem is proved.
\end{proof}

Thus we have an equivalence of categories $F:\resprod
\catfont{C}_i \rightarrow \text{Comod}_{A_R}$ given by the
previous theorem.  As $\resprod \catfont{C}_i$ is a tensor
category under $\otimes$, it induces a similar structure on
$\text{Comod}_{A_R}$ through $F$, which we call
$\overline{\otimes}$, whose action is given as follows.  Any two
objects of $\text{Comod}_{A_R}$ pull back under $F$ to objects of
$\resprod \catfont{C}_i$ which look like
\[ [X_i,\rho_i:X_i \rightarrow X_i \otimes A_i]\]
\[ [Y_i, \mu_i:Y_i \rightarrow Y_i \otimes A_i] \]
Their tensor product in $\resprod \catfont{C}_i$ is defined as
\[ [X_i \otimes Y_i, X_i \otimes Y_i \stackrel{\rho_i \otimes
\mu_i}{\vlongrightarrow} X_i \otimes A_i \otimes Y_i \otimes A_i
\stackrel{1 \otimes T_i \otimes 1}{\vvlongrightarrow} X_i \otimes
Y_i \otimes A_i \otimes A_i \stackrel{1 \otimes 1 \otimes
\text{mult}_i}{\vvlongrightarrow} X_i \otimes Y_i \otimes A_i] \]
and we push this new object back through $F$ to yield a new object
in $\text{Comod}_{A_R}$, having underlying vector space $\uprod
X_i \otimes Y_i$ and comodule map given by the composition
\begin{gather*}
\uprod X_i \otimes Y_i \stackrel{[\rho_i \otimes
\mu_i]}{\vvlongrightarrow} \uprod X_i \otimes A_i \otimes Y_i
\otimes A_i \stackrel{[1 \otimes T_i \otimes
1]}{\vvlongrightarrow} \uprod X_i \otimes Y_i \otimes A_i \otimes
A_i \\
 \stackrel{[1 \otimes 1 \otimes
\text{mult}_i]}{\vvvlongrightarrow} \uprod X_i \otimes Y_i \otimes
A_i \stackrel{\Phi^{-1}}{\longrightarrow} (\uprod X_i \otimes Y_i)
\otimes \uprod A_i  \supset (\uprod X_i \otimes Y_i) \otimes A_R
\end{gather*}
Playing the same game we see that two morphisms in
$\text{Comod}_{A_R}$ pull back to morphisms $[\psi_i:X_i
\rightarrow V_i],[\xi_i:Y_i \rightarrow W_i]$ in $\resprod
\catfont{C}_i$, and upon taking their tensor product in $\resprod
\catfont{C}_i$ and pushing them back through $F$ we obtain the
image under $\overline{\otimes}$ of these morphisms:
\[ \uprod X_i \otimes Y_i \stackrel{[\psi_i \otimes
\xi_i]}{\vvlongrightarrow} \uprod V_i \otimes W_i \] Now let us
modify $\overline{\otimes}$ a bit; simply tack on $\Phi$ to both
ends of the above to yield
\begin{gather*}
 \uprod X_i \otimes \uprod Y_i
\stackrel{\Phi}{\longrightarrow} \\
\uprod X_i \otimes Y_i \stackrel{[\rho_i \otimes
\mu_i]}{\vvlongrightarrow} \uprod X_i \otimes A_i \otimes Y_i
\otimes A_i \stackrel{[1 \otimes T_i \otimes
1]}{\vvlongrightarrow} \uprod X_i \otimes Y_i \otimes A_i \otimes
A_i \\
 \stackrel{[1 \otimes 1 \otimes
\text{mult}_i]}{\vvvlongrightarrow} \uprod X_i \otimes Y_i \otimes
A_i \stackrel{\Phi^{-1}}{\longrightarrow} (\uprod X_i \otimes Y_i)
\otimes \uprod A_i  \supset (\uprod X_i \otimes Y_i) \otimes A_R
\\
 \stackrel{\Phi^{-1} \otimes 1}{\vlongrightarrow} \uprod X_i
\otimes \uprod Y_i \otimes A_R
\end{gather*}
 and instead of $[\psi_i \otimes \xi_i]$, write
\[ \uprod X_i \otimes \uprod Y_i \stackrel{[\psi_i] \otimes
[\xi_i]}{\vvlongrightarrow} \uprod V_i \otimes \uprod W_i \] The
naturality of the isomorphism $\Phi$ guarantees that this new
functor is naturally isomorphic to $\overline{\otimes}$; let us
relabel this new functor as $\overline{\otimes}$.

The next proposition simplifies the description of
$\overline{\otimes}$, one which doesn't require first pulling an
object back to $\resprod \catfont{C}_i$.

\begin{prop}
\label{rephopfalgprop1} If $(X,\rho)$, $(Y,\mu)$ are objects of
$\text{Comod}_{A_R}$, then $\overline{\otimes}$ sends this pair to
the vector space $X \otimes Y$, with comodule map given by the
composition
\begin{gather*}
 X \otimes Y \stackrel{\rho \otimes \mu}{\vlongrightarrow} X
\otimes A_R \otimes Y \otimes A_R \subset X \otimes \uprod A_i
\otimes Y \otimes \uprod A_i \\
 \stackrel{1 \otimes T \otimes 1}{\vvlongrightarrow} X \otimes Y
\otimes \uprod A_i \otimes \uprod A_i \stackrel{1 \otimes 1
\otimes \text{mult}}{\vvlongrightarrow} X \otimes Y \otimes \uprod
A_i \supset X \otimes Y \otimes A_R
\end{gather*}
 where $\text{mult}$ denotes
the natural coordinate wise multiplication on $\uprod A_i$.
\end{prop}

\begin{proof}
As $F:\resprod \catfont{C}_i \rightarrow \text{Comod}_{A_R}$ is an
equivalence, there is an object $[(X_i,\rho_i)]$ of $\resprod
\catfont{C}_i$ such that $X = \uprod X_i$ and $\rho$ is equal to
the composition
\[ \uprod X_i \stackrel{[\rho_i]}{\vlongrightarrow} \uprod X_i
\otimes A_i \stackrel{\Phi^{-1}}{\longrightarrow} \uprod X_i
\otimes \uprod A_i \supset \uprod X_i \otimes A_R \] and similarly
for $(Y,\mu)$.  Consider

\begin{diagram}
X \otimes Y & \rEq & \uprod X_i \otimes \uprod Y_i & \rEq & \uprod X_i \otimes \uprod Y_i \\
& & \dTo_{[\rho_i] \otimes [\mu_i]} & & \dIso_\Phi \\
& & \uprod X_i \otimes A_i \otimes \uprod Y_i \otimes A_i & & \uprod X_i \otimes Y_i \\
\dTo^{\rho \otimes \mu} & & \dTo_{\Phi^{-1} \otimes \Phi^{-1}} & & \dTo_{[\rho_i \otimes \mu_i]}\\
& & \uprod X_i \otimes \uprod A_i \otimes \uprod Y_i \otimes \uprod A_i & \rIso^{\Phi} & \uprod X_i \otimes A_i \otimes Y_i \otimes A_i \\
& & \uInto_{\subset \otimes \subset}& & \dIso_{[1 \otimes T_i \otimes 1]} \\
X \otimes A_R \otimes Y \otimes A_R & \rEq & \uprod X_i \otimes A_R \otimes \uprod Y_i \otimes A_R & & \uprod X_i \otimes Y_i \otimes A_i \otimes A_i \\
\dInto^{\subset} & & \dInto_{\subset}   & \ruIso(2,4)_\Phi  &  \dTo_{[1 \otimes 1 \otimes \text{mult}_i]} \\
X \otimes \uprod A_i \otimes Y \otimes \uprod A_i & \rEq & \uprod X_i \otimes \uprod A_i \otimes \uprod Y_i \otimes \uprod A_i & & \uprod X_i \otimes Y_i \otimes A_i \\
\dIso^{1 \otimes T \otimes 1} & & \dIso^{1 \otimes T \otimes 1} & & \dIso_{\Phi} \\
X \otimes Y \otimes \uprod A_i \otimes \uprod A_i & \rEq & \uprod X_i \otimes \uprod Y_i \otimes \uprod A_i \otimes \uprod A_i & & \uprod X_i \otimes Y_i \otimes \uprod A_i \\
\dTo^{1 \otimes 1 \otimes \text{mult}} & & \dTo_{1 \otimes 1 \otimes \text{mult}} & \ruTo_{\Phi \otimes 1} & \uInto_{\subset} \\
X \otimes Y \otimes \uprod A_i & \rEq & \uprod X_i \otimes \uprod Y_i \otimes \uprod A_i & & (\uprod X_i \otimes Y_i) \otimes A_R \\
\uInto^{\subset} & & \uInto_{\subset} & & \dIso_{\Phi \otimes 1} \\
X \otimes Y \otimes A_R & \rEq & \uprod X_i \otimes \uprod Y_i \otimes A_R & \rEq & \uprod X_i \otimes \uprod Y_i \otimes A_R \\
\end{diagram}
which the author would at this time like to nominate to the
Academy as the ugliest diagram of all time.  The top left
rectangle commutes by the previous remarks, and all of the other
simple sub-polygons, though numerous, are easy to check; thus this
entire diagram commutes. The composition that starts at the top
right and goes all the way down is the image of $(X,\rho)$ and
$(Y,\mu)$ under $\overline{\otimes}$ as described previously; the
composition that starts at the top left and goes all the way down
is the composition given in the statement of the proposition.
These are equal, and the proposition is proved.
\end{proof}

So then, we have a bifunctor $\overline{\otimes}$ on
$\text{Comod}_{A_R}$ which sends two comodules to a new comodule
whose underlying vector space is the tensor product of the
underlying vector spaces of the comodules.  Proposition
\ref{tensortomultprop} tells us that this functor is induced by a
unique $k$-homomorphism $u:A_R \otimes A_R \rightarrow A_R$; that
is, $\overline{\otimes}$ sends the objects $(X,\rho)$ and
$(Y,\mu)$ to $X \otimes Y$, with $A_R$-comodule structure given by
the composition
\begin{gather*}
 X \otimes Y \stackrel{\rho \otimes
\mu}{\vlongrightarrow} X \otimes A_R \otimes Y \otimes A_R
\stackrel{1 \otimes T \otimes 1}{\vvlongrightarrow} X \otimes Y
\otimes A_R \otimes A_R \\
 \stackrel{1 \otimes 1 \otimes u}{\vvlongrightarrow} X \otimes Y
\otimes A_R
\end{gather*}
 We claim that this $u$ is nothing more than the
natural multiplication on $\uprod A_i$ restricted to $A_R$.

\begin{lem}
Let $(C,\Delta)$ be a coalgebra.  Define the subset $S$ of $C$ to
consist of those elements $a \in C$ with the following property:
there exists a finite dimensional comodule $(X,\rho)$ over $C$
such that, for some element $x \in X$, $\rho(x) = \sum_i x_i
\otimes a_i$, with the $x_i$ linearly independent and $a = a_i$
for some $i$.  Then $S$ spans $C$.
\end{lem}

\begin{proof}
Let $c \in C$, and $C^\prime \subset C$ the finite dimensional
subcoalgebra it generates.  Write $\Delta(c) = \sum_i a_i \otimes
b_i$ with the $a_i$ linearly independent.  Then obviously all of
the $b_i$ are in $S$ if we view $C^\prime$ as a finite dimensional
comodule over $C$. Apply the coalgebra identity
\begin{diagram}
C & \rTo^{\Delta} & C \otimes C \\
& \rdTo_{\isomorphic} & \dTo_{\varepsilon \otimes 1} \\
& & k \otimes C \\
\end{diagram}
to yield $c = \sum_i \varepsilon(a_i) b_i$, showing $c$ to be in
the span of the $b_i$.
\end{proof}

\begin{thm}
$u$ is equal to the natural multiplication on $\uprod A_i$
restricted to $A_R$.
\end{thm}

\begin{proof}
Let $(X,\rho)$,$(Y,\mu)$ be finite dimensional comodules over
$A_R$, and consider
\begin{diagram}
X \otimes Y & \rTo^{\rho \otimes \mu} & X \otimes A_R \otimes Y
\otimes A_R & \rInto^\subset & X \otimes \uprod A_i \otimes Y
\otimes \uprod A_i \\
\dTo^{\rho \otimes \mu} &  &  &  &  \dTo_{1 \otimes T \otimes 1} \\
&  &  &  &  X \otimes Y \otimes \uprod A_i \otimes \uprod A_i \\
X \otimes A_R \otimes Y \otimes A_R &  &  &  &  \dTo_{1 \otimes 1 \otimes \text{mult}} \\
\dTo^{1 \otimes T \otimes 1} &  &  &  &  X \otimes Y \otimes \uprod A_i \\
&  &  &  &  \uInto_\subset \\
X \otimes Y \otimes A_R \otimes A_R &  & \rTo_{1 \otimes 1 \otimes u}  &  &  X \otimes Y \otimes A_R \\
\end{diagram}
The composition that starts at the top left, goes across, and then
all the way down is the image of the pair $(X,\rho)$, $(Y,\mu)$
under $\overline{\otimes}$ as proved in proposition
\ref{rephopfalgprop1}.  The one that starts at the top left, goes
down, and then across is the bifunctor induced by $u$; this
diagram commutes by assumption.

In the notation of the previous lemma, let $a,b \in S \subset
A_R$.  This means that there is an $(X,\rho)$ such that, for some
$x \in X$, $\rho(x) = \sum_i x_i \otimes a_i$ with the $x_i$
linearly independent and $a=a_i$ for some $i$.  Similarly, there
is a $(Y,\mu)$ and $y \in Y$ such that $\mu(y) = \sum_j y_j
\otimes b_j$, with the $y_j$ linearly independent and $b = b_j$
for some $j$.  Now if we chase the element $x \otimes y$ in the
above diagram around both ways, it gives
\[ \sum_{i,j} x_i \otimes y_j \otimes u(a_i \otimes b_j) = \sum_{i,j} x_i
\otimes y_j \otimes \text{mult}(a_i \otimes b_j) \]
Since the $x_i
\otimes y_j$ are linearly independent, by matching coefficients we
must have $u(a_i \otimes b_j) = \text{mult}(a_i \otimes b_j)$ for
every $i$ and $j$; in particular, $u(a \otimes b) = \text{mult}(a
\otimes b)$.

Thus we have shown that $u$ and $\text{mult}$ are equal on
elements of the form $a \otimes b$, with $a,b \in S$.  But the
previous lemma shows that $S \otimes S$ spans $A_R \otimes A_R$.
As both $u$ and $\text{mult}$ are $k$-linear, they must be equal
everywhere.  This completes the proof.
\end{proof}

\section{Examples}

\label{sectionExamples}

\subsection{Finite Groups}

\label{subsectionFiniteGroups}

Let $G$ be a finite group defined over $\mathbb{Z}$, $A$ its
representing Hopf algebra, $k_i$ a sequence of fields, $A_i = A
\otimes k_i = $ the representing Hopf algebra of $G$ over $k_i$,
and $\catfont{C}_i = \text{Comod}_{A_i}$.  We claim that the
representing Hopf algebra of $\resprod \catfont{C}_i$ is nothing
more than $A \otimes \uprod k_i$.  The following observation shows
this not to be surprising.

\begin{prop}
The category $\resprod \catfont{C}_i$ is equal to the principal
subcategory generated by $[A \otimes k_i]$.
\end{prop}

\begin{proof}
By $[A \otimes k_i]$ we mean the object of $\resprod
\catfont{C}_i$ consisting of the regular representation of $G$
over $k$ in every slot; it is clearly of constant finite
dimension. Let $[X_i]$ be an object of $\resprod \catfont{C}_i$,
say of dimension $n$.  Then by theorem \ref{regreptheorem} each
$X_i$ is a subobject of a quotient of $(A \otimes k_i)^n$; this is
a first-order statement, and so $[X_i]$ is a subobject of a
quotient of $[A \otimes k_i]^n$, showing $[X_i]$ to be in the
principal subcategory generated by $[A \otimes k_i]$.
\end{proof}

Proposition 2.20 of \cite{deligne} tells us that the property of
being singularly generated in fact characterizes those neutral
tannakian categories whose representing Hopf algebra is finite
dimensional.  In fact we can identify this Hopf algebra as
$\text{End}(\omega \vert \gen{[A \otimes k_i]})^\circ$, since $[A
\otimes k_i]$ generates the entire category.  Then by proposition
\ref{everythingisilluminatedprop} and theorem
\ref{adualitytheorem} we have that $A_\infty$, as a coalgebra, can
be identified as
\[ A_\infty = \text{End}(\omega \vert \gen{[A \otimes k_i]})^\circ
\isomorphic \uprod \text{End}(\omega_i \vert \gen{A \otimes
k_i})^\circ = \uprod A \otimes k_i \] Thus $A_\infty$ is equal to
the full ultraproduct of the $A \otimes k_i$.  Note that in this
case $A_R$, the restricted ultraproduct of the $A \otimes k_i$, is
in fact equal to the \textit{full} ultraproduct, since the $A
\otimes k_i$ are of constant finite dimension.

Since the $A \otimes k_i$ are of constant finite dimension, we
have an isomorphism $\uprod A \otimes k_i \isomorphic \uprod A
\otimes \uprod k_i$; it is not hard to see that the latter can be
further identified as $A \otimes \uprod k_i$ as an algebra,
coalgebra, and indeed Hopf algebra.

Thus, for $G$ finite, the category $\resprod \text{Rep}_{k_i} G$
can be identified with $\text{Rep}_{\uprod k_i} G$.

\subsection{The Multiplicative Group}

\label{subsectionTheGroupGm}

Here we compute $A_\infty$ for the multiplicative group $G_m$
using the work done in this chapter. Let $\catfont{C}_i =
\text{Rep}_{k_i} G_m$ with $k_i$ being some sequence of fields.
Let $A_i$ denote the representing Hopf algebra of $G_m$ over the
field $k_i$, which we identify as
\begin{gather*}
 A_i = k_i[x,x^{-1}] \\
 \Delta_i: x \mapsto x \otimes x \\
 \text{mult}:x^r \otimes x^s \mapsto x^{r+s}
\end{gather*}
We know then that $A_\infty$ is isomorphic to $A_R$, the
restricted ultraproduct of the Hopf algebras $A_i$, which we set
about now identifying.

Fix a field $k$, and let $A = k[x,x^{-1}]$, with $\Delta$,
$\text{mult}$ defined as above.  Then $A$ can be realized as the
increasing union of the finite dimensional subcoalgebras $B_0
\subset B_1 \subset \ldots $, defined as
\[ B_n = \text{span}_k(x^{-n}, x^{-(n-1)}, \ldots, x^{-1}, 1, x, x^2,
\ldots, x^n) \] The dual algebra $B_n^*$ to $B_n$ we identify as
\begin{gather*}
B_n^* = \text{span}_k(\alpha_{-n}, \ldots, \alpha_0 , \ldots,
\alpha_n) \\
 \text{mult}_n: \alpha_r \otimes \alpha_s \mapsto \delta_{rs}
\alpha_r
\end{gather*}
 As the $B_n$ form a direct system under the inclusion
mappings, the $B_n^*$ form an inverse system under the duals to
these inclusion mappings.  This map $B_n^* \leftarrow B_{n+1}^*$
is given by, for $\alpha_r:B_{n+1} \rightarrow k$, the image of
$\alpha_r$ is $\alpha_r$ if $|r| \leq n$ and $0$ otherwise. Then
we leave it to the reader to verify

\begin{prop}
The inverse limit $\invlim B_n^*$ of the $B_n^*$ can be identified
as $\text{span}_k(\alpha_i : i \in \mathbb{Z})$, with
$\text{mult}$ defined by
\[ \text{mult}:\alpha_r \otimes \alpha_s \mapsto \delta_{rs}
\alpha_r \] and the canonical mapping $T_m:\invlim B_n^*
\rightarrow B_m^*$ given by
\[ \alpha_r \mapsto \left\{
\begin{array}{c}
  \alpha_r \text{ if $|r| \leq m$} \\
  0 \text{ otherwise} \\
\end{array}
\right.
\]
\end{prop}

Next we must identify the map $\phi:A \rightarrow \left(\invlim
B_n^*\right)^\circ$ giving us the notion of `rank' in $A$.  For
$x^r \in A$, we pull it back to $x^r \in B_m$ for some $m \geq
|r|$, pass to the isomorphism $B_m \isomorphic B_m^{*\circ}$, and
then up through $T_m^\circ$.  Thus we obtain

\begin{prop}
The linear functional $\phi(x^r)$ acts on $\invlim B_n^*$ by

\[ \phi(x^r):\alpha_s \mapsto \delta_{rs} \]
\end{prop}

Let $f = c_1 x^{m_1} + c_2 x^{m_2} + \ldots + c_n x^{m_n}$ be an
arbitrary element of $A$, with $m_i \in \mathbb{Z}$ and $c_i \in
k$. Then $\phi(f)$ kills the ideal
\[ I = \text{span}_k(\alpha_r : r \in \mathbb{Z} - \{m_1, m_2,
\ldots, m_n\}) \ideal \invlim B_n^* \] This ideal is the largest
ideal that $\phi(f)$ kills and it has codimension $n$. Therefore

\begin{prop}
\index{rank} The rank of the element $f = c_1 x^{m_1} + c_2
x^{m_2} + \ldots + c_n x^{m_n} \in A$ is equal to $n$, the number
of distinct monomials occurring as terms.
\end{prop}

Let $A_i = k_i[x,x^{-1}]$ be the representing Hopf algebra of
$G_m$ over the field $k_i$. According to the previous proposition,
the rank of a polynomial in $A_i$ is equal the number of monomial
terms occurring in it. Thus, in order for an element $[f_i] \in
\uprod A_i$ to have bounded rank, it is necessary and sufficient
for it to have almost everywhere bounded monomial length. If this
bound is $n$, then in almost every slot $f_i$ has length among the
finite set $\{0,1, \ldots, n\}$, and so by lemma
\ref{coloringlemma} we may as well assume that the $f_i$ have
constant length.  Then we have

\begin{prop}
The restricted ultraproduct $A_R$ of the $A_i$ can be identified
as
\begin{gather*}
A_R = \{ [f_i] \in \uprod A_i : f_i \text{ has constant length }
\} \\
 \Delta:[f_i] \mapsto [\Delta_i(f_i)] \\
 \text{mult}: [f_i] \otimes [g_i] \mapsto [\text{mult}_i(f_i
\otimes g_i)]
\end{gather*}
\end{prop}

Let us find a tighter description of $A_R$.  Let $\uprod
\mathbb{Z}$ denote the ultrapower of the integers and let
$A^\prime$ denote the $k = \uprod k_i$-span of the formal symbols
$x^{[z_i]}$, $[z_i] \in \uprod \mathbb{Z}$.  Define the following
Hopf algebra structure on this vector space as follows:
\begin{gather*}
 A^\prime = \text{span}_k(x^{[z_i]}:[z_i] \in \uprod \mathbb{Z})
 \\
 \Delta:x^{[z_i]} \mapsto x^{[z_i]} \otimes x^{[z_i]} \\
\text{mult}:x^{[z_i]} \otimes x^{[w_i]} \mapsto x^{[z_i + w_i]}
\\
 \varepsilon:x^{[z^i]} \mapsto 1 \\
 S:x^{[z_i]} \mapsto x^{[-z_i]}
\end{gather*}
  Now every element of $A_R$
looks like
\[ [a_i^1 x^{z_i^1} + \ldots + a_i^m x^{z_i^m} ] \]
with $a_i^j \in k_i$ for every $i$ and $j$, and $z_i^j \in
\mathbb{Z}$ with $z_i^1 < z_i^2 < \ldots < z_i^m$.  Then define a
map $A_R$ to $A^\prime$ by
\[ [a_i^1 x^{z_i^1} + \ldots + a_i^m x^{z_i^m} ] \mapsto
[a_i^1] x^{[z_i^1]} + \ldots + [a_i^m] x^{[z_i^m]} \] We leave it
to the reader to verify

\begin{prop}
The map just defined is an isomorphism of Hopf algebras.
\end{prop}

Finally, let us build the equivalence of categories $\resprod
\catfont{C}_i \rightarrow \text{Comod}_{A_R}$ using the
description of $A_R$ given above.  Let $[X_i,\rho_i]$ be an object
of $\resprod \catfont{C}_i$ of dimension $m$.  It is well known
that any module for $G_m$ over a field is simply a diagonal sum of
characters; that is, in some basis, it has matrix formula
\[
\left(%
\begin{array}{cccc}
  x^{z_1} &  &  &  \\
   & x^{z_2} &  &  \\
   &  & \ddots &  \\
   &  &  & x^{z_m} \\
\end{array}%
\right)
\]
for some collection of non-negative integers $z_1, \ldots, z_m$.
For each $i$ then, fix a basis $e_i^1, \ldots, e_i^m$ of $X_i$ for
which the action of $G_m$ is a diagonal sum of characters.  Then
we can write
\[ \rho_i:e_i^j \mapsto e_i^j \otimes x^{z_i^j} \]
The vectors $[e_i^1], \ldots, [e_i^m]$ form a basis for $\uprod
X_i$, and in this basis the $A_R$-comodule structure of $\uprod
X_i$ is given by
\[ \rho:[e_i^j] \mapsto [e_i^j] \otimes x^{[z_i^j]} \]
That is, the action of $G_\infty = $ the group represented by
$A_R$ on $\uprod X_i$ is simply
\[
\left(%
\begin{array}{cccc}
  x^{[z_i^1]} &  &  &  \\
   & x^{[z_i^2]} &  &  \\
   &  & \ddots &  \\
   &  &  & x^{[z_i^m]} \\
\end{array}%
\right)
\]

\chapter{A Combinatorial Approach to the Representation Theory of Unipotent Algebraic Groups}
\label{themethodchapter}

As promised in the introduction, for the next several chapters we
take a break entirely from working with ultraproducts, and instead
focus on working out the concrete representation theories of
certain unipotent algebraic groups.  Here we outline the approach
we will be taking for all of the proofs constructed throughout.

 Let $G$ be an \index{algebraic group} algebraic group over the
field $k$ with Hopf algebra $A = k[x_1, \ldots, x_n]/I$, where $I$
is the ideal of $k[x_1, \ldots, x_n]$ generated by the defining
polynomial equations of $G$ (e.g., if $G=\text{SL}_2$, then $A =
k[x_1,x_2,x_3,x_4]/(x_1x_4-x_2x_3-1)$). If $(V, \rho)$ is a
comodule over $A$ with basis $e_1, \ldots, e_m$, we can write
\[ \rho:e_j \mapsto \sum_i e_i \otimes a_{ij} \]
\[ a_{ij} = \sum_{\vec{r} = (r_1, \ldots, r_n)} c_{ij}^{\vec{r}} x_1^{r_1} \ldots
x_n^{r_n} \] where $c_{ij}^{\vec{r}}$ is a scalar for every $i,j$
and $\vec{r}$, and the summation runs over some finite collection
of $n$-tuples of non-negative integers.  If for each $\vec{r}$ we
think of \index{$(c_{ij})$} $(c_{ij})^{\vec{r}}$ as an $m \times
m$ matrix over $k$, its significance is that it consists of the
coefficients of the monomial $x_1^{r_1}\ldots x_n^{r_n}$ in the
matrix formula for the representation in the basis $e_1, \ldots,
e_m$.  For example, consider the group $G_a \times G_a$, where
$G_a$ denotes the \index{$G_a$} additive group (see chapter
\ref{chapterTheAdditiveGroup}).  $G_a \times G_a$ has as its
representing Hopf algebra $k[x,y]$.  Consider the
\index{representation} representation defined by
\[
\left(%
\begin{array}{ccc}
  1 & 2y+x & 2y^2+2yx+\frac{1}{2} x^2 \\
  0 & 1 & 2y + x \\
  0 & 0 & 1 \\
\end{array}
\right)
\]
Then in this basis the $(c_{ij})$ matrices are given by
\[
 (c_{ij})^{(0,0)} = \left(
\begin{array}{ccc}
  1 & 0 & 0 \\
  0 & 1 & 0 \\
  0 & 0 & 1 \\
\end{array}%
\right)
 \hspace{1cm} (c_{ij})^{(1,0)} =
\left(
\begin{array}{ccc}
  0 & 1 & 0 \\
  0 & 0 & 1 \\
  0 & 0 & 0 \\
\end{array}%
\right)
\]
\[(c_{ij})^{(0,1)} =
\left(%
\begin{array}{ccc}
  0 & 2 & 0 \\
  0 & 0 & 2 \\
  0 & 0 & 0 \\
\end{array}%
\right) \hspace{1cm} (c_{ij})^{(2,0)} =
\left(%
\begin{array}{ccc}
  0 & 0 & \frac{1}{2} \\
  0 & 0 & 0 \\
  0 & 0 & 0 \\
\end{array}%
\right)
\]
\[(c_{ij})^{(0,2)} =
\left(%
\begin{array}{ccc}
  0 & 0 & 2 \\
  0 & 0 & 0 \\
  0 & 0 & 0 \\
\end{array}%
\right)  \hspace{1cm} (c_{ij})^{(1,1)} =
\left(%
\begin{array}{ccc}
  0 & 0 & 2 \\
  0 & 0 & 0 \\
  0 & 0 & 0 \\
\end{array}%
\right)
\]
with $(c_{ij})^{\vec{r}} = 0$ for all other $\vec{r}$.  Now the
two diagrams asserting that a given vector space $V$ and linear
map $\rho:V \rightarrow V \otimes A$ is a comodule over $A$ are
\[
\begin{diagram}
V  & \rTo^\rho & V \otimes A \\
& \rdTo_{\sim} & \dTo_{1 \otimes \varepsilon} \\
& & V \otimes k \\
\end{diagram}
\hspace{3cm}
\begin{diagram}
V & \rTo^\rho & V \otimes A \\
\dTo^\rho & & \dTo_{1 \otimes \Delta} \\
V \otimes A & \rTo_{\rho \otimes 1} & V \otimes A \otimes A \\
\end{diagram}
\]
In the first diagram, if we chase $e_j$ along both paths we arrive
at
\[ e_j \otimes 1 = \sum_i e_i \otimes \varepsilon(a_{ij}) \]
and we see by matching up coefficients that $\varepsilon(a_{ij}) =
1$ if $i=j$, zero otherwise.  This simply says that
\begin{equation}
\label{comoduleequation1}
 (\varepsilon(a_{ij})) = \text{Id}
\end{equation}
which will be some matrix expression among the $(c_{ij})$
matrices.  For the second diagram we chase $e_j$ along both paths
and arrive at
\[ \sum_i e_i \otimes (\sum_k a_{ik} \otimes a_{kj}) = \sum_i e_i
\otimes \Delta(a_{ij}) \] and again, by matching coefficients,
this reduces to
\begin{equation}
\label{comoduleequation2} \sum_k a_{ik} \otimes a_{kj} =
\Delta(a_{ij})
\end{equation}
Notice that the left hand side is precisely the
$(i,j)^{\text{th}}$ entry of the matrix product $(a_{ij} \otimes
1)(1 \otimes a_{ij})$. In practice these two equations will allow
us to derive matrix equalities between the various
$(c_{ij})^{\vec{r}}$, which will serve as necessary and sufficient
conditions for them to define representations over a given group.

This approach is particularly amenable to the study of unipotent
groups.  It is well known that all unipotent algebraic groups (and
quite obviously the ones we will be studying) have Hopf algebras
which are algebra-isomorphic to $A = k[x_1,x_2, \ldots, x_n]$;
that is, isomorphic to a polynomial algebra with no relations.  As
such, the collection of all monomial tensors $x_1^{r_1} x_2^{r_2}
\ldots x_n^{r_n} \otimes x_1^{s_1} x_2^{s_2} \ldots x_n^{s_n}$
form a basis for $A \otimes A$, and a great deal of our work will
involve looking at equalities between large summations in $A
\otimes A$ and trying to match coefficients on a basis.  Thus, for
unipotent groups, a logical choice of basis with which to attempt
this is always available.

We would like to briefly mention that, in a round-about way, what
all this amounts to is working with the Lie algebra of the group
in the characteristic $0$ case, and the distribution algebra in
the characteristic $p > 0$ case (see chapter $7$ of \cite{jantzen}
for a good account of the latter).  In fact, for a given
$G$-module, these $(c_{ij})^{\vec{r}}$ matrices correspond to the
images of certain distributions under the associated
$\text{Dist}(G)$ module, and these matrix equalities we shall be
deriving essentially amount to working out the multiplication law
in $\text{Dist}(G)$; compare for instance equation
\ref{Gafundamentalrelation} with equation $(1)$ of page $101$ of
\cite{jantzen}.  We have chosen however to proceed without this
machinery; it gives us no advantage to our purposes, and in any
case puts less of a burden on the reader (and the author for that
matter).

\section{Morphisms}

None of this work would be worth much if we couldn't say something
about morphisms between comodules.

Again let $G$ be an algebraic group over the field $k$, with Hopf
algebra $A = k[x_1, \ldots, x_n]/I$, and this time fix a monomial
basis $\{x_1^{r_1}\ldots x_n^{r_n} : \vec{r} \in R\}$ of $A$. For
brevity, for $\vec{r} \in R$, denote by $x^{\vec{r}}$ the monomial
 $x_1^{r_1} \ldots x_n^{r_n}$. If $(V, \rho)$ is a comodule over
$A$ with basis $e_1, \ldots, e_n$, we can write
\[ \rho:e_j \mapsto \sum_{i=1}^n e_i \otimes a_{ij} \]
\[ a_{ij} = \sum_{\vec{r} \in R} c_{ij}^{\vec{r}} x^{\vec{r}} \]
with each $c_{ij}^{\vec{r}}$ a scalar.  Similarly let $(W,\mu)$ be
a comodule over $A$ with basis $f_1, \ldots, f_m$, and write
\[ \mu:f_j \mapsto \sum_{i=1}^m f_i \otimes b_{ij} \]
\[ b_{ij} = \sum_{\vec{r} \in R} d_{ij}^{\vec{r}} x^{\vec{r}} \]

\begin{thm}

\label{morphismsInTheMethodThm}

Let $\phi:V \rightarrow W$ be a linear map, and write it as the
matrix $(k_{ij})$ in the given bases.  Then $\phi$ is a morphism
of $V$ and $W$ as $A$-comodules if and only if for every $\vec{r}
\in R$
\[ (k_{ij}) (c_{ij})^{\vec{r}} = (d_{ij})^{\vec{r}} (k_{ij}) \]

\end{thm}

\begin{proof}
Consider the diagram
\[
\begin{diagram}
V & \rTo^\phi & W \\
\dTo^\rho & & \dTo_\mu \\
V \otimes A & \rTo_{\phi \otimes 1} & W \otimes A \\
\end{diagram}
\]
whose commutativity is equivalent to $\phi$ being a morphism
between $V$ and $W$. If we chase $e_j$ along both paths we arrive
at
\[ \sum_{i=1}^n \left( \sum_{s=1}^m k_{si} f_s \right) \otimes
a_{ij} = \sum_{i=1}^m k_{ij} \left( \sum_{s=1}^m f_s \otimes
b_{si} \right) \] Replacing the $a's$ and $b's$ with their
definitions in terms of the $c's$ and $d's$ and re-arranging a
bit, we have
\[ \sum_{s,\vec{r}} \left( \sum_{i=1}^n k_{si} c_{ij}^{\vec{r}} \right) f_s \otimes
x^{\vec{r}} = \sum_{s,\vec{r}} \left(\sum_{i=1}^m d_{si}^{\vec{r}}
k_{ij} \right) f_s \otimes x^{\vec{r}} \] As $f_s \otimes
x^{\vec{r}}$ for varying $s=1 \ldots m$ and $\vec{r} \in R$ is a
free basis for $W \otimes A$, we must have, for every $j,s$ and
$\vec{r} \in R$
\[ \sum_{i=1}^n k_{si} c_{ij}^{\vec{r}} = \sum_{i=1}^m d_{si}^{\vec{r}} k_{ij}  \]
But the left hand side is the $(s,j)^\text{th}$ entry of
$(k_{ij})(c_{ij})^{\vec{r}}$, and the right the $(s,j)^\text{th}$
entry of $(d_{ij})^{\vec{r}} (k_{ij})$.

\end{proof}

A combinatorial lemma we will need later:

\begin{lem}
\label{commutingRepsInTheMethodThm} Let $(V,\rho)$,$(V,\mu)$ be
two representation of $G$ on the finite dimensional vector space
$V$, given by

\begin{equation*}
\begin{split}
&\rho:e_j \mapsto \sum_i e_i \otimes a_{ij} \hspace{1.5cm}
 a_{ij} = \sum_{\vec{r}} c_{ij}^{\vec{r}}x^{\vec{r}} \\
&\mu: e_j \mapsto \sum_i e_i \otimes b_{ij} \hspace{1.5cm} b_{ij}
= \sum_{\vec{r}} d_{ij}^{\vec{r}} x^{\vec{r}} \end{split}
\end{equation*}
Then the matrices $(a_{ij} \otimes 1)$ and $(1 \otimes b_{ij})$
(taking their entries from $A \otimes A$) commute if and only if
for every $\vec{r}, \vec{s} \in R$, the matrices
$(c_{ij})^{\vec{r}}$ and $(d_{ij})^{\vec{s}}$ commute.

\end{lem}

\begin{proof}
The $(i,j)^{\text{th}}$ entry of $(a_{ij} \otimes 1)(1 \otimes
b_{ij})$ is
\[  \sum_k a_{ik} \otimes b_{kj} = \sum_k
\left(\sum_{\vec{r}} c_{ik}^{\vec{r}} x^{\vec{r}})\right) \otimes
\left(\sum_{\vec{s}} d_{kj}^{\vec{s}} x^{\vec{s}}\right) =
\sum_{\vec{r},\vec{s}} \left(\sum_k c_{ik}^{\vec{r}}
d_{kj}^{\vec{s}} \right) x^{\vec{r}} \otimes x^{\vec{s}} \] while
the $(i,j)^{\text{th}}$ entry of $(1 \otimes b_{ij})(a_{ij}
\otimes 1)$ is
\[ \sum_k a_{kj} \otimes b_{ik} = \sum_k \left(\sum_{vec{r}}
c_{kj}^{\vec{r}} x^{\vec{r}} \right) \otimes \left( \sum_{\vec{s}}
d_{ik}^{\vec{s}} x^{\vec{s}} \right)  = \sum_{\vec{r},\vec{s}}
\left(\sum_k d_{ik}^{\vec{s}} c_{kj}^{\vec{r}} \right) x^{\vec{r}}
\otimes x^{\vec{s}} \] By matching coefficients on the free basis
$x^{\vec{r}} \otimes x^{\vec{s}}$ for $A \otimes A$ we must have,
for every $i,j \leq n$ and $\vec{r},\vec{s} \in R$
\[ \left(\sum_k c_{ik}^{\vec{r}}
d_{kj}^{\vec{s}} \right) = \left(\sum_k d_{ik}^{\vec{s}}
c_{kj}^{\vec{r}} \right) \] i.e.
\[ (c_{ij})^{\vec{r}} (d_{ij})^{\vec{s}} = (d_{ij})^{\vec{s}}
(c_{ij})^{\vec{r}} \]

\end{proof}

\chapter{Representation Theory of Direct Products}

\index{algebraic group!direct products of}

\label{representationtheoryofdirectproductschapter}

Here we are interested in the following: given that one has a
handle on the representation theory of the algebraic groups $G$
and $H$, what can be said about the representation theory of $G
\times H$? Any representation of $G \times H$ is evidently a
representation of both $G$ and $H$ in a natural way.  It is also
evident that one cannot paste together any two representations of
$G$ and $H$ to get one for $G \times H$; they must somehow be
compatible. The goal of this section is to prove necessary and
sufficient conditions for representations of $G$ and $H$ to
together define one for $G \times H$, and to provide a formula for
it.

I will spoil the suspense: two representations $\Phi$ and $\Psi$
for $G$ and $H$ on the vector space $V$ define one for $G \times
H$ on $V$ if and only if they commute; that is, the linear maps
$\Phi(g)$ and $\Psi(h)$ commute for every pair $g \in G$, $h \in
H$.  The matrix formula for the $G \times H$-module they define is
nothing more than the product of the matrix formulas of the
constituent modules.  Morphisms for the new module are exactly
those that are morphisms for both of the constituent modules, and
likewise direct sums and tensor products behave as we hope they
will.

We leave it to the reader to verify the following: if $(A,
\Delta_A, \varepsilon_A)$ and $(B, \Delta_B, \varepsilon_B)$ are
the representing Hopf algebras of $G$ and $H$, then the Hopf
algebra for $G \times H$ is $(A \otimes B, \Delta, \varepsilon)$,
defined by
\begin{gather*}
 \Delta: A \otimes B \stackrel{\Delta_A \otimes
\Delta_B}{\vvlongrightarrow} A \otimes A \otimes B \otimes B
\stackrel{1 \otimes \text{Twist} \otimes 1}{\vvlongrightarrow} A
\otimes B \otimes A \otimes B \\
 \varepsilon: A \otimes B \stackrel{\varepsilon_A \otimes
\varepsilon_B}{\vvlongrightarrow} k \otimes k \isomorphic k
\end{gather*}
The natural embedding $G \rightarrow G \times H$ is induced by the
Hopf algebra map $A \otimes B \stackrel{1 \otimes
\varepsilon_B}{\vlongrightarrow} A \otimes k \isomorphic A$,
similarly for $H$.  Then for any vector space $V$ and $A \otimes
B$-comodule structure $\rho:V \rightarrow V \otimes A \otimes B$,
we get an $A$-comodule structure on $V$ given by the composition
\[ V \stackrel{\rho}{\longrightarrow} V \otimes A \otimes B
\stackrel{1 \otimes 1 \otimes \varepsilon_B}{\vvlongrightarrow} V
\otimes A \otimes k \isomorphic V \otimes A \] Similarly we get a
$B$-comodule structure by instead tacking on $1 \otimes
\varepsilon_A \otimes 1$.

Let $(V,\rho),(V,\mu)$ be two comodule structures on the vector
space $V$, the first for $G$, the second for $H$.  Write
\[ \rho:e_j \mapsto \sum_i e_i \otimes a_{ij} \]
\[ \mu:e_j \mapsto \sum_i e_i \otimes b_{ij} \]
\begin{thm}
\label{directproducttheorem} Let $(V,\rho)$, $(V,\mu)$ be modules
for $G$ and $H$ as above. Define
\[ z_{ij} = \sum_k a_{ik} \otimes b_{kj} \]
Then the map $\sigma: V \rightarrow V \otimes A \otimes B$ defined
by
\[ \sigma: e_j \mapsto \sum_i e_i \otimes z_{ij} \]
is a valid module structure for \index{algebraic group!direct
products of} $G \times H$ if and only if the matrices $(a_{ij})$
and $(b_{ij})$ commute with one another. $\sigma$ restricts
naturally to $\rho$ and $\mu$ via the canonical embeddings, and
this is the only possible comodule structure on $V$ that does so.

\end{thm}

Some explanation is in order.  The matrices $(a_{ij})$ and
$(b_{ij})$ take their entries from different algebras, so it
doesn't make much sense to say they commute.  What we really mean
is that the matrix products $(a_{ij} \otimes 1)(1 \otimes b_{ij})$
and $(1 \otimes b_{ij})(a_{ij} \otimes 1)$, which take their
entries from $A \otimes B$, are equal.  The content of the theorem
then is that, if they do commute, the matrix $(z_{ij}) = (a_{ij}
\otimes 1)(1 \otimes b_{ij})$ provides a valid module structure
for $G \times H$, and this is the only one that restricts to
$\rho$ and $\mu$. In what follows we shall simply write $a_{ij}$
and $b_{ij}$, even when we wish to consider them as elements of $A
\otimes B$.

\begin{proof}
Given that this actually is a representation, the representation
induced on $G$ from it is given by
\begin{gather*}
 e_j \stackrel{\sigma}{\mapsto}
\sum_i e_i \otimes z_{ij} \stackrel{1 \otimes 1 \otimes
\varepsilon_B}{\mapsto} \sum_i e_i \otimes \sum_k a_{ik} \otimes
\varepsilon_B(b_{kj}) \\
= \sum_i e_i \otimes \sum_k a_{ik} \otimes \delta_{kj} = \sum_i
e_i \otimes a_{ij}
\end{gather*}
 and similarly for $H$.  Thus $\sigma$ does
indeed restrict to $\rho$ and $\mu$.

To prove that this actually is a representation, as always, we
must check that the equations $\varepsilon(z_{ij}) = \delta_{ij}$
and $\sum_k z_{ik} \otimes z_{kj} = \Delta(z_{ij})$ are satisfied.
For the first, we have
\begin{equation*}
\begin{split}
 \varepsilon(z_{ij}) &=  \varepsilon(\sum_k a_{ik} \otimes
b_{kj}) \\ &= \sum_k \varepsilon_A(a_{ik}) \varepsilon_B(b_{kj})\\
&= \sum_k \delta_{ik} \delta_{kj}\\ &= \delta_{ij}
\end{split}
\end{equation*}
as required. For the second equation, we have
\begin{equation*}
\begin{split}
 \Delta(z_{ij}) &= \Delta(\sum_k a_{ik} \otimes b_{kj}) \\
 &= \sum_k \Delta_A(a_{ik}) \otimes \Delta_B(b_{kj}) \\
 &= \sum_k (\sum_l a_{il} \otimes a_{lk}) \otimes (\sum_m b_{km}
\otimes b_{mj}) \\
 &\stackrel{\text{Twist}}{\isomorphic} \sum_{k,l,m} a_{il} \otimes
b_{km} \otimes a_{lk} \otimes b_{mj}
\end{split}
\end{equation*}
 where we have used the fact
that $\rho$ and $\mu$ are comodule structures for $G$ and $H$, and
applied the twist at the end. Note that the last expression is
equal to the $(i,j)^\text{th}$ entry of the matrix
$(a_{ij})(b_{ij})(a_{ij})(b_{ij})$.  For the other side:
\begin{equation*}
\begin{split}
 \sum_k z_{ik} \otimes z_{kj} &= \sum_k (\sum_l a_{il} \otimes
b_{lk}) \otimes (\sum_m a_{km} \otimes b_{mj}) \\
&= \sum_{k,l,m} a_{il} \otimes b_{lk} \otimes a_{km} \otimes
b_{mj}
\end{split}
\end{equation*}
This is equal to the $(i,j)^\text{th}$ entry of the matrix
$(a_{ij})(a_{ij})(b_{ij})(b_{ij})$.  Thus, all this reduces to the
matrix equality
\[ (a_{ij})(b_{ij})(a_{ij})(b_{ij}) =
(a_{ij})(a_{ij})(b_{ij})(b_{ij})\] And since the matrices
$(a_{ij})$ and $(b_{ij})$ are necessarily invertible, we can
multiply both sides on the left by $(a_{ij})^{-1}$ and on the
right by $(b_{ij})^{-1}$, and we are left with
\[ (b_{ij})(a_{ij}) = (a_{ij})(b_{ij}) \]
This proves the first part of the theorem.  Lastly, we wish to see
that $\sigma$ is the only possible comodule structure on $V$
restricting to $\rho$ and $\mu$.  Let $\tau:V \rightarrow V
\otimes A \otimes B$ be any such comodule, and write
\[ \tau: e_j \mapsto \sum_i e_i \otimes w_{ij} \]
$\tau$ restricting to $\rho$ and $\mu$ means that $1 \otimes
\varepsilon_B:w_{ij} \mapsto a_{ij}$ and that $\varepsilon_A
\otimes 1:w_{ij} \mapsto b_{ij}$.  By virtue of $\tau$ being a
comodule map we have
\[ (\Delta_A \otimes \Delta_B)(w_{ij}) = \sum_k w_{ik} \otimes
w_{kj} \] and, using the Hopf algebra identity $(\varepsilon
\otimes 1) \circ \Delta = (1 \otimes \varepsilon) \circ \Delta =
1$,
\begin{equation*}
\begin{split}
 (\Delta_A \otimes \Delta_B)(\sum_k a_{ik} \otimes b_{kj}) &=
(\Delta_A \otimes \Delta_B)(\sum_k (1 \otimes
\varepsilon_B)(w_{ik}) \otimes (\varepsilon_A \otimes 1)(w_{kj})
\\
 &= \sum_k w_{ik} \otimes w_{kj}
\end{split}
\end{equation*}
  But $\Delta_A \otimes \Delta_B
= \Delta$ is a 1-1 map (as all co-multiplication maps are), and
sends $w_{ij}$ and $\sum_i a_{ik} \otimes b_{kj}$ to the same
thing.  We conclude then that $w_{ij} = \sum_k a_{ik} \otimes
b_{kj}$.  This completes the proof.

\end{proof}

\begin{section}{Constructions and Morphisms}

Here we record some facts about certain constructions and
morphisms on representations of $G \times H$, relative to the
induced representations on $G$ and $H$.  Throughout, let $A$ and
$B$ be the representing Hopf algebras of $G$ and $H$, $V$ and $W$
fixed finite dimensional vector spaces spanned by $\{e_i\}$ and
$\{f_j\}$ respectively, and endowed with $A \otimes B$-comodule
structures $\sigma$ and $\tau$, given by
\begin{equation*}
\begin{split}
 \sigma: &e_j \mapsto \sum_i e_i \otimes z_{ij} \\
 \tau: &f_j \mapsto \sum_i f_i \otimes w_{ij}
\end{split}
\end{equation*}
Further, let $\sigma_A, \sigma_B, \tau_A, \tau_B$ be the induced
comodule structures on $A$ and $B$, given by
\begin{equation*}
\begin{split}
 \sigma_A: &e_j \mapsto \sum_i e_i \otimes a_{ij} \hspace{1.5cm}
 \tau_A: f_j \mapsto \sum_i f_i \otimes l_{ij} \\
 \sigma_B: &e_j \mapsto \sum_i e_i \otimes b_{ij}
\hspace{1.5cm} \tau_B: f_j \mapsto \sum_i f_i \otimes m_{ij}
\end{split}
\end{equation*}
Then we know from the previous section that the matrix $(z_{ij})$
is equal to the matrix product $(a_{ij})(b_{ij})$, and similarly
$(w_{ij}) = (l_{ij})(m_{ij})$, that the matrix $(a_{ij})$ commutes
with $(b_{ij})$, and that $(l_{ij})$ commutes with $(m_{ij})$.
Since $\sigma$ and $\tau$ are the unique representations
restricting to the given ones, we shall say that they are
\emph{induced} by $\sigma_A,\sigma_B$ and $\tau_A, \tau_B$
respectively.

\begin{prop}
The direct sum $\sigma \oplus \tau$ is induced by the direct sums
 $\sigma_A \oplus \tau_A$ and $\sigma_B \oplus \tau_B$.
\end{prop}

\begin{proof}
If we view $(z_{ij})$ and $(w_{ij})$ as the matrix formulas for
the representations $\sigma$ and $\tau$, then the matrix formula
for $\sigma \oplus \tau$ in the basis $\{e_i\} \cup \{f_i\}$ is
\[
\left(%
\begin{array}{cc}
  (z_{ij}) &  \\
   & (w_{ij}) \\
\end{array}%
\right)
\]
which we can write as
\[
\left(%
\begin{array}{cc}
  (a_{ij})(b_{ij}) &  \\
   & (l_{ij})(m_{ij}) \\
\end{array}%
\right) =
\left(%
\begin{array}{cc}
  (a_{ij}) &  \\
   & (l_{ij}) \\
\end{array}
\right)
\left(%
\begin{array}{cc}
  (b_{ij}) &  \\
   & (m_{ij}) \\
\end{array}%
\right)
\]
This last formula is exactly that of the representation on $G
\times H$ induced by that of $\sigma_A \oplus \tau_A$ and
$\sigma_B \oplus \tau_B$.
\end{proof}

\begin{prop}
\label{tensorForDirProdsProp} The tensor product of
representations $\sigma \otimes \tau$ is induced by $\sigma_A
\otimes \tau_A$ and $\sigma_B \otimes \tau_B$.
\end{prop}

\begin{proof}
This is merely the observation that tensor product of matrices
commutes with matrix multiplication.  The matrix formula for the
representation $\sigma \otimes \tau$ is
\[ (z_{ij}) \otimes (w_{ij}) = (a_{ij})(b_{ij}) \otimes
(l_{ij})(m_{ij}) = [(a_{ij}) \otimes (l_{ij})][(b_{ij}) \otimes
(m_{ij})] \] which is the representation induced by $\sigma_A
\otimes \tau_A$ and $\sigma_B \otimes \tau_B$.
\end{proof}

\begin{prop}
\label{dirProdHomprop} A linear map $\phi: V \rightarrow W$ is a
morphism between the representations $(V,\sigma)$ and $(W,\tau)$
in the category $\text{Rep}_k G \times H$ if and only if it is
both a morphism between the representations $(V,\sigma_A)$ and
$(W,\tau_A)$ in the category $\text{Rep}_k G$ and between
$(V,\sigma_B)$ and $(W,\tau_B)$ in the category $\text{Rep}_k H$.

\end{prop}

\begin{proof}
Write $\phi$ as the matrix $(k_{ij})$ in the relevant bases.
Recall that, $\phi$ being a morphism in $\text{Rep}_k G \times H$
is equivalent to the matrix equality $(k_{ij})(z_{ij}) = (w_{ij})
(k_{ij})$.  Then the `if' direction of the theorem is obvious,
since by assumption $(z_{ij}) = (a_{ij})(b_{ij}) =
(b_{ij})(a_{ij})$, similarly for $(w_{ij})$.

Conversely, consider the diagram
\begin{diagram}
V & \rTo^\phi & W \\
\dTo^\sigma & & \dTo_\tau \\
V \otimes A \otimes B & \rTo_{\phi \otimes 1 \otimes 1} & V \otimes A \otimes B \\
\dTo^{1 \otimes 1 \otimes \varepsilon_B} & & \dTo_{1 \otimes 1
\otimes \varepsilon_B} \\
V \otimes A & \rTo_{\phi \otimes 1} & V \otimes A \\
\end{diagram}
Commutativity of the top rectangle is the statement that $\phi$ is
a morphism between $(V,\sigma)$ and $(W,\tau)$, and commutativity
of the bottom rectangle is true no matter what $\phi$ is.  This
gives us commutativity of the outermost rectangle, which is the
assertion that $\phi$ is a morphism between $(V,\sigma_A)$ and
$(W,\tau_A)$ in $\text{Rep}_k G$.  An identical result holds in
$\text{Rep}_k H$ if we consider $\varepsilon_A$ instead.

\end{proof}

\end{section}

\chapter{The Additive Group}

\label{chapterTheAdditiveGroup}

In this chapter we give a complete characterization of the finite
dimensional representations of the additive group \index{$G_a$}
$G_a$ over any field, using no more than the combinatorial methods
described in chapter \ref{themethodchapter}.  The goal, as with
both of the unipotent groups we'll be investigating, is to show
that, for characteristic $p > 0$ large with respect to dimension,
modules for $G_a$ in characteristic $p$ look exactly like modules
for $G_a^n$ (direct product of copies of $G_a$) in characteristic
zero.  This is by far the easiest case we'll consider, as even
when $p << \text{dim}$, the analogy is still very strong (which is
atypical; modules for the Heisenberg group, discussed later, in
$\text{dim}
>> p$ do not share this property).

We start by considering the group $G_a^\infty$, the countable
direct product of $G_a$; all we need for the group $G_a$ will
emerge as a special case. For a fixed field $k$ we identify the
Hopf algebra $A$ for the group $G_a^\infty$ as $k[x_1,x_2,
\ldots]$, the free algebra on countably many commuting variables,
with $\Delta$ and $\varepsilon$ defined by
\begin{gather*}
A = k[x_1, x_2, \ldots] \\
 \Delta: x_i \mapsto 1 \otimes x_i + x_i \otimes 1 \\
 \varepsilon: x_i \mapsto 0
\end{gather*}
Let $(V,\rho)$ be a comodule over $A$, fix a basis $\{e_i\}$ of
$V$, and write \label{cijGaDefn}
\begin{gather*}
\rho:e_j \mapsto \sum_i e_i \otimes a_{ij} \\
 a_{ij} = \sum_{\vec{r} = (r_1, \ldots, r_n)} c_{ij}^{\vec{r}}
x^{\vec{r}}
\end{gather*}
 where we adopt the notation $x^{\vec{r}} = x_1^{r_1}
x_2^{r_2} \ldots x_n^{r_n}$. Notice that only finitely many of the
variables $x_i$ can show up in any of the $a_{ij}$, so for all
intents and purposes, this is really just a comodule over a
finitely generated slice of $A$, namely $k[x_1, \ldots, x_n]$, the
representing Hopf algebra of the group $G_a^n$.

\section{Combinatorics}

We look first at equation \ref{comoduleequation1}:
\[ (\varepsilon(a_{ij})) = \text{Id} \]
Note that $\varepsilon$ acts on $a_{ij}$ by simply picking off its
constant term $c_{ij}^{\vec{0}}$.  The above thus reduces to the
matrix equality
\[ (c_{ij})^{\vec{0}} =  \text{Id} \]
This equality is intuitively obvious, for when we evaluate the
matrix formula of the representation at $x_1 = \ldots = x_n = 0$
we are left with $(c_{ij})^{(0,\ldots,0)}$, which should indeed be
the identity matrix.

Next we look equation \ref{comoduleequation2}, namely $\sum_k
a_{ik} \otimes a_{kj} = \Delta(a_{ij})$.  Working with the right
hand side first we have, making repeated use of the binomial
theorem
\begin{equation*}
\begin{split}
 \Delta(a_{ij}) &= \sum_{\vec{r}} c_{ij}^{\vec{r}}
\Delta(x_1)^{r_1} \ldots \Delta(x_n)^{r_n} \\
&= \sum_{\vec{r}} c_{ij}^{\vec{r}} (1 \otimes x_1 + x_1 \otimes
1)^{r_1} \ldots (1 \otimes x_n + x_n \otimes 1)^{r_n} \\
&=\sum_{\vec{r}} c_{ij}^{\vec{r}} \left(\sum_{k_1 + l_1 =
r_1}{k_1+l_1 \choose l_1} x_1^{k_1}\otimes  x_1^{l_1}\right)
\ldots \left(\sum_{k_n + l_n = r_n} {k_n+l_n \choose l_n}
x_n^{k_n} \otimes x_n^{l_n} \right) \\
&= \sum_{\vec{r}} c_{ij}^{\vec{r}} \sum_{\vec{k} + \vec{l} =
\vec{r}} {k_1 + l_1 \choose l_1} \ldots {k_n+l_n \choose l_n}
x_1^{k_1} \ldots x_n^{k_n} \otimes x_1^{l_1}\ldots x_n^{l_n} \\
&= \sum_{\vec{r}} c_{ij}^{\vec{r}} \sum_{\vec{k} + \vec{l} =
\vec{r}} {\vec{k} + \vec{l} \choose \vec{l}} x^{\vec{k}} \otimes
x^{\vec{l}}
\end{split}
\end{equation*}
Here we have adopted the notation, for two $n$-tuples of
non-negative integers $\vec{m} = (m_1, \ldots, m_n)$ and $\vec{r}
= (r_1, \ldots, r_n)$,
\[\vec{m}+\vec{r} = (m_1 + r_1, \ldots, m_n + r_n)\]
and
\[{\vec{m} \choose \vec{r}} = {m_1 \choose r_1}{m_2 \choose r_2}
\ldots {m_n \choose r_n}\] For the left hand side we have
\begin{equation*}
\begin{split}
 \sum_k a_{ik} \otimes a_{kj} &= \sum_k \left(\sum_{\vec{r}}
c_{ik}^{\vec{r}}x^{\vec{r}}\right) \otimes \left(\sum_{\vec{s}}
c_{kj}^{\vec{s}}x^{\vec{s}} \right) \\
&= \sum_k \sum_{\vec{r},\vec{s}} c_{ik}^{\vec{r}} c_{kj}^{\vec{s}}
x^{\vec{r}} \otimes x^{\vec{s}}\\
&= \sum_{\vec{r},\vec{s}} \left( \sum_k c_{ik}^{\vec{r}}
c_{kj}^{\vec{s}} \right) x^{\vec{r}} \otimes x^{\vec{s}}
\end{split}
\end{equation*}
Thus, equation \ref{comoduleequation2} reduces to
\[ \sum_{\vec{r},\vec{s}} \left( \sum_k c_{ik}^{\vec{r}}
c_{kj}^{\vec{s}} \right) x^{\vec{r}} \otimes x^{\vec{s}}  =
\sum_{\vec{r}} c_{ij}^{\vec{r}} \sum_{\vec{k} + \vec{l} = \vec{r}}
{\vec{k} + \vec{l} \choose \vec{l}} x^{\vec{k}} \otimes
x^{\vec{l}} \]

Now the polynomial ring $k[x_1, \ldots, x_n]$ has no relations,
whence the collection of all monomial tensors $x^{\vec{r}} \otimes
x^{\vec{s}}$ for varying $\vec{r}$ and $\vec{s}$ constitutes a
free basis for $k[x_1, \ldots, x_n] \otimes k[x_1, \ldots, x_n]$;
we can therefore simply match coefficients. In the left hand side
of the above equation, clearly each monomial tensor occurs exactly
once, and its coefficient is $\sum_k c_{ik}^{\vec{r}}
c_{kj}^{\vec{s}}$. In the right hand side it is also true that
each monomial tensor occurs exactly once, for if you choose
$\vec{k}$ and $\vec{l}$, there is only one $\vec{r}$ in whose
summation the term $x^{\vec{k}} \otimes x^{\vec{l}}$ will occur,
and there it occurs exactly once. The coefficient of the monomial
tensor $x^{\vec{r}} \otimes x^{\vec{s}}$ on the right hand side is
thus ${\vec{r}+\vec{s} \choose \vec{s}} c_{ij}^{(\vec{r} +
\vec{s})}$. Then we have
\[ {\vec{r}+\vec{s} \choose \vec{s}} c_{ij}^{(\vec{r} + \vec{s})} = \sum_k
c_{ik}^{\vec{r}} c_{kj}^{\vec{s}} \] For every $i,j,\vec{r}$ and
$\vec{s}$.  But the right hand side is simply the
$(i,j)^{\text{th}}$ entry of the matrix
$(c_{ij})^{\vec{r}}(c_{ij})^{\vec{s}}$, and the left hand side is
the $(i,j)^{\text{th}}$ entry of the matrix ${\vec{r}+\vec{s}
\choose \vec{s}} (c_{ij})^{\vec{r}+\vec{s}}$.  Equation
\ref{comoduleequation2} is therefore equivalent to the matrix
equality
\begin{equation}
\index{fundamental relation!for $G_a^\infty$}
\label{GaInfinityfundamentalrelation}(c_{ij})^{\vec{r}}(c_{ij})^{\vec{s}}
= {\vec{r}+\vec{s} \choose \vec{s}}(c_{ij})^{(\vec{r}+\vec{s})}
\end{equation}
for every $\vec{r}$ and $\vec{s}$.  This equation, along with
$(c_{ij})^{\vec{0}} = \text{Id}$, and the requirement that
$(c_{ij})^{\vec{r}}$ should vanish for all but finitely many
$\vec{r}$, are necessary and sufficient for a collection of
matrices $(c_{ij})^{\vec{r}}$ to define a representation of
$G_a^\infty$ or $G_a^n$ over any field $k$.

In the case of $G_a = G_a^1$ the above equations reduces to
\[ (c_{ij})^0 = \text{Id} \]
and
\begin{equation}
\index{$(c_{ij})$} \index{fundamental relation!for $G_a$}
\label{Gafundamentalrelation} (c_{ij})^r (c_{ij})^s = {r+s \choose
r} (c_{ij})^{r+s}
\end{equation}
Again, these equations, along with the requirement that
$(c_{ij})^r$ vanish for large $r$, are necessary and sufficient to
define a representation of $G_a$.

For the rest of this chapter we restrict to the case of $G_a$, and
treat the case of zero and positive characteristic separately.

\section{Characteristic Zero}

Let $k$ have characteristic zero.

\begin{thm}
\label{Gacharzerothm} Every $n$-dimensional representation of
$G_a$ over $k$ is of the form $e^{xN}$, where $N$ is an $n \times
n$ nilpotent matrix with entries in $k$.  Further, any $n \times
n$ nilpotent matrix over $k$ gives a representation according to
this formula.
\end{thm}

\begin{proof}
By $e^{xN}$ we mean the sum
\[ 1 + xN + \frac{x^2 N^2}{2} + \ldots + \frac{x^m N^m}{m!} \]
which of course terminates since $N$ is nilpotent.  Obviously such
a formula gives a representation, in view of the matrix identity
$e^{xN}e^{yN} = e^{(x+y)N}$ (see lemma
\ref{exponentialsAndWhatnotLemma}). For the converse let $\rho:e_j
\mapsto \sum_i e_i \otimes a_{ij}$ be any representation, and set
$N = (c_{ij})^1$. Then examination of equation
\ref{Gafundamentalrelation} gives, for any $r > 0$
\[ (c_{ij})^r = \frac{N^r}{r!} \]
Since $(c_{ij})^r$ must vanish for large $r$, $N^r = 0$ for large
$r$, whence $N$ is nilpotent. Recalling that the matrix formula
for this representation is
\[ (c_{ij})^0 + (c_{ij})^1 x + \ldots + (c_{ij})^n x^n \]
where $n$ is the largest non-zero $(c_{ij})$, this representation
is indeed of the form $e^{xN}$.
\end{proof}

In the preceding proof we used the fact that $\text{char}(k) = 0$
in the form of assuming that $\frac{1}{r!}$ is defined for all
non-negative integers $r$.

Example: take the $4 \times 4$ nilpotent matrix $N = $
\[
\left(
\begin{array}{cccc}
  0 & 1 & 0 & 0 \\
  0 & 0 & 1 & 0 \\
  0 & 0 & 0 & 1 \\
  0 & 0 & 0 & 0 \\
\end{array}
\right)
\]
Then the representation it defines is $e^{xN} = $
\[
\left(
\begin{array}{cccc}
  1 & x & \frac{x^2}{2} & \frac{x^3}{6} \\
   & 1 & x & \frac{x^2}{2} \\
   &  & 1 & x \\
   &  &  & 1 \\
\end{array}
\right)
\]

\begin{prop}
Two representation $e^{xN},e^{xM}$ are isomorphic if and only if
the nilpotent matrices $N$ and $M$ are conjugate.
\end{prop}

\begin{proof}
Sufficiency is immediate.  For the necessity, suppose $e^{xN}$ and
$e^{xM}$ are isomorphic via a base change matrix $P$.  Then
\[ Pe^{xN} P^{-1} = 1 + xPNP^{-1} + \ldots = 1 + xM + \ldots \]
which forces $PNP^{-1} = M$.
\end{proof}

\begin{cor}
Let $k$ have characteristic zero.
\begin{enumerate}
\item{For an $n$-dimensional representation of $G_a$ over $k$, the
polynomials occurring as matrix entries cannot have degree larger
than $n-1$} \item{For a given dimension $n$, there are only
finitely many non-isomorphic $n$-dimensional representations of
$G_a$ over $k$}

\end{enumerate}
\end{cor}

\begin{proof}
If $N$ is nilpotent, it is nilpotent of order no greater than the
dimension of the representation; this proves 1.  Any nilpotent
matrix is conjugate to a Jordan matrix with $0's$ on the main
diagonal, of which there are only finitely many for a given
dimension; this proves 2.
\end{proof}

\section{Characteristic $p$}

Let $k$ have characteristic $p > 0$.  In this section we prove

\begin{thm}
\label{Gacharptheorem} Any representation of $G_a$ over $k$ is of
the form
\[ e^{xN_0}e^{x^pN_1}e^{x^{p^2}N_2} \ldots e^{x^{p^m}N_m} \]
with each of the factors commuting (and so necessarily and
sufficiently all of the $N_i$ commuting), and each $N_i$ being
nilpotent of order $\leq p$.  Further, any finite collection of
commuting, $p$-nilpotent matrices defines a representation
according to the above formula.
\end{thm}

Note that, for a sequence of assignments to the $(c_{ij})^n$ as in
equation \ref{Gafundamentalrelation}, satisfying the relations
$(c_{ij})^0 = \text{Id}$ along with
\[ (c_{ij})^r (c_{ij})^s = {r+s \choose r} (c_{ij})^{r+s} \]
are still necessary and sufficient to determine a representation,
no matter the characteristic. The gist of the theorem is that,
while in characteristic $0$ the entire representation is
determined by the nilpotent matrix $(c_{ij})^1$, in characteristic
$p$ the binomial coefficient ${ r+s \choose r}$ can and often does
vanish, which relaxes the relations that the $(c_{ij})$ matrices
must satisfy. This is to be expected, since for example
\[
\left(
\begin{array}{cc}
1 & x^{p^m} \\
0 & 1 \\
\end{array}
\right)
\]
defines a perfectly respectable representation of $G_a$ in
characteristic $p$, for any $m$ as large as we like.  This
illustrates that, for example, $(c_{ij})^{p^m}$ is not determined
by the matrix $(c_{ij})^1$.

It turns out that one can choose the matrices $(c_{ij})^{p^m}$
freely, subject only to the condition that they commute and are
nilpotent of order $\leq p$, and that these completely determine
the rest of the representation. These $(c_{ij})^{p^m}$ matrices
for $m \geq 1$ should thus be thought of as accounting for the
``Frobenius'' parts of the representation.

To start, we need some number theory concerning the behavior of
binomial and multinomial coefficients modulo a prime. The notation
${n \choose a,b,\ldots, z}$ denotes the usual multinomial
expression $\frac{n!}{a!b! \ldots z!}$.

\begin{thm}\index{Lucas' theorem} \label{Lucas'Thm} (Lucas' theorem)
Let $n$ and $a,b, \ldots ,z$ be non-negative integers with $a + b
+ \ldots + z = n$, $p$ a prime.  Write $n = n_m p^m + n_{m-1}
p^{m-1} + \ldots + n_0$ in $p$-ary notation, similarly for $a, b,
\ldots , z$. Then, modulo $p$,
\[ {n \choose a,b, \ldots, z}= \left\{
\begin{array}{c}
  0 \hspace{.3cm}\text{ \emph{if for some }$i$, $a_i + b_i + \ldots + z_i \geq p$} \\
  {n_0 \choose a_0, b_0, \ldots, z_0}{n_1 \choose a_1, b_1,
\ldots, z_1} \ldots {n_m \choose a_m, b_m, \ldots , z_m} \hspace{.3cm} \text{ \emph{otherwise}} \\
\end{array}
\right.
\]

\end{thm}

Some corollaries we will need later:

\begin{cor} Let $p$ be a prime, $n,r$ and $s$ non-negative
integers. \label{Lucas'Cor}
\begin{enumerate}
\item{The binomial coefficient ${n \choose r}$ is non-zero if and
only if every $p$-digit of $n$ is greater than or equal to the
corresponding digit of $r$} \item{If $n$ is not a power of $p$,
then for some $0 < r < n$, ${n \choose r}$ is non-zero} \item{If
$n$ is a power of $p$, for every $0 < r < n$, ${n \choose r}$ is
zero} \item{${r+s \choose r}$ is non-zero if and only if there is
no $p$-digit rollover (i.e.~carrying) for the sum $r+s$.}

\end{enumerate}
\end{cor}

See \cite{LucasThm} for a proof of these facts.

\begin{thm}
A representation of $G_a$ over $k$ given by the matrices
$(c_{ij})^n$ is completely determined by the assignments
\[ (c_{ij})^{p^0} = X_0, \hspace{.3cm} (c_{ij})^{p^1} = X_1, \hspace{.3cm} \ldots,
(c_{ij})^{p^m} = X_m \] (with the understanding that
$(c_{ij})^{p^k} = 0$ for $k > m$). The $X_i$ must necessarily
commute with each other and satisfy $X_i^{p} = 0$.

\end{thm}

\begin{proof}
All is proved by examining equation \ref{Gafundamentalrelation}.
It is easy to see by induction that the values of $(c_{ij})^n$ are
determined by the $X_i$.  If $n$ is a power of $p$ then its value
is given, and if not, by 2.~of corollary \ref{Lucas'Cor} there is
$0 < r < n$ with ${n \choose r} \neq 0$ forcing
\[(c_{ij})^n  = {n \choose r}^{-1} (c_{ij})^r (c_{ij})^{n-r} \]

For the commutativity condition, if $n \neq m$, then by theorem
\ref{Lucas'Thm} ${p^m + p^n \choose p^m}$ is non-zero, and we must
have
\[ (c_{ij})^{p^m+p^n} = {p^m + p^n \choose p^m}^{-1} X_m X_n = {p^m +
p^n \choose p^n}^{-1} X_n X_m \]

To prove the nilpotency claim, consider
\begin{equation*}
\begin{split}
 (c_{ij})^{p^m} &= X_m \\
 (c_{ij})^{2p^m} &= {2p^m \choose p^m}^{-1} X_m^2 \\
  &\vdots \\
 (c_{ij})^{(p-1)p^m} &= \left[ \prod_{k=1}^{p-1} {kp^m \choose
p^m} \right]^{-1} X_m^{p-1}
\end{split}
\end{equation*}
 Noting that there is carrying in
computing the sum $(p^{m+1} - p^{m}) + p^m$, corollary
\ref{Lucas'Cor} tells us that ${p^{m+1} \choose p^m} = 0$, and we
have
\[ 0 = {p^{m+1} \choose p^m} X_{m+1} = (c_{ij})^{(p-1)p^m}
(c_{ij})^{p^m} = \left[ \prod_{k=2}^{p-1} {kp^m \choose p^m}
\right]^{-1} X_m^p \]
forcing $X_m^p = 0$.

\end{proof}

We have shown thus far that commutativity and $p$-nilpotency of
the $X_i$ are necessary to define a representation; we must now
show sufficiency.  This will become clear once we have a closed
expression for $(c_{ij})^n$ in terms of the $X_i$.  For $n = n_m
p^m + n_{m-1} p^{m-1} + \ldots + n_0$ in $p$-ary notation, let
\index{$\Gamma(n)$} $\Gamma(n) = n_0! n_1! \ldots n_m!$.
Obviously $\Gamma(n)$ is always non-zero mod $p$.

\begin{prop}
\label{Gacharpprop} Let $X_0, \ldots, X_m$ be pair-wise commuting
$p$-nilpotent matrices, and let $n= n_m p^m + \ldots + n_0$ be the
$p$-ary expansion of $n$. Then the assignment
\[ (c_{ij})^n = \Gamma(n)^{-1} X_0^{n_0} X_1^{n_1} \ldots
X_m^{n_m} \] defines a representation of $G_a$ over $k$.

\end{prop}

\begin{proof}

Obviously these assignments satisfy $(c_{ij})^0 = \text{Id}$, with
$(c_{ij})^n$ vanishing for large $n$ (we define $X_i = 0$ for $i >
m$).  Then it remains to check the equation
\[ {r+s \choose r} (c_{ij})^{r+s} = (c_{ij})^r (c_{ij})^s \]
Let $r=r_m p^m + \ldots + r_0$, $s=s_mp^m + \ldots + s_0$, and
suppose first that ${r+s \choose r} = 0$.  This means, by
corollary \ref{Lucas'Cor}, that there is some digit rollover in
the computation of $r+s$, i.e.~$r_i + s_i \geq p $ for some $i$.
Looking at the right hand side in view of the given assignments we
see that $X_i^{r_i + s_i}$ will occur as a factor.  But $X_i$ is
nilpotent of order less than or equal to $p$, so the right hand
side will be zero as well.

On the other hand, if ${r+s \choose r} \neq 0$ let $r+s=z_m p^m +
\ldots +z_0$, so that necessarily $r_i + s_i = z_i$ for all $i$.
Then the given assignments give the same power of each $X_i$ on
either side, so it only remains to check the coefficients.  This
reduces to
\[ {r+s \choose r}\Gamma(r) \Gamma(s) = \Gamma(r+s) \]
After applying theorem \ref{Lucas'Thm} for the term ${r+s \choose
r}$, the equality is clear.

\end{proof}

We can now prove theorem \ref{Gacharptheorem}.  Let $X_0, \ldots,
X_m$ be commuting, $p$-nilpotent matrices over $k$.  Then
according to the previous theorem, the representation they define
is
\begin{equation*}
\begin{split}
 \sum_{r=0}^{p^{m+1}-1} (c_{ij})^{r} x^r &= \sum_{r=0}^{p^{m+1}-1}
\Gamma(r)^{-1} X_0^{r_0} X_1^{r_1} \ldots X_m^{r_m} x^{r_0 + r_1p
+ \ldots + r_m p^m} \\
 &= \sum_{r_0=0}^{p-1} \sum_{r_1=0}^{p-1} \ldots
\sum_{r_m=0}^{p-1} \frac{1}{r_0!}\ldots\frac{1}{r_m!} X_0^{r_0}
\ldots X_m^{r_m} x^{r_0} x^{r_1 p} \ldots x^{r_m p^m} \\
&= \left(\sum_{r_0=0}^{p-1} \frac{1}{r_0!} X_0^{r_0}
x^{r_0}\right)\left(\sum_{r_1=0}^{p-1} \frac{1}{r_1!}X_1^{r_1}
x^{r_1 p}\right) \ldots \left(\sum_{r_m=0}^{p-1}
\frac{1}{r_m!}X_m^{r_m} x^{r_mp^m}\right) \\
&= e^{xX_0}e^{x^p X_1} \ldots e^{x^{p^m} X_m}
\end{split}
\end{equation*}
 as claimed.  All of
the factors of course commute, since the $X_i$ do.

Example: define the following matrices:
\[
X_0 = X_1 =
\left(%
\begin{array}{ccc}
  0 & 1 & 0 \\
  0 & 0 & 1 \\
  0 & 0 & 0 \\
\end{array}%
\right), \hspace{.5cm} X_2 =
\left(%
\begin{array}{ccc}
  0 & 0 & 2 \\
  0 & 0 & 0 \\
  0 & 0 & 0 \\
\end{array}%
\right)
\]
These all commute and are nilpotent of order less than or equal to
$3$. Then the representation they define in characteristic $3$ is
$e^{xX_0} e^{x^3 X_1} e^{x^9 X_2} = $
\[
\left(%
\begin{array}{ccc}
  1 & x+x^3 & 2x^2+x^4+2x^6+2x^9 \\
   & 1 & x+x^3 \\
   &  & 1 \\
\end{array}%
\right)
\]

With a view toward defining the height-restricted ultraproduct
later, we make the following simple but important observation.

\begin{thm} \textit{ }
\label{Gacharzeropanaloguetheorem}
\begin{enumerate}

\item{Let $k$ have characteristic zero.  Then the $n$-dimensional
representations of $G_a^\infty$ over $k$ are in $1-1$
correspondence with the finite ordered sequences $N_i$ of $n
\times n$ commuting nilpotent matrices over $k$, according to the
formula
\[ e^{x_0N_0} e^{x_1N_1} \ldots e^{x_mN_m} \]
} \item{Let $k$ have positive characteristic $p$.  Then if $p >>
n$, the $n$-dimensional representations of $G_a$ over $k$ are in
$1-1$ correspondence with the finite ordered sequences $N_i$ of $n
\times n$ commuting nilpotent matrices over $k$, according to the
formula
\[ e^{xN_0} e^{x^p N_1} \ldots e^{x^{p^m} N_m} \]
}

\end{enumerate}

\end{thm}

\begin{proof}
By the work done in chapter
\ref{representationtheoryofdirectproductschapter}, all
representations of $G_a^\infty$ over $k$ are given by commuting
finite products of individual representations of $G_a$ over $k$.
It is easy to see that the representations $e^{xN}$ and $e^{xM}$
commute if and only if $N$ and $M$ do; this proves 1.  2.~follows
immediately from \ref{Gacharptheorem}, with the additional
realization that if $p$ is greater than or equal to dimension,
being nilpotent and $p$-nilpotent are identical concepts.

\end{proof}

\chapter{The Heisenberg Group}

\label{TheHeisenbergGroup}

In this chapter we investigate the group \index{$H_1$} $H_1$ of
all $3 \times 3$ unipotent upper triangular matrices
\[
\left(%
\begin{array}{ccc}
  1 & x & z \\
  0 & 1 & y \\
  0 & 0 & 1 \\
\end{array}%
\right)
\]
for arbitrary $x,y$ and $z$.  Our intent is to prove a theorem
analogous to theorem \ref{Gacharzeropanaloguetheorem} for the
group $G_a$; that is, if one is content to keep $p$ large with
respect to dimension (larger than twice the dimension in fact),
modules for $H_1$ in characteristic $p$ `look like' modules for
$H_1^{\infty}$ in characteristic zero. This result will be more
precisely stated in a later chapter, as we consider the
`height-restricted ultraproduct' of the categories
$\text{Rep}_{k_i} H_1$ for a collection of fields $k_i$ of
increasing positive characteristic.  The results of this chapter
are perhaps more surprising than the previous in that, unlike
modules for $G_a$, modules for $H_1$ in characteristic $p <<
\text{dim}$ look hardly at all like representation for
$H_1^{\infty}$ in characteristic zero; it is only when $p$ becomes
large enough with respect to dimension that the resemblance is
apparent.

It was the author's original intent to prove these results for all
of the generalized Heisenberg groups $H_n$, but that is not
attempted here. Nonetheless, we shall at least work out the
fundamental relations for the groups $H_n$, as all we need for
$H_1$ shall arise as a special case.

\section{Combinatorics for $H_n$}

Let \index{$H_n$} $H_n$, the $n^{\text{th}}$ generalized
Heisenberg group, be the group of all $(n+2) \times (n+2)$
matrices of the form
\[
\left(%
\begin{array}{ccccc}
  1 & x_1 & \ldots & x_n & z \\
   & 1 &  & 0 & y_1 \\
   &  & \ddots &  & \vdots \\
   &  &  & 1 & y_n \\
   &  &  &  & 1 \\
\end{array}%
\right)
\]
That is, upper triangular matrices with free variables on the top
row, right-most column, $1's$ on the diagonal, and $0's$
elsewhere. The Hopf algebra for $H_n$ is
\begin{gather*}
A = k[x_1, \ldots, x_n, y_1, \ldots, y_n, z] \\
 \Delta: x_i \mapsto 1 \otimes x_i + x_i \otimes 1, \hspace{.5cm}
y_i \mapsto 1 \otimes y_i + y_i \otimes 1, \hspace{.5cm}  z
\mapsto z \otimes 1 + 1 \otimes z + \sum_{i=1}^n x_i \otimes y_i
\\  \varepsilon: x_i, y_i, z \mapsto 0
\end{gather*}

Let us adopt the notation, for an $n$-tuple of non-negative
integers $\vec{r}$, $x^{\vec{r}} = x_1^{r_1} \ldots x_n^{r_n}$,
and similarly for $y^{\vec{r}}$.  Let $(V,\rho)$ be a comodule for
$G$ with basis $\{e_i\}$, and $\rho$ given by
\[ \rho:e_j \mapsto \sum_i e_i \otimes a_{ij} \]
and write
\[ a_{ij} = \sum_{(\vec{r},\vec{s},t)}
c_{ij}^{(\vec{r},\vec{s},t)} x^{\vec{r}} y^{\vec{s}} z^t \]

\begin{lem}
With $a_{ij}$ as above, $\Delta(a_{ij})$ is equal to
\begin{gather*}
 \hspace{-2cm} \sum_{{(\vec{r}_1, \vec{s}_1, a)} \atop
{(\vec{r}_2, \vec{s}_2, b)}} \left( \sum_{0 \leq \vec{t} \leq
\vec{r}_1, \vec{s}_2} {\vec{r}_1 + \vec{r}_2 - \vec{t} \choose
\vec{r}_2}{\vec{s}_1+\vec{s}_2-\vec{t} \choose \vec{s}_1}{a+b +
\abs{\vec{t}} \choose a, b, t_1, \ldots, t_n}
c_{ij}^{(\vec{r}_1+\vec{r}_2 -
\vec{t},\vec{s}_1+\vec{s}_2-\vec{t},a+b+\abs{\vec{t}})} \right) \\
 x^{\vec{r}_1} y^{\vec{s}_1} z^a \otimes x^{\vec{r}_2} y^
{\vec{s}_2} z^b
\end{gather*}

\end{lem}

Remark: the summation condition $0 \leq \vec{t} \leq \vec{r}_1,
\vec{s}_2$ is understood to mean all $\vec{t}$ such that every
entry in $\vec{t}$ is no larger than either of the corresponding
entries of in $\vec{r}_1$ or $\vec{s}_2$.  $\abs{\vec{t}}$ means
$t_1 + \ldots + t_n$.

\begin{proof}
We start by computing
\[ \Delta(a_{ij}) = \sum_{\vec{r},\vec{s},t}
c_{ij}^{(\vec{r},\vec{s},\vec{t})} \Delta(x^{\vec{r}})
\Delta(y^{\vec{s}}) \Delta(z)^t
\]
We have
\begin{equation*}
\begin{split}
 \Delta(x^{\vec{r}}) &= \Delta(x_1)^{r_1} \ldots \Delta(x_n)^{r_n}
 \\
&= (x_1 \otimes 1 + 1 \otimes x_1)^{r_1} \ldots (x_n \otimes 1 + 1
\otimes x_n)^{r_n} \\
&= \left(\sum_{l_1 + m_1 = r_1}{l_1 + m_1 \choose m_1} x_1^{l_1}
\otimes x_1^{m_1}\right) \ldots \left(\sum_{l_n + m_n = r_n}{l_n +
m_n \choose m_n} x_n^{l_n} \otimes x_n^{m_n} \right) \\
&= \sum_{\vec{l} + \vec{m} = \vec{r}} {\vec{l} + \vec{m} \choose
\vec{m}} x^{\vec{l}} \otimes x^{\vec{m}}
\end{split}
\end{equation*}
where ${\vec{l} + \vec{m} \choose \vec{m}}$ is shorthand for the
product ${l_1 + m_1 \choose m_1} \ldots {l_n + m_n \choose m_n}$.
Similarly we have
\[ \Delta(y^{\vec{s}}) = \sum_{\vec{f}+\vec{g} =
\vec{s}}{\vec{f}+\vec{g} \choose \vec{g}} y^{\vec{f}} \otimes
y^{\vec{g}} \] and
\begin{equation*}
\begin{split}
\Delta(z^t) &= (z \otimes 1 + 1 \otimes z + \sum_i x_i \otimes
y_i)^t \\
&= \sum_{a+b+c =t}{a+b+c \choose a,b,c} (z \otimes 1)^a(1\otimes
z)^b (\sum_i x_i \otimes y_i)^c \\
&= \sum_{a+b+c=t}{a+b+c \choose a,b,c} z^a \otimes z^b
\sum_{\abs{\vec{t}}=c}{\abs{\vec{t}} \choose t_1, \ldots, t_n}
x^{\vec{t}} \otimes y^{\vec{t}} \\
&= \sum_{a+b+\abs{\vec{t}} = t}{a+b+\abs{t} \choose
a,b,t_1,\ldots,t_n} x^{\vec{t}} z^a \otimes y^{\vec{t}} z^b
\end{split}
\end{equation*}
where $\abs{\vec{t}} \stackrel{\text{def}}{=} t_1 + \ldots + t_n$.
Thus $\Delta(a_{ij})$ is equal to
\[ \sum_{\vec{r},\vec{s},t}
c_{ij}^{\vec{r},\vec{s},t} \sum_{{\vec{l}+\vec{m} = \vec{r}} \atop
{{\vec{f}+\vec{g} = \vec{s}} \atop {a+b+\abs{t} = t}}} {\vec{l} +
\vec{m} \choose \vec{m}}{\vec{f}+\vec{g} \choose
\vec{g}}{a+b+\abs{\vec{t}} \choose a,b,t_1, \ldots, t_n}
x^{\vec{l}+\vec{t}} y^{\vec{f}} z^a \otimes x^{\vec{m}}
y^{\vec{g}+\vec{t}} z^b
\]
We seek to write this as a sum over distinct monomial tensors,
i.e.~in the form
\[ \sum_{\vec{r}_1,\vec{s}_1,a \atop \vec{r}_2, \vec{s}_2, b} \chi\left({\vec{r}_1,\vec{s}_1,a \atop \vec{r}_2, \vec{s}_2, b}\right)
 x^{\vec{r}_1} y^{\vec{s}_1} z^a \otimes x^{\vec{r}_2}
y^{\vec{s}_2} z^b \] for some collection of scalars $\chi$, which
is merely a question of how many times a given monomial tensor
shows up as a term in our summation expression of
$\Delta(a_{ij})$.  That is, how many solutions are there to
\[ x^{\vec{l}+\vec{t}} y^{\vec{f}} z^a \otimes x^{\vec{m}}
y^{\vec{g}+\vec{t}} z^b= x^{\vec{r}_1} y^{\vec{s}_1} z^a \otimes
x^{\vec{r}_2} y^{\vec{s}_2} z^b \] Clearly the values of
$\vec{f},\vec{m},a$ and $b$ are determined. Further, once one
chooses $\vec{t}$, the values of both $\vec{l}$ and $\vec{g}$
follow; thus, we can parameterize by $\vec{t}$.  For the $n$-tuple
$\vec{t}$ to induce a solution, it is necessary and sufficient
that none of its entries be larger than the corresponding entries
in $\vec{r}_1$ or $\vec{s}_2$; we shall express this condition by
$\vec{0} \leq \vec{t} \leq \vec{r}_1,\vec{s}_2$.  Then
\[ \chi\left({\vec{r}_1,\vec{s}_1,a \atop \vec{r}_2, \vec{s}_2,
b}\right) = \sum_{\vec{0} \leq \vec{t} \leq \vec{r}_1,\vec{s}_2}
c_{ij}^{(\vec{r},\vec{s},t)}{\vec{l}+\vec{m} \choose
\vec{m}}{\vec{f} + \vec{g} \choose \vec{g}}{a+b+\abs{\vec{t}}
\choose a,b,t_1, \ldots, t_n} \] and upon substituting
\[ \vec{m} = \vec{r}_2 \eqspace \vec{f} = \vec{s}_1 \eqspace
\vec{s} = \vec{s}_1 + \vec{s}_2 - \vec{t} \eqspace \vec{r} =
\vec{r}_1 + \vec{r}_2 - \vec{t} \eqspace t = a+b+\abs{\vec{t}} \]
we get
\begin{gather*}
\chi\left({\vec{r}_1,\vec{s}_1,a \atop \vec{r}_2, \vec{s}_2,
b}\right) \\
 = \sum_{0 \leq \vec{t} \leq \vec{r}_1, \vec{s}_2} {\vec{r}_1 +
\vec{r}_2 - \vec{t} \choose \vec{r}_2}{\vec{s}_1+\vec{s}_2-\vec{t}
\choose \vec{s}_1}{a+b + \abs{\vec{t}} \choose a, b, t_1, \ldots,
t_n} c_{ij}^{(\vec{r}_1+\vec{r}_2 -
\vec{t},\vec{s}_1+\vec{s}_2-\vec{t},a+b+\abs{\vec{t}})}
\end{gather*}
which proves the lemma.

\end{proof}

\begin{thm}
\label{Hnfundamentalrelationthm} A finite collection of $(c_{ij})$
matrices defines a module for $H_n$ if and only if
$(c_{ij})^{(\vec{0}, \vec{0},0)} = \text{Id}$, and for all
$\vec{r}_1,\vec{r}_2,a,\vec{s}_1, \vec{s}_2, b$, the following
matrix equation holds: \index{fundamental relation!for $H_n$}
\begin{gather}
\label{Hnfundamentalrelation}
 (c_{ij})^{(\vec{r}_1,\vec{s}_1,a)}
(c_{ij})^{(\vec{r}_2,\vec{s}_2,b)} \\
 =\sum_{0 \leq \vec{t} \leq \vec{r}_1, \vec{s}_2} {\vec{r}_1 +
\vec{r}_2 - \vec{t} \choose \vec{r}_2}{\vec{s}_1+\vec{s}_2-\vec{t}
\choose \vec{s}_1}{a+b + \abs{\vec{t}} \choose a, b, t_1, \ldots,
t_n} (c_{ij})^{(\vec{r}_1+\vec{r}_2 -
\vec{t},\vec{s}_1+\vec{s}_2-\vec{t},a+b+\abs{\vec{t}})} \nonumber
\end{gather}

\end{thm}

\begin{proof}
This follows by matching coefficients for the equation
$\Delta(a_{ij}) = \sum_k a_{ik} \otimes a_{kj}$.  The coefficient
of the monomial tensor $x^{\vec{r}_1} y^{\vec{s}_1} z^a \otimes
x^{\vec{r}_2} y^{\vec{s}_2} z^b$ for $\Delta(a_{ij})$ is the right
hand side of the above equation, as proved in the previous lemma,
while for $\sum_k a_{ik} \otimes a_{kj}$ it is the left hand side
of the above equation, as is easy to verify.
\end{proof}

\section{Combinatorics for $H_1$}

The Hopf algebra $(A,\Delta,\varepsilon)$ for the group $H_1$ over
the field $k$ is
\begin{gather*}
A = k[x,y,z] \\
 \Delta: x \mapsto 1 \otimes x+ x \otimes 1, \hspace{.5cm} y
\mapsto 1 \otimes y + y \otimes 1, \hspace{.5cm}
 z \mapsto 1 \otimes z + x \otimes y + z \otimes 1 \\
 \varepsilon: x,y,z \mapsto 0
\end{gather*}
  Let $V$ be a finite dimensional
vector space over $k$, $\rho:V \rightarrow V \otimes A$ a
$k$-linear map.  Fix a basis $\{e_i\}$ of $V$, and write
\[
\rho:e_j \mapsto \sum_i e_i \otimes a_{ij}
\]
Each $a_{ij} \in A$ is a polynomial in the variables $x,y$ and
$z$, so write
\[ a_{ij} = \sum_{\vec{r}} c_{ij}^{\vec{r}} x^{r_1}y^{r_2}z^{r_3}
\]
where the summation is over all $3$-tuples of non-negative
integers, remembering of course that $c_{ij}^{\vec{r}} = 0$ for
all but finitely many $\vec{r}$.

\begin{thm}
A finite collection of $(c_{ij})^{\vec{r}}$ matrices defines a
module for $H_1$ if and only if  they satisfy $(c_{ij})^{(0,0,0)}
= \text{Id}$, and for all $3$-tuples $\vec{r}$ and $\vec{s}$

\index{fundamental relation!for $H_1$}
\begin{equation} \index{$(c_{ij})$}
\label{H1fundamentalrelation} (c_{ij})^{\vec{r}}(c_{ij})^{\vec{s}}
= \sum_{l=0}^{\text{min}(r_1,s_2)}{ r_1 + s_1 -l \choose s_1}
{r_2+s_2-l \choose r_2}{r_3 + s_3 + l \choose r_3, s_3, l}
(c_{ij})^{\vec{r} + \vec{s} + (-l,-l,l)}
\end{equation}

\end{thm}

\begin{proof}
Apply theorem \ref{Hnfundamentalrelationthm} to the case of $n=1$.
\end{proof}

We will also make good use of the fact that $G$ contains three
copies of the additive group $G_a$, one for every coordinate. This
says that, for example, the collection of all matrices of the form
$(c_{ij})^{(r,0,0)}$ (the matrices representing the $x$-coordinate
in the representation) must in isolation satisfy equation
\ref{Gafundamentalrelation}:
\[ (c_{ij})^{(r,0,0)}(c_{ij})^{(s,0,0)} = {r+s \choose r}
(c_{ij})^{(r+s,0,0)} \] An identical statement holds for matrices
of the form $(c_{ij})^{(0,r,0)}$ and those of the form
$(c_{ij})^{(0,0,r)}$, the $y$ and $z$ parts respectively.  These
relations could of course just as well have been read off of
equation \ref{H1fundamentalrelation}.

From here on we treat the cases of zero and prime characteristic
separately.

\section{Characteristic Zero}

Let $k$ be a field of characteristic zero, and $(V,\rho)$ a
representation of $H_1$ as in the previous section.  Set $X =
(c_{ij})^{(1,0,0)}$, $Y = (c_{ij})^{(0,1,0)}$, $Z =
(c_{ij})^{(0,0,1)}$.

\begin{thm}
\label{H1CharZeroThm} A representation of $H_1$ over $k$ is
completely determined by the assignments $X$ and $Y$.  Necessarily
$Z = [X,Y]$, each of $X$ and $Y$ must commute with $Z$, and $X$,
$Y$ and $Z$ must all be nilpotent.  Further, any $X$ and $Y$
satisfying these relations defines a representation of $G$ over
$k$.

\end{thm}

\begin{proof}
We know from our previous work with the additive group $G_a$ that
the following identities must hold:
\[ (c_{ij})^{(r,0,0)} = \frac{1}{r!} X^r \hspace{.8cm} (c_{ij})^{(0,r,0)} = \frac{1}{r!} Y^r \hspace{.8cm}
 (c_{ij})^{(0,0,r)} = \frac{1}{r!} Z^r \]
We work with the fundamental relation for $H_1$, equation
\ref{H1fundamentalrelation}:
\[(c_{ij})^{\vec{r}}(c_{ij})^{\vec{s}} = \sum_{l=0}^{\text{min}(r_1, s_2)} {r_1 + s_1 -l \choose s_1}{r_2 +
s_2 -l \choose r_2}{r_3 + s_3 +l \choose r_3, s_3, l}
(c_{ij})^{(r_1 + s1 -l, r_2+s_2-1,l)} \]
We have
\[ (c_{ij})^{(0,m,0)}(c_{ij})^{(n,0,0)} = \sum_{l=0}^{0}{n -l
\choose n} {m-l \choose m}{l \choose l} (c_{ij})^{(n-l,m-l,l)} \]
which says that
\[ (c_{ij})^{(n,m,0)} = \frac{1}{m!n!} Y^m X^n \]
Using again the fundamental relation, we also have
\[ (c_{ij})^{(n,m,0)}(c_{ij})^{(0,0,k)} = (c_{ij})^{(n,m,k)} \]
which together with the last equation gives
\[ (c_{ij})^{(n,m,k)} = \frac{1}{n!m!k!}Y^m X^n Z^k \]
Thus, all of the $(c_{ij})$ are determined by $X$, $Y$ and $Z$,
according to the above formula. Further,
\[ XY = \sum_{l=0}^1 {1 -l \choose
0}{1 -l \choose 0}{l \choose l} (c_{ij})^{1 -l, 1-1,l)} = YX + Z
\]
and so $Z = [X,Y]$ as claimed.  Each of $Y$ and $X$ must commute
with $Z$, for if we apply the fundamental relation to each of $XZ$
and $ZX$, in each case we obtain $(c_{ij})^{(1,0,1)}$, showing $XZ
= ZX$, and an identical computation shows $YZ = ZY$.  And by our
work on $G_a$ we know that each of $X$, $Y$ and $Z$ must be
nilpotent.

We must now show sufficiency of the given relations.  Let $X$, $Y$
and $Z$ be any three nilpotent matrices satisfying $Z = XY - YX$,
with each of $X$ and $Y$ commuting with $Z$.  We need to show that
the fundamental relation, equation \ref{H1fundamentalrelation}, is
always satisfied.  We assign
\[ (c_{ij})^{(n,m,k)} = \frac{1}{n!m!k!} Z^k Y^m X^n \]
Since each of $X$, $Y$ and $Z$ are nilpotent, $(c_{ij})^{\vec{r}}$
will vanish for all by finitely many $\vec{r}$, as required. The
fundamental relation, with these assignments, reduces to (after
shuffling all coefficients to the right-hand side and some
cancellation)
\[ Z^{r_3 + s_3} Y^{r_2} X^{r_1} Y^{s_2} X^{s_1} =
\sum_{l=0}^{\text{min}(r_1, s_2)} l!{r_1 \choose l}{s_2 \choose l}
Z^{r_3+s_3 + l} Y^{r_2+s_2-l} X^{r_1+s_1-l} \] Each term in the
summation has the term $Z^{r_3+s_3}$ in the front and $X^{s_1}$ in
the rear, and so does the left-hand side.  So it suffices to show
\[ Y^{r_2} X^{r_1} Y^{s_2} = \sum_{l=0}^{\text{min}(r_1, s_2)} l!{r_1 \choose l}{s_2 \choose l}
Z^l Y^{r_2+s_2-l} X^{r_1-l} \] and since $Y$ commutes with $Z$,
the summation term (minus coefficients) can be written as
$Y^{r_2}Z^l Y^{s_2-l}X^{r_1-l}$. We can now take off the $Y^{r_2}$
term from the front of either side, so it suffices to show
\[ X^n Y^m = \sum_{l=0}^{\text{min}(n,m)} l!{n \choose l}{m
\choose l} Z^l Y^{m-l}X^{n-l} \] where we have renamed $r_1$ and
$s_2$ with the less cumbersome $n$ and $m$.

We proceed by a double induction on $n$ and $m$.  The case of $n$
or $m$ being zero is trivial, and if $n=m=1$, the above equation
is $XY=Z+YX$, which is true by assumption.  Consider then $X^nY$,
and by induction suppose that the equation holds for $X^{n-1} Y$,
so that $X^{n-1}Y = YX^{n-1} + (n-1)ZX^{n-2}$. Then using the
relation $XY = Z+YX$ and $X$ commuting with $Z$, we have
\begin{equation*}
\begin{split}
 X^nY &= X^{n-1}XY \\ &
 = X^{n-1}(Z+YX) \\
 &= ZX^{n-1} + (X^{n-1}Y)X \\
 &= ZX^{n-1} + (YX^{n-1} + (n-1)ZX^{n-2})X \\
 &= nZX^{n-1} + YX^n
\end{split}
\end{equation*}
 and
so the equation is true when $m=1$.  Now suppose that $m \leq n$,
so that $\text{min}(n,m) = m$.  Then we have
\begin{equation*}
\begin{split}
 X^n Y^m &= (X^nY)Y^{m-1} \\
 &= (YX^n + nZX^{n-1})Y^{m-1} \\
 &= Y(X^nY^{m-1}) + nZ(X^{n-1}Y^{m-1})
\end{split}
\end{equation*}
which by induction is equal to
\begin{gather*}
 Y \left(\sum_{l=0}^{m-1} l!{n \choose l}{m-1 \choose
l}Z^lY^{m-1-l}X^{n-l} \right) + n Z \left(\sum_{l=0}^{m-1}l!{n-1
\choose l}{m-1 \choose l}Z^lY^{m-1-l}X^{n-1-l} \right) \\
= \sum_{l=0}^{m-1} l! {n \choose l}{m-1 \choose
l}Z^lY^{m-l}X^{m-l} + \sum_{l=0}^{m-1} n l!{n-1 \choose l}{m-1
\choose l} Z^{l+1}Y^{m-1-l}X^{n-1-l} \\
= Y^mX^n + \sum_{l=1}^{m-1}
l!{n \choose l}{m-1 \choose l}Z^l Y^{m-l}X^{n-l} + \sum_{l=1}^m
n(l-1)!{n-1 \choose l-1}{m-1 \choose l-1}Z^l Y^{m-l}X^{n-l}
\end{gather*}
where, in the last step, we have chopped off the first term of the
first summation and shifted the index $l$ of the second summation.
If we chop off the last term of the second summation we obtain
\begin{gather*}
 = Y^mX^n + \sum_{l=1}^{m-1} l!{n \choose l}{m-1 \choose l}Z^l
Y^{m-l}X^{n-l} \\
+ \sum_{l=1}^{m-1} n(l-1)!{n-1 \choose l-1}{m-1 \choose l-1}Z^l
Y^{m-l}X^{n-l} + n(m-1)!{n-1 \choose m-1}{m-1 \choose m-1}Z^m
X^{n-m}
\end{gather*}
and upon merging the summations, we have
\begin{equation*}
\begin{split}
 &= Y^m X^n + \sum_{l=1}^{m-1} \left[l!{n \choose l}{m-1 \choose
l} + n(l-1)!{n-1 \choose l-1}{m-1 \choose l-1} \right] Z^l Y^{m-1}
X^{n-l} \\
& \qquad + n(m-1)!{n-1 \choose m-1}{m-1 \choose m-1} Z^mX^{n-m} \\
&= Y^m X^n + \sum_{l=1}^{m-1} \left[l!{n \choose l}{m-1 \choose l}
+ n(l-1)!{n-1 \choose l-1}{m-1 \choose l-1} \right] Z^l Y^{m-1}
X^{n-l} \\
& \qquad + n!{n \choose m}{m \choose m} Z^mX^{n-m}
\end{split}
\end{equation*}
 The two terms
outlying the summation are exactly the first and last terms of
what the fundamental relation predicts them to be. To finish then,
it suffices to show that the term in brackets is equal to $l!{n
\choose l}{m \choose l}$, which is a straightforward computation
left to the reader.  This completes the case of $m \leq n$, and
the case $n \geq m$ is hardly any different, and left to the
reader.

\end{proof}

\section{Characteristic $p$}

Here we are not interested in giving a complete combinatorial
classification of characteristic $p$ representations of $H_1$.
Rather, we shall only be interested in the case where $p$ is
sufficiently large when compared to the dimension of the module.
Doing so, we obtain a result analogous to theorem
\ref{Gacharzeropanaloguetheorem} for the group $G_a$, namely that
such representations `look like' representations of $H_1^n$ in
characteristic zero.

Let $(V,\rho)$ be a comodule for $G$ over the field $k$ of
characteristic $p>0$, given by the matrices $(c_{ij})^{\vec{r}}$
over $k$.  Again, matrices of the form $(c_{ij})^{(r,0,0)}$,
$(c_{ij})^{(0,r,0)}$ and $(c_{ij})^{(0,0,r)}$, the $x$, $y$ and
$z$ parts of the representation respectively, must in isolation
define representations of $G_a$ over $k$.  We know then that
(proposition \ref{Gacharpprop}), for example, all matrices of the
form $(c_{ij})^{(r,0,0)}$ are completely determined by the
assignments
\[ X_0 = (c_{ij})^{(p^0,0,0)}, \hspace{.5cm} X_1 = (c_{ij})^{(p^1,0,0)}, \ldots, X_m
= (c_{ij})^{(p^m,0,0)} \] and abide by the formula, for $r = r_m
p^m + r_{m-1} p^{m-1} + \ldots r_0$ in $p$-ary notation
\[ (c_{ij})^{(r,0,0)} = \Gamma(r)^{-1} X_0^{r_0} X_1^{r_1} \ldots
X_m^{r_m} \] and that the $X_i$ must commute and be $p$-nilpotent.
An identical statement holds for the matrices $Y_m =
(c_{ij})^{(0,p^m,0)}$ and $Z_m = (c_{ij})^{(0,0,p^m)}$.

From here on we adopt the notation $X_{(i)} = (c_{ij})^{(i,0,0)}$,
similarly for $Y_{(i)}$ and $Z_{(i)}$.  Note that
\index{$X_{(i)}$} $X_{(i)}$ and $X_i$ are not the same thing.

\begin{thm}
\label{H1CharpThm} Let $k$ have characteristic $p>0$.  A
representation of $H_1$ over $k$ is completely determined by the
$X_i$ and $Y_i$. The $X_i$ must commute with one another, same for
the $Y_i$ and $Z_i$, and each $X_i$ must commute with every $Z_j$,
same for $Y_i$ and $Z_j$.
\end{thm}

\begin{proof}

We work again with the fundamental relation for $H_1$, equation
\ref{H1fundamentalrelation}:
\[(c_{ij})^{\vec{r}}(c_{ij})^{\vec{s}} = \sum_{l=0}^{\text{min}(r_1, s_2)} {r_1 + s_1 -l \choose s_1}{r_2 +
s_2 -l \choose r_2}{r_3 + s_3 +l \choose r_3, s_3, l}
(c_{ij})^{(r_1 + s1 -l, r_2+s_2-1,l)} \] taking care of course to
realize when a given binomial coefficient is or is not zero mod
$p$.  We begin with
\begin{equation*}
\begin{split}
 Y_{(m)} X_{(n)} &= (c_{ij})^{(0,m,0)}(c_{ij})^{(n,0,0)} \\
&= \sum_{l=0}^0 {n \choose 0}{m \choose 0}{l \choose l}
(c_{ij})^{(n-l,m-l,l)} \\
 &= (c_{ij})^{(n,m,0)}
\end{split}
\end{equation*}
and
\begin{equation*}
\begin{split}
 Y_{(m)} X_{(n)} Z_{(k)} &= (c_{ij})^{(n,m,0)} (c_{ij})^{(0,0,k)}
 \\
 &= \sum_{l=0}^0 {r_1 \choose l}{r_2 \choose l}{k+l \choose k,l}
(c_{ij})^{(n-l,m-l,k+l)} \\
 &= (c_{ij})^{(n,m,k)}
\end{split}
\end{equation*}
  Thus we have a
formula for an arbitrary $(c_{ij})$ matrix:
\[ (c_{ij})^{(n,m,k)} = Y_{(m)} X_{(n)} Z_{(k)} \]
We now show that each of $Z_i$ are determined by the $X_i$ and
$Y_i$.  The fundamental relation gives, just as in characteristic
zero
\[ X_0 Y_0 = (c_{ij})^{(1,0,0)}(c_{ij})^{(0,1,0)} = Y_0X_0 + Z_0 \]
showing $Z_0 = [X_0,Y_0]$.  Now assume by induction that $Z_i$ is
determined by the $X_i$ and $Y_i$ for $i < m$, and we have
\begin{equation*}
\begin{split}
X_m Y_m &= (c_{ij})^{(p^m,0,0)}(c_{ij})^{(0,p^m,0)} \\
 &= \sum_{l=0}^{p^m} {p^m -l \choose 0}{p^m - l \choose 0} {l \choose
l} Y_{(p^m-l)} X_{(p^m-l)} Z_{(l)} \\
 &= \left(\sum_{l=0}^{p^m-1} {p^m -l \choose 0}{p^m - l \choose 0}
{l \choose l} Y_{(p^m-l)} X_{(p^m-l)} Z_{(l)}\right) + Z_m
\end{split}
\end{equation*}
 For $l<
p^m$, $Z_{(l)}$ is determined by the $Z_i$ for $i<m$, who in turn,
by induction, are determined by the $X_i$ and $Y_i$. Every term in
the summation is thus determined by the $X_i$ and $Y_i$, hence so
is $Z_m$, the outlying term.  This shows that the entire
representation is determined by the $X_i$ and $Y_i$.

To see that each $X_i$ commutes with every $Z_j$, simply apply the
fundamental relation to both $X_i Z_j$ and $Z_j X_i$, for which
you get the same answer.  Do the same for $Y_i$ and $Z_j$, and
this completes the proof.

\end{proof}

We ask the reader to note that it is \emph{not} generally the case
that $Z_m = [X_m,Y_m]$ for $m > 0$, nor is it the case that $X_i$
and $Y_j$ commute for $i \neq j$ (the author verified this with
several counter-examples which he will not burden you with).
However, we will see now that these relations do in fact hold so
long as $p$ is sufficiently large when compared to the dimension
of a module.

\begin{lem}
\label{H1remark} Suppose that $p$ is greater than twice the
dimension of a module, and that the sum $r + s$ carries.  Then at
least one of $P_{(r)}$ or $Q_{(s)}$ must be zero, where $P$ and
$Q$ can be any of $X$, $Y$ or $Z$.
\end{lem}

\begin{proof}
The key fact is that since the $X_i$, $Y_i$, and $Z_i$ are all
nilpotent, they are nilpotent of order less than or equal to the
the dimension of the module, which we assume is no greater than
$p/2$. Since the sum $r + s$ carries, we have $r_i + s_i \geq p$
for some $i$, whence, say, $r_i \geq p/2$.  Then
\[ P_{(r)} = \Gamma(r)^{-1}P_0^{r_0} \ldots P_i^{r_i} \ldots
P_i^{r_m} \] is zero, since $P_i^{r_i}$ is.

\end{proof}

\begin{prop}
\label{H1CharpLargeThm1} Suppose $p$ is greater than or equal to
twice the dimension of a module. Then the following relations must
hold: $Z_m = [X_m,Y_m]$ for every $m$, and $X_m Y_n = Y_n X_m$ for
every $m \neq n$.

\end{prop}

\begin{proof}
Consider the fundamental relation, equation
\ref{H1fundamentalrelation}, applied to $X_m Y_m$:
\[ X_m Y_m =  Y_m X_m + \left(\sum_{l=1}^{p^m-1} {p^m -l \choose 0}{p^m - l \choose
0} {l \choose l} Z_{(l)} Y_{(p^m-l)} X_{(p^m-l)} \right) + Z_m \]
For every $0 < l < p^m$ there is clearly some carrying in
computing the sum $(p^m -l) + l$, so lemma \ref{H1remark} says
that the summation term $Z_{(l)}Y_{(p^m-l)}X_{(p^m-l)}$ is always
zero, since at least one of $Z_{(l)}$ or $Y_{(p^m-l)}$ is zero.
This gives $Z_m = [X_m,Y_m]$ as claimed.

Now let $n \neq m$, and consider the fundamental relation applied
to $X_m Y_n$:
\[ X_m Y_n = Y_n X_m + \left(\sum_{l=1}^{\text{min}(p^n,p^m)} {p^m -l \choose 0}{p^n - l \choose
0} {l \choose l} Z_{(l)} Y_{(p^n-l)} X_{(p^m-l)} \right) \] In
case $m < n$, for every value of $l$ in the above summation, $(p^n
-l) + l$ has digit rollover, again forcing at least one of
$Z_{(l)}$ or $Y_{(p^n-l)}$ to be zero, forcing every term in the
summation to be zero. A similar statement holds in case $n < m$.
This proves $X_m Y_n = Y_n X_m$, as claimed.

\end{proof}

Thus far we have shown that, for $p \geq 2d$, every
$d$-dimensional module must satisfy at least those relations that
representations of $G \times G \times \ldots$ over a field of
characteristic zero must satisfy.  We now show sufficiency.

\begin{lem}
\label{H1lemma2} If $p$ is greater than or equal to twice the
dimension of a module, then for any $r$ and $s$
\[ Z_{(r)} Z_{(s)} = {r+s \choose r} Z_{(r+s)} \]
The same holds if we replace $Z$ with $X$ or $Y$.

\end{lem}

\begin{proof}
In case the sum $r + s$ does not carry, we know from our previous
work with $G_a$ (or direct verification) that the equation is
true, just by checking the assignments of the $Z_{(i)}$ in terms
of the $Z_{i}$ (This is true even without the hypothesis that $p$
be large).  If on the other hand the sum does carry, then the
binomial coefficient on the right is zero by corollary
\ref{Lucas'Cor}. But so is the product on the left, by lemma
\ref{H1remark}.
\end{proof}

We can now prove

\begin{thm}
\label{H1CharpLargeThm2} Suppose $p \geq 2d$.  Let $X_i$, $Y_i$
and $Z_i$ be a finite sequence of $d \times d$ matrices satisfying

\begin{enumerate}
\item{The $X_i$, $Y_i$, and $Z_i$ are all nilpotent} \item{$Z_i
=[X_i,Y_i]$ for every $i$} \item{$[X_i,Z_i] = [Y_i,Z_i] = 0$ for
every $i$} \item{For every $i \neq j$, $X_i,Y_i,Z_i$ all commute
with $X_j,Y_j,Z_j$}

\end{enumerate}
Let $n = n_m p^m + n_{m-1} p^{m-1} + \ldots + n_1p+n_0$, and
assign
\[ X_{(n)} = \Gamma(n)^{-1} X_m^{n_m} \ldots X_0^{n_0} \]
and similarly for $Y_{(n)}$ and $Z_{(n)}$.  Set
\[ (c_{ij})^{(n,m,k)} = Z_{(k)}Y_{(m)}X_{(n)} \]
Then these assignments define a valid $d$-dimensional
representation of $G$ over $k$.

\end{thm}

\begin{proof}

For arbitrary $n,m,k,r,s$ and $t$, the equation we must verify is
\begin{gather*}
(c_{ij})^{(n,m,k)}(c_{ij})^{(r,s,t)} \\
= \sum_{l=0}^{\text{min}(n,s)} {n+r-l \choose r}{m+s-l \choose
m}{k+t+l \choose k,t,l} (c_{ij})^{(n+r-l,m+s-l,k+t+l)}
\end{gather*}
 which, with
the given assignments and assumptions, can be written
\begin{gather*}
 Z_{(k)}Z_{(t)}Y_{(m)}X_{(n)} Y_{(s)} X_{(r)} \\
 = \sum_{l=0}^{\text{min}(n,s)} {n+r-l \choose r}{m+s-l \choose
m}{k+t+l \choose k,t,l} Z_{(k+t+l)} Y_{(m+s-l)} X_{(n+r-l)}
\end{gather*}
Lemma \ref{H1lemma2} gives the identities
\begin{equation*}
\begin{split}
 Z_{(k)} Z_{(t)} &= {k+t \choose t} Z_{(k+t)} \\
 Y_{(m)}Y_{(s-l)}  &= {m+s-l \choose m}Y_{(m+s-l)} \\
X_{(n-l)} X_{(r)}  &= {n+r-l \choose r} X_{(n+r-l)}
\end{split}
\end{equation*}
 so we can
rewrite our equation as
\[ {k+t \choose t} Z_{(k+t)}Y_{(m)}X_{(n)} Y_{(s)} X_{(r)} = \sum_{l=0}^{\text{min}(n,s)} {k+t+l \choose k,t,l} Z_{(k+t+l)}
Y_{(m)}Y_{(s-l)} X_{(n-l)} X_{(r)} \] First suppose that the sum
$k+t$ carries. In this case the equation is true, since the left
hand side
 binomial coefficient vanishes, and the right hand side multinomial coefficient
 vanishes for every $l$, causing both sides to be zero.  We assume
 then that $k+t$ does not carry, so we can divide both sides
 by ${k+t \choose t}$ to yield
 \[ Z_{(k+t)}Y_{(m)}X_{(n)} Y_{(s)} X_{(r)} = \sum_{l=0}^{\text{min}(n,s)} {k+t+l \choose l} Z_{(k+t+l)}
Y_{(m)}Y_{(s-l)} X_{(n-l)} X_{(r)} \] Now apply ${k+t+l \choose
l}Z_{(k+t+l)} = Z_{(k+t)} Z_{(l)}$:
 \[ Z_{(k+t)}Y_{(m)}X_{(n)} Y_{(s)} X_{(r)} = \sum_{l=0}^{\text{min}(n,s)} Z_{(k+t)} Z_{(l)} Y_{(m)}Y_{(s-l)}
X_{(n-l)} X_{(r)} \] We have $Z_{(k+t)}$ in the front and
$X_{(r)}$ in the rear of both sides, so it suffices to show
 \[ Y_{(m)}X_{(n)} Y_{(s)} = \sum_{l=0}^{\text{min}(n,s)} Z_{(l)} Y_{(m)}Y_{(s-l)} X_{(n-l)}
\]
and since $Y_{(m)}$ commutes with $Z_{(l)}$, we can move it to the
front of the right hand side, and then take it off both sides, so
it suffices to show
\begin{equation}
\label{H1charpProofEq} X_{(n)} Y_{(m)} =
\sum_{l=0}^{\text{min}(n,m)} Z_{(l)} Y_{(m-l)} X_{(n-l)}
\end{equation}
where we have replaced $s$ with the more traditional $m$.

Now we begin to replace the $X_{(i)}'s$ with their definitions in
terms of the $X_i's$, similarly for $Y$ and $Z$, so that the left
hand side of equation \ref{H1charpProofEq} is
\[ \left[\Gamma(n) \Gamma(m)\right]^{-1}X_0^{n_0} \ldots
X_k^{n_k}Y_0^{m_0} \ldots Y_k^{m_k} \] and since everything
commutes except $X_i$ and $Y_j$ when $i = j$, we can write
\[ \left[\Gamma(n) \Gamma(m)\right]^{-1}(X_0^{n_0}Y_0^{m_0}) \ldots
(X_k^{n_k}Y_k^{m_k}) \] Moving all coefficients to the right, we
must show
\[ (X_0^{n_0}Y_0^{m_0}) \ldots
(X_k^{n_k}Y_k^{m_k}) = \Gamma(n) \Gamma(m)
\sum_{l=0}^{\text{min}(n,m)} Z_{(l)} Y_{(m-l)} X_{(n-l)} \] We
proceed by induction on $k$, maximum number of $p$-digits of
either $m$ or $n$. If $k=0$ the equation is
\begin{equation*}
\begin{split}
X_0^{n_0}Y_0^{m_0} &= n_0! m_0! \sum_{l=0}^{\text{min}(n_0,m_0)}
Z_{(l)} Y_{(m_0-l)} X_{(n_0 -l)} \\
 &= \sum_{l=0}^{\text{min}(n_0,m_0)}
\frac{n_0!m_0!}{(m_0-l)!(n_0-l)!l!} Z_0^l Y_0^{m_0-l} X_0^{n_0-l}
\\
 &= \sum_{l=0}^{\text{min}(n_0,m_0)} l!{ n_0 \choose l}{ m_0 \choose
l}  Z_0^l Y_0^{m_0-l} X_0^{n_0-l}
\end{split}
\end{equation*}
The reader may recall that this was exactly the equation to be
verified halfway through the proof of theorem \ref{H1CharZeroThm}
in the characteristic zero case for $X$, $Y$ and $Z$.  Nowhere in
that section of the proof did we use the characteristic of the
field; the same hypotheses hold here for $X_0,Y_0$ and $Z_0$, and
the proof goes through just the same, so we do not repeat it. Now
suppose the equation is true when $n$ and $m$ have no more than
$k-1$ digits. Let $n=n_{k-1}p^{k-1} + \ldots + n_0$ and let
$n^\prime = n_{k}p^{k} + n_{k-1} p^{k-1} + \ldots + n_0$, and
similarly for $m$. Then by induction we have
\begin{equation*}
\begin{split}
\Gamma(n^\prime)& \Gamma(m^\prime) X_{(n^\prime)} Y_{(m^\prime)} =
\left[(X_0^{n_0}Y_0^{m_0}) \ldots
(X_{k-1}^{n_{k-1}}Y_{k-1}^{m_{k-1}})\right](X_{k}^{n_{k}}Y_{k}^{m_{k}})
 \\
 &= \left(\Gamma(n) \Gamma(m) \sum_{l=0}^{\text{min}(n,m)} Z_{(l)}
Y_{(m-l)} X_{(n-l)}\right)  \left(
\sum_{l^\prime=0}^{\text{min}(n_{k},m_{k})} l^\prime!{ n_{k}
\choose l^\prime}{ m_{k} \choose l^\prime}  Z_{k}^{l^\prime}
Y_{k}^{m_{k}-l^\prime} X_{k}^{n_{k}-l^\prime} \right) \\
 &= n_k!\Gamma(n)m_k!\Gamma(m) \sum_{l,l^\prime} \left(\frac{Z_{(l)}
Z_k^{l^\prime}}{l^\prime!} \right) \left(\frac{Y_{(m-l)}
Y_k^{m_k-l^\prime}}{(m_k-l^\prime)!} \right) \left(\frac{X_{(n-l)}
X_k^{n_k-l^\prime}}{(n_k-l^\prime)!} \right)
\end{split}
\end{equation*}
Note that these divisions are valid, since for every value of
$l^\prime$ in the summation, $l^\prime \leq m_k,n_k < p$.  Note
also that, since $l \leq p^{k-1}$ and $l^\prime < p$ for all
values of $l,l^\prime$ in the summation, Lucas' theorem gives that
${ l + l^\prime p^k \choose l} = 1$ for all such $l$ and
$l^\prime$.  For similar reasons we have ${(m-l)+(m_k-l^\prime)p^k
\choose m-l} = {(n-l)+(n_k - l^\prime)p^k \choose n-l} = 1$.  Then
we have the identities
\begin{gather*}
 n_k!\Gamma(n) = \Gamma(n^\prime)  \hspace{1cm} \hspace{1cm}  m_k!\Gamma(m) =
\Gamma(m^\prime) \\
 \frac{Z_{(l)} Z_k^{l^\prime}}{l^\prime!} = Z_{(l)}Z_{(l^\prime
p^k)} = {l+l^\prime p^k \choose l} Z_{(l+l^\prime p^k)} =
Z_{(l+l^\prime p^k)} \\
 \frac{Y_{(m-l)} Y_k^{m_k-l^\prime}}{(m_k-l^\prime)!} =
Y_{(m-l)}Y_{((m_k-l^\prime)p^k)} = {(m-l)+(m_k-l^\prime)p^k
\choose m-l} Y_{((m+m_kp^k)-(l+l^\prime)p^k)} =
Y_{(m^\prime-(l+l^\prime p^k))}
\end{gather*}
 and similarly
\[ \frac{X_{(n-l)} X_k^{n_k-l^\prime}}{(n_k-l^\prime)!} =
X_{(n^\prime-(l+l^\prime p^k))} \] These substitutions transform
the right hand side of our equation into
\[ = \Gamma(n^\prime) \Gamma(m^\prime) \sum_{l,l^\prime}
Z_{(l+l^\prime p^k)} Y_{(m^\prime-(l+l^\prime p^k))}
X_{(n^\prime-(l+l^\prime p^k))} \] But, if we look at the
summation limits of $l = 0 \ldots \text{min}(n,m)$ and $l^\prime =
0 \ldots \text{min}(n_k,m_k)$, we see that it is really a single
summation running from $0$ to $\text{min}(n^\prime,m^\prime)$,
with $l+l^\prime p^k$ as the summation variable. That is
\[ = \Gamma(n^\prime) \Gamma(m^\prime)
\sum_{l=0}^{\text{min}(n^\prime,m^\prime)} Z_{(l)}
Y_{(m^\prime-l)} X_{(n^\prime-l)} \] which finally gives
\[ X_{(n^\prime)} Y_{(m^\prime)} = \sum_{l=0}^{\text{min}(n^\prime,m^\prime)} Z_{(l)}
Y_{(m^\prime-l)} X_{(n^\prime-l)} \] as required. This completes
the proof.

\end{proof}

\section{Baker-Campbell-Hausdorff Formula For $H_1$ in Positive Characteristic}

There is a much more compact way to state all of this.  What we
have really recovered is, in characteristic zero, the familiar
Baker-Campbell-Hausdorff formula for the group $H_1$, and for
characteristic $p >> \text{dimension}$, something very close to
it.

In theorem 3.1 of \cite{hall} it is proven that, if $X,Y$ and $Z$
are matrices over $\mathbb{R}$ such that $Z = [X,Y]$ and $[Z,X] =
[Z,Y] = 0$, then
\[ e^X e^Y = e^{X + Y + \frac{1}{2}Z} \]
Our first step is to extend this result to the case of a field of
sufficiently large characteristic when compared to the dimensions
of the matrices $X,Y$ and $Z$, and under the additional hypothesis
that they be nilpotent.  The proof we give is an almost exact
replica of that given in \cite{hall}; the only difference is that
we replace the notion of derivative with `formal derivative' of
polynomials.

For the remainder, by a polynomial $f(t)$, we shall mean a
polynomial in the commuting variable $t$ with coefficients which
are matrix expressions among the matrices $X$,$Y$ and $Z$ over a
given field; for example, $f(t) = XY + 2(Z-YX)t +
\frac{Y}{2}t^2-t^3$. We define the \textbf{formal derivative} of
$f(t)$ in the usual manner; for example, $f^\prime(t) = 2(Z-YX) +
Yt - 3t^2$. Then the following facts hold just as well for formal
differentiation as they do for standard differentiation.

\begin{lem}
\label{formaldifferenationforeverLemma} Let $f(t)$, $g(t)$ be
polynomials, and suppose that the field is either of
characteristic zero, or of positive characteristic greater than
the degrees of both $f(t)$ and $g(t)$.

\begin{enumerate}

\item{(product rule) $(fg)^\prime = f^\prime g + f g^\prime$}
\item{(uniqueness of antiderivatives) If $f^\prime(t) =
g^\prime(t)$, and if $f(0) = g(0)$, then $f(t) = g(t)$}
\item{(uniqueness of solutions to differential equations) Let $M$
be some matrix expression among $X$,$Y$ and $Z$.  Then if
$f^\prime(t) = M f(t)$, and if $g^\prime(t) = M g(t)$, and if
$f(0) = g(0)$, then $f(t) = g(t)$}

\end{enumerate}
\end{lem}

Remark: the assumption that $\text{char}(k) > \text{degree}$ is
essential.  For example, in characteristic $2$, the derivatives of
the polynomials $t^2$ and $0$ are both zero, and they are both
zero when evaluated at $t=0$, but they are obviously not
themselves equal.

\begin{proof}
1.~is true even without any hypothesis on the characteristic. Let
$f= \sum_{k=0}^m a_k t^k$, $g = \sum_{k=0}^m b_k t^k$, where the
$a_i,b_i$ are matrix expressions in $X,Y$ and $Z$.  Then
\begin{equation*}
\begin{split}
(fg)^\prime &= \left[\left(\sum_{k=0}^m a_k t^k\right)\left(
\sum_{l=0}^m b_l t^l\right) \right]^\prime \\
 &=
\left[\sum_{r=0}^{2m} \left(\sum_{k+l=r} a_k b_l \right) t^r
\right]^\prime \\
 &= \sum_{r=0}^{2m} r \left(\sum_{k+l=r} a_k b_l
\right) t^{r-1} \\
 &= \sum_{r=1}^{2m} r \left( \sum_{k+l=r} a_k b_l \right) t^{r-1} \\
 &= \sum_{r=0}^{2m-1} (r+1) \left(\sum_{k+l=r+1} a_k b_l \right) t^r
\end{split}
\end{equation*}
and
\begin{equation*}
\begin{split}
 f^\prime g + f g^\prime &= \left( \sum_{k=0}^m k a_k t^{k-1}
\right) \left( \sum_{l=0}^m b_l t^l \right) + \left( \sum_{k=0}^m
a_k t^k \right) \left( \sum_{l=0}^m l b_l t^{l-1} \right) \\
  &= \sum_{k,l=0}^m \left(k a_k b_l \right) t^{k-1+l} +
\sum_{k,l=0}^m \left(l a_k b_l \right) t^{k-1+l} \\
&= \sum_{k,l=0}^m (k+l)(a_k b_l) t^{k+l-1} \\
 &= \sum_{r=0}^{2m-1} \left(\sum_{k+l-1=r} (k+l) a_k b_l \right)
 t^r \\
&=\sum_{r=0}^{2m-1} (r+1) \left( \sum_{k+l = r+1} a_k b_l \right)
t^r
\end{split}
\end{equation*}
 which proves 1.

For 2., let $f$ and $g$ be as before.  To say that $f^\prime =
g^\prime$ is to say that $n a_n = n b_n, (n-1) a_{n-1} = (n-1)
b_{n-1}, \ldots, a_1 = b_1$, and to say that $f(0) = g(0)$ is to
say that $a_0 = b_0$.  Under the given hypotheses all of $n,n-1,
\ldots, 1$ are invertible, which forces $a_n = b_n$, $a_{n-1} =
b_{n-1}$, $\ldots a_1 = b_1$ and $a_0 = b_0$, whence $f=g$. This
proves 2.

For 3., suppose $f^\prime = M f$ and $g^\prime = M g$.  Then by
matching coefficients for the various powers of $t$ this forces
the equalities
\[
\begin{array}{cc}
  M a_n = 0 & M b_n = 0\\
  n a_n = M a_{n-1} & n b_n = M b_{n-1} \\
  (n-1) a_{n-1} =M a_{n-2}  & \hspace{1cm} (n-1) b_{n-1} =M b_{n-2} \\
  \vdots & \vdots \\
  2a_2 = M a_1 & 2b_2 = M b_1 \\
  a_1 = M a_0 & b_1 = M b_0 \\
\end{array}%
\]
$f(0)=g(0)$ again forces $a_0 = b_0$.  Noting again that all of
$n,n-1, \ldots, 1$ are invertible, we can work backwards to see
that $a_1 = M a_0 = M b_0 = b_1$, that $a_2 = \frac{1}{2} M a_1 =
\frac{1}{2} M b_1 = b_2$, $\ldots$, $a_n = \frac{1}{n} M a_{n-1} =
\frac{1}{n} M b_{n-1} = b_n$, whence $f = g$.  This proves 3.

\end{proof}

\begin{lem}

\label{exponentialsAndWhatnotLemma}

Let $X$ and $Y$ be commuting nilpotent matrices over a field $k$
such that $k$ is of characteristic zero, or of positive
characteristic greater than or equal to the dimension of $X$ and
$Y$. Then

\begin{enumerate}

\item{$\left(e^{tX}\right)^\prime = e^{tX} X$}
 \item{$\left(e^{t^2
X}\right)^\prime = e^{t^2X}(2tX)$}
 \item{$e^X e^Y = e^{X+Y}$}

\end{enumerate}
\end{lem}

Remark: the first two are obvious corollaries to the usual chain
rule for differentiation, but the chain rule is in general
\emph{not} valid for polynomials in non-commuting coefficients. It
is convenient for our purposes just to treat these cases
separately.

\begin{proof}
We note firstly that, if $\text{char}(k) = p \geq \text{dim}$,
then all of the above expressions make sense, since their series
expansions will vanish before we get to see denominators divisible
by $p$. For 1., compute:
\begin{equation*}
\begin{split}
(e^{tX})^\prime &= \left(1 + tX + \frac{t^2 X^2}{2!} + \ldots +
\frac{t^n X^n}{n!}\right)^\prime \\
&= X + \frac{2tX^2}{2!} + \ldots + \frac{n t^{n-1} X^n}{n!} \\
 &= X \left(1+tX + \ldots + \frac{t^{n-1} X^{n-1}}{(n-1)!} +
\frac{t^n X^n}{n!}\right) \\
&= X e^{tX}
\end{split}
\end{equation*}
 Note that, in the second to last expression, we are
justified in tacking on the term $\frac{t^nX^n}{n!}$ since
multiplication by $X$ will annihilate it anyway.  This proves 1.

For 2., compute again:
\begin{equation*}
\begin{split}
 (e^{t^2X})^\prime &= \left(1+t^2X+ \frac{t^4 X^2}{2!} + \ldots +
\frac{t^{2n} X^n}{n!} \right)^\prime \\
 &= 2tX+ \frac{4t^3 X^2}{2!} + \ldots + \frac{2n t^{2n-1} X^n}{n!}
 \\
 &= 2tX\left(1 + \frac{2t^2X}{2!} + \ldots + \frac{n t^{2(n-1)}
X^n}{n!}\right) \\
 &= 2tX \left(1+ t^2 X + \frac{t^4 X^2}{2!} + \ldots +
\frac{t^{2(n-1)} X^{n-1}}{(n-1)!} + \frac{t^{2n} X^n}{n!} \right)
\\
 &= 2tX e^{t^2X}
\end{split}
\end{equation*}
  where, again, in the second to last expression,
we are justified in tacking on the term $\frac{t^{2n} X^n}{n!}$
since $X$ will annihilate it anyhow.  This proves 2.

For 3., we shall prove that $e^{tX} e^{tY} = e^{t(X+Y)}$ as
polynomials; evaluating at $t=1$ gives the desired result.  Note
that the right hand side is defined; if $X$ and $Y$ commute, they
can be put in simultaneous upper triangular form, and so $X+Y$ is
nilpotent.  By 3.~of lemma \ref{formaldifferenationforeverLemma},
since they are equal when evaluated at $t=0$, it is enough to show
that they satisfy the same differential equation:
\begin{equation*}
\begin{split}
\left(e^{tX} e^{tY} \right)^\prime &= \left( e^{tX} \right)^\prime
e^{tY} + e^{tX} \left(e^{tY} \right)^\prime \\
 &= X e^{tX} e^{tY} +
e^{tX} Y e^{tY} \\
&=  e^{tX} e^{tY}(X+Y)
\end{split}
\end{equation*}
 and
\[ \left( e^{t(X+Y)} \right)^\prime = e^{t(X+Y)}(X+Y) \]
This completes the proof.

\end{proof}

\begin{lem}
\label{H1BakerLemma} Let $X$ and $Y$ be nilpotent matrices over a
field, commuting with their nilpotent commutator $Z$.  If the
field is either of characteristic zero, or of positive
characteristic larger than twice the dimension of the matrices,
then
\[ e^{X} e^{Y} = e^{X+Y+\frac{1}{2} Z} \]
\end{lem}

\begin{proof}

We shall prove something stronger, namely that
\[ e^{tX} e^{tY} = e^{tX+tY+\frac{t^2}{2} Z} \]
as polynomials; evaluating at $t=1$ will give the desired result.

We note first that if $\text{char} = p \geq 2\text{dim}$, all of
the above expressions make sense, since e.g.~the series expansion
for $e^{tX}$ will vanish before we get to see denominators
divisible by $p$.  Note also that the results of the previous two
lemmas apply, since the maximum degree of any of the above
polynomials is $2\text{dim} - 2$.

Note also that $tX+tY+\frac{t^2}{2}Z$ must be also be nilpotent.
If $X$,$Y$ and $Z$ are matrices satisfying the given hypotheses,
they define a representation of $H_1$ according to either theorem
\ref{H1CharZeroThm} or theorem \ref{H1CharpLargeThm2}.  As any
representation of a unipotent algebraic group can be put in upper
triangular form, it follows that $X,Y$ and $Z$ can be put in
simultaneous upper triangular form.  It is obvious then that any
linear combination of $X,Y$ and $Z$ is nilpotent, and so the right
hand side makes sense as well (except when $p = 2$; but this
forces dimension to be $\leq 1$, and in this case the result is
trivial).

The proof proceeds exactly as in theorem 3.1 of \cite{hall} for
the Lie group case.  Since $Z$ commutes with both $X$ and $Y$, we
can rewrite the above equation as
\[ e^{tX} e^{tY} e^{\frac{-t^2}{2} Z} = e^{t(X+Y)} \]
Denote by $A(t)$ the left hand side of this equation, $B(t)$ the
right hand side.  These are both equal to $1$ when evaluated at
$t=0$, so by 3.~of lemma \ref{formaldifferenationforeverLemma} it
suffices to show that they both satisfy the same linear
differential equation.  Working first with $A(t)$, using the
iterated product rule we have
\begin{equation*}
\begin{split}
 A^\prime(t) &= e^{tX} X e^{tY} e^{\frac{-t^2}{2} Z} + e^{tX}
e^{tY} Y e^{\frac{-t^2}{2}Z} + e^{tX}e^{tY}
e^{\frac{-t^2}{2}Z}(-tZ) \\
 &= e^{tX} e^{tY} \left(e^{-tY} X e^{tY} \right) e^{\frac{-t^2}{2}
Z} + e^{tX} e^{tY} Y e^{\frac{-t^2}{2}Z} + e^{tX}e^{tY}
e^{\frac{-t^2}{2}Z}(-tZ)
\end{split}
\end{equation*}
 We claim that $e^{-tY} X e^{tY}$ is equal
to $X + tZ$.  They are both equal to $X$ when evaluated at $t=0$,
and $(X+tZ)^\prime = Z$, so it suffices to show by part 2.~of
lemma \ref{formaldifferenationforeverLemma} that the derivative of
$e^{-tY} X e^{tY}$ is equal to $Z$:
\begin{equation*}
\begin{split}
 \left(e^{-tY} X e^{tY} \right)^\prime &= e^{-tY}(-YX) e^{tY} +
e^{-tY} XY e^{tY} \\
 &= e^{-tY}(XY-YX) e^{tY}\\
  &= e^{-tY} Z e^{tY} \\
  &= Z
\end{split}
\end{equation*}
as required.  Thus
\begin{equation*}
\begin{split}
A^\prime(t) &=  e^{tX} e^{tY} \left(X+tZ\right) e^{\frac{-t^2}{2}
Z} + e^{tX} e^{tY} Y e^{\frac{-t^2}{2}Z} + e^{tX}e^{tY}
e^{\frac{-t^2}{2}Z}(-tZ) \\
 &= e^{tX} e^{tY} e^{\frac{-t^2}{2} Z} (X + tZ + Y - tZ) \\
  &= e^{tX} e^{tY} e^{\frac{-t^2}{2} Z}(X+Y) \\
  &= A(t)(X+Y)
\end{split}
\end{equation*}
   To finish then, it suffices to show that
$B^\prime(t) = B(t)(X+Y)$; but this is obvious by part 1.~of lemma
\ref{exponentialsAndWhatnotLemma}.  This completes the proof.

\end{proof}

\begin{thm}
\label{H1charZeroBakerThm} Let $M(x,y,z)$ be (the matrix formula
for) a finite dimensional module for $H_1$ in characteristic zero
given by the nilpotent matrices $X,Y$, and $Z$, in the notation of
theorem \ref{H1CharZeroThm}. Then
\[ M(x,y,z) = e^{xX + yY + (z-xy/2)Z} \]
\end{thm}

\begin{proof}
We will first prove that the formula given is actually a
representation of $H_1$, which amounts to verifying the matrix
equality
\[ e^{xX+yY+(z-xy/2)Z} e^{rX+sY+(t-rs/2)Z}
 = e^{(x+r)X+(y+s)Y+(z+xs+t-(x+r)(y+s)/2)Z} \]
If $X$ and $Y$ are nilpotent and commute with their nilpotent
commutator $Z$, then $xX+yY$ and $rX+sY$ are also nilpotent, and
also commute with their nilpotent commutator $(xs-yr)Z$, and so
lemma \ref{H1BakerLemma} applies:
\[\text{exp}(xX+yY)\text{exp}(rX+sY) =
\text{exp}((x+r)X+(y+s)Y+\frac{(xs-yr)}{2}Z) \] Recalling that
$e^R e^S = e^{R+S}$ whenever $R$ and $S$ commute, the left hand
side of our first equation can be written
\begin{equation*}
\begin{split}
 \text{exp}&(xX+yY)\text{exp}(rX+sY)\text{exp}((z+t-(xy+rs)/2)Z) \\
 &= \text{exp}((x+r)X+(y+s)Y+
\frac{xs-yr}{2}Z)\text{exp}((z+t-(xy+rs)/2)Z) \\
&= \text{exp}((x+r)X+(y+s)Y+(z+xs+t-(x+r)(y+s)/2)Z)
\end{split}
\end{equation*}
 The expression
given in the statement of the theorem is therefore indeed a
representation of $H_1$.  To see that they are equal, simply
verify that, in the notation of theorem \ref{H1CharZeroThm}, the
matrix $(c_{ij})^{(1,0,0)}$ is actually $X$, and the matrix
$(c_{ij})^{(0,1,0)}$ is actually $Y$; as these completely
determine the rest of the representation, we conclude that
$M(x,y,z)$ and the given expression are in fact equal.
\end{proof}

In the characteristic $p>0$ case, if we assume $p \geq 2
\text{dim}$, we obtain a result analogous to theorem
\ref{Gacharptheorem} for the group $G_a$.

\begin{thm}
\label{H1charpBakerThm} Let $k$ have characteristic $p > 0$, and
suppose $p \geq 2d$. Then every $d$-dimensional representation of
$H_1$ over $k$ is of the form
\[ e^{xX_0+yY_0+(z-xy/2)Z_0} e^{x^pX_1+y^pY_1+(z^p-x^py^p/2)Z_1} \ldots
e^{x^{p^m} X_m + y^{p^m} Y_m+(z^{p^m} - x^{p^m} y^{p^m}/2)Z_m} \]
with all of the factors commuting.  Further, any collection
$X_i,Y_i,Z_i$ of $d$-dimensional matrices satisfying the
hypotheses of theorem \ref{H1CharpLargeThm2} gives a
representation according to the above formula.
\end{thm}

\begin{proof}
In the notation of theorem \ref{H1CharpLargeThm2}, let $X_0,
\ldots, X_s$, $Y_0, \ldots, Y_s$, $Z_0, \ldots, Z_s$ be given.
Then the matrix formula for the representation they define is
\begin{equation*}
\begin{split}
 M(x,y,z) &= \sum_{n,m,k} (c_{ij})^{(n,m,k)} x^n y^m z^k =
\sum_{n,m,k} Z_{(k)}Y_{(m)}X_{(n)} x^n y^m z^k \\
 &= \sum_{n,m,k} \Gamma(n)^{-1} \Gamma(m)^{-1} \Gamma(k)^{-1}
Z_0^{k_0} \ldots Z_s Y_0^{m_0} \ldots Y_s^{m_s}X_0^{n_0} \ldots
X_s^{n_s} \\
 & \qquad x^{n_0+n_1p+ \ldots +n_sp^s}y^{m_0+m_1p+ \ldots +m_sp^s}z^{k_0 +
k_1p+\ldots k_s p^s} \\
 &= \left(\sum_{n_0,m_0,k_0=0}^{p-1}
\frac{1}{n_0!m_0!k_0!}Z_0^{k_0}Y_0^{m_0}X_0^{n_0} x^{n_0} y^{m_0}
z^{m_0} \right) \\
& \qquad \ldots \left(\sum_{n_s,m_s,k_s=0}^{p-1}
\frac{1}{n_s!m_s!k_s!} Z_s^{k_s} Y_s^{m_s} X_s^{n_s}
x^{n_sp^s}y^{m_sp^s}z^{k_sp^s} \right)
\end{split}
\end{equation*}
 We note that, for fixed
$r$, the matrices $X_r,Y_r$, and $Z_r$, by theorem
\ref{H1CharpLargeThm2}, satisfy the hypotheses of lemma
\ref{H1BakerLemma}.  Working through the proof of theorem
\ref{H1charZeroBakerThm}, we see that nowhere was the
characteristic of the field used; only that lemma
\ref{H1BakerLemma} was satisfied.  In other words, theorem
\ref{H1charZeroBakerThm} establishes a purely combinatorial fact
that, if $X_0,Y_0,Z_0$ satisfy lemma \ref{H1BakerLemma}, then
\[ \sum_{n_0,m_0,k_0=0}^{m}
\frac{1}{n_0!m_0!k_0!}Z_0^{k_0}Y_0^{m_0}X_0^{n_0} x^{n_0} y^{m_0}
z^{m_0}  = \text{exp}(xX_0+yY_0+(z-xy/2)Z_0) \] whenever $m$ is
greater than or equal to the nilpotent orders of $X_0,Y_0$, and
$Z_0$.  We conclude that
\[ \sum_{n_0,m_0,k_0=0}^{p-1}
\frac{1}{n_0!m_0!k_0!}Z_0^{k_0}Y_0^{m_0}X_0^{n_0} x^{n_0} y^{m_0}
z^{m_0}  = \text{exp}(xX_0+yY_0+(z-xy/2)Z_0) \] We can of course
replace $x,y$ and $z$ with $x^{p^r},y^{p^r}$ and $z^{p^r}$ to
likewise obtain
\[\sum_{n_r,m_r,k_r=0}^{p-1} \frac{1}{n_r!m_r!k_r!}
Z_r^{k_r} Y_r^{m_r} X_r^{n_r} x^{n_rp^r}y^{m_rp^r}z^{k_rp^r}
 =
\text{exp}(x^{p^r}X_r+y^{p^r}Y_r+(z^{p^r}-x^{p^r}y^{p^r}/2)Z_r)
\]
for any $r$, and hence
\[ \left(\sum_{n_0,m_0,k_0=0}^{p-1}
\frac{1}{n_0!m_0!k_0!}Z_0^{k_0}Y_0^{m_0}X_0^{n_0} x^{n_0} y^{m_0}
z^{m_0} \right) \]
\[ \ldots \left(\sum_{n_s,m_s,k_s=0}^{p-1} \frac{1}{n_s!m_s!k_s!}
Z_s^{k_s} Y_s^{m_s} X_s^{n_s} x^{n_sp^s}y^{m_sp^s}z^{k_sp^s}
\right) \]
\[= \text{exp}(xX_0+yY_0+(z-xy/2)Z_0) \text{exp}(x^pX_1+y^pY_1+(z^p-x^py^p/2)Z_1) \]
\[ \ldots \text{exp}(x^{p^s} X_s + y^{p^s} Y_s+(z^{p^s} - x^{p^s}
y^{p^s}/2)Z_s)
\]
which proves the theorem.  Note that all of the factors commute,
since so do $X_i,Y_i,Z_i$ and $X_j,Y_j,Z_j$ when $i \neq j$.

\end{proof}

\chapter{The Height-Restricted Ultraproduct}
\index{height-restricted ultraproduct} In the previous chapters,
we saw that for the unipotent groups $G$ that we studied, the
representation theories of $G^n$ in characteristic zero and for
$G$ in characteristic $p
>> \text{dimension}$ are in perfect analogy. The appropriate
context, we believe, in which to interpret these results is in
consideration of the so-called height-restricted ultraproduct,
which we formally define now.

Let $G$ be any of our so far studied unipotent groups, and let $k$
be a field of characteristic $p>0$.  They all have Hopf algebras
isomorphic to $k[x_1, \ldots, x_n]$ for some $n$, so the following
definition makes sense:

\begin{defn}
\index{height} The \textbf{height} of a representation $M$ of $G$
over $k$ is the largest $m$ such that, for some $i$,
$x_i^{p^{m-1}}$ occurs as a coefficient in the matrix formula of
$M$.  In case no such occurs (i.e. $M$ is a trivial
representation), we say $M$ has height zero.
\end{defn}

Since all of the Hopf algebras at issue are isomorphic to $k[x_1,
\ldots, x_n]$, height is an isomorphism invariant. Lemma
\ref{comoduleCoefficientsLemma} shows that applying any base
change to the matrix formula of a representation yields two
matrices, each of whose entries will be linear combinations of the
entries of the other.

Example: the representation
\[
\left(%
\begin{array}{cc}
  1 & x+x^{p^2} \\
  0 & 1 \\
\end{array}%
\right)
\]
for $G_a$ has height $3$.

If $p$ is large with respect to dimension, we know that every
representation of $G_a$ or $H_1$ can be factored into a commuting
product of representations, each accounting for one of its
Frobenius layers. In this case the height of $M$ is equal to the
number of these layers (even in the case of a trivial
representation, which has no layers), hence the motivation for the
definition.

Let $k_i$ be a collection of fields of strictly increasing
positive characteristic, $\catfont{C}_i = \text{Rep}_{k_i} G$.

\begin{defn}
\label{defnHightResUprod} The \index{height} \textbf{height} of an
object $[X_i]$ of $\resprod \catfont{C}_i$ is defined to be the
essential supremum of $\{\text{height}(X_i):i \in I\}$.  The
\index{height-restricted ultraproduct} \textbf{height-restricted
ultraproduct} of the $\catfont{C}_i$, denoted \index{$\hprod
\catfont{C}_i$} $\hprod \catfont{C}_i$, is the full subcategory of
$\resprod \catfont{C}_i$ consisting of those objects having finite
height. For $n \in \mathbb{N}$, we denote by \index{$\hnprod{n}
\catfont{C}_i$} $\hnprod{n} \catfont{C}_i$ the full subcategory of
$\hprod \catfont{C}_i$ consisting of those objects of height no
greater than $n$.
\end{defn}

Note that, as a subcategory of $\resprod \catfont{C}_i$, we demand
the objects of $\hprod \catfont{C}_i$ to be of bounded dimension
as well as height.

The remainder of this chapter is devoted to proving

\begin{thm}
\index{height-restricted ultraproduct}
\label{heightResMainTheorem} Let $G$ be any of the so far studied
unipotent groups, $k_i$ a sequence of fields of strictly
increasing positive characteristic. Then $\hprod \text{Rep}_{k_i}
G$ is a neutral tannakian subcategory of $\resprod
\text{Rep}_{k_i} G$, and is tensorially equivalent to
$\text{Rep}_{\uprod k_i} G^\infty$. Likewise, for any $n \in
\mathbb{N}$, $\hnprod{n} \text{Rep}_{k_i} G$ is tensorially
equivalent to $\text{Rep}_{\uprod k_i} G^n$.
\end{thm}

Note that the groups $G^\infty$ and $G^n$ in this theorem are
\emph{independent} of the choice of non-principal ultrafilter, and
while the field $\uprod k_i$ does indeed vary, it will in all
cases have characteristic zero.

We shall prove the theorem for the group $H_1$, leaving it to the
reader to convince himself that the same proof applies to the
group $G_a$. The proof is quite straightforward; we shall
construct an explicit equivalence between the two categories, show
that it is tensor preserving, and it will be immediate that the
usual forgetful functor on $\text{Rep}_{\uprod k_i} G^\infty$ can
be identified as the restriction of the fibre functor
$\omega:\resprod \catfont{C}_i \rightarrow \text{Vec}_{\uprod
k_i}$ defined in chapter \ref{themaintheoremchapter} to $\hprod
\text{Rep}_{k_i} G$.

For the remainder, let $G$ denote the group $H_1$, $k_i$ a
sequence of fields of strictly increasing positive characteristic,
$\catfont{C}_i = \text{Rep}_{k_i} G$, and $k =\uprod k_i$.

\section{Labelling the Objects of $\prod_H \hspace{-.03cm} \catfont{C}_i$}

Now, what does an object $[V_i]$ of $\hprod \catfont{C}_i$
actually look like?  Firstly, the vector spaces $V_i$ are of
bounded dimension. And since the $k_i$ are of strictly increasing
characteristic, this tells us that, for all but finitely many $i$,
$V_i$ is of the form given by theorem \ref{H1CharpLargeThm2}; that
is, it is determined by a finite sequence of nilpotent
transformations on $V_i$
\[ X_0^i, \ldots, X_m^i,Y_0^i, \ldots, Y_m^i,Z_0^i, \ldots,Z_m^i
\]
according to the formula given in theorem \ref{H1charpBakerThm},
and under the conditions given in theorem \ref{H1CharpLargeThm2}.
Secondly, it is of finite height, whence we can take $m$ to be
constant for almost every $i$.

\begin{thm}
\label{H1hprodLabelingThm} Each object $[V_i]$ of $\hprod
\catfont{C}_i$ is completely determined by a sequence $X_0,
\ldots, X_m,Y_0, \ldots, Y_m,Z_0, \ldots ,Z_m$ of linear
transformations on $\uprod V_i$ satisfying
\begin{enumerate}
\item{The $X_j$, $Y_j$, and $Z_j$ are all nilpotent} \item{$Z_j =
[X_j,Y_j]$ for every $j$} \item{$[X_j,Z_j] = [Y_j,Z_j] = 0$ for
every $j$} \item{For every $i \neq j$, the matrices $X_i,Y_i,Z_i$
commute with $X_j,Y_j,Z_j$}
\end{enumerate}
Further, any such sequence of linear transformations on a finite
dimensional vector space over $k$ gives an object of $\hprod
\catfont{C}_i$.

\end{thm}

\begin{proof}
We define these transformations as the ultraproduct of the
transformations given above:
\[ X_0 = [X_0^i], \ldots, X_m = [X_m^i] \]
\[ Y_0 = [Y_0^i], \ldots, Y_m = [X_m^i] \]
\[ Z_0 = [Z_0^i], \ldots, Z_m = [Z_m^i] \]
By theorem \ref{transformationsarenicethm}, all four of the above
conditions are valid among the $X_j,Y_j$ and $Z_j$ if and only if,
for almost every $i$, all four are valid among the $X_j^i,Y_j^i$
and $Z_j^i$.  Thus every object of $\hprod \catfont{C}_i$
determines such a collection of transformations on a $\uprod
k_i$-vector space.

Conversely, suppose that we are given a sequence $X_0, \ldots,
X_m,Y_0, \ldots, Y_m,Z_0, \ldots ,Z_m$ of linear transformations
on an $n$-dimensional $\uprod k_i$-vector space $V$ satisfying all
the above; we claim there is an object $[V_i]$ of $\hprod
\catfont{C}_i$, unique up to isomorphism, to which these
transformations correspond. By proposition \ref{prop1} let $V_i$
be a collection of $n$-dimensional $k_i$-vector spaces such that
$\uprod V_i \isomorphic V$.  By proposition \ref{homlemma}, $X_0$
is uniquely of the form $[X_0^i]$, where each $X_0^i$ is a linear
transformation on $V_i$; the same goes for all of the $X_k,Y_k$
and $Z_k$.  Finally, note that given relations among the $X_k,Y_k$
and $Z_k$ amount to a finite number of equations involving
composition of maps, and so by theorem
\ref{transformationsarenicethm} these relations are almost
everywhere valid among the $X_k^i,Y_k^i$ and $Z_k^i$.  As such,
almost everywhere, they define a valid $H_1$-module structure on
$V_i$ according to theorem \ref{H1CharpLargeThm2}. The object we
seek then is $[V_i]$.  That $[V_i]$ is unique up to isomorphism is
clear from the description of morphisms in $\hprod \catfont{C}_i$
given in the following paragraphs.
\end{proof}

And what about morphisms?  By definition, a morphism
$[\phi_i]:[V_i] \rightarrow [W_i]$ in the category $\hprod
\catfont{C}_i$ is such that, for almost every $i$, $\phi_i:V_i
\rightarrow W_i$ is a morphism in the category $\catfont{C}_i$.
And by theorem \ref{morphismsInTheMethodThm}, for large enough
$i$, such $\phi_i$ are exactly those which commute with the
$X_j^i$, $Y_j^i$, and $Z_j^i$ for every $j$. Again by theorem
\ref{transformationsarenicethm} this is equivalent to saying

\begin{thm}
Let $V$, $W$ be objects of $\hprod \catfont{C}_i$, given by
(according to the previous theorem) the transformations $X_j,Y_j$
and $Z_j$ for $V$ and $R_j,S_j$ and $T_j$ for $W$.  Then $\phi =
[\phi_i]$ is a morphism between $V$ and $W$ if and only if, for
every $j$, $\phi$ satisfies
\[ X_j \circ \phi = \phi \circ R_j \hspace{1cm} Y_j \circ \phi =
\phi \circ S_j \hspace{1cm} Z_j \circ \phi = \phi \circ T_j \]
\end{thm}

We can therefore identify the category $\hprod \catfont{C}_i$ as
the collection of all finite dimensional vector spaces $V$ over $k
= \uprod k_i$, each endowed with a collection of linear
transformations
\[ X_0, \ldots, X_m, Y_0, \ldots, Y_m, Z_0, \ldots, Z_m \]
satisfying the relations given in theorem
\ref{H1hprodLabelingThm}, with morphisms being those linear maps
commuting with the $X's$, $Y's$, and $Z's$.

\section{Labelling the Objects of $\text{Rep}_k G^\infty$}

Let $k= \uprod k_i$.  What does the category $\text{Rep}_{k}
G^\infty$ look like?  By theorem \ref{directproducttheorem}
representations of $G^\infty$ on the $k$-vector space $V$ are
exactly the finite commuting products of representations of $G$ on
$V$.  And as $k$ has characteristic zero, according to theorem
\ref{H1CharZeroThm}, an individual representation of $G$ on $V$ is
determined by a triple $X,Y,Z$ of nilpotent linear transformations
on $V$ satisfying $Z = [X,Y]$, $XZ=ZX$, $YZ= ZY$, and any such
triple gives a representation.  Thus, every object of
$\text{Rep}_{k} G^\infty$ is a finite dimensional vector space $V$
with an attached collection
\[ X_0, \ldots, X_m, Y_0, \ldots, Y_m, Z_0, \ldots, Z_m \]
of nilpotent linear transformations on $V$, such that the
representation determined by $X_i,Y_i,Z_i$ commutes with the
representation determined by $X_j,Y_j,Z_j$ for $i \neq j$.  By
theorem \ref{commutingRepsInTheMethodThm}, commutativity of these
representations is equivalent to requiring that $X_i,Y_i,Z_i$ all
commute with $X_j,Y_j,Z_j$ for $i \neq j$.  Thus

\begin{thm}
\label{H1DirProdLabelingThm} Each object of $\text{Rep}_{k}
G^\infty$ is a finite dimensional $k$-vector space $V$ with an
attached sequence $X_0, \ldots, X_m,Y_0, \ldots, Y_m,Z_0, \ldots
,Z_m$ of linear transformations on $V$ satisfying

\begin{enumerate}
\item{The $X_j$, $Y_j$, and $Z_j$ are all nilpotent} \item{$Z_j =
[X_j,Y_j]$ for every $j$} \item{$[X_j,Z_j] = [Y_j,Z_j] = 0$ for
every $j$} \item{For every $i \neq j$, the matrices $X_i,Y_i,Z_i$
commute with the matrices $X_j,Y_j,Z_j$}

\end{enumerate}
Further, any such sequence of linear transformations on a finite
dimensional vector space over $k$ gives an object of
$\text{Rep}_{k} G^\infty$.

\end{thm}

Let $V$, $W$ be objects of $\text{Rep}_k G^\infty$, given by the
nilpotent transformations $X_0, \ldots, X_m$, $Y_0, \ldots, Y_m$,
$Z_0, \ldots , Z_m$ and $R_0, \ldots, R_m$, $S_0, \ldots, S_m$, $
T_0, \ldots, T_m$ respectively.  What is a morphism between these
two objects? By proposition \ref{dirProdHomprop} it is a linear
map $\phi:V \rightarrow W$ such that, for every $j$, $\phi$ is a
morphism between the representations of $G$ on $V$ determined by
$X_j,Y_j,Z_j$ and $R_j,S_j,T_j$.  And by theorem
\ref{morphismsInTheMethodThm} this is equivalent to

\begin{thm}
Let $V$ and $W$ be objects of $\text{Rep}_k G^\infty$, given by
 the transformations $X_j,Y_j$
and $Z_j$ for $V$ and $R_j,S_j$ and $T_j$ for $W$.  Then a linear
map $\phi:V \rightarrow W$ is a morphism between $V$ and $W$ if
and only if, for every $j$, $\phi$ satisfies
\[ X_j \circ \phi = \phi \circ R_j \hspace{1cm} Y_j \circ \phi =
\phi \circ S_j \hspace{1cm} Z_j \circ \phi = \phi \circ T_j \]
\end{thm}

We can therefore identify the category $\text{Rep}_k G^\infty$ as
the collection of all finite dimensional vector spaces $V$ over
$k$, each endowed with a collection of linear transformations
\[ X_0, \ldots, X_m, Y_0, \ldots, Y_m, Z_0, \ldots, Z_m \]
satisfying the relations given in theorem
\ref{H1DirProdLabelingThm}, with morphisms being those linear maps
commuting with the $X's$, $Y's$, and $Z's$.

\section{The Equivalence $\prod_H \hspace{-.03cm} \catfont{C}_i \rightarrow
\text{Rep}_{\prod_\filtfont{U} k_i} G^\infty$}

\label{theEquivalenceHresGnSection}

With the characterization for objects and morphisms in the
categories $\hprod \text{Rep}_{k_i} G$ and $\text{Rep}_{\uprod
k_i} G^\infty$ given in the previous two sections, the equivalence
(as $k$-linear abelian categories) is obvious.

What is left to verify is the not quite obvious fact that this
equivalence is tensor preserving.  First, let us examine the
tensor product on $\text{Rep}_k G^1$.  Fix two objects $V$ and
$W$, given by the transformations $X,Y,Z$ and $R,S,T$ on $V$ and
$W$ respectively.  Then their tensor product has the matrix
formula
\[ \left(\sum_{n,k,k} (c_{ij})^{(n,m,k)} x^n y^m z^k \right)
\otimes \left(\sum_{r,s,t} (d_{ij})^{(r,s,t)} x^r y^s z^t \right)
\]
\[ = \sum_{n,m,k)}\sum_{r,s,t} (c_{ij})^{(n,m,k)} \otimes
(d_{ij})^{(r,s,t)} x^{n+r} y^{m+s} z^{k+t} \]

Recalling that $(c_{ij})^{(n,m,k)} = \frac{1}{n!m!k!}Z^kY^mZ^n$
and similarly for $(d_{ij})$, we see that the coefficient matrix
for $x$ in the representation $V \otimes W$ is actually $X \otimes
R$, that for $y$ is $Y \otimes S$, and for $z$ is $Z \otimes T$.

Now consider two representations $V$ and $W$ for $G^\infty$ over
$k$, given by the transformations $X_j,Y_j,Z_j$ and $R_j, S_j,
T_j$ respectively.  Proposition \ref{tensorForDirProdsProp} tells
that the `layers' for the tensor product of $V$ and $W$ is just
the tensor product of the individual layers.  That is

\begin{prop}
\label{tensorGinftyprop} The tensor product of the representations
$V$ and $W$ is given by the sequence of transformations $X_j
\otimes R_j$, $Y_j \otimes S_j$, $Z_j \otimes T_j$.
\end{prop}

In positive characteristic the situation is slightly more
delicate. Consider for example the natural representation of $H_1$
\[
\left(%
\begin{array}{ccc}
  1 & x & z \\
   & 1 & y \\
   &  & 1 \\
\end{array}%
\right)
\]
where we consider it as a representation of $G$ in characteristic
$2$, that is, as a height-$1$ representation given by the matrices
\[ X_0 =
\left(%
\begin{array}{ccc}
  0 & 1 & 0 \\
  0 & 0 & 0 \\
  0 & 0 & 0 \\
\end{array}%
\right),
 Y_0 =
\left(%
\begin{array}{ccc}
  0 & 0 & 0 \\
  0 & 0 & 1 \\
  0 & 0 & 0 \\
\end{array}%
\right),
 Z_0 =
\left(%
\begin{array}{ccc}
  0 & 0 & 1 \\
  0 & 0 & 0 \\
  0 & 0 & 0 \\
\end{array}%
\right)
\]
The tensor product of this representation with itself is a $9
\times 9$ representation in which $x^2$ will occur as an entry,
causing it to have height $2$.  We see then that, in general, the
taking of tensor products in positive characteristic often causes
Frobenius layers to `spill over' into one another, and we do not
always have a situation analogous to proposition
\ref{tensorGinftyprop}.

However, because objects of $\hprod \text{Rep}_{k_i} G$ are
demanded to be of bounded dimension, for large enough $i$ this
difficulty will always vanish.

\begin{prop}
Let $k$ have characteristic $p>0$, and let $V$ and $W$ be
$G$-modules over $k$, say of height $m$, given by the
transformations $X_j,Y_j,Z_j$ and $R_j,S_j,T_j$. Then if $p$ is
large compared to the dimensions of both $V$ and $W$, the
representation $V \otimes W$ will also be of height $m$, given by
the transformations
\[ X_j \otimes R_j, Y_j \otimes S_j, Z_j \otimes T_j \]
for all $j = 0, \ldots, m$.
\end{prop}

\begin{proof}
We will first prove the theorem in the case where $V$ and $W$
contain a single, mutual non-zero Frobenius layer, say the
$r^{\text{th}}$ layer, given by $X, Y,Z$ and $R,S,T$; the case of
an arbitrary number of layers is an easy corollary.  We show that,
if $p$ is large compared to the dimensions of both, there is no
possibility of that layer `spilling over' into the next one.  We
can assume that $p$ is large enough so that all of $V$, $W$, and
$V \otimes W$ are of the form given by theorem
\ref{H1CharpLargeThm2}, and in particular, that all the relevant
matrices have nilpotent order $\leq p/2$. This means that we can
write the representation $V$ as
\[ \sum_{n,m,k=0}^{p/2-1} \frac{1}{n!m!k!}Z^k Y^m X^n x^{np^r} y^{mp^r} z^{kp^r} \]
and $W$ as
\[ \sum_{a,b,c=0}^{p/2-1} \frac{1}{a!b!c!}T^c S^b R^a x^{ap^r}
y^{bp^r} z^{cp^r} \] and their tensor product as
\[ = \sum_{n,m,k,a,b,c=0}^{p/2-1} \frac{1}{n!m!k!a!b!c!} (Z^kY^mX^n
\otimes T^cS^bR^a) x^{(n+a)p^r} y^{(m+b)p^r} z^{(k+c)p^r} \] Note
that this new representation still only has a single non-zero
Frobenius layer, the $r^{\text{th}}$ one, since, e.g., $(n+a)p^r <
p^{r+1}$ for every $n$ and $a$ in the summation.  Notice also the
coefficient matrix for the monomial $x^{p^r}$ is exactly $X
\otimes R$, that for $y^{p^r}$ is exactly $Y \otimes S$, and that
for $z^{p^r}$ is $Z \otimes T$.  Thus the theorem is true in the
single layer case, and the case of an arbitrary number of layers
easily follows.
\end{proof}

So, given two objects $[V_i],[W_i]$ of $\hprod \text{Rep}_{k_i}
G$, for large enough $i$ the previous proposition applies, whence
the tensor products on $\hprod \text{Rep}_{k_i} G$ and
$\text{Rep}_{\uprod k_i} G^\infty$, via our equivalence, are
compatible, and this equivalence is indeed tensor preserving.

We leave it to the reader to convince himself that all of the
arguments of this chapter can be slightly modified to prove

\begin{thm} If $k_i$ is a sequence of fields of strictly
increasing positive characteristic, then for any $n \in
\mathbb{N}$, the category $\hnprod{n} \text{Rep}_{k_i} G$ is
tensorially equivalent to $\text{Rep}_{\uprod k_i} G^n$.
\end{thm}

\chapter{Height-Restricted Generic Cohomology}

One application of theorem \ref{heightResMainTheorem} of the
previous chapter is to give quick and intuitive large
characteristic, `height-restricted' generic cohomology results for
the two unipotent groups we have studied, at least in the case of
$\text{Ext}^1$.

\section{First-Order Definability of $\text{Ext}^1$}

\label{FODofExt1Section}

Let $M$ and $N$ be objects in some tannakian category over the
field $k$, fix $n$, and let $\xi_1, \ldots, \xi_m$ be a sequence
of diagrams of the form
\[ \xi_j: 0 \rightarrow N \rightarrow X_j \rightarrow M \rightarrow 0 \]
To prevent ourselves from having to repeat the same long-winded
sentence over and over again, we define the formula
\index{$\text{LISE}$} $\text{LISE}(\xi_1, \ldots, \xi_m,M,N)$ to
mean `` $\xi_1, \ldots, \xi_m$ is a linearly independent sequence
of $1$-fold extensions of $M$ by $N$.''

\begin{thm}
\label{FODofLISEthm}  For fixed $m$, the formula
$\text{LISE}(\xi_1, \ldots, \xi_m,M,N)$, modulo the theory of
tannakian categories, is expressible as a first-order formula in
the language of abelian tensor categories.
\end{thm}

\begin{proof}

A full proof that $\text{LISE}$ is a first-order formula would be
an unnecessarily mind-numbing exercise; we shall instead be
content to give an outline of such a proof, leaving it to the
reader to fill in the necessary details.

Define the formula $\text{Exten}(N,\iota,X,\pi,M)$ to mean that
these objects and morphisms comprise a $1$-fold extension of $M$
by $N$.  This amounts to demanding that $N$, $X$, and $M$ are
objects, that $\iota$ and $\pi$ are morphisms, that the morphisms
point between the objects we want them to, that $\iota$ is
injective, that $\pi$ is surjective, and that the sequence is
exact at $X$. This can be translated into a first-order sentence.

Now the formula $\text{LISE}(\xi_1, \ldots, \xi_n,M,N)$ doesn't
make sense on its face, since we are treating the extensions
$\xi_i$ as if they were elements of our category, which they are
not.  If we were being strictly formal, we should instead use the
objects and morphisms comprising the extensions as the variables,
and make the additional assertions that they are all extensions of
$M$ by $N$.  But since $\text{Exten}$ is first-order, this can
certainly be done, so we are justified in using this abbreviation.

The formula $\text{LISE}(\xi_1, \ldots, \xi_m,M,N)$ should go
something like, ``$\xi_1, \ldots, \xi_m$ are extensions of $M$ by
$N$, and for any scalars $k_1, \ldots, k_n$, if $k_1 \xi_1 \oplus
\ldots \oplus k_n \xi_n$ is equivalent to the trivial extension,
then $k_1 = \ldots = k_n = 0$.''  It is then simply a matter of
showing that the concepts of being a scalar, of scalar
multiplication of extensions, of Baer sum of extensions, of being
a trivial extension, and of being the zero scalar are all
first-order.

The first is obvious; to be a scalar simply means that it is an
endomorphism of the identity object, which is clearly first-order.
If $\phi:X \rightarrow Y$ is a morphism and $k$ a scalar, then we
define the scalar multiplication of $k$ on $\phi$ to be the
composition
\[ X \stackrel{\text{unit}_X}{\vlongrightarrow} \underline{1} \otimes X   \stackrel{k \otimes \phi}{\vlongrightarrow}
\underline{1} \otimes Y
\stackrel{\text{unit}_Y^{-1}}{\vlongrightarrow} Y
\]
which again is first-order.  This allows us, under the definition
given in section \ref{CohomologyOfComodulesSection}, to define the
scalar multiplication of an extension in a first-order way.  (We
would of course have to make separate definitions for the case
when $k = 0$ or $k \neq 0$; but this is no problem, since ``$k$ is
the zero scalar'' and ``$\xi$ is a trivial extension'' are both
first-order, as is shown below.)

As for the Baer sum, we ask the reader to see the definition of it
given in section \ref{CohomologyOfComodulesSection}.  It involves
such concepts as ``being a pullback'', ``being the unique map
pushing through a pullback'', ``being a cokernel'', ``being the
unique map pushing through a cokernel'', etc.  All of these
concepts are expressed in terms of universal properties, which are
quite amenable to being expressed in a first-order fashion.  They
simply state that, given a collection of morphisms making such and
such a diagram commute, there is a unique morphism making such and
such a diagram commute. These types of statements are plainly
first-order.

The statement that two extensions are equivalent is first-order;
it is merely the assertion that there exists a morphism making an
equivalence diagram between the two extensions commute (this is
the only point in the proof at which it is necessary to restrict
to $\text{Ext}^1$ as opposed to higher $\text{Ext}$; see section
\ref{TheDifficultyWithSection} for more on this). Then to say that
$\chi$ is equivalent to the trivial extension would go something
like ``if $\xi$ is any extension, then $\xi \oplus \chi$ is
equivalent to $\xi$''.

Finally, to say that the scalar $k$ is the zero scalar is simply
to say that it is the additive identity of the field
$\text{End}(\underline{1})$, which is clearly first-order.

\end{proof}

\begin{cor}
For fixed $n$, the formula ``$\extDim \text{Ext}^1(M,N) = n$'' is
first-order.
\end{cor}

\begin{proof}
It is equivalent to ``there exist $n$ linear independent $1$-fold
extensions of $M$ by $N$, and there do not exist $n+1$ of them''.
\end{proof}

\section{Generic Cohomology for $\text{Ext}^1$}

\label{GenericCohomologyforExt1}

In this section $G$ denotes any unipotent group, defined over
$\mathbb{Z}$, for which the conclusion of theorem
\ref{heightResMainTheorem} is true (and in particular, has Hopf
algebra isomorphic to $A = \mathbb{Z}[x_1, \ldots, x_n]$ for some
$n$), $k_i$ is a sequence of fields of strictly increasing
positive characteristic, $\catfont{C}_i = \text{Rep}_{k_i} G$, and
$k$ is the ultraproduct of the fields $k_i$.

\begin{defn}
Let $k$ be a field of characteristic $p>0$, $M$ and $N$ modules
for $G$ over $k$, and let $n,h \in \mathbb{N}$.  Then we define
\index{$\text{Ext}^{n,h}(M,N)$}
\index{extension!height-restricted}
\[ \text{Ext}^{n,h}_{G(k)}(M,N) \]
to be the subset of $\text{Ext}^n_{G(k)}(M,N)$ consisting of those
(equivalence classes of) $n$-fold extensions of $M$ by $N$ such
that, up to equivalence, each of the extension modules can be
taken to have height less than or equal to $h$.
\end{defn}

Example: the (equivalence class of) the extension
\[ 0 \rightarrow k \rightarrow
\left(%
\begin{array}{cc}
  1 & x^{p^2} \\
  0 & 1 \\
\end{array}%
\right) \rightarrow k \rightarrow 0 \]
is a member of
$\text{Ext}^{1,3}_{G_a(k)}(k,k)$, but not of
$\text{Ext}^{1,2}_{G_a(k)}(k,k)$.

\begin{lem}
\label{heightRespectsSubmodulesAndQuotientsLemma} If $M$ is a
$G$-module of height no greater than $h$, then any submodule or
quotient of $M$ also has height no greater than $h$.  If $M$ and
$N$ have height no greater than $h$, so does $M \oplus N$.
\end{lem}

\begin{proof}
The case of subobjects and quotients follows immediately from
lemma \ref{comoduleCoefficientsLemma}: any subobject or quotient
of $M$ will have matrix formula with entries who are linear
combinations of the entries of $M$.  The case of $M \oplus N$ is
even easier to see, examining the usual matrix representation for
a direct sum.
\end{proof}

\begin{thm}
Let $M$ and $N$ be modules for $G$ over a field $k$ of
characteristic $p>0$, of height no greater than $h$.  Then
$\text{Ext}^{1,h}_{G(k)}(M,N)$ is a subspace of
$\text{Ext}^1_{G(k)}(M,N)$.
\end{thm}

\begin{proof}
Let $\xi,\chi$ be extensions in $\text{Ext}^{1,h}_{G(k)}(M,N)$,
with the extension modules of both $\xi$ and $\chi$ having height
no greater than $h$. We examine the definitions given for the Baer
sum and scalar multiplications in section
\ref{CohomologyOfComodulesSection}. Clearly a non-zero scalar
multiple of either of them is still in
$\text{Ext}^{1,h}_{G(k)}(M,N)$.  As for the Baer sum $\xi \oplus
\chi$, we recall the concrete constructions of a pullback or
pushout of $G$-modules.  The former is defined as a certain
subobject of the direct sum of two modules, and the other a
certain quotient of their direct sum.  By the previous lemma both
of these constructions yield modules of height no greater than
those of the originals.  Thus $\text{Ext}^{1,h}_{G(k)}(M,N)$ is
closed under the Baer sum.  The trivial extension has extension
module isomorphic to the direct sum of $M$ and $N$, again by the
previous lemma, of height no greater than that of $M$ or $N$.
\end{proof}

Let $M$ and $N$ be modules for $G$ over $\mathbb{Z}$.  Then it of
course makes sense to consider them as modules for $G$ over any
field.  Further, for any $n \in \mathbb{N}$, and indeed for $n =
\infty$, we can consider them as modules for $G^n$ over any field,
that is, as representations with a single layer, according to
theorem \ref{directproducttheorem}.

Our goal for the rest of this section is to prove

\begin{thm}
\label{genCohomExt1Thm} Let $h \in \mathbb{N}$, $M,N$ modules for
$G$ over $\mathbb{Z}$. Suppose that the computation $\extdim
\text{Ext}^1_{G^h(k)}(M,N) = m$ is the same for any characteristic
zero field $k$.  Let $k_i$ be a sequence of fields of strictly
increasing characteristic.

\begin{enumerate}

\item{If $m$ is finite, then for sufficiently large $i$
\index{$\text{Ext}^{n,h}(M,N)$}
\[ \extdim \text{Ext}^{1,h}_{G(k_i)}(M,N) = m \]
} \item{If $m = \infty$, then for any sequence of fields $k_i$ of
strictly increasing characteristic, $\extdim
\text{Ext}^{1,h}_{G(k_i)}(M,N)$ diverges to infinity with
increasing $i$.}

\end{enumerate}
\end{thm}

For each $i$, suppose we have a diagram in $\catfont{C}_i$
\[ \xi^i: 0 \rightarrow N_i \stackrel{\iota^i}{\longrightarrow}
X_i \stackrel{\pi^i}{\longrightarrow} M_i \rightarrow 0 \] Denote
by $[\xi^i]$ corresponding diagram in $\uprod \catfont{C}_i$:
\[ [\xi^i]: 0 \rightarrow [N_i] \stackrel{[\iota^i]}{\longrightarrow}
[X_i] \stackrel{[\pi^i]}{\longrightarrow} [M_i] \rightarrow 0 \]

\begin{prop}
For each $i$, let $\xi_i^1, \ldots, \xi_i^m$ be the sequence of
diagrams in $\catfont{C}_i$
\begin{gather*}
 \xi_i^1: 0 \rightarrow N_i \rightarrow X_i^1 \rightarrow M_i
\rightarrow 0 \\
 \vdots \\
 \xi_i^m: 0 \rightarrow N_i \rightarrow X_i^m \rightarrow M_i
\rightarrow 0
\end{gather*}
 Then the formula $\text{LISE}(\xi_i^1, \ldots,
\xi_i^m,M_i,N_i)$ holds in almost every $\catfont{C}_i$ if and
only if the formula $\text{LISE}([\xi_i^1], \ldots,
[\xi_i^m],[M_i],[N_i])$ holds in $\uprod \catfont{C}_i$.
\end{prop}

\begin{proof}
Apply theorems \ref{FODofLISEthm} and \ref{Los'sTheorem}.
\end{proof}

Now fix two modules $M$ and $N$ for $G$ over $\mathbb{Z}$, and for
each $i$, let $\xi_i^1, \ldots, \xi_i^m$ be the sequence of
diagrams in $\catfont{C}_i$
\begin{gather*}
 \xi_i^1: 0 \rightarrow N \rightarrow X_i^1 \rightarrow M
\rightarrow 0 \\
 \vdots \\
 \xi_i^m: 0 \rightarrow N \rightarrow X_i^m \rightarrow M
\rightarrow 0
\end{gather*}
 Further, assume that each $\xi_i^j$ is a member of
$\text{Ext}^{1,h}_{G(k_i)}(M,N)$; this is merely the assertion
that every $X_i^j$ has height no greater than $h$.  As $M$ and $N$
are constant over $i$ and have height $1$, each of $[M],[N]$, and
$[X_i^j]$ have bounded height, whence $[\xi_i]_j
\stackrel{\text{def}}{=} [\xi_i^j]$ is an extension in $\hnprod{h}
\catfont{C}_i$.

\begin{prop}
\label{Ext1InequalityProp} For fixed $n,h \in \mathbb{N}$ and
modules $M$ and $N$ for $G$ over $\mathbb{Z}$, the statement
``$\extdim \text{Ext}^{1,h}_{G(k_i)}(M,N) \geq n$'' holds for
almost every $i$ if and only if $\extdim \text{Ext}^1_{G^h(\uprod
k_i)}(M,N) \geq n$.
\end{prop}

\begin{proof}
Suppose $\extdim \text{Ext}^{1,h}_{G(k_i)}(M,N) \geq n$ holds for
almost every $i$.  This means that, for almost every $i$, we have
a linearly independent sequence of $1$-fold extensions of $M$ by
$N$
\begin{gather*}
 \xi_i^1: 0 \rightarrow N \rightarrow X_i^1 \rightarrow M
\rightarrow 0 \\
 \vdots \\
 \xi_i^m: 0 \rightarrow N \rightarrow X_i^m \rightarrow M
\rightarrow 0
\end{gather*}
 with each $X_i^j$ being of height $\leq h$. Note
that the objects $[X_i]^1, \ldots, [X_i]^m$ are of bounded height
and dimension; then these project to the sequence of diagrams
$[\xi_i]^1, \ldots, [\xi_i]^m$ in $\hnprod{h} \catfont{C}_i$.  As
the formula $\text{LISE}$ is first-order, these extensions,
considered as diagrams in the full ultracategory $\uprod
\catfont{C}_i$, are also linearly independent, and by proposition
\ref{linIndExtensionsInSubcategoriesProp}, so also are they in the
undercategory $\hnprod{h} \catfont{C}_i$.  Note also that, under
the equivalence $\hnprod{h} \catfont{C}_i \isomorphic
\text{Rep}_{\uprod k_i} G^h$ given in section
\ref{theEquivalenceHresGnSection} the objects $[M]$ and $[N]$ in
$\hnprod{h} \catfont{C}_i$ actually correspond to the objects $M$
and $N$ in $\text{Rep}_{\uprod k_i} G^h$.  This gives a collection
of $m$ linearly independent extensions of $N$ by $M$ in the
category $\text{Rep}_{\uprod k_i} G^h$; thus, $\extdim
\text{Ext}^1_{G^h(\uprod k_i)}(M,N) \geq n$.

The converse is proved similarly; if $\extdim
\text{Ext}^1_{G^h(\uprod k_i)}(M,N) \geq n$, take a linearly
independent sequence of extensions $[\xi_i]^1, \ldots, [\xi_i]^n$
of $N$ by $M$ in $\hnprod{h} \catfont{C}_i$, which project back
to, for almost $i$, a linearly independent sequence of extensions
$\xi_i^1, \ldots, \xi_i^n$ of $M$ by $N$ in $\catfont{C}_i$,
showing $\extdim \text{Ext}^{1,h}_{G(k_i)}(M,N) \geq n$ for almost
every $i$.
\end{proof}

\begin{cor}
The statement $\extdim \text{Ext}^1_{G^h(\uprod k_i)}(M,N) = n$
holds if and only if $\extdim \text{Ext}^{1,h}_{G(k_i)}(M,N) = n$
holds for almost every $i$.
\end{cor}

\begin{proof}
The above statement is equivalent to the conjunction ``$\extdim
\text{Ext}^1_{G^h(\uprod k_i)}(M,N) \geq n$'' and NOT ``$\extdim
\text{Ext}^1_{G^h(\uprod k_i)}(M,N) \geq n+1$''.  Apply the
previous proposition.
\end{proof}

We can now prove theorem \ref{genCohomExt1Thm}.  Suppose that the
computation $\extdim \text{Ext}^1_{G^h(k)}(M,N) = n$ is both
finite and the same for \emph{any} characteristic zero field $k$.
Then in particular, it is the same for the field $\uprod k_i$ for
\emph{any} choice of non-principal ultrafilter.  Let $J \subset I$
be the set on which $\extdim \text{Ext}^{1,h}_{G(k_i)}(M,N) = n$
is true.  By the previous corollary, $J$ is large for every choice
of non-principal ultrafilter, and by corollary
\ref{cofinitelemma}, $J$ is cofinite.  This proves the first part
of the theorem.

If instead $\extdim \text{Ext}^1_{G^h(k)}(M,N)$ is infinite, then
for any $n \in \mathbb{N}$, the statement $\extdim
\text{Ext}^{1,h}_{G(k_i)} = n$ is false for almost every $i$, for
every choice of non-principal ultrafilter, whence the statement
$\extdim \text{Ext}^{1,h}_{G(k_i)} = n$ is false on a cofinite
set. This goes for every $n \in \mathbb{N}$, whence $\extdim
\text{Ext}^{1,h}_{G(k_i)}$ is divergent.

\subsection{An Example}

\label{GenCohomExampleSection}

We shall illustrate an application of theorem
\ref{genCohomExt1Thm} with a simple, easily verifiable example.

Let $G=G_a$ and consider \index{$\text{Ext}^n(M,N)$}
$\text{Ext}^1_{G_a(k)}(k,k)$, where $k$ has characteristic $p>0$.
As the extension module of any extension of $k$ by $k$ has
dimension $2$, and as $p \geq 2$ for all primes, theorem
\ref{Gacharptheorem} applies, whence any $2$-dimensional
representation of $G_a$ is given by a finite sequence $X_0,
\ldots, X_m$ of commuting nilpotent matrices over $k$.  We can
take $X_0$ to be in Jordan form times some scalar, namely
\[
\left(%
\begin{array}{cc}
  0 & c_0 \\
  0 & 0 \\
\end{array}%
\right)
\]
for some scalar $c_0$.  The centralizers of this matrix are
exactly those of the form
\[
\left(%
\begin{array}{cc}
  a & b \\
  0 & a \\
\end{array}%
\right)
\]
and if we demand them to be nilpotent, we must have $a=0$.  Thus
the $X_i$ can be taken to be
\[
X_0 =
\left(%
\begin{array}{cc}
  0 & c_0 \\
  0 & 0 \\
\end{array}%
\right) , \ldots, X_m =
\left(%
\begin{array}{cc}
  0 & c_m \\
  0 & 0 \\
\end{array}%
\right)
\]
for some scalars $c_0, \ldots, c_m$.  The representation they
generate according to theorem \ref{Gacharptheorem} is
\[
\left(
\begin{array}{cc}
  1 & c_0x+c_1x^p + \ldots + c_m x^{p^m} \\
  0 & 1 \\
\end{array}
\right)
\]
These are all extensions of $k$ by $k$ with the obvious injection
$1 \mapsto (1,0)$ and projection $(1,0) \mapsto 0,(0,1) \mapsto
1$, and any extension of $k$ by $k$ must be of this form.
Therefore \index{extension} extensions of the form
\[ 0 \rightarrow k \rightarrow
\left(
\begin{array}{cc}
  1 & c_0x+c_1x^p + \ldots + c_m x^{p^m} \\
  0 & 1 \\
\end{array}
\right) \rightarrow k \rightarrow 0 \] constitute all extensions
of $k$ by $k$.  Denote by $\xi_m$ the extension
\[\xi_m: 0 \rightarrow k \rightarrow
\left(
\begin{array}{cc}
  1 & x^{p^m} \\
  0 & 1 \\
\end{array}
\right) \rightarrow k \rightarrow 0 \] Then direct computation
shows that the \index{extension!Baer sum of} Baer sum of $\xi_m$
and $\xi_n$ is the extension
\[ \xi_m \oplus \xi_n: 0 \rightarrow k \rightarrow
\left(
\begin{array}{cc}
  1 & x^{p^m}+x^{p^n} \\
  0 & 1 \\
\end{array}
\right) \rightarrow k \rightarrow 0
\]
and that, for $c \neq 0$, the \index{extension!scalar
multiplication of} scalar multiplication $c \xi_m$ is
\index{extension!equivalence of} (equivalent to)
\[ c\xi_m: 0 \rightarrow k \rightarrow
\left(
\begin{array}{cc}
  1 & c x^{p^m}\\
  0 & 1 \\
\end{array}
\right) \rightarrow k \rightarrow 0
\]
A basis for $\text{Ext}^1_{G_a(k)}(k,k)$ is therefore given by
$\xi_1, \xi_1, \ldots $.  If instead we restrict to
\index{extension!height-restricted}
$\text{Ext}^{1,h}_{G_a(k)}(k,k)$, then this is a finite
dimensional subspace spanned by $\xi_0, \ldots, \xi_{h-1}$.

Now consider $\text{Ext}^1_{G_a^h(k)}(k,k)$, where $k$ now has
characteristic zero.  Using theorems \ref{Gacharzerothm} and
\ref{directproducttheorem} virtually identical computations to the
above show it to be spanned by the linearly independent extensions
$\chi_0, \ldots, \chi_{h-1}$, given by
\[ \chi_m: 0 \rightarrow k \rightarrow
\left(
\begin{array}{cc}
  1 & x_m\\
  0 & 1 \\
\end{array}
\right) \rightarrow k \rightarrow 0
\]
where $x_m$ denotes the $m^{th}$ free variable of the Hopf algebra
$k[x_0, \ldots, x_{h-1}]$, and that the Baer sum and scalar
multiplication of extensions give analogous results to that of the
above.  We see then that
\[ \extDim \text{Ext}^{1,h}_{G_a(k)}(k,k) =  \extDim
\text{Ext}^1_{G_a^h(k^\prime)}(k^\prime,k^\prime) \] when $k$ has
characteristic $p$ and $k^\prime$ has characteristic zero.  In
particular we conclude that, if $k_i$ is a sequence of fields of
increasing positive characteristic, then
\[ \extDim \text{Ext}^{1,h}_{G_a(k_i)}(k_i,k_i) \hspace{.2cm}
\vlongrightarrow\hspace{.2cm} \extDim \text{Ext}^1_{G_a^h(\uprod
k_i)}(\uprod k_i,\uprod k_i) \] which is predicted by theorem
\ref{genCohomExt1Thm}.

The reader should note that this example is misleading, in that
the generic value of $\text{Ext}^{1,h}_{G_a(k)}(k,k)$ is attained
for \emph{any} positive characteristic $p \geq 2$.  This was
simply due to the fact that theorem
\ref{Gacharzeropanaloguetheorem} applies to all characteristics in
dimension $2$, i.e.~because any $2 \times 2$ nilpotent matrix is
nilpotent of order $\leq 2$.  If instead we were to consider
$\text{Ext}^{1,h}_{G_a(k_i)}(k_i \oplus k_i, k_i \oplus k_i)$,
then (assuming the computation on the right does not depend on the
particular characteristic zero field) we would still have
\[ \extDim \text{Ext}^{1,h}_{G_a(k_i)}(k_i \oplus k_i,k_i \oplus k_i) \hspace{.2cm}
\vlongrightarrow\hspace{.2cm} \extDim \text{Ext}^1_{G_a^h(\uprod
k_i)}(\uprod k_i \oplus \uprod k_i, \uprod k_i \oplus \uprod k_i)
\]
only this time we would have to wait for $\text{char}(k_i) = 5$
for the generic value to be obtained.

\section{The Difficulty with Higher Ext}

\label{TheDifficultyWithSection}

To finish, we mention a few of the reasons why our attempts to
apply this machinery to $\text{Ext}^n$ for $n>1$ have so far
proved unfruitful.

In the previous section we saw that there is a $1-1$
correspondence between extensions in $\text{Ext}^1_{G^h(\uprod
k_i)}(M,N)$ and almost everywhere extensions in
$\text{Ext}^{1,h}_{G(k_i)}(M,N)$. But for higher Ext, this does
not always work. Here is what can go wrong.  For concreteness'
sake consider $\text{Ext}^{2,h}_{G(k_i)}(M,N)$, and suppose that,
for each $i$, we have an element $\xi_i \in
\text{Ext}^{2,h}_{G(k_i)}(M,N)$.  This means that each $\xi_i$ is
of the form
\[ \xi_i:0 \rightarrow M \rightarrow X_i \rightarrow Y_i
\rightarrow N \rightarrow 0 \] with every $X_i$ and $Y_i$ being,
up to equivalence of extensions, of height $\leq h$.  But this
says nothing about the \emph{dimensions} of $X_i$ and $Y_i$, and
indeed there is every reason to suspect that $\text{dim}(X_i)$ and
$\text{dim}(Y_i)$ diverge as $i$ becomes large.  As the objects of
$\hnprod{h} \text{Rep}_{k_i} G$ are demanded to have bounded
dimension as well as height, the objects $[X_i]$ and $[Y_i]$ will
not belong to $\hnprod{h} \text{Rep}_{k_i} G$, and hence the
extension $[\xi_i] \in \uprod \text{Rep}_{k_i} G$ will not belong
to $\hnprod{h} \text{Rep}_{k_i} G$.

Another problem we face in the case of higher $\text{Ext}$ is in
trying to define equivalence of extensions in a first-order way.
For $\text{Ext}^1$ this was no problem, since if two $1$-fold
extensions are equivalent, there is necessarily an actual
equivalence map between them, which is easily asserted in a
first-order way. But this is not so for $n$-fold extensions in
general; equivalent extensions need not have an actual equivalence
mapping between them (see section
\ref{CohomologyOfComodulesSection}).

To illustrate the problem, suppose we have, for each $i$, two
equivalent extensions $\xi_i$ and $\chi_i$ in the category
$\catfont{C}_i$. What this says is that, for each $i$, there
exists a finite sequence of $m_i$ extensions $\rho_i^1,\rho_i^2,
\ldots, \rho_i^{m_i}$ forming a chain of concrete equivalencies
leading from $\xi_i$ to $\chi_i$.  But there is every reason to
suspect that $m_i$ diverges to infinity as $i$ becomes large. As
such, these equivalencies between $\xi_i$ and $\chi_i$ in the
categories $\catfont{C}_i$ do not necessarily project to an
equivalence between the extensions $[\xi_i]$ and $[\chi_i]$ in the
category $\uprod \catfont{C}_i$.  If such a collection were found
(and we have none in mind), this would in fact prove that the
property of being equivalent is not first-order.

With these difficulties in mind, we tried instead to prove the
following inequality:

\begin{thm}
Let $n,h \in \mathbb{N}$, and let $M$ and $N$ be modules for $G$
over $\mathbb{Z}$.  Suppose that the computation $\extdim
\text{Ext}^n_{G^h(k)}(M,N) = m$ (where $m$ could possibly be
infinite) is the same for every characteristic zero field $k$.
Then for any sequence of fields $k_i$ of increasing positive
characteristic
\[ \extdim \text{Ext}^{n,h}_{G(k_i)}(M,N) \geq m \]
for all sufficiently large $i$.
\end{thm}

But the obvious attempt at a proof of this falls apart as well.
Suppose we had a sequence $[\xi_i]^1, \ldots, [\xi_i]^n$ of
linearly independent extensions in the category $\hnprod{h}
\text{Rep}_{k_i} G$; then we would like to see that these project
back to an almost everywhere sequence of linearly independent
extensions $\xi_i^1, \ldots, \xi_i^n$ in the categories
$\catfont{C}_i$.  But even this, as far as we can tell, is not
guaranteed.  Linear independence means that, whenever $a_i^1
\xi_i^1 \oplus \ldots \oplus a_i^n \xi_i^n$ is a trivial
extension, then $a_i^1 = \ldots = a_i^n = 0$.  Being a trivial
extension in turn means that, whenever $\chi_i$ is any extension,
$\chi_i \oplus (a_i^1 \xi_i^1 \oplus \ldots \oplus a_i^n \xi_i^n)$
is equivalent to $\chi_i$.  But again, equivalence of extensions
is not necessarily first-order, and so neither is the property of
being a trivial extension.  We see then that a linear dependence
among the $\xi_i^1, \ldots, \xi_i^n$ does not necessarily project
to a linear dependence among the $[\xi_i]^1, \ldots, [\xi_i]^n$,
and we cannot automatically conclude that the $a_i^1, \ldots,
a_i^n$ are equal to zero for almost every $i$.





\bibliography{common}

\begin{thebibliography}{10}

\bibitem{benson}
D.~J. Benson.
\newblock {\em Representations and cohomology: Volume I}.
\newblock Cambridge University Press, New York, 1991.

\bibitem{borel}
Armand Borel.
\newblock {\em Linear Algebraic Groups}.
\newblock Graduate Texts in Mathematics. Springer-Verlag, New York,
  $2^\text{nd}$ edition, 1991.

\bibitem{breen}
Lawrence Breen.
\newblock {\em Motives}, volume~55 of {\em Proc. Symp. Pure Math.}, chapter
  Tannakian Categories, pages 337--376.
\newblock 1994.

\bibitem{hopfalgebras}
Nastasescu Dascalescu and Raianu.
\newblock {\em Hopf Algebras: An Introduction}.
\newblock Pure and Applied Mathematics. Marcel Dekker, New York, 2001.

\bibitem{deligne}
P.~Deligne and J.~S. Milne.
\newblock {\em Hodge Cycles, Motives, and Shimura Varieties}, chapter Tannakian
  Categories, pages 101--137.
\newblock Lecture Notes in Mathematics. Springer-Verlag, New York, 1982.

\bibitem{LucasThm}
N.J. Fine.
\newblock Binomial coefficients modulo a prime.
\newblock {\em Amer. Math. Monthly}, pages 589--592, Dec 1974.

\bibitem{freyd}
Peter Freyd.
\newblock {\em Abelian Categories: An Introduction to the Theory of Functors}.
\newblock Harper and Row, New York, 1964.

\bibitem{FP}
E~M Friedlander and B~J Parshall.
\newblock Cohomology of infinitesimal and discrete groups.
\newblock {\em Mathematische Annalen}, pages 353--374, 1986.

\bibitem{greub}
W.~H. Greub.
\newblock {\em Multilinear Algebra}.
\newblock Springer-Verlag, New York, 1967.

\bibitem{hall}
Brian~C. Hall.
\newblock {\em Lie Groups, Lie Algebras, and Representations: An Elementary
  Introduction}.
\newblock Graduate Texts in Mathematics. Springer-Verlag, New York, 2003.

\bibitem{hodges}
Wilfrid Hodges.
\newblock {\em A Shorter Model Theory}.
\newblock Cambridge University Press, New York, 1997.

\bibitem{humphreys}
James~E. Humphreys.
\newblock {\em Linear Algebraic Groups}.
\newblock Graduate Texts in Mathematics. Springer-Verlag, New York, 1981.

\bibitem{jantzen}
Jens~Carsten Jantzen.
\newblock {\em Representations of Algebraic Groups}, volume 131 of {\em Pure
  and Applied Mathematics}.
\newblock Academic Press, Orlando, FL, 1987.

\bibitem{maclane}
Saunders~Mac Lane.
\newblock {\em Categories for the Working Mathematician}.
\newblock Graduate Texts in Mathematics. Springer-Verlag, New York,
  $2^{\text{nd}}$ edition, 1998.

\bibitem{saavedra}
Neantro~Saavedra Rivano.
\newblock {\em Cat\'egories Tannakiennes}, volume 265 of {\em Lecture Notes in
  Mathematics}.
\newblock Springer-Verlag, New York, 1972.

\bibitem{waterhouse}
William~C. Waterhouse.
\newblock {\em Introduction to Affine Group Schemes}.
\newblock Graduate Texts in Mathematics. Springer-Verlag, New York, 1979.

\bibitem{weibel}
Charles~A. Weibel.
\newblock {\em An Introduction to Homological Algebra}, volume~38 of {\em
  Cambridge Studies in Advanced Mathematics}.
\newblock Cambridge University Press, New York, 1994.

\end{thebibliography}
\bibliographystyle{plain}


\appendix

\chapter{Model Theory and First-Order Languages}

Here is a very basic and sometimes imprecise introduction to the
notions of \index{model} models and first-order languages; it will
be just enough to get by.  The reader is encouraged to consult
\cite{hodges} for an excellent introduction to the subject.

We abandon the term `model' for the moment and instead focus on
the notion of \index{relational structure} \textbf{relational
structure}. This is by definition a set $X$ (called the domain of
the structure) endowed with the following: a collection $\{f_i\}$
of $n_i$-ary functions on $X$ (functions from $X^{n_i}$ to $X$,
where $n_i \in \mathbb{N}$), a collection $\{r_j\}$ of $n_j$-ary
relations on $X$ (subsets of $X^{n_j}$) , and a collection of
`constants' $\{c_k\}$, certain distinguished elements of the
domain. The various labels given to these functions, relations and
constants is called the \index{signature} \textbf{signature} of
the structure.   Many (but not all) of the usual mathematical
structures one comes across can be realized as relational
structures. We do not at all demand that a signature be finite,
but all of the examples given in this dissertation will have
finite signatures.

Example: a field $k$ can be realized as a relational structure.  A
natural choice for signature might be the two binary functions $+$
and $\multsymbol$, the unary function $-$, and the two constants
$0$ and $1$, representing the obvious.  We abusively call this the
\index{first-order language!of fields} `signature of fields',
realizing that a random structure in this signature is not at all
guaranteed to be a field.

The whole point of bothering with which symbols you choose to
attach to a relational structure is three-fold. Firstly, it
determines the definition of a `homomorphism' of relational
structures (always assumed to be between structures in the same
signature); namely, a homomorphism is demanded to preserve
relations, functions, and send constants to constants. Secondly,
it determines the notion of a `substructure' $A$ of a structure
$B$, which by definition must contain all constants, be closed
under all functions, and such that the relations on $A$ are
compatible with those on $B$.  Note for instance that we included
the symbol $-$ in the language of fields, whence any substructure
of a field must be closed under negation.  If we were to omit this
symbol, this would no longer be the case; e.g.~$\mathbb{N}$ would
now qualify as a substructure of $\mathbb{Q}$.

Thirdly, and most importantly for us, the signature of a structure
determines the structure's \index{first-order language}
\textbf{first-order language}. Roughly speaking, the first-order
language of a structure is the collection of all meaningful
`formulae' one can form, in certain prescribed ways, using the
symbols of the signature as the primitive elements of the
language.

Any language, at the least, needs certain primitive verbs and
nouns.  In the context of first-order languages verbs are called
\textbf{predicates} and nouns are called \textbf{terms}.  For a
given signature we define the terms of our language as follows:

\begin{enumerate}
\item{Any variable is a term (a variable is any convenient symbol
you might choose not being used by the language already, e.g.
$x,y,a,b$, etc.)} \item{Any constant symbol is a term} \item{If
$f$ is an $n$-ary function symbol in the language and $t_1,
\ldots, t_n$ are terms, then so is $f(t_1, \ldots, t_n)$.}
\end{enumerate}

In the case of fields, $1$ is a term, so is $x$, so is $0
\multsymbol x$, and so is $(x + y) \multsymbol (1+(-z))$. These
represent the `nouns' of our language.

Next we need predicates, ways to say stuff about our nouns.
 This is the role fulfilled by the \emph{relational} symbols of our
language, as well as the binary relational symbol `=',
representing equality, which we always reserve for ourselves. We
define the \textbf{atomic formulae}, which one can think of as the
most basic sentences belonging to our language, as follows:

\begin{enumerate}
\item{If $s$ and $t$ are terms, then $s=t$ is an atomic formula.}
\item{If $r$ is an $n$-ary relational symbol, and $t_1, \ldots,
t_n$ are terms, then $r(t_1, \ldots, t_n)$ is an atomic formula.}

\end{enumerate}

In the case of fields, $1=0$ is an atomic formula, and so is
$x+y=1\multsymbol z$.  The signature we chose for fields did not
include any relational symbols other than `=', so all atomic
formulae in this signature must be built from this.

We are of course not content to restrict ourselves to these
primitive formulae; we want to able to put them together using the
usual logical symbols.  Our primitive logical symbols are $\myand,
\myor$, and $\mynot$, representing `and', `or', and `not'. Thus,
the following are all formulae in the first-order language of
fields: $\mynot(1=0) \myand x=y$, $\mynot(1+x=x)$, and
$\mynot(1=0) \myand \mynot(1+1=0) \myand \mynot(1+1+1=0) \myand
\mynot(1+1+1+1=0)$. In higher order languages, there is indeed a
notion of conjunction or disjunction of an infinite collection of
formulae, but the definition of a first-order language explicitly
disallows this. All logical combinations of formulae take place
over finite collections of formulae.

We finally have two more symbols, namely $\forall$ and $\exists$,
representing universal and existential quantification.  For any
formulae $\Phi$ in our language, and any variable $x$, we also
have the formula $\forall x \Phi$ and $\exists x \Phi$.  So, for
example, in the language of fields, the following are formulae:
$\forall x (0 \multsymbol x = 0)$, $(\forall x)(\forall y)(x
\multsymbol y = y \multsymbol x)$, and $\mynot(\exists x)(x
\multsymbol 0 = 1)$.  It is important to remember that
quantification is always understood to be over the \emph{elements}
of a structure; in particular, we have no concept in first-order
logic of quantification over \emph{subsets} of a structure.

To make things manageable, we shall not hesitate to use
abbreviations. For two formulae $\Phi$ and $\Psi$, $\Phi
\myimplies \Psi$ is shorthand for $\Psi \myor \mynot \Phi$, and
$\Phi \myiff \Psi$ is shorthand for $(\Phi \myimplies \Psi) \myand
(\Psi \myimplies \Phi)$.  If $x$ is a variable and $\Phi(x)$ is a
formula in which the free variable (unbound by quantification) $x$
occurs, then $(\exists!x)\Phi(x)$ is shorthand for $(\exists
x)(\Phi(x) \myand (\forall y)(\Phi(y) \myimplies x = y))$.  We
shall be making several such abbreviations as we go along, and
usually leave it the reader to convince himself that the intended
meaning can be achieved using only the primitive symbols of our
language.

We say that a first-order formula is a \textbf{sentence} if it has
no free variables. A (perhaps infinite) collection of sentences in
a given first-order language is called a \textbf{theory}.  If $M$
is a relational structure and $\Phi = \{\phi_i : i \in I\}$ is a
theory, we say that $M$ is a \index{model} \textbf{model} of
$\Phi$ if every sentence of $\Phi$ is true in the structure $M$.
Obviously not all collections of sentences have models; $\{ 1=0,
\mynot(1=0) \}$ obviously has no model, whatever you interpret $0$
and $1$ to be.

\begin{thm} \index{compactness theorem} (Compactness theorem for first-order logic)
Let $\Phi$ be a collection of first-order sentences such that
every finite subset of $\Phi$ has a model.  Then $\Phi$ has a
model.
\end{thm}

\begin{proof}
See theorem 5.1.1 of \cite{hodges}.
\end{proof}

Some oft used corollaries:

\begin{prop}
\label{compactnesscorollaries} The following are all corollaries
of the compactness theorem.

\begin{enumerate}
\item{If the first-order sentence $\phi$ is equivalent to the
infinite conjunction of the first-order sentences $\{\psi_i:i \in
I\}$, then $\phi$ is equivalent to some finite conjunction of
them.} \item{If the first-order sentence $\phi$ is implied by the
infinite conjunction of the first-order sentences $\{\psi_i:i \in
I\}$, then $\phi$ is implied by some finite conjunction of them.}
\end{enumerate}
\end{prop}

As an easy example of an application of compactness, let
$+,\multsymbol,-,0,1$ be the language of fields. See proposition
\ref{ultraprodoffieldsisafieldprop} for the fairly obvious
observation that ``is a field'' is expressible by a first-order
sentence in this language.

\begin{prop} Let $L$ be the language of fields.
\begin{enumerate}
\item{The statement ``has characteristic zero'', modulo the theory
of fields, is not expressible by a first-order sentence of $L$.}
\item{If $\phi$ is a first-order $L$-sentence which is true of
every characteristic zero field, then $\phi$ is true for all
fields of sufficiently large positive characteristic.}
\end{enumerate}
\end{prop}

\begin{proof}
For a fixed prime number $p$, define $\text{char}_p$ to be the
first-order sentence $1+1+ \ldots + 1 = 0$ ($p$-occurrences of
$1$).  Modulo the theory of fields, this is obviously equivalent
to the assertion that the field is of characteristic $p$. Now the
statement ``is of characteristic zero'' is by definition
equivalent to the infinite conjunction of the sentences $\mynot
\text{char}_p$ for $p = 2,3,5,\ldots$.  By 1. of proposition
\ref{compactnesscorollaries}, if this were expressible as a
first-order sentence, it would be equivalent to some finite subset
of this collection.  But we know this is absurd; no finite
collection of the sentences $\mynot \text{char}_p$ can guarantee a
field to be of characteristic zero. We conclude that ``is of
characteristic zero'' is not first-order.

Now suppose that the first-order sentence $\phi$ were true in
every characteristic zero field.  This means that the infinite
conjunction of the sentences $\mynot \text{char}_p$ implies
$\phi$. By 2. of proposition \ref{compactnesscorollaries}, $\phi$
is implied by some finite subset of them.  Any field of large
enough positive characteristic satisfies this finite collection of
sentences, and hence satisfies $\phi$ as well.
\end{proof}


\chapter{Ultrafilters}

The notion of a filter is sometimes given as a slightly more
general definition then we give, but it suffices for our purposes.

\begin{defn}
\label{defnFilter} Let $I$ be a set.  A \textbf{filter} on $I$ is
a non-empty collection $\filtfont{F}$ of subsets of $I$ satisfying

\begin{enumerate}
\item{$\filtfont{F}$ is closed under the taking of pairwise
intersections} \item{If $Y$ is a superset of some element of
$\filtfont{F}$, then $Y$ is in $\filtfont{F}$} \item{The empty set
is not in $\filtfont{F}$}

\end{enumerate}

A filter is called an \index{ultrafilter} \textbf{ultrafilter} if
it is maximal with respect to inclusion among all filters.  An
ultrafilter is called \textbf{principal} if it is of the form $\{X
\subset I : x \in X\}$ for some element $x$ of $I$.

\end{defn}

We sometimes call the elements of a filter \textbf{large} sets. If
$\phi(i)$ is some statement about elements of $I$ we say that
$\phi$ holds \textbf{almost everywhere} or \textbf{for almost
every $i$} if the set on which $\phi(i)$ is true is large.

\begin{prop}
\label{ultrafilterlemma} A filter $\filtfont{F}$ on $I$ is an
ultrafilter if and only if for any subset $X$ of $I$, either $X$
or its complement is in $\filtfont{F}$.
\end{prop}

\begin{proof}
Suppose $\filtfont{F}$ does not contain $X$ or its complement; we
claim that $\filtfont{F}$ can be enlarged to a new filter
containing one or the other. First suppose that $R,S \in
\filtfont{F}$ are such that $X \cap R = \nothing$ and $X^C \cap S
= \nothing$. Then $R \subset X^C$ and $S \subset X$, whence $R
\cap S = \nothing$; but this cannot happen since $\filtfont{F}$ is
closed under intersections and does not contain $\nothing$. Thus
at least one of $X$ or its complement is not disjoint from
anything in $\filtfont{F}$, let's say $X$.  Then define
$\filtfont{F}^\prime = \filtfont{F} \cup \{S \subset I: S \supset
X\} \cup \{S \cap R:R \in \filtfont{F}, S \supset X \}$, which is
easily seen to be a new filter properly containing $\filtfont{F}$.

Conversely, if $\filtfont{F}$ contains every set or its
complement, then it is necessarily maximal, since there are no new
sets we can throw in; any such $X$ would intersect with its
complement to arrive at $\nothing \in \filtfont{F}$.
\end{proof}

Examples: The collection of all subsets having Lebesgue measure
$1$ is a filter on the interval $[0,1]$, and the collection of all
cofinite subsets is a filter on $\mathbb{N}$.  These are both
obviously non-principal and non-ultra.

Principal ultrafilters are boring and useless; we need
non-principal ultrafilters.

\begin{prop}
\label{singletonlemma} An ultrafilter \index{$\filtfont{U}$}
$\filtfont{U}$ is non-principal if and only if it contains no
finite sets if and only if it contains no singleton sets.
\end{prop}

\begin{proof}
If $\filtfont{U}$ is principal, say generated by $x \in I$, then
obviously $\filtfont{U}$ contains the singleton set $\{x\}$.
Conversely, suppose $\filtfont{U}$ contains the finite set
$X=\{x_1, \ldots, x_n\}$, $n>1$.  Then at least one of the sets
$\{x_1, \ldots,x_{n-1}\}$ or its complement is in $\filtfont{U}$.
In the latter case we intersect with $X$ to obtain $\{x_n\} \in
\filtfont{U}$, and in either case we have a new subset with less
than $n$ elements.  Applying this process finitely many times will
eventually yield some singleton $\{x\}$ in $\filtfont{U}$.  Then
any subset containing $x$ is in $\filtfont{U}$, no subset not
containing $x$ can be in $\filtfont{U}$, and thus $\filtfont{U}$
is principal.
\end{proof}

\begin{prop}
Let $I$ be an infinite set, and $X \subset I$ any infinite subset
of $I$.  Then there exists a non-principal ultrafilter on $I$
containing $X$.
\end{prop}

\begin{proof}
Let $\filtfont{C}$ be the filter on $I$ consisting of all cofinite
sets, and enlarge it, as in the proof of proposition
\ref{ultrafilterlemma}, to contain $X$.  Partially order the
collection of all filters containing $X$ by inclusion, which we
just showed is non-empty.  The union over any chain of filters
qualifies as an upper bound for that chain; take a maximal element
by Zorn's Lemma.  It is guaranteed to be non-principal since it
contains no finite sets, by proposition \ref{singletonlemma}.
\end{proof}

\begin{cor}
\label{cofinitelemma} The subsets of $I$ that are contained in
every non-principal ultrafilter are exactly the cofinite subsets.
\end{cor}

\begin{proof}
If $X$ is not cofinite, the previous proposition shows that its
complement is contained in some non-principal ultrafilter,
necessarily not containing $X$.  If $X$ is cofinite, then
proposition \ref{singletonlemma} shows that every non-principal
ultrafilter does not contain $X^C$, and so contains $X$.
\end{proof}

\begin{lem}
\label{coveringlemma} Let $\filtfont{U}$ be an ultrafilter on $I$,
$J$ a member of $\filtfont{U}$, and $X_1, \ldots, X_n$ a finite
collection of subsets of $I$ which cover $J$.  Then at least one
of the $X_i$ is in $\filtfont{U}$.
\end{lem}

\begin{proof}
Suppose none of them are in $\filtfont{U}$.  Then all of their
complements are in $\filtfont{U}$, as well as the intersection of
their complements, which is contained in $J^C$; but this cannot
be, since $J^C \notin \filtfont{U}$.
\end{proof}

\begin{lem}
\label{coloringlemma} If $\filtfont{U}$ is an ultrafilter on $I$,
$J$ a member of $\filtfont{U}$, and $X_1, \ldots, X_n$ a finite
disjoint partition of $J$, then exactly one of the $X_i$ is
contained in $\filtfont{U}$.
\end{lem}

\begin{proof}
At least one of them is in $\filtfont{U}$ by the previous lemma,
and no two of them can be, lest we take their intersection and
arrive at $\nothing \in \filtfont{U}$.
\end{proof}

With a view towards defining ultraproducts in the next section, we
close with

\begin{prop}
\label{ultraprodequivalencelemma} Let $X_i$ be a collection of
sets indexed by $I$, $\filtfont{U}$ an ultrafilter on $I$.  Define
a relation on $\prod_{i \in I} X_i$ (cartesian product of the
$X_i$) as follows. For tuples $(x_i),(y_i) \in \prod_{i \in I}
X_i$, $(x_i) \sim (y_i)$ if and only if the set $\{i \in I: x_i =
y_i \}$ is in $\filtfont{U}$. Then $\sim$ is an equivalence
relation.
\end{prop}

\begin{proof}
Reflexivity is clear since necessarily $I \in \filtfont{F}$, and
symmetry is obvious.  For transitivity, suppose $(x_i) \sim (y_i)$
and $(y_i) \sim (z_i)$.  Then the set
\[ \{ i \in I : x_i = z_i \} \]
contains at least the set
\[ \{i \in I: x_i = y_i \} \cap \{i \in I: y_i = z_i \} \]
which is in $\filtfont{U}$ by intersection closure.  Then so is
$\{ i \in I : x_i = z_i \}$, by superset closure.
\end{proof}

We say that two such tuples are equal \textbf{almost everywhere}
or \textbf{on a large set} if they are related through this
relation, and we denote by $[x_i]$ the equivalence class of the
tuple $(x_i)$.


\chapter{Ultraproducts}

\index{ultraproduct}

Let $M_i$ be a collection of relational structures in a common
signature $L$, indexed by the set $I$, and fix a non-principal
ultrafilter $\filtfont{U}$ on $I$. Then we define the
\textbf{ultraproduct} of these structures relative to
$\filtfont{U}$,  denoted \index{$\uprod M_i$} $M = \uprod M_i$, to
be a new $L$-structure defined as follows.

The domain of $M$ is the collection of all equivalence classes
$[x_i]$ of tuples $(x_i) \in \prod_{i \in I} M_i$ (cartesian
product of the $M_i$) as defined in proposition
\ref{ultraprodequivalencelemma}. For an $n$-ary relation symbol
$r$, we define $r([x_i]_1, \ldots, [x_i]_n)$ to hold if and only
if, for almost every $i$, $r(x_{i,1},\ldots,x_{i,n})$ holds in the
structure $M_i$.  For an $n$-ary function symbol $f$,
$f([x_i]_1,\ldots,[x_i]_n)$ is the element
$[f(x_{i,1},\ldots,x_{i,n})]$ of $M$, and the constant $c$
corresponds to the element $[x_i]$, where $x_i$ is the element of
$X_i$ corresponding to the constant $c$.

\begin{prop}
For any ultrafilter $\filtfont{U}$, the definition just given for
$M = \uprod M_i$ is well-defined.
\end{prop}

\begin{proof}
We must show that the definitions given are independent of the
choice of tuple $(x_i)$ one uses to represent the equivalence
class $[x_i]$. Suppose then that $(x_i)_1 \sim (y_i)_1, \ldots ,
(x_i)_n \sim (y_i)_n$, with $x_{i,1} = y_{i,1}$ holding on the
large set $J_1$, similarly for $J_2, \ldots, J_n$.  Let $r$ be an
$n$-ary relational symbol, and suppose that the relation
$r(x_{i,1},\ldots,x_{i,n})$ holds for almost every $i$, say on the
large set $J \subset I$.  Then the relation
$r(y_{i,1},\ldots,y_{i,n})$ holds at least on the set $J_1 \cap
\ldots \cap J_n \cap J$, which is large by intersection closure.
Thus deciding if $r([x_i]_1, \ldots, [x_i]_n)$ holds in $M$ is
independent of the choice of representatives. Identical arguments
hold for function and constant symbols.
\end{proof}

Since we are assumed to be working over an ultrafilter, we can say
something stronger:

\begin{prop}
The relation $r([x_i]_1,\ldots,[x_i]_n)$ does not hold in $M$ if
and only if, for almost every $i$, $r(x_{i,1},\ldots,x_{i,n})$
does not hold in $M_i$.
\end{prop}

\begin{proof}
The `if' direction is true even in a non-ultra filter.  For the
converse, If $r([x_i]_1, \ldots, [x_i]_n)$ does not hold, it is
because the set on which $r(x_{i,1},\ldots,x_{i,n})$ holds is not
large. Then as $\filtfont{U}$ is an ultrafilter, its complement is
large, namely the set on which $r(x_{i,1},\ldots,x_{i,n})$ does
not hold.
\end{proof}

This is the reason we demand our filters to be ultra; otherwise
$M$ preserves the primitive relations $r_i$, but not necessarily
their negations. The reason we demand our ultrafilters to be
non-principal is because

\begin{prop}
If $\filtfont{U}$ is a principal ultrafilter, say generated by $j
\in I$, then $M$ is isomorphic to $M_j$.
\end{prop}

\begin{proof}
Two tuples $(x_i),(y_i)$ are then equivalent if and only if the
set on which they are equal contains $j$, if and only if $x_j =
y_j$.  The map $[x_i] \mapsto x_j$ is thus easily seen to be an
isomorphism of $L$-structures, preserving all relations and
whatnot.
\end{proof}

The `fundamental theorem of ultraproducts', what makes them worth
studying at all, would have to be

\begin{thm}\index{Los' theorem}($\L$os' Theorem)
\label{Los'sTheorem} Let $\filtfont{U}$ be an ultrafilter on $I$,
$M = \uprod M_i$ the ultraproduct of the structures $M_i$ with
respect to $\filtfont{U}$. Let $\Phi(x_1, \ldots, x_n)$ be a
first-order formula in the language $L$ in the variables $x_1,
\ldots, x_n$, and let $[a_i]_1, \ldots, [a_i]_n$ be a collection
of elements of $M$. Then $\Phi([a_i]_1, \ldots, [a_i]_n)$ is true
of $M$ if and only if $\Phi(a_{i,1}, \ldots, a_{i,n})$ is true of
$M_i$ for almost every $i$.
\end{thm}

\begin{proof}
See theorem 8.5.3 of \cite{hodges}.
\end{proof}

We've proved this theorem already in the case of atomic formulae
or their negations.  The rest of the proof proceeds by induction
on the complexity (i.e.~length) of the formula.  For example, if
the theorem is true for the formulae $\Phi(x_1, \ldots, x_n)$ and
$\Psi(y_1, \ldots, y_m)$ then it is also true for their
conjunction, by considering the intersection of two large sets,
which is also large.

\begin{cor}
\label{loscor} If $\Phi$ is a first-order statement in the
language $L$, then $\Phi$ is true of $\uprod M_i$ if and only if
it is true of almost every $i$.
\end{cor}

\begin{proof}
Sentences are just a particular type of formulae; apply $\L$os'
theorem.
\end{proof}

Ultimately, we are not particularly interested in what sorts of
statements might hold in $\uprod M_i$ for a particular choice of
non-principal ultrafilter, but rather those first-order statements
that hold for \emph{every} non-principal ultrafilter.

\begin{prop}
Let $\Phi$ be a first-order statement that holds in $\uprod M_i$
for every choice of non-principal ultrafilter on $I$.  Then $\Phi$
holds in $M_i$ for all but finitely many $i \in I$.
\end{prop}

\begin{proof}
Apply corollary \ref{loscor} and corollary \ref{cofinitelemma}.
\end{proof}

\addcontentsline{toc}{chapter}{Index}

\printindex

\end{document} 